%% file: ComplexGinzburgLandauEq.tex
\documentclass[leqno]{article}

\input{0000_preamble}

\title{Stochastic complex Ginzburg-Landau equation with space-time white noise
}
\author{Masato Hoshino, Yuzuru Inahama and Nobuaki Naganuma}
\date{\empty}

\begin{document}
\maketitle

\begin{abstract}
We study the stochastic cubic complex Ginzburg-Landau equation
with complex-valued space-time white noise on the three dimensional torus.
This nonlinear equation is so singular that it can only be understood in a renormalized sense.
In the first half of this paper we prove local well-posedness of this equation
in the framework of regularity structure theory.
In the latter half we prove local well-posedness
in the framework of paracontrolled distribution theory.
\end{abstract}



\tableofcontents

\input{0100_introduction}
\input{0330_RegStrDet}
\input{0380_RegStrProb}
\input{0550_WellposednessOfCGLByPara}
\input{0560_ConvOfDrivers}
\appendix
\input{0700_ItoWienerInt}

\section*{Acknowledgement}
The authors thank Professor Reika Fukuizumi of Tohoku University
for stimulating discussions and valuable advice.
They also thank Professor Makoto Katori
of Chuo University for helpful comments.

The first named author was partially supported by JSPS KAKENHI Grant Number JP16J03010.
The second named author was partially supported by JSPS KAKENHI Grant Number JP15K04922.
The third named author was partially supported by JSPS KAKENHI Grant Number JP17K14202.


\bigskip
\address{Masato Hoshino}
{
School of Fundamental Science and Engineering\\
Waseda University\\
Okubo, Shinjuku-ku, Tokyo, 169-8555\\
Japan
}
{hoshino@ms.u-tokyo.ac.jp}

\address{Yuzuru Inahama}
{
Graduate School of Mathematics\\
Kyushu University\\
Motooka, Nishi-ku, Fukuoka, 819-0395\\
Japan
}
{inahama@math.kyushu-u.ac.jp}

\address{Nobuaki Naganuma}
{
Graduate School of Engineering Science\\
Osaka University\\
Machikaneyama, Toyonaka, Osaka, 560-8531\\
Japan
}
{naganuma@sigmath.es.osaka-u.ac.jp}

\end{document}

%% file: 0000_preamble.tex
\usepackage{amsmath}
\usepackage{amssymb}
\usepackage{amsthm}
\usepackage{enumerate}
\usepackage{hyperref}
\usepackage{stmaryrd}
\usepackage{tikz}
\usetikzlibrary{intersections, calc}
\usepackage{multirow}
\usepackage{longtable}

\allowdisplaybreaks

\makeatletter
	
	\@addtoreset{equation}{section}
\makeatother

\makeatletter
\def\address#1#2#3{\begingroup
\noindent\parbox[t]{7.8cm}{%
\small{\scshape\ignorespaces#1}\par\vskip1ex
\small{\scshape\ignorespaces#2}\par\vskip1ex
\noindent\small{\itshape E-mail address}%
\/: #3\par\vskip4ex}\hfill%
\endgroup}%
\makeatother

\newtheorem{theorem}{Theorem}[section]
\newtheorem{proposition}[theorem]{Proposition}
\newtheorem{lemma}[theorem]{Lemma}

\newtheorem{assumption}[theorem]{Assumption}
\theoremstyle{definition}
\newtheorem{definition}[theorem]{Definition}
\newtheorem{remark}[theorem]{Remark}

\newcommand{\tref}[2][]
	{\ifx#1""{Theorem~\textup{\ref{#2}}}%
	\else{Theorems~\textup{\ref{#1}} and~\textup{\ref{#2}}}\fi}
\newcommand{\trefs}[4][]
	{\ifx#1""{Theorems~\textup{\ref{#2}},~\textup{\ref{#3}} and~\textup{\ref{#4}}}%
	\else{Theorems~\textup{\ref{#1}},~\textup{\ref{#2}},~\textup{\ref{#3}} and~\textup{\ref{#4}}}\fi}
\newcommand{\pref}[2][]
	{\ifx#1""{Proposition~\textup{\ref{#2}}}%
	\else{Propositions~\textup{\ref{#1}} and~\textup{\ref{#2}}}\fi}
\newcommand{\prefs}[4][]
	{\ifx#1""{Propositions~\textup{\ref{#2}},~\textup{\ref{#3}} and~\textup{\ref{#4}}}%
	\else{Propositions~\textup{\ref{#1}},~\textup{\ref{#2}},~\textup{\ref{#3}} and~\textup{\ref{#4}}}\fi}
\newcommand{\lref}[2][]
	{\ifx#1""{Lemma~\textup{\ref{#2}}}%
	\else{Lemmas~\textup{\ref{#1}} and \textup{\ref{#2}}}\fi}
\newcommand{\lrefs}[4][]
	{\ifx#1""{Lemma~\textup{\ref{#2}}, \textup{\ref{#3}} and~\textup{\ref{#4}}}%
	\else{Lemmas~\textup{\ref{#1}}, \textup{\ref{#2}}, \textup{\ref{#3}} and \textup{\ref{#4}}}\fi}
\newcommand{\rref}[1]{Remark~\textup{\ref{#1}}}
\newcommand{\tblref}[1]{Table~\textup{\ref{#1}}}
\newcommand{\aref}[1]{Assumption~\textup{\ref{#1}}}

\newcommand{\secref}[2][]
	{\ifx#1""{Section~\textup{\ref{#2}}}%
	\else{Sections~\textup{\ref{#1}} and \textup{\ref{#2}}}\fi}

\newcommand{\CmplNum}{\mathbf{C}}
\newcommand{\RealNum}{\mathbf{R}}
\newcommand{\Integers}{\mathbf{Z}}
\newcommand{\NaturalNum}{\mathbf{N}}
\newcommand{\Torus}{\mathbf{T}}

\newcommand{\poly}{\mathrm{poly}}
\newcommand{\cgl}{\mathrm{cgl}}
\renewcommand{\$}{|\!|\!|}
\newcommand{\bbb}{|\hspace{-2.5pt}|\hspace{-2.5pt}|\hspace{-2.5pt}|\hspace{-2.5pt}|\hspace{-2.5pt}|}

\newcommand{\unit}{\mathbf{1}}
\newcommand{\id}{\textrm{Id}}

\newcommand{\drawC}[2]{
\draw[very thick] #1 -- #2;
\filldraw #1 circle (0.8pt);
\filldraw #2 circle (1.6pt);}

\newcommand{\drawD}[2]{
\draw[thin] #1 circle (1pt);
\draw[thin,double] #1 -- #2;
\filldraw[white] #1 circle (0.2pt);
\draw[draw=black,fill=white, thin] #2 circle (1.6pt);}

\newcommand{\drawI}[2]{
\draw[very thick,black] #1 -- #2;
\filldraw[black] #1 circle (0.8pt);
\filldraw[black] #2 circle (0.8pt);}

\newcommand{\drawJ}[2]{
\filldraw[black] #1 circle (1pt);
\filldraw[black] #1 circle (1pt);
\draw[thin,double] #1 -- #2;
\filldraw[white] #1 circle (0.2pt);
\filldraw[white] #2 circle (0.2pt);}

\newsavebox{\boxCCD}
\sbox{\boxCCD}{\,\begin{tikzpicture}[baseline=0pt, scale=6/12]
\coordinate (O1) at (0,0);
\coordinate (X1) at (-6pt,11pt);
\coordinate (X2) at (+0pt,12pt);
\coordinate (X3) at (+6pt,11pt);
\drawC{(O1)}{(X1)}
\drawC{(O1)}{(X2)}
\drawD{(O1)}{(X3)}
\end{tikzpicture}\,}
\newcommand{\CCD}{\usebox{\boxCCD}}

\newsavebox{\boxCC}
\sbox{\boxCC}{\,\begin{tikzpicture}[baseline=0pt, scale=6/12]
\coordinate (O1) at (0,0);
\coordinate (X1) at (-4pt,12pt);
\coordinate (X2) at (+4pt,12pt);
\drawC{(O1)}{(X1)}
\drawC{(O1)}{(X2)}
\end{tikzpicture}\,}
\newcommand{\CC}{\usebox{\boxCC}}

\newsavebox{\boxCD}
\sbox{\boxCD}{\,\begin{tikzpicture}[baseline=0pt, scale=6/12]
\coordinate (O1) at (0,0);
\coordinate (X1) at (-4pt,12pt);
\coordinate (X2) at (+4pt,12pt);
\drawC{(O1)}{(X1)}
\drawD{(O1)}{(X2)}
\end{tikzpicture}\,}
\newcommand{\CD}{\usebox{\boxCD}}

\newsavebox{\boxIX}
\sbox{\boxIX}{\,\begin{tikzpicture}[baseline=0pt, scale=6/12]
\coordinate (O1) at (0,0);
\coordinate (X1) at (0pt,12pt);
\drawC{(O1)}{(X1)}
\end{tikzpicture}\,}
\newcommand{\IX}{\usebox{\boxIX}}

\newsavebox{\boxJX}
\sbox{\boxJX}{\,\begin{tikzpicture}[baseline=0pt, scale=6/12]
\coordinate (O1) at (0,0);
\coordinate (X1) at (0pt,12pt);
\drawD{(O1)}{(X1)}
\end{tikzpicture}\,}
\newcommand{\JX}{\usebox{\boxJX}}

\newsavebox{\boxCDICC}
\sbox{\boxCDICC}{\,\begin{tikzpicture}[baseline=1pt, scale=6/12]
\coordinate (O1) at (0pt,0pt);
\coordinate (O2) at (0pt,7pt);
\coordinate (X1) at (-3pt,15pt);
\coordinate (X3) at (+3pt,15pt);
\coordinate (Y1) at (-5pt,6pt);
\coordinate (Y2) at (+5pt,6pt);
\drawC{(O2)}{(X1)}
\drawC{(O2)}{(X3)}
\drawI{(O1)}{(O2)}
\drawC{(O1)}{(Y1)}
\drawD{(O1)}{(Y2)}
\end{tikzpicture}\,}
\newcommand{\CDICC}{\usebox{\boxCDICC}}

\newsavebox{\boxCCJDD}
\sbox{\boxCCJDD}{\,\begin{tikzpicture}[baseline=1pt, scale=6/12]
\coordinate (O1) at (0pt,0pt);
\coordinate (O2) at (0pt,7pt);
\coordinate (X1) at (-3pt,15pt);
\coordinate (X3) at (+3pt,15pt);
\coordinate (Y1) at (-5pt,6pt);
\coordinate (Y2) at (+5pt,6pt);
\drawD{(O2)}{(X1)}
\drawD{(O2)}{(X3)}
\drawJ{(O1)}{(O2)}
\drawC{(O1)}{(Y1)}
\drawC{(O1)}{(Y2)}
\end{tikzpicture}\,}
\newcommand{\CCJDD}{\usebox{\boxCCJDD}}

\newsavebox{\boxCDICD}
\sbox{\boxCDICD}{\,\begin{tikzpicture}[baseline=1pt, scale=6/12]
\coordinate (O1) at (0pt,0pt);
\coordinate (O2) at (0pt,7pt);
\coordinate (X1) at (-3pt,15pt);
\coordinate (X3) at (+3pt,15pt);
\coordinate (Y1) at (-5pt,6pt);
\coordinate (Y2) at (+5pt,6pt);
\drawC{(O2)}{(X1)}
\drawD{(O2)}{(X3)}
\drawI{(O1)}{(O2)}
\drawC{(O1)}{(Y1)}
\drawD{(O1)}{(Y2)}
\end{tikzpicture}\,}
\newcommand{\CDICD}{\usebox{\boxCDICD}}

\newsavebox{\boxCCJDC}
\sbox{\boxCCJDC}{\,\begin{tikzpicture}[baseline=1pt, scale=6/12]
\coordinate (O1) at (0pt,0pt);
\coordinate (O2) at (0pt,7pt);
\coordinate (X1) at (-3pt,15pt);
\coordinate (X3) at (+3pt,15pt);
\coordinate (Y1) at (-5pt,6pt);
\coordinate (Y2) at (+5pt,6pt);
\drawD{(O2)}{(X1)}
\drawC{(O2)}{(X3)}
\drawJ{(O1)}{(O2)}
\drawC{(O1)}{(Y1)}
\drawC{(O1)}{(Y2)}
\end{tikzpicture}\,}
\newcommand{\CCJDC}{\usebox{\boxCCJDC}}

\newsavebox{\boxCDICCD}
\sbox{\boxCDICCD}{\,\begin{tikzpicture}[baseline=1pt, scale=6/12]
\coordinate (O1) at (0pt,0pt);
\coordinate (O2) at (0pt,7pt);
\coordinate (X1) at (-5pt,14pt);
\coordinate (X2) at (+0pt,15pt);
\coordinate (X3) at (+5pt,14pt);
\coordinate (Y1) at (-7pt,6pt);
\coordinate (Y2) at (+7pt,6pt);
\drawC{(O2)}{(X1)}
\drawC{(O2)}{(X2)}
\drawD{(O2)}{(X3)}
\drawI{(O1)}{(O2)}
\drawC{(O1)}{(Y1)}
\drawD{(O1)}{(Y2)}
\end{tikzpicture}\,}
\newcommand{\CDICCD}{\usebox{\boxCDICCD}}

\newsavebox{\boxCCJDDC}
\sbox{\boxCCJDDC}{\,\begin{tikzpicture}[baseline=1pt, scale=6/12]
\coordinate (O1) at (0pt,0pt);
\coordinate (O2) at (0pt,7pt);
\coordinate (X1) at (-5pt,14pt);
\coordinate (X2) at (+0pt,15pt);
\coordinate (X3) at (+5pt,14pt);
\coordinate (Y1) at (-7pt,6pt);
\coordinate (Y2) at (+7pt,6pt);
\drawD{(O2)}{(X1)}
\drawD{(O2)}{(X2)}
\drawC{(O2)}{(X3)}
\drawJ{(O1)}{(O2)}
\drawC{(O1)}{(Y1)}
\drawC{(O1)}{(Y2)}
\end{tikzpicture}\,}
\newcommand{\CCJDDC}{\usebox{\boxCCJDDC}}

\newsavebox{\boxCICCD}
\sbox{\boxCICCD}{\,\begin{tikzpicture}[baseline=1pt, scale=6/12]
\coordinate (O1) at (0pt,0pt);
\coordinate (O2) at (0pt,7pt);
\coordinate (X1) at (-5pt,14pt);
\coordinate (X2) at (+0pt,15pt);
\coordinate (X3) at (+5pt,14pt);
\coordinate (Y2) at (+7pt,6pt);
\drawC{(O2)}{(X1)}
\drawC{(O2)}{(X2)}
\drawD{(O2)}{(X3)}
\drawI{(O1)}{(O2)}
\drawC{(O1)}{(Y2)}
\end{tikzpicture}\,}
\newcommand{\CICCD}{\usebox{\boxCICCD}}

\newsavebox{\boxDICCD}
\sbox{\boxDICCD}{\,\begin{tikzpicture}[baseline=1pt, scale=6/12]
\coordinate (O1) at (0pt,0pt);
\coordinate (O2) at (0pt,7pt);
\coordinate (X1) at (-5pt,14pt);
\coordinate (X2) at (+0pt,15pt);
\coordinate (X3) at (+5pt,14pt);
\coordinate (Y2) at (+7pt,6pt);
\drawC{(O2)}{(X1)}
\drawC{(O2)}{(X2)}
\drawD{(O2)}{(X3)}
\drawI{(O1)}{(O2)}
\drawD{(O1)}{(Y2)}
\end{tikzpicture}\,}
\newcommand{\DICCD}{\usebox{\boxDICCD}}

\newsavebox{\boxCJDDC}
\sbox{\boxCJDDC}{\,\begin{tikzpicture}[baseline=1pt, scale=6/12]
\coordinate (O1) at (0pt,0pt);
\coordinate (O2) at (0pt,7pt);
\coordinate (X1) at (-5pt,14pt);
\coordinate (X2) at (+0pt,15pt);
\coordinate (X3) at (+5pt,14pt);
\coordinate (Y2) at (+7pt,6pt);
\drawD{(O2)}{(X1)}
\drawD{(O2)}{(X2)}
\drawC{(O2)}{(X3)}
\drawJ{(O1)}{(O2)}
\drawC{(O1)}{(Y2)}
\end{tikzpicture}\,}
\newcommand{\CJDDC}{\usebox{\boxCJDDC}}

\newcommand{\ICCD}{{\,
\begin{tikzpicture}[baseline=1pt, scale=6/12]
\coordinate (O1) at (0pt,0pt);
\coordinate (O2) at (0pt,7pt);
\coordinate (X1) at (-5pt,14pt);
\coordinate (X2) at (+0pt,15pt);
\coordinate (X3) at (+5pt,14pt);
\drawC{(O2)}{(X1)}
\drawC{(O2)}{(X2)}
\drawD{(O2)}{(X3)}
\drawI{(O1)}{(O2)}
\end{tikzpicture}\,}}

\newcommand{\JDDC}{{\,
\begin{tikzpicture}[baseline=1pt, scale=6/12]
\coordinate (O1) at (0pt,0pt);
\coordinate (O2) at (0pt,7pt);
\coordinate (X1) at (-5pt,14pt);
\coordinate (X2) at (+0pt,15pt);
\coordinate (X3) at (+5pt,14pt);
\drawD{(O2)}{(X1)}
\drawD{(O2)}{(X2)}
\drawC{(O2)}{(X3)}
\drawJ{(O1)}{(O2)}
\end{tikzpicture}\,}}

\newcommand{\IC}{{\,
\begin{tikzpicture}[baseline=1pt, scale=6/12]
\coordinate (O1) at (0pt,0pt);
\coordinate (O2) at (0pt,8pt);
\coordinate (X3) at (+4pt,14pt);
\drawC{(O2)}{(X3)}
\drawI{(O1)}{(O2)}
\end{tikzpicture}\,}}

\newcommand{\JD}{{\,
\begin{tikzpicture}[baseline=1pt, scale=6/12]
\coordinate (O1) at (0pt,0pt);
\coordinate (O2) at (0pt,8pt);
\coordinate (X3) at (+4pt,14pt);
\drawD{(O2)}{(X3)}
\drawJ{(O1)}{(O2)}
\end{tikzpicture}\,}}

\newcommand{\CIC}{{\,
\begin{tikzpicture}[baseline=1pt, scale=6/12]
\coordinate (O1) at (0pt,0pt);
\coordinate (O2) at (0pt,8pt);
\coordinate (X3) at (+4pt,14pt);
\coordinate (Y2) at (+6pt,6pt);
\drawC{(O2)}{(X3)}
\drawI{(O1)}{(O2)}
\drawC{(O1)}{(Y2)}
\end{tikzpicture}\,}}

\newcommand{\DIC}{{\,
\begin{tikzpicture}[baseline=1pt, scale=6/12]
\coordinate (O1) at (0pt,0pt);
\coordinate (O2) at (0pt,8pt);
\coordinate (X3) at (+4pt,14pt);
\coordinate (Y2) at (+6pt,6pt);
\drawC{(O2)}{(X3)}
\drawI{(O1)}{(O2)}
\drawD{(O1)}{(Y2)}
\end{tikzpicture}\,}}

\newcommand{\CJD}{{\,
\begin{tikzpicture}[baseline=1pt, scale=6/12]
\coordinate (O1) at (0pt,0pt);
\coordinate (O2) at (0pt,8pt);
\coordinate (X3) at (+4pt,14pt);
\coordinate (Y2) at (+6pt,6pt);
\drawD{(O2)}{(X3)}
\drawJ{(O1)}{(O2)}
\drawC{(O1)}{(Y2)}
\end{tikzpicture}\,}}

\newcommand{\CDIC}{{\,
\begin{tikzpicture}[baseline=1pt, scale=6/12]
\coordinate (O1) at (0pt,0pt);
\coordinate (O2) at (0pt,8pt);
\coordinate (X3) at (+3pt,14pt);
\coordinate (Y1) at (-5pt,6pt);
\coordinate (Y2) at (+5pt,6pt);
\drawC{(O2)}{(X3)}
\drawI{(O1)}{(O2)}
\drawC{(O1)}{(Y1)}
\drawD{(O1)}{(Y2)}
\end{tikzpicture}\,}}

\newcommand{\CCJD}{{\,
\begin{tikzpicture}[baseline=1pt, scale=6/12]
\coordinate (O1) at (0pt,0pt);
\coordinate (O2) at (0pt,8pt);
\coordinate (X3) at (+3pt,14pt);
\coordinate (Y1) at (-5pt,6pt);
\coordinate (Y2) at (+5pt,6pt);
\drawD{(O2)}{(X3)}
\drawJ{(O1)}{(O2)}
\drawC{(O1)}{(Y1)}
\drawC{(O1)}{(Y2)}
\end{tikzpicture}\,}}

\newcommand{\ICC}{{\,
\begin{tikzpicture}[baseline=1pt, scale=6/12]
\coordinate (O1) at (0pt,0pt);
\coordinate (O2) at (0pt,7pt);
\coordinate (X1) at (-3pt,15pt);
\coordinate (X3) at (+3pt,15pt);
\drawC{(O2)}{(X1)}
\drawC{(O2)}{(X3)}
\drawI{(O1)}{(O2)}
\end{tikzpicture}\,}}

\newcommand{\JDD}{{\,
\begin{tikzpicture}[baseline=1pt, scale=6/12]
\coordinate (O1) at (0pt,0pt);
\coordinate (O2) at (0pt,7pt);
\coordinate (X1) at (-3pt,15pt);
\coordinate (X3) at (+3pt,15pt);
\drawD{(O2)}{(X1)}
\drawD{(O2)}{(X3)}
\drawJ{(O1)}{(O2)}
\end{tikzpicture}\,}}

\newcommand{\ICD}{{\,
\begin{tikzpicture}[baseline=1pt, scale=6/12]
\coordinate (O1) at (0pt,0pt);
\coordinate (O2) at (0pt,7pt);
\coordinate (X1) at (-3pt,15pt);
\coordinate (X3) at (+3pt,15pt);
\drawC{(O2)}{(X1)}
\drawD{(O2)}{(X3)}
\drawI{(O1)}{(O2)}
\end{tikzpicture}\,}}

\newcommand{\DHE}{{\,
\begin{tikzpicture}[baseline=1pt, scale=6/12]
\coordinate (O1) at (0pt,0pt);
\coordinate (O2) at (3pt,7pt);
\coordinate (X1) at (-1pt,14pt);
\coordinate (X2) at (+3pt,15pt);
\coordinate (X3) at (+7pt,14pt);
\coordinate (Z1) at (-11pt,11pt);
\coordinate (Z2) at (-7pt,12pt);
\coordinate (Z3) at (-3pt,11pt);
\coordinate (Y1) at (-7pt,4pt);
\coordinate (Y2) at (+7pt,6pt);
\drawC{(O2)}{(X1)}
\drawC{(O2)}{(X2)}
\drawD{(O2)}{(X3)}
\drawC{(Y1)}{(Z1)}
\drawC{(Y1)}{(Z2)}
\drawD{(Y1)}{(Z3)}
\drawI{(O1)}{(O2)}
\drawI{(O1)}{(Y1)}
\drawD{(O1)}{(Y2)}
\end{tikzpicture}\,}}

\newcommand{\putdots}[4]{
\node (#1) [fill=white,circle] at (#2) {};
\fill ($(#1)-(1.73pt,1.73pt)$) rectangle ($(#1)+(1.73pt,1.73pt)$);
\node (#1') [#3] at (#1) {$#4$};
}

\newcommand{\putdot}[4]{
\node (#1) [fill=white,circle] at (#2) {};
\fill (#1) circle (2pt);
\node (#1') [#3] at (#1) {$#4$};
}

\newcommand{\putdotw}[4]{
\node (#1) [fill=white,circle] at (#2) {};
\draw (#1) circle (2pt);
\node (#1') [#3] at (#1) {$#4$};
}

\newcommand{\putdotg}[4]{
\node (#1) [fill=white,circle] at (#2) {};
\filldraw[lightgray] (#1) circle (2pt);
\node (#1') [#3] at (#1) {$#4$};
}

\newcommand{\dputw}[3]{
\draw[->,double] (#1)--(#2);
\draw[very thick] ($(#3)+(0,0.1)$)--($(#3)-(0,0.1)$);
}

\newcommand{\hputw}[3]{
\draw[->,double] (#1)--(#2);
\draw[very thick] ($(#3)+(0.1,0)$)--($(#3)-(0.1,0)$);
}

\newcommand{\putw}[2]{
\draw[->,double] (#1)--(#2);
}

\newcommand{\putdw}[2]{
\draw[->,densely dotted,double] (#1)--(#2);
}

\newcommand{\putcdw}[4]{
\draw[->,densely dotted,double] (#1)..controls(#2)and(#3)..(#4);
}

\newcommand{\putb}[2]{
\draw[->,very thick] (#1)--(#2);
}

\newcommand{\dputb}[3]{
\draw[->,very thick] (#1)--(#2);
\draw[very thick] ($(#3)+(0,0.1)$)--($(#3)-(0,0.1)$);
}

\newcommand{\hputb}[3]{
\draw[->,very thick] (#1)--(#2);
\draw[very thick] ($(#3)+(0.1,0)$)--($(#3)-(0.1,0)$);
}

\newcommand{\putdb}[2]{
\draw[->,densely dotted,very thick] (#1)--(#2);
}

\newcommand{\putcdb}[4]{
\draw[->,densely dotted,very thick] (#1)..controls(#2)and(#3)..(#4);
}

\newcommand{\dputcb}[5]{
\draw[->,very thick] (#1)..controls(#2)and(#3)..(#4);
\draw[very thick] ($(#5)+(0,0.1)$)--($(#5)-(0,0.1)$);
}

\newcommand{\dputcw}[5]{
\draw[->,double] (#1)..controls(#2)and(#3)..(#4);
\draw[very thick] ($(#5)+(0,0.1)$)--($(#5)-(0,0.1)$);
}

\newcommand{\putcl}[8]{
\draw[very thick] (#1)..controls(#2)and(#3)..(#4);
\filldraw[white] (#5) rectangle (#6);
\draw (#5) rectangle (#6);
\node at (#7) {\tiny$#8$};
}

\newcommand{\putl}[6]{
\draw[very thick] (#1)--(#2);
\filldraw[white] (#3) rectangle (#4);
\draw (#3) rectangle (#4);
\node at (#5) {\tiny$#6$};
}

\newcommand{\cD}{\mathcal{D}}
\newcommand{\cE}{\mathcal{E}}
\newcommand{\cL}{\mathcal{L}}

\newcommand{\ImUnit}{\mathsf{i}}
\newcommand{\CmplConj}[1]{\overline{#1}}

\newcommand{\probSp}{\Omega}
\newcommand{\sigmaField}{\mathcal{F}}
\newcommand{\prob}{\boldsymbol P}
\newcommand{\expect}{\boldsymbol E}
\newcommand{\whiteNoise}{\xi}

\newcommand{\SmoothFunc}[3]{C^{#1}({#2},{#3})}
\newcommand{\HolBesSp}[1]{\mathcal{C}^{#1}}

\newcommand{\drivers}[2]{{\mathcal{X}_{#1}^{#2}}}
\newcommand{\sols}[3]{{\mathcal{D}_{#1}^{#2,#3}}}

\newcommand{\FourierTrans}{\mathcal{F}}
\newcommand{\FTTime}{\FourierTrans_{\mathrm{time}}}
\newcommand{\FourierCoeff}[1]{\hat{#1}}
\newcommand{\FourierBase}[1][]{\mathbf{e}_{#1}}
\newcommand{\DyaPartOfUnit}[1][]{\rho_{#1}}
\newcommand{\LPBlock}[1][]{\triangle_{#1}}

\newcommand{\LaplaceOp}{\triangle}

\newcommand{\indicator}[1]{\mathtt{1}_{#1}}

\DeclareMathOperator{\supp}{supp}

\newcommand{\WienerIntCmpl}[2]{{\cal J}_{#1,#2}}

\newcommand{\const}[1][]{C_{#1}} 
\newcommand{\constRe}[2][]{\mathfrak{c}_{#1}^{#2}} 

\newcommand{\reso}{\varodot}
\newcommand{\rpara}{\varolessthan}
\newcommand{\lpara}{\varogreaterthan}
\DeclareMathOperator{\com}{com}
\newcommand{\comR}{R}

\newcommand{\drawA}[2]
	{
		\draw[thick] #1 -- #2;
		\draw[fill=black] #1 circle (0.3pt);
		\draw[fill=black] #2 circle (1.2pt);
	}
\newcommand{\drawB}[2]
	{
		\draw[densely dotted, thick] #1 -- #2;
		\draw[fill=black] #1 circle (0.3pt);
		\draw[fill=black] ($#2+(-1.2pt,-1.2pt)$)--($#2+(+1.2pt,-1.2pt)$)--($#2+(+1.2pt,+1.2pt)$)--($#2+(-1.2pt,+1.2pt)$)--cycle;
	}
\newcommand{\drawInt}[2]
	{
		\draw[thick] #1 -- #2;
		\draw[fill=black] #1 circle (0.3pt);
		\draw[fill=black] #2 circle (0.2pt);
	}
\newcommand{\drawReso}[1]
	{
		\draw[blue, fill=white] #1 circle (1.2pt); \draw[blue, fill=blue] #1 circle (0.2pt);
	}

\newcommand{\A}
	{
		{
			\begin{tikzpicture}[baseline=0.5pt, scale=6/12]
				\coordinate (O1) at (0pt,0pt);
				\coordinate (X1) at (0pt,12pt);
				\drawA{(O1)}{(X1)}
			\end{tikzpicture}
		}
	}
\newcommand{\B}
	{
		{
			\begin{tikzpicture}[baseline=0.5pt, scale=6/12]
				\coordinate (O1) at (0pt,0pt);
				\coordinate (X1) at (0pt,12pt);
				\drawB{(O1)}{(X1)}
			\end{tikzpicture}
		}
	}
\renewcommand{\AA}
	{
		{
			\begin{tikzpicture}[baseline=0.5pt, scale=6/12]
				\coordinate (O1) at (0,0);
				\coordinate (X1) at (-3pt,12pt);
				\coordinate (X2) at (+3pt,12pt);
				\drawA{(O1)}{(X1)}
				\drawA{(O1)}{(X2)}
			\end{tikzpicture}
		}
	}
\newcommand{\AB}
	{
		{
			\begin{tikzpicture}[baseline=0.5pt, scale=6/12]
				\coordinate (O1) at (0,0);
				\coordinate (X1) at (-3pt,12pt);
				\coordinate (X2) at (+3pt,12pt);
				\drawA{(O1)}{(X1)}
				\drawB{(O1)}{(X2)}
			\end{tikzpicture}
		}
	}
\newcommand{\BB}
	{
		{
			\begin{tikzpicture}[baseline=0.5pt, scale=6/12]
				\coordinate (O1) at (0,0);
				\coordinate (X1) at (-3pt,12pt);
				\coordinate (X2) at (+3pt,12pt);
				\drawB{(O1)}{(X1)}
				\drawB{(O1)}{(X2)}
			\end{tikzpicture}
		}
	}
\newcommand{\AAB}
	{
		{
			\begin{tikzpicture}[baseline=1pt, scale=6/12]
				\coordinate (O1) at (0,0);
				\coordinate (X1) at (-5pt,12pt);
				\coordinate (X2) at (+0pt,14pt);
				\coordinate (X3) at (+5pt,12pt);
				\drawA{(O1)}{(X1)}
				\drawA{(O1)}{(X2)}
				\drawB{(O1)}{(X3)}
			\end{tikzpicture}
		}
	}

\newcommand{\IAA}
	{
		{
			\begin{tikzpicture}[baseline=1pt, scale=8/12]
				\coordinate (O1) at (0pt,0pt);
				\coordinate (O2) at (0pt,6pt);
				\coordinate (X1) at (-4pt,12pt);
				\coordinate (X2) at (+4pt,12pt);
				\drawA{(O2)}{(X1)}
				\drawA{(O2)}{(X2)}
				\drawInt{(O1)}{(O2)}
			\end{tikzpicture}
		}
	}
\newcommand{\IAB}
	{
		{
			\begin{tikzpicture}[baseline=1pt, scale=8/12]
				\coordinate (O1) at (0pt,0pt);
				\coordinate (O2) at (0pt,6pt);
				\coordinate (X1) at (-4pt,12pt);
				\coordinate (X2) at (+4pt,12pt);
				\drawA{(O2)}{(X1)}
				\drawB{(O2)}{(X2)}
				\drawInt{(O1)}{(O2)}
			\end{tikzpicture}
		}
	}
\newcommand{\IAAB}
	{
		{
			\begin{tikzpicture}[baseline=1pt, scale=8/12]
				\coordinate (O1) at (0pt,0pt);
				\coordinate (O2) at (0pt,5pt);
				\coordinate (X1) at (-4pt,11pt);
				\coordinate (X2) at (+0pt,12pt);
				\coordinate (X3) at (+4pt,11pt);
				\drawA{(O2)}{(X1)}
				\drawA{(O2)}{(X2)}
				\drawB{(O2)}{(X3)}
				\drawInt{(O1)}{(O2)}
			\end{tikzpicture}
		}
	}

\newcommand{\IAABoA}
	{
		{
			\begin{tikzpicture}[baseline=0pt, scale=8/12]
				\coordinate (O1) at (0pt,-4pt);

				\coordinate (O2) at (0pt,3pt);
				\coordinate (X1) at (-4pt,11pt);
				\coordinate (X2) at (0pt,12pt);
				\coordinate (X3) at (4pt,11pt);

				\coordinate (X4) at (4pt,+4pt);

				\drawA{(O2)}{(X1)}
				\drawA{(O2)}{(X2)}
				\drawB{(O2)}{(X3)}
				\drawInt{(O1)}{(O2)}

				\drawA{(O1)}{(X4)}

				\drawReso{(O1)}
			\end{tikzpicture}
		}
	}
\newcommand{\IAABoB}
	{
		{
			\begin{tikzpicture}[baseline=0pt, scale=8/12]
				\coordinate (O1) at (0pt,-4pt);

				\coordinate (O2) at (0pt,4pt);
				\coordinate (X1) at (-4pt,11pt);
				\coordinate (X2) at (0pt,12pt);
				\coordinate (X3) at (4pt,11pt);

				\coordinate (X4) at (4pt,+4pt);

				\drawA{(O2)}{(X1)}
				\drawA{(O2)}{(X2)}
				\drawB{(O2)}{(X3)}
				\drawInt{(O1)}{(O2)}

				\drawB{(O1)}{(X4)}

				\drawReso{(O1)}
			\end{tikzpicture}
		}
	}
\newcommand{\IAAoAB}
	{
		{
			\begin{tikzpicture}[baseline=0pt, scale=8/12]
				\coordinate (O1) at (0pt,-4pt);

				\coordinate (O2) at (0pt,4pt);
				\coordinate (X1) at (-4pt,11pt);
				\coordinate (X2) at (4pt,11pt);

				\coordinate (X3) at (-4pt,+4pt);
				\coordinate (X4) at (+4pt,+4pt);

				\drawA{(O2)}{(X1)}
				\drawA{(O2)}{(X2)}
				\drawInt{(O1)}{(O2)}

				\drawA{(O1)}{(X3)}
				\drawB{(O1)}{(X4)}

				\drawReso{(O1)}
			\end{tikzpicture}
		}
	}
\newcommand{\IAAoBB}
	{
		{
			\begin{tikzpicture}[baseline=0pt, scale=8/12]
				\coordinate (O1) at (0pt,-4pt);

				\coordinate (O2) at (0pt,4pt);
				\coordinate (X1) at (-4pt,11pt);
				\coordinate (X2) at (4pt,11pt);

				\coordinate (X3) at (-4pt,+4pt);
				\coordinate (X4) at (+4pt,+4pt);

				\drawA{(O2)}{(X1)}
				\drawA{(O2)}{(X2)}
				\drawInt{(O1)}{(O2)}

				\drawB{(O1)}{(X3)}
				\drawB{(O1)}{(X4)}

				\drawReso{(O1)}
			\end{tikzpicture}
		}
	}
\newcommand{\IABoAB}
	{
		{
			\begin{tikzpicture}[baseline=0pt, scale=8/12]
				\coordinate (O1) at (0pt,-4pt);

				\coordinate (O2) at (0pt,4pt);
				\coordinate (X1) at (-4pt,11pt);
				\coordinate (X2) at (4pt,11pt);

				\coordinate (X3) at (-4pt,+4pt);
				\coordinate (X4) at (+4pt,+4pt);

				\drawA{(O2)}{(X1)}
				\drawB{(O2)}{(X2)}
				\drawInt{(O1)}{(O2)}

				\drawA{(O1)}{(X3)}
				\drawB{(O1)}{(X4)}

				\drawReso{(O1)}
			\end{tikzpicture}
		}
	}
\newcommand{\IABoBB}
	{
		{
			\begin{tikzpicture}[baseline=0pt, scale=8/12]
				\coordinate (O1) at (0pt,-4pt);

				\coordinate (O2) at (0pt,4pt);
				\coordinate (X1) at (-4pt,11pt);
				\coordinate (X2) at (4pt,11pt);

				\coordinate (X3) at (-4pt,+4pt);
				\coordinate (X4) at (+4pt,+4pt);

				\drawA{(O2)}{(X1)}
				\drawB{(O2)}{(X2)}
				\drawInt{(O1)}{(O2)}

				\drawB{(O1)}{(X3)}
				\drawB{(O1)}{(X4)}

				\drawReso{(O1)}
			\end{tikzpicture}
		}
	}
\newcommand{\IAABoAB}
	{
		{
			\begin{tikzpicture}[baseline=0pt, scale=8/12]
				\coordinate (O1) at (0pt,-4pt);

				\coordinate (O2) at (0pt,3pt);
				\coordinate (X1) at (-4pt,11pt);
				\coordinate (X2) at (0pt,12pt);
				\coordinate (X3) at (4pt,11pt);

				\coordinate (X4) at (-4pt,+4pt);
				\coordinate (X5) at (4pt,+4pt);

				\drawA{(O2)}{(X1)}
				\drawA{(O2)}{(X2)}
				\drawB{(O2)}{(X3)}
				\drawInt{(O1)}{(O2)}

				\drawA{(O1)}{(X4)}
				\drawB{(O1)}{(X5)}

				\drawReso{(O1)}
			\end{tikzpicture}
		}
	}
\newcommand{\IAABoBB}
	{
		{
			\begin{tikzpicture}[baseline=0pt, scale=8/12]
				\coordinate (O1) at (0pt,-4pt);

				\coordinate (O2) at (0pt,3pt);
				\coordinate (X1) at (-4pt,11pt);
				\coordinate (X2) at (0pt,12pt);
				\coordinate (X3) at (4pt,11pt);

				\coordinate (X4) at (-4pt,+4pt);
				\coordinate (X5) at (4pt,+4pt);

				\drawA{(O2)}{(X1)}
				\drawA{(O2)}{(X2)}
				\drawB{(O2)}{(X3)}
				\drawInt{(O1)}{(O2)}

				\drawB{(O1)}{(X4)}
				\drawB{(O1)}{(X5)}

				\drawReso{(O1)}
			\end{tikzpicture}
		}
	}

\newcommand{\drawAC}[2]{
\draw[thick] #1 -- #2;
\draw[fill=black] #1 circle (0.3pt);
\draw[draw=black,fill=white] #2 circle (1.2pt);
}

\newcommand{\drawBC}[2]{
\draw[densely dotted, thick] #1 -- #2;
\draw[fill=black] #1 circle (0.3pt);
\draw[draw=black,fill=white] #2 circle (1.2pt);}

\newcommand{\bigIAAB}{{
\begin{tikzpicture}[baseline=0pt, scale=14/12]
\coordinate (O1) at (0pt,0pt);
\coordinate (O2) at (0pt,5pt);
\coordinate (X1) at (-4pt,11pt);
\node (X1') [above] at (X1) {\tiny$1$};
\coordinate (X2) at (+0pt,12pt);
\node (X2') [above] at (X2) {\tiny$2$};
\coordinate (X3) at (+4pt,11pt);
\node (X3') [above] at (X3) {\tiny$1$};
\drawA{(O2)}{(X1)}
\drawA{(O2)}{(X2)}
\drawB{(O2)}{(X3)}
\drawInt{(O1)}{(O2)}
\end{tikzpicture}}}

\newcommand{\bigIAA}{{
\begin{tikzpicture}[baseline=1pt, scale=12/12]
\coordinate (O1) at (0pt,0pt);
\coordinate (O2) at (0pt,6pt);
\coordinate (X1) at (-4pt,12pt);
\node (X1') [above] at (X1) {\tiny$1$};
\coordinate (X2) at (+4pt,12pt);
\node (X2') [above] at (X2) {\tiny$2$};
\drawA{(O2)}{(X1)}
\drawA{(O2)}{(X2)}
\drawInt{(O1)}{(O2)}
\end{tikzpicture}}}

\newcommand{\bigIAB}{{
\begin{tikzpicture}[baseline=1pt, scale=12/12]
\coordinate (O1) at (0pt,0pt);
\coordinate (O2) at (0pt,6pt);
\coordinate (X1) at (-4pt,12pt);
\node (X1') [above] at (X1) {\tiny$1$};
\coordinate (X2) at (+4pt,12pt);
\node (X2') [above] at (X2) {\tiny$1$};
\drawA{(O2)}{(X1)}
\drawB{(O2)}{(X2)}
\drawInt{(O1)}{(O2)}
\end{tikzpicture}}}

\newcommand{\bigA}[1]{{
\begin{tikzpicture}[baseline=0.5pt, scale=10/12]
\coordinate (O1) at (0pt,0pt);
\coordinate (X1) at (0pt,12pt);
\node (X1') [above] at (X1) {\tiny$#1$};
\drawA{(O1)}{(X1)}
\end{tikzpicture}}}

\newcommand{\bigB}[1]{{
\begin{tikzpicture}[baseline=0.5pt, scale=10/12]
\coordinate (O1) at (0pt,0pt);
\coordinate (X1) at (0pt,12pt);
\node (X1') [above] at (X1) {\tiny$#1$};
\drawB{(O1)}{(X1)}
\end{tikzpicture}}}

\newcommand{\bigAB}[2]{{
\begin{tikzpicture}[baseline=0.5pt, scale=10/12]
\coordinate (O1) at (0,0);
\coordinate (X1) at (-3pt,12pt);
\node (X1') [above] at (X1) {\tiny$#1$};
\coordinate (X2) at (+3pt,12pt);
\node (X2') [above] at (X2) {\tiny$#2$};
\drawA{(O1)}{(X1)}
\drawB{(O1)}{(X2)}
\end{tikzpicture}}}

\newcommand{\bigBB}[2]{{
\begin{tikzpicture}[baseline=0.5pt, scale=10/12]
\coordinate (O1) at (0,0);
\coordinate (X1) at (-3pt,12pt);
\node (X1') [above] at (X1) {\tiny$#1$};
\coordinate (X2) at (+3pt,12pt);
\node (X2') [above] at (X2) {\tiny$#2$};
\drawB{(O1)}{(X1)}
\drawB{(O1)}{(X2)}
\end{tikzpicture}}}

\newcommand{\IAADoC}{{
\begin{tikzpicture}[baseline=0pt, scale=8/12]
\coordinate (O1) at (0pt,-4pt);
\coordinate (O2) at (0pt,4pt);
\coordinate (X1) at (-3pt,11pt);
\coordinate (X2) at (3pt,11pt);
\coordinate (X3) at (6pt,1pt);
\coordinate (X4) at (6pt,1pt);
\drawA{(O2)}{(X1)}
\drawA{(O2)}{(X2)}
\drawBC{(O2)}{(X3)}
\drawInt{(O1)}{(O2)}
\drawAC{(O1)}{(X4)}
\drawReso{(O1)}
\end{tikzpicture}}}

\newcommand{\IACBoD}{{
\begin{tikzpicture}[baseline=0pt, scale=8/12]
\coordinate (O1) at (0pt,-4pt);
\coordinate (O2) at (0pt,4pt);
\coordinate (X1) at (-3pt,11pt);
\coordinate (X2) at (3pt,11pt);
\coordinate (X3) at (6pt,1pt);
\coordinate (X4) at (6pt,1pt);
\drawA{(O2)}{(X1)}
\drawB{(O2)}{(X2)}
\drawAC{(O2)}{(X3)}
\drawInt{(O1)}{(O2)}
\drawBC{(O1)}{(X4)}
\drawReso{(O1)}
\end{tikzpicture}}}

\newcommand{\IACoAD}{{
\begin{tikzpicture}[baseline=0pt, scale=8/12]
\coordinate (O1) at (0pt,-4pt);
\coordinate (O2) at (0pt,4pt);
\coordinate (X1) at (-3pt,10pt);
\coordinate (X2) at (6pt,1pt);
\coordinate (X3) at (-5pt,2pt);
\coordinate (X4) at (6pt,1pt);
\drawA{(O2)}{(X1)}
\drawAC{(O2)}{(X2)}
\drawInt{(O1)}{(O2)}
\drawA{(O1)}{(X3)}
\drawBC{(O1)}{(X4)}
\drawReso{(O1)}
\end{tikzpicture}}}

\newcommand{\IACoBD}{{
\begin{tikzpicture}[baseline=0pt, scale=8/12]
\coordinate (O1) at (0pt,-4pt);
\coordinate (O2) at (0pt,4pt);
\coordinate (X1) at (-3pt,10pt);
\coordinate (X2) at (6pt,1pt);
\coordinate (X3) at (-5pt,2pt);
\coordinate (X4) at (6pt,1pt);
\drawA{(O2)}{(X1)}
\drawAC{(O2)}{(X2)}
\drawInt{(O1)}{(O2)}
\drawB{(O1)}{(X3)}
\drawBC{(O1)}{(X4)}
\drawReso{(O1)}
\end{tikzpicture}}}

\newcommand{\ICBoAD}{{
\begin{tikzpicture}[baseline=0pt, scale=8/12]
\coordinate (O1) at (0pt,-4pt);
\coordinate (O2) at (0pt,4pt);
\coordinate (X1) at (-3pt,10pt);
\coordinate (X2) at (6pt,1pt);
\coordinate (X3) at (-5pt,2pt);
\coordinate (X4) at (6pt,1pt);
\drawB{(O2)}{(X1)}
\drawAC{(O2)}{(X2)}
\drawInt{(O1)}{(O2)}
\drawA{(O1)}{(X3)}
\drawBC{(O1)}{(X4)}
\drawReso{(O1)}
\end{tikzpicture}}}

\newcommand{\IADoCB}{{
\begin{tikzpicture}[baseline=0pt, scale=8/12]
\coordinate (O1) at (0pt,-4pt);
\coordinate (O2) at (0pt,4pt);
\coordinate (X1) at (-3pt,10pt);
\coordinate (X2) at (6pt,1pt);
\coordinate (X3) at (-5pt,2pt);
\coordinate (X4) at (6pt,1pt);
\drawA{(O2)}{(X1)}
\drawBC{(O2)}{(X2)}
\drawInt{(O1)}{(O2)}
\drawB{(O1)}{(X3)}
\drawAC{(O1)}{(X4)}
\drawReso{(O1)}
\end{tikzpicture}}}

\newcommand{\IADoBD}{{
\begin{tikzpicture}[baseline=0pt, scale=8/12]
\coordinate (O1) at (0pt,-4pt);
\coordinate (O2) at (0pt,4pt);
\coordinate (X1) at (-3pt,10pt);
\coordinate (X2) at (6pt,1pt);
\coordinate (X3) at (-5pt,2pt);
\coordinate (X4) at (6pt,1pt);
\drawB{(O2)}{(X1)}
\drawAC{(O2)}{(X2)}
\drawInt{(O1)}{(O2)}
\drawB{(O1)}{(X3)}
\drawBC{(O1)}{(X4)}
\drawReso{(O1)}
\end{tikzpicture}}}

\newcommand{\ICCoDD}{{
\begin{tikzpicture}[baseline=0pt, scale=8/12]
\coordinate (O1) at (0pt,-4pt);
\coordinate (O2) at (0pt,6pt);
\coordinate (X1) at (-6pt,1pt);
\coordinate (X2) at (6pt,1pt);
\coordinate (X3) at (-6pt,1pt);
\coordinate (X4) at (6pt,1pt);
\drawAC{(O2)}{(X1)}
\drawAC{(O2)}{(X2)}
\drawInt{(O1)}{(O2)}
\drawBC{(O1)}{(X3)}
\drawBC{(O1)}{(X4)}
\drawReso{(O1)}
\end{tikzpicture}}}

\newcommand{\ICDoCD}{{
\begin{tikzpicture}[baseline=0pt, scale=8/12]
\coordinate (O1) at (0pt,-4pt);
\coordinate (O2) at (0pt,6pt);
\coordinate (X1) at (-6pt,1pt);
\coordinate (X2) at (6pt,1pt);
\coordinate (X3) at (-6pt,1pt);
\coordinate (X4) at (6pt,1pt);
\drawAC{(O2)}{(X1)}
\drawBC{(O2)}{(X2)}
\drawInt{(O1)}{(O2)}
\drawBC{(O1)}{(X3)}
\drawAC{(O1)}{(X4)}
\drawReso{(O1)}
\end{tikzpicture}}}

\newcommand{\IACBoAD}{{
\begin{tikzpicture}[baseline=0pt, scale=8/12]
\coordinate (O1) at (0pt,-4pt);
\coordinate (O2) at (0pt,4pt);
\coordinate (X1) at (-3pt,11pt);
\coordinate (X2) at (3pt,11pt);
\coordinate (X3) at (6pt,1pt);
\coordinate (X4) at (-5pt,2pt);
\coordinate (X5) at (6pt,1pt);
\drawA{(O2)}{(X1)}
\drawB{(O2)}{(X2)}
\drawAC{(O2)}{(X3)}
\drawInt{(O1)}{(O2)}
\drawA{(O1)}{(X4)}
\drawBC{(O1)}{(X5)}
\drawReso{(O1)}
\end{tikzpicture}}}

\newcommand{\IACBoBD}{{
\begin{tikzpicture}[baseline=0pt, scale=8/12]
\coordinate (O1) at (0pt,-4pt);
\coordinate (O2) at (0pt,4pt);
\coordinate (X1) at (-3pt,11pt);
\coordinate (X2) at (3pt,11pt);
\coordinate (X3) at (6pt,1pt);
\coordinate (X4) at (-5pt,2pt);
\coordinate (X5) at (6pt,1pt);
\drawA{(O2)}{(X1)}
\drawB{(O2)}{(X2)}
\drawAC{(O2)}{(X3)}
\drawInt{(O1)}{(O2)}
\drawB{(O1)}{(X4)}
\drawBC{(O1)}{(X5)}
\drawReso{(O1)}
\end{tikzpicture}}}

\newcommand{\IAADoBC}{{
\begin{tikzpicture}[baseline=0pt, scale=8/12]
\coordinate (O1) at (0pt,-4pt);
\coordinate (O2) at (0pt,4pt);
\coordinate (X1) at (-3pt,11pt);
\coordinate (X2) at (3pt,11pt);
\coordinate (X3) at (6pt,1pt);
\coordinate (X4) at (-5pt,2pt);
\coordinate (X5) at (6pt,1pt);
\drawA{(O2)}{(X1)}
\drawA{(O2)}{(X2)}
\drawBC{(O2)}{(X3)}
\drawInt{(O1)}{(O2)}
\drawB{(O1)}{(X4)}
\drawAC{(O1)}{(X5)}
\drawReso{(O1)}
\end{tikzpicture}}}

\newcommand{\IACDoCD}{{
\begin{tikzpicture}[baseline=0pt, scale=8/12]
\coordinate (O1) at (0pt,-4pt);
\coordinate (O2) at (0pt,6pt);
\coordinate (O3) at (0pt,12pt);
\coordinate (X1) at (-6pt,1pt);
\coordinate (X2) at (6pt,1pt);
\coordinate (X3) at (-6pt,1pt);
\coordinate (X4) at (6pt,1pt);
\drawAC{(O2)}{(X1)}
\drawBC{(O2)}{(X2)}
\drawInt{(O1)}{(O2)}
\drawA{(O2)}{(O3)}
\drawBC{(O1)}{(X3)}
\drawAC{(O1)}{(X4)}
\drawReso{(O1)}
\end{tikzpicture}}}

\newcommand{\IBCCoDD}{{
\begin{tikzpicture}[baseline=0pt, scale=8/12]
\coordinate (O1) at (0pt,-4pt);
\coordinate (O2) at (0pt,6pt);
\coordinate (O3) at (0pt,12pt);
\coordinate (X1) at (-6pt,1pt);
\coordinate (X2) at (6pt,1pt);
\coordinate (X3) at (-6pt,1pt);
\coordinate (X4) at (6pt,1pt);
\drawAC{(O2)}{(X1)}
\drawAC{(O2)}{(X2)}
\drawInt{(O1)}{(O2)}
\drawB{(O2)}{(O3)}
\drawBC{(O1)}{(X3)}
\drawBC{(O1)}{(X4)}
\drawReso{(O1)}
\end{tikzpicture}}}

%% file: 0100_introduction.tex

\section{Introduction}\label{sec_1503967525}

The cubic complex Ginzburg-Landau (CGL) equation
is one of the most important nonlinear
partial differential equations (PDEs) in applied mathematics and physics.
It describes various physical phenomena such as
nonlinear waves, second-order phase transition,
superconductivity, superfluidity among others.
See \cite{AransonKramer2002} for example.

There are also many papers on
its stochastic version, the CGL with a noise term
(\cite{Barton-Smith2004b, Barton-Smith2004a, KuksinShirikyan2004, Odasso2006, PuGuo2011, Yang2004}
to name but a few).
In these preceding works, however, the noise is either non-white
or multiplicative.
Except when the space dimension $d=1$ in \cite{Hairer2002},
the stochastic cubic CGL with \textit{additive} space-time white noise
has not been solved.

The difficulty in the case $d \ge 2$ is as follows.
Since space-time white noise is so rough,
a solution $u_t (x)=u(t, x)$ would be a Schwartz distribution in $x$,
not a function, even if it existed.
Consequently, the cubic nonlinear term $|u_t|^2 u_t$
does not make sense in the usual way.
For this reason, well-definedness of the equation itself was unclear
and the cubic CGL with space-time white noise
was considered too singular when $d \ge 2$.

However,
two new theories emerged recently,
which can deal with quite singular stochastic PDEs of this kind.
One is regularity structure theory \cite{Hairer2014a}
 and the other is
 paracontrolled distribution theory \cite{GubinelliImkellerPerkowski2015}.
They are both descendants of rough path theory and
their deterministic part looks somewhat similar
to the counterpart in rough path theory at least in spirit.
However, their probabilistic part is more complicated
than the counterpart  in rough path theory
since non-trivial renormalization of the noise has to be done.
(There is another theory based on the theory of
renormalization groups \cite{Kupiainen2016},
which will not be discussed in this paper, however.)

Although they are clearly different theories,
examples of stochastic PDEs they can deal with are very similar.
A partial list of singular stochastic PDEs
which have been solved (locally in time)
by these theories is as follows:
Parabolic Anderson Model ($d=2, 3$)
\cite{GubinelliImkellerPerkowski2015, Hairer2014a, BailleulBernicotFrey2015Arvix},
KPZ equation and its variants ($d=1$)
\cite{FrizHairer2014, GubinelliPerkowski2017, Hoshino2016, FunakiHoshino2016Arxiv},
the dynamic $\Phi^4_3$-model ($d=3$)  \cite{Hairer2014a, CatellierChouk2013Arxiv},
Navier-Stokes equation with space-time white noise ($d=3$)
\cite{ZhuZhu2015},
FitzHugh-Nagumo equation with space-time
white noise ($d=3$)
\cite{BerglundKuehn2016}.

The main objective of
this paper is to prove local well-posedness of the
stochastic cubic complex Ginzburg-Landau equation
on the three-dimensional torus $\Torus^3=(\RealNum/\Integers)^3$
of the following form by using these two theories:
\begin{align}\label{eq:cgl}
	\partial_t u
	=
		(\ImUnit+\mu)\LaplaceOp u
		+
		\nu(1-|u|^2)u
		+
		\whiteNoise,
	\qquad
	t>0,\quad
	x\in\Torus^3.
\end{align}
Here, $\ImUnit=\sqrt{-1}$, $\mu>0$, $\nu\in\CmplNum$ are constants
and $\xi$ is \textit{complex-valued} space-time white noise, that is,
a centered complex Gaussian random field with covariance
\begin{align*}
	\expect[\xi(t,x)\xi(s,y)]&=0,
	&
	\expect[\xi(t,x)\overline{\xi(s,y)}]
	&
		=\delta(t-s)\delta(x-y),
\end{align*}
where $\delta$ denotes the Dirac delta function.

We replace $\whiteNoise$ by smeared noise $\whiteNoise^\epsilon$ with a parameter $0<\epsilon<1$
so that $\whiteNoise^\epsilon\to\whiteNoise$ as $\epsilon\downarrow 0$ in an appropriate topology
and consider a renormalized equation
\begin{align}\label{eq:rcgl}
	\partial_t u^\epsilon
	=
		(\ImUnit+\mu)\LaplaceOp u^\epsilon
		+
		\nu(1-|u^\epsilon|^2)u^\epsilon
		+
		\nu C^\epsilon u^\epsilon
		+
		\whiteNoise^\epsilon,
	\qquad
	t>0,\quad
	x\in\Torus^3,
\end{align}
where $C^\epsilon$ is a suitably chosen complex
constant (specified later) which diverges as $\epsilon\downarrow 0$.
We show that the solution to \eqref{eq:rcgl}
converges to some process in an appropriate topology.
To this end, we use the theory of regularity structure by Hairer \cite{Hairer2014a}
and the theory of paracontrolled distributions by Gubinelli-Imkeller-Perkowski \cite{GubinelliImkellerPerkowski2015}.
In the two main results (Theorems \ref{0330_main result} and \ref{thm_2017013154042}),
we use different approximations of $\xi$. However, we can choose the same approximation $\xi^\epsilon$ in both theories.
See \rref{0550 tx regularize is ok}.
Consequently, we can see that the solutions obtained in these two theories ``essentially coincide'',
even though the idea behind these theories are quite different.
(It should be noted, however, that we do not have a rigorous proof
of the exact coincidence of the two solutions.
To prove it, a further  investigation of
the renormalization constants is needed, which could be
an interesting future task.)

We now make a comment on the space dimension.
When $d \ge 4$, CGL \eqref{eq:cgl} is not subcritical in the sense of
\cite{Hairer2014a} and therefore the equation cannot be solved
(or does not even make sense) by any existing method.
Though we do not give a proof in this paper,
we believe that
the case $d=2$ is actually much easier than our case $d=3$.

This paper is organized as follows.
In \secref[sec_20161107073826]{sec_20161107073513}, following  \cite{Hairer2014a},
we apply the theory of regularity structures to
the stochastic CGL \eqref{eq:cgl}.
At the beginning of \secref{sec_20161107073826}
we first present our main result (\tref{0330_main result}) in a precise form.
Then we construct a regularity structure for \eqref{eq:cgl}
and prove local-wellposedness of \eqref{eq:cgl}
in
a deterministic way.
\secref{sec_20161107073513} is devoted to the probabilistic step, in particular,
the renormalization procedure.

In \secref[sec_20161107073643]{sec_20161107073600}, we apply the paracontrolled calculus to \eqref{eq:cgl}.
In \secref{sec_20161107073643},
we precisely present our main result (\tref{thm_2017013154042})
and deterministically solve \eqref{eq:cgl} locally in time
in a similar way to \cite{MourratWeber2016Arxiv}.
We prove the probabilistic part in \secref{sec_20161107073600}
using a new method developped in
\cite{GubinelliPerkowski2017}.

\secref{sec:itowienerintegral} is an appendix, in which
 we recall the definition of complex multiple It\^o-Wiener integrals.
The product formula for them
is frequently used in \secref[sec_20161107073513]{sec_20161107073600}.

Notations:~
We use the following notations:
For two functions $f$ and $g$, we write $f\lesssim g$
if there exists a positive constant $C$ such that $f(x)\leq C g(x)$ for any $x$.
We write
$f(x)\approx g(x)$ if both
$f(x)\lesssim g(x)$ and $g(x)\lesssim f(x)$ hold.
To indicate the argument $x$ of a function $f$, we use both symbols
$f(x)$ and $f_x$.

%% file: 0330_RegStrDet.tex

\section{CGL by the theory of regularity structures}\label{sec_20161107073826}

In this and the next sections, we study CGL equation by the theory of regularity structures. We begin by presenting the main result in \tref{0330_main result} below.

We denote by $\xi$ periodic space-time white noise on $\RealNum\times\Torus^3$, which is extended periodically to $\RealNum^4$.
We replace $\xi$ by space-time smeared noise $\xi^\epsilon=\xi*\rho^\epsilon$ for $\epsilon>0$,
where $\rho$ is non-negative, smooth and compactly supported function on $\RealNum^4$ with $\int\rho=1$,
and $\rho^\epsilon(t,x)=\epsilon^{-5}\rho(\epsilon^{-2}t,\epsilon^{-1}x)$. We consider the classical solution $u^\epsilon$ of the equation
\[\partial_tu^\epsilon=(\mathsf{i}+\mu)\triangle u^\epsilon+\nu(1-|u^\epsilon|^2+C^\epsilon)u^\epsilon+\xi^\epsilon,\quad(t,x)\in\RealNum_+\times\RealNum^3,\]
with initial condition $u_0$,
where $C^\epsilon=2C_1^\epsilon-2\overline{\nu}C_{2,1}^\epsilon-4\nu C_{2,2}^\epsilon$
is a sum of diverging constants as $\epsilon\downarrow0$ and precise behaviors of them are stated in \pref{section4_5:estimate of C}.
We write $\RealNum_+=(0,\infty)$. For $\eta\in\RealNum$, we define
\[
  \mathcal{C}^\eta
  =
    \mathcal{C}^\eta(\Torus^3,\CmplNum)
  =
    \{
      u\in\mathcal{B}_{\infty,\infty}^\eta(\RealNum^3,\CmplNum)\,;\,
      \text{$u(\cdot+k)=u(\cdot)$ for any $k\in\Integers^3$}
    \},
\]
where $\mathcal{B}_{\infty,\infty}^\eta(\RealNum^3,\CmplNum)$ is a usual inhomogeneous Besov space.
We denote by $C([0,T],\mathcal{C}^\eta)$ the set of all $\mathcal{C}^\eta$-valued continuous functions on $[0,T]$
endowed with the supremum norm $\|\cdot\|_{C([0,T],\mathcal{C}^\eta)}$.

\begin{theorem}\label{0330_main result}
Let $\eta\in(-\frac{2}{3},-\frac{1}{2})$. Then for every $u_0\in\mathcal{C}^\eta$,
the sequence $\{u^\epsilon\}$ converges to a limit $u$ in probability as $\epsilon\downarrow0$.
Precisely speaking, this means that there exists an a.s. strictly positive random time $T$ depending on $u_0$ and $\xi$,
such that $u$ and $u^\epsilon$ for every $\epsilon>0$ belong to the space $C([0,T],\mathcal{C}^\eta)$ and we have
\[
\|u^\epsilon-u\|_{C([0,T],\mathcal{C}^\eta)}\to0
\]
in probability. Furthermore, $u$ is independent of the choice of $\rho$.
\end{theorem}

We use the following notations in \secref[sec_20161107073826]{sec_20161107073513}:
\begin{itemize}
\item For $z=(t,x_1,x_2,x_3)\in\RealNum^4$, we define $\|z\|_s=|t|^{\frac{1}{2}}+|x_1|+|x_2|+|x_3|$.
\item For $k=(k_i)_{i=0}^3\in\Integers_+^4$, we define $|k|_s=2k_0+k_1+k_2+k_3$ and $\partial^k=\partial_t^{k_0}\partial_{x_1}^{k_1}\partial_{x_2}^{k_2}\partial_{x_3}^{k_3}$. Here $\Integers_+=\{0,1,2,\dots\}$.
\item For $\varphi\in C(\RealNum^4,\CmplNum)$ and $\delta>0$, we define the space-time scaling around $z=(t,x_1,x_2,x_3)\in\RealNum^4$ by
\[\varphi_z^\delta(t',x_1',x_2',x_3')=\delta^{-5}\varphi(\delta^{-2}(t'-t),\delta^{-1}(x_1'-x_1),\delta^{-1}(x_2'-x_2),\delta^{-1}(x_3'-x_3)).\]
\end{itemize}

We define the parabolic H\"older-Besov space $\mathcal{C}_s^\alpha$ on $\RealNum^4$ for $\alpha\in\RealNum$. At this stage, we do not impose periodicity for elements of $\mathcal{C}_s^\alpha$.
\begin{itemize}
\item For $\alpha>0$, we denote by $\mathcal{C}_s^\alpha$ the space of complex-valued functions $\varphi$ on $\RealNum^4$ such that
\begin{align}\label{3.1:def of C^alpha}
\bigl|\partial^k\varphi(z')-\sum_{|k+l|_s<\alpha}\frac{(z'-z)^l}{l!}\partial^{k+l}\varphi(z)\bigr|\lesssim\|z'-z\|_s^{\alpha-|k|_s}
\end{align}
holds locally in $z,z'\in\RealNum^4$ and for every $k$ with $|k|_s<\alpha$.
\item Denote by $\mathcal{C}_s^0=L_{\mathrm{loc}}^\infty(\RealNum^4,\CmplNum)$ the space of locally bounded functions.
\item For $r>0$, let $\mathcal{B}_r$ be the set of complex-valued smooth functions $\varphi$ on $\RealNum^4$ supported in the ball $B_s(0,1)=\{z;\|z\|_s\le1\}$ and such that their derivatives of order up to $r$ are bounded by $1$. Let $\alpha<0$ and $r=\lceil-\alpha\rceil$. Denote by $\mathcal{C}_s^\alpha$ be the space of Schwartz distributions $\xi\in\mathcal{S}'=\mathcal{S}'(\RealNum^4,\CmplNum)$ such that
\[\|\xi\|_{\alpha;K}:=\sup_{z\in K,\varphi\in\mathcal{B}_r,\delta\in(0,1]}\delta^{-\alpha}|\langle\xi,\varphi_z^\delta\rangle|<\infty\]
for every compact set $K\subset\RealNum^4$.
\end{itemize}

\subsection{Results on regularity structures}

First we recall basic concepts from the theory of regularity structures \cite{Hairer2014a}.

\begin{definition}
We say that a triplet $\mathcal{T}=(A,T,G)$ is a \emph{regularity structure} with \emph{index set} $A$, \emph{model space} $T$ and \emph{structure group} $G$, if
\begin{itemize}
\item $A$ is a locally finite set of real numbers bounded from below and $0\in A$.
\item $T=\bigoplus_{\alpha\in A}T_\alpha$ with complex Banach spaces $(T_\alpha,\|\cdot\|_\alpha)$. Furthermore, $T_0\simeq\CmplNum$ and its unit vector is denoted by $\unit$.
\item $G$ is a subgroup of $\mathcal{L}(T)$, the set of continuous linear operators on $T$, such that, for every $\Gamma\in G$, $\alpha\in A$, and $\tau\in T_\alpha$,
\[\Gamma\tau-\tau\in T_\alpha^-:=\bigoplus_{\beta<\alpha}T_\beta.\]
Furthermore, $\Gamma\unit=\unit$ for every $\Gamma\in G$.
\end{itemize}
\end{definition}

\begin{definition}
Let $\mathcal{T}$ be a regularity structure. We say that a subspace $V=\bigoplus_{\beta\in A}V_\beta$ with $V_\beta\subset T_\beta$ is a \emph{sector} of regularity $\alpha\le0$ if $V$ is invariant under $G$ (i.e. $\Gamma V\subset V$ for every $\Gamma\in G$) and $\alpha$ is the minimal index such that $V_\alpha\neq\{0\}$.
A sector with regularity $0$ is called \emph{function-like}.
\end{definition}


For $\tau\in T$, we write $\|\tau\|_\alpha=\|\tau_\alpha\|_\alpha$, where $\tau_\alpha$ is the component of $\tau$ in $T_\alpha$.

\begin{definition}
Let $\mathcal{T}$ be a regularity structure and let $r=\lceil-\inf A\rceil$. A \emph{model} $Z=(\Pi,\Gamma)$ is a pair of maps $\Gamma:\RealNum^4\times\RealNum^4\ni(z,z')\mapsto\Gamma_{zz'}\in G$ and $\Pi:\RealNum^4\ni z\mapsto\Pi_z\in\mathcal{L}(T,\mathcal{S}')$, the set of continuous linear operators from $T$ to $\mathcal{S}'$, which satisfy
\[\Gamma_{zz'}\Gamma_{z'z''}=\Gamma_{zz''},\quad
\Pi_z\Gamma_{zz'}=\Pi_{z'}\]
for every $z,z',z''\in\RealNum^4$, and
\begin{align*}
&\|\Gamma\|_{\gamma;K}
:=\sup_{\substack{\beta<\alpha<\gamma,\,
\tau\in T_\alpha,\\
(z,z')\in K^2}}
\frac{\|\Gamma_{zz'}\tau\|_\beta}{\|\tau\|_\alpha\|z-z'\|_s^{\alpha-\beta}}<\infty,\\
&\|\Pi\|_{\gamma;K}
:=\sup_{\substack{\alpha<\gamma,\,\tau\in T_\alpha,\\
z\in K,\,\varphi\in\mathcal{B}_r,\,\delta\in(0,1]}}
\frac{|\langle\Pi_z\tau,\varphi_z^\delta\rangle|}{\|\tau\|_\alpha\delta^\alpha}<\infty
\end{align*}
for every $\gamma>0$ and compact set $K\subset\RealNum^4$. For models $Z=(\Pi,\Gamma)$ and $Z'=(\Pi',\Gamma')$ on $\mathcal{T}$, we write
\[\$Z\$_{\gamma;K}=\|\Gamma\|_{\gamma;K}+\|\Pi\|_{\gamma;K},\quad\$Z-Z'\$_{\gamma;K}=\|\Gamma-\Gamma'\|_{\gamma;K}+\|\Pi-\Pi'\|_{\gamma;K}.\]
\end{definition}

Following \cite[Section~6]{Hairer2014a}, we define the space of modelled distributions with singularity at $P=\{z=(t,x)\in\RealNum^4\,;t=0\}$. For a subset $K\subset\RealNum^4$, we denote by
\[K_P:=\{(z,z')\in(K\setminus P)^2\,;z\neq z',\|z-z'\|_s\le1\wedge|t|^{\frac{1}{2}}\wedge|t'|^{\frac{1}{2}}\}.\]

\begin{definition}
Let $Z=(\Pi,\Gamma)$ be a model on $\mathcal{T}$, $\gamma>0$ and $\eta\in\RealNum$. For a function $f:\RealNum^4\to T_\gamma^-$ and a subset $K\subset\RealNum^4$, we define
\begin{align*}
\|f\|_{\gamma,\eta;K}
&:=\sup_{\substack{\beta<\gamma\\z=(t,x)\in K\setminus P}}
(1\wedge|t|^{\frac{\beta-\eta}{2}\vee0})\|f(z)\|_\beta,\\
\$f\$_{\gamma,\eta;K}
&:=\|f\|_{\gamma,\eta;K}
+\sup_{\substack{\beta<\gamma\\(z,z')\in K_P}}
(1\wedge|t|\wedge|t'|)^{\frac{\gamma-\eta}{2}}\frac{\|f(z)-\Gamma_{zz'}f(z')\|_\beta}{\|z-z'\|_s^{\gamma-\beta}}.
\end{align*}
We write $f\in\mathcal{D}^{\gamma,\eta}_P=\mathcal{D}^{\gamma,\eta}_P(Z)$ if $\$f\$_{\gamma,\eta;K}<\infty$ for every compact subset $K\subset\RealNum^4$. If $f$ takes value in a sector $V$, we write $f\in\mathcal{D}_P^{\gamma,\eta}(V;Z)$.

For models $Z$, $Z'$ and $f\in\mathcal{D}_P^{\gamma,\eta}(Z)$, $f'\in\mathcal{D}_P^{\gamma,\eta}(Z')$, we define
\begin{align*}
\$f;f'\$_{\gamma,\eta;K}
&=\|f-f'\|_{\gamma,\eta;K}\\
&\quad+\sup_{\substack{\alpha<\gamma,\\(z,z')\in K_P}}
(1\wedge|t|\wedge|t'|)^{\frac{\gamma-\eta}{2}}
\frac{\|f(z)-f'(z)-\Gamma_{zz'}f(z')+\Gamma_{zz'}'f'(z')\|_\alpha}{\|z-z'\|_s^{\gamma-\alpha}}.
\end{align*}
We denote by $\mathcal{M}\ltimes\mathcal{D}_P^{\gamma,\eta}$ the set of all pairs $(Z,f)$ of a model $Z$ and $f\in\mathcal{D}_P^{\gamma,\eta}(Z)$. The topology on $\mathcal{M}\ltimes\mathcal{D}_P^{\gamma,\eta}$ is defined by the family of pseudo-metrics $\{\$\cdot\,;\cdot\$_{\gamma,\eta;K}\}$.
\end{definition}

\begin{theorem}[{\cite[Theorem~3.10 and Proposition~6.9]{Hairer2014a}}]\label{3.1:reconstruction}
Let $Z=(\Pi,\Gamma)$ be a model on $\mathcal{T}$. Let $V$ be a sector with regularity $\alpha\le0$ and let $r=\lceil-\alpha\rceil$. If $\gamma>0$, $\eta\le\gamma$, and $\alpha\wedge\eta>-2$, then there exists a unique continuous linear map $\mathcal{R}:\mathcal{D}_P^{\gamma,\eta}(V)\to\mathcal{C}_s^{\alpha\wedge\eta}$ such that, if $K$ and $K'$ are compact subsets of $\RealNum^4$ such that $K$ is included in the interior of $K'$, then we have
\begin{align}\label{eq:reconstruction}
|\langle\mathcal{R}f-\Pi_zf(z),\varphi_z^\delta\rangle|
\lesssim\delta^\gamma\|\Pi\|_{\gamma;K'}\$f\$_{\gamma,\eta;K'},
\end{align}
uniformly over $f\in\mathcal{D}_P^{\gamma,\eta}(V)$, $\delta\in(0,1]$, $z\in K$ and $\varphi\in\mathcal{B}_r$ with $\varphi_z^\delta$ supported in $K'$ and uniformly away from $P$. Furthermore, the map $\mathcal{M}\ltimes\mathcal{D}_P^{\gamma,\eta}(V)\ni(Z,f)\to\mathcal{R}f\in\mathcal{C}_s^{\alpha\wedge\eta}$ is locally uniformly continuous.
\end{theorem}

\begin{remark}[{\cite[Lemma~6.7]{Hairer2014a}}]
The reconstruction operator $\mathcal{R}$ is local in the sense that, the behavior of $\mathcal{R}f$ on the compact set $K\subset\RealNum^4$ is uniquely determined by the values of $f$ and $\Pi$ in an arbitrary neighborhood of $K$.
\end{remark}

Next we introduce specific symbols and operators to describe \eqref{eq:cgl} by regularity structure: the polynomial structure, product, integration against Green's function, and the complex conjugate.

We have the regularity structure $\mathcal{T}^\poly$ given by all polynomials in the symbols $X_0,X_1,X_2,X_3$, which denote the time and space directions, respectively. Denote $X^k=\prod_{i=0}^3X_i^{k_i}$ for a multi-index $k\in\Integers_+^4$, and $\unit=X^{(0,0,0,0)}$. We endow these with the parabolic degrees $|X^k|_s=|k|_s$. Now we define the model space $T^\poly=\bigoplus_{n\in\Integers_+}T_n^\poly$, where
\[T_n^\poly=\langle X^k\,;|k|_s=n\rangle.\]
The group $G=\RealNum^4$ acts on $T^\poly$ by defining $\Gamma_hX^k=(X-h\unit)^k$ for every $h\in\RealNum^4$. Now we have the regularity structure $\mathcal{T}^\poly=(\Integers_+,T^\poly,\RealNum^4)$. Furthermore, we have the canonial model $(\Pi,\Gamma)$ on $\mathcal{T}^\poly$ given by
\begin{align}\label{3.1:model on poly}
(\Pi_zX^k)(z')=(z'-z)^k,\quad\Gamma_{zz'}=\Gamma_{z'-z},
\end{align}
for every $z,z'\in\RealNum^4$.

Throughout this section, the regularity structure $\mathcal{T}=(A,T,G)$ contains $\mathcal{T}^\poly$, i.e. $T^\poly$ is contained as a sector and the restriction of $G$ on $T^\poly$ coincides with $\{\Gamma_h\,;h\in\RealNum^4\}$. The model $(\Pi,\Gamma)$ acts on $T^\poly$ by \eqref{3.1:model on poly}. Furthermore, we assume that $T_n=T_n^\poly$ for every $n\in\Integers_+$.

\begin{proposition}[{\cite[Proposition~3.28]{Hairer2014a}}]\label{3.1:funclike reconst}
Let $V$ be a function-like sector which contains $T^\poly$ and such that $V\subset T^\poly+T_\alpha^+$ for some $\alpha>0$, where $T_\alpha^+:=\bigoplus_{\alpha\le\beta}T_\beta$. Let $\gamma>\alpha$, $\eta\in\RealNum$. Then for every $f\in\mathcal{D}_P^{\gamma,\eta}(V)$, $\mathcal{R}f$ coincides with the component of $f$ in $V_0=\langle\unit\rangle$ and belongs to $\mathcal{C}_s^\alpha((0,\infty)\times\RealNum^3)$, the space of functions $\varphi$ such that the estimate \eqref{3.1:def of C^alpha} holds uniformly over $z,z'\in K$ for every compact set $K\subset(0,\infty)\times\RealNum^3$.
\end{proposition}

For a pair of sectors $(V,W)$, a \emph{product} $*:V\times W\to T$ is a continuous bilinear map such that
\begin{itemize}
\item $V_\alpha*W_\beta\subset T_{\alpha+\beta}$ for every $\alpha,\beta\in A$,
\item $\unit*w=w$ for every $w\in W$ and $v*\unit=v$ for every $v\in V$,
\item $\Gamma(v*w)=(\Gamma v)*(\Gamma w)$ for every $(v,w)\in V\times W$ and $\Gamma\in G$.
\end{itemize}
The canonical product on $T^\poly$ is given by $X^k*X^l=X^{k+l}$.

\begin{proposition}[{\cite[Proposition~6.12]{Hairer2014a}}]\label{3.1:product}
Let $(V,W)$ be a pair of sectors with regularities $\alpha_1,\alpha_2$, respectively, and product $*:V\times W\to T$. For every $f_1\in\mathcal{D}_P^{\gamma_1,\eta_1}(V)$ and $f_2\in\mathcal{D}_P^{\gamma_2,\eta_2}(W)$, the function $f=f_1*f_2$ (projected onto $T_\gamma^-$) belongs to $\mathcal{D}_P^{\gamma,\eta}$ with $\gamma=(\gamma_1+\alpha_2)\wedge(\gamma_2+\alpha_1)$ and $\eta=(\eta_1+\alpha_2)\wedge(\eta_2+\alpha_1)\wedge(\eta_1+\eta_2)$. Furthermore, this bilinear map is locally uniformly continuous with respect to the topology of $\mathcal{M}\ltimes\mathcal{D}_P^{\gamma,\eta}$.
\end{proposition}

We say that a function $K:\RealNum^4\setminus\{0\}\to\CmplNum$ is a \emph{regularizing kernel} (of order $2$) if it can be written by $K=\sum_{n\ge0}K_n$, where $\{K_n\}$ satisfies the following assumptions.

\begin{assumption}\phantomsection\label{3.1:regularizing kernel}
\begin{itemize}
\item $K_n:\RealNum^4\to\CmplNum$ is smooth and supported in a ball $B_s(0,2^{-n})$.
\item There exists a constant $C>0$ such that
$\sup_z|\partial^kK_n(z)|\le C2^{(3+|k|_s)n}$
for every $n\ge0$ and $k\in\Integers_+^4$.
\item There exists $r>0$ such that
$\int_{\RealNum^4}K_n(z)z^kdz=0$
for every $n\ge0$ and $k$ with $|k|_s\le r$.
\end{itemize}
\end{assumption}

For a sector $V$, an \emph{abstract integration map} $\mathcal{I}:V\to T$ is a continuous linear map such that
\begin{itemize}
\item $\mathcal{I}V_\alpha\subset T_{\alpha+2}$ for every $\alpha\in A$ such that $\alpha+2\in A$,
\item $\mathcal{I}\tau=0$ for every $\tau\in V\cap T^\poly$,
\item $(\mathcal{I}\Gamma-\Gamma\mathcal{I})V\subset T^\poly$ for every $\Gamma\in G$.
\end{itemize}
Given a sector $V$ and an abstract integration map $\mathcal{I}$, we say that a model $(\Pi,\Gamma)$ \emph{realizes} a regularizing kernel $K$ for $\mathcal{I}$, if for every $\alpha\in A$, $\tau\in V_\alpha$ and $z\in\RealNum^4$ we have
\[\Pi_z\mathcal{I}\tau=K*(\Pi_z\tau)-\Pi_z\mathcal{J}(z)\tau,\]
where $\mathcal{J}(z)\tau=\sum_{|k|_s<\alpha+2}\frac{1}{k!}X^k(\partial^kK*\Pi_z\tau)(z)$. It is a consequence of \aref{3.1:regularizing kernel} that $(\partial^kK*\Pi_z\tau)(z)$ is defined for all $k$ with $|k|_s<\alpha+2$.

For $f\in\mathcal{D}_P^{\gamma,\eta}(V)$, we define the modelled distribution $\mathcal{K}_\gamma f$ by
\[(\mathcal{K}_\gamma f)(z)=\mathcal{I}f(z)+\mathcal{J}(z)f(z)+(\mathcal{N}_\gamma f)(z),\]
where $(\mathcal{N}_\gamma f)(z)=\sum_{|k|_s<\gamma+2}\frac{1}{k!}X^k\partial^kK*(\mathcal{R}f-\Pi_zf(z))(z)$.

\begin{proposition}[{\cite[Proposition~6.16]{Hairer2014a}}]\label{3.1:integration against K}
Let $V$ be a sector of regularity $\alpha$ and with an abstract integration map $\mathcal{I}$. Let $\gamma>0$ and $\eta<\gamma$. Assume that $\eta\wedge\alpha>-2$ and $\gamma+2,\eta+2\notin\Integers_+$. Then, for every model $Z$ realizing $K$ for $\mathcal{I}$, the operator $\mathcal{K}_\gamma$ maps $\mathcal{D}_P^{\gamma,\eta}(V)$ into $\mathcal{D}_P^{\gamma',\eta'}$ with $\gamma'=\gamma+2$, $\eta'=\eta\wedge\alpha+2$, and for every $f\in\mathcal{D}_P^{\gamma,\eta}(V)$, we have
\[\mathcal{R}\mathcal{K}_\gamma f=K*\mathcal{R}f.\]
Furthermore, the map $\mathcal{K}_\gamma:\mathcal{M}\ltimes\mathcal{D}_P^{\gamma,\eta}(V)\to\mathcal{D}_P^{\gamma',\eta'}$ is locally uniformly continuous.
\end{proposition}

\begin{remark}
Even in the case that $\eta\wedge\alpha\le-2$, if there exists a distribution $\mathcal{R}f\in\mathcal{C}_s^{\eta\wedge\alpha}$ which satisfies \eqref{eq:reconstruction}, then \pref{3.1:integration against K} still holds.
\end{remark}

For a sector $V$, a \emph{complex conjugate} map $V\ni\tau\mapsto\overline{\tau}\in T$ is a map such that
\begin{itemize}
\item $\tau\mapsto\overline{\tau}$ is continuous and antilinear, in the sense that $\overline{\lambda_1\tau_1+\lambda_2\tau_2}=\overline{\lambda_1}\overline{\tau_1}+\overline{\lambda_2}\,\overline{\tau_2}$ for $\lambda_1,\lambda_2\in\CmplNum$ and $\tau_1,\tau_2\in V$,
\item $\overline{V_\alpha}\subset T_\alpha$ for every $\alpha\in A$,
\item $\overline{X^k}=X^k$ for every $k\in\Integers_+^4$,
\item $\Gamma\overline{\tau}=\overline{\Gamma\tau}$ for every $\tau\in V$ and $\Gamma\in G$.
\end{itemize}
For such $V$, the set $\overline{V}=\{\overline{\tau}\,;\tau\in V\}$ is also a sector.
We assume that a model $(\Pi,\Gamma)$ is compatible with the complex conjugate, i.e.
\[\Pi_z\overline{\tau}=\overline{\Pi_z\tau}\]
for every $z\in\RealNum^4$ and $\tau\in T$.
Then we can see that the map $\mathcal{D}_P^{\gamma,\eta}(V)\ni f\mapsto\overline{f}\in\mathcal{D}_P^{\gamma,\eta}$ is continuous antilinear, and $\mathcal{R}\overline{f}=\overline{\mathcal{R}f}$ holds.

\subsection{Regularity structures associated with CGL and admissible models}

For smooth $\xi$, the CGL equation \eqref{eq:cgl} is equivalent to the mild form
\begin{align}\label{3.2:mild form}
u=G*\{\unit_{t>0}(\xi+\nu(1-u\overline{u})u)\}+Gu_0,
\end{align}
where $*$ denotes the space-time convolution, $G$ is the fundamental solution of
\begin{align}\label{3.2:fundamental solution}
\partial_tG=(\mathsf{i}+\mu)\triangle G,\quad (t,x)\in\RealNum_+\times\RealNum^3,
\end{align}
with initial condition $G(0,\cdot)=\delta_0$,
extended into the function $G:\RealNum^4\setminus\{0\}\to\CmplNum$ by $G(t,\cdot)\equiv0$ if $t\le0$, and $Gu_0$ denotes the solution of \eqref{3.2:fundamental solution} with initial condition $u_0$.

We construct a regularity structure associated with \eqref{3.2:mild form} by following \cite[Section~8.1]{Hairer2014a}. We assumed that polynomials $\{X^k\}$ are contained in our regularity structure. Additionally we have symbols $\Xi$ (noise), an abstract integration $\mathcal{I}$ (space-time convolution with $G$), and the complex conjugate. Inspired from \eqref{3.2:mild form}, we can define $\tilde{\cal{F}}$ as the smallest set of symbols such that $\{\Xi,X^k\}\subset\tilde{\mathcal{F}}$ and closed for the operations:
\begin{itemize}
\item If $\tau,\tau'\in\tilde{\mathcal{F}}$, then $\tau\tau'=\tau'\tau\in\tilde{\mathcal{F}}$.
\item If $\tau\in\tilde{\mathcal{F}}$, then $\overline{\tau}\in\tilde{\mathcal{F}}$, where we set $\overline{X^k}=X^k$.
\item If $\tau\in\tilde{\mathcal{F}}\setminus\{X^k\}$, then $\mathcal{I}\tau\in\tilde{\mathcal{F}}$.
\end{itemize}
For a fixed number $\alpha<-\frac{5}{2}$, we define the homogeneity of each variable by
\begin{align*}
&|\Xi|_s=\alpha,\quad|X^k|_s=|k|_s,\quad
|\tau\tau'|_s=|\tau|_s+|\tau'|_s,\quad
|\overline{\tau}|_s=|\tau|_s,\quad
|\mathcal{I}\tau|_s=|\tau|_s+2.
\end{align*}
However, $\tilde{\mathcal{F}}$ is too big. Precisely we consider the subsets $\mathcal{U}$ and $\mathcal{W}$, which are defined by the smallest sets such that $\{\Xi,X^k\}\subset\mathcal{W}$, $\{X^k\}\subset\mathcal{U}$ and
\[\tau\in\mathcal{W}\Rightarrow\mathcal{I}\tau\in\mathcal{U},\quad
\tau_1,\tau_2,\tau_3\in\mathcal{U}\Rightarrow\tau_1\tau_2\overline{\tau_3}\in\mathcal{W}.\]
We set $\mathcal{F}=\mathcal{U}\cup\mathcal{W}$, and define
\[T_\beta=\langle\tau\in\mathcal{F}\,;|\tau|_s=\beta\rangle,\quad
T=\langle\mathcal{F}\rangle,\quad
U=\langle\mathcal{U}\rangle,\quad
W=\langle\mathcal{W}\rangle.\]
We can see that $T$ contains all polynomials $T^\poly$, and furthermore, the abstract integration $\mathcal{I}$, the complex conjugate, and the product $U\times U\times \overline{U}\to W$ are well defined. Here $\overline{U}=\langle\overline{\tau}\,;\tau\in\mathcal{U}\rangle$.

\begin{remark}
We do not assume identifications of symbols
\[\overline{\overline{\tau}}=\tau,\quad\overline{\tau_1\tau_2}=\overline{\tau_1}\,\overline{\tau_2}\]
since $\overline{\overline{\tau}}$ and $\overline{\tau_1}\,\overline{\tau_2}$ are not involved in the definition of $\mathcal{F}$.
\end{remark}

In order to define $T$ as a model space of a regularity structure, the set $\{|\tau|_s\,;\tau\in\mathcal{F}\}$ must be bounded from below. A nonlinear SPDE is called \emph{subcritical}, if the nonlinear terms formally disappear in some scaling which keeps the linear part and the noise term invariant. This is equivalent to the property that all symbols except $\Xi$ defined as above have homogeneities strictly greater than $|\Xi|_s$ (\cite[Assumption~8.3]{Hairer2014a}). In the present case, this is equivalent to $|(\mathcal{I}\Xi)^2\overline{\mathcal{I}\Xi}|_s=3(2+\alpha)>\alpha$, or $\alpha>-3$.

We need to define the structure group acting on $T$. Let $T^+$ be the complex free commutative algebra generated by abstract symbols
\[\mathcal{F}^+=\{X^k\}\cup\{\mathcal{J}_k\tau,\overline{\mathcal{J}_k\tau}\,; \tau\in\mathcal{F}\setminus\{X^k\},|k|_s<|\tau|_s+2\}.\]
We define the homogeneity of each variable by
\[|X^k|_s=|k|_s,\quad|\tau\tau'|_s=|\tau|_s+|\tau'|_s,\quad
|\mathcal{J}_k\tau|_s=|\overline{\mathcal{J}_k\tau}|_s=|\tau|_s+2-|k|_s.\]
In the following, we will view $\mathcal{J}_k$ as a map from $T$ to $T^+$, by defining $\mathcal{J}_k\tau=0$ if $\tau=X^k$ or $|\tau|_s+2-|k|_s\le0$, and linearly extending it for all $\tau\in T$.

We construct two linear maps $\Delta$ and $\Delta^+$ recursively as follows. The linear map $\Delta:T\to T\otimes T^+$ is defined by
\begin{align*}
&\Delta\unit=\unit\otimes\unit,\quad
\Delta X_i=X_i\otimes\unit+\unit\otimes X_i\ \ (i=0,1,2,3),\quad
\Delta\Xi=\Xi\otimes\unit,\\
&\Delta(\tau\tau')=(\Delta\tau)(\Delta\tau'),\quad
\Delta\overline{\tau}=\overline{\Delta\tau},\\
&\Delta\mathcal{I}\tau=(\mathcal{I}\otimes\id_{T^+})\Delta\tau+\sum_{l,m}\frac{X^l}{l!}\otimes\frac{X^m}{m!}\mathcal{J}_{l+m}\tau.
\end{align*}
The linear map $\Delta^+:T^+\to T^+\otimes T^+$ is defined by
\begin{align*}
&\Delta^+\unit=\unit\otimes\unit,\quad
\Delta^+ X_i=X_i\otimes\unit+\unit\otimes X_i\ \ (i=0,1,2,3),\\
&\Delta^+(\tau\tau')=(\Delta^+\tau)(\Delta^+\tau'),\quad
\Delta^+\overline{\tau}=\overline{\Delta^+\tau},\\
&\Delta^+\mathcal{J}_k\tau=\sum_l\left(\mathcal{J}_{k+l}\otimes\frac{(-X)^l}{l!}\right)\Delta\tau+\unit\otimes\mathcal{J}_k\tau.
\end{align*}
Then by \cite[Theorem~8.16]{Hairer2014a}, the pair $(T^+,\Delta^+)$ is a Hopf algebra, i.e. $\Delta^+$ satisfies the identity
\[(\id_{T^+}\otimes\Delta^+)\Delta^+=(\Delta^+\otimes\id_{T^+})\Delta^+\]
and the algebra homomorphism $\unit^*:T^+\to T^+$ defined by $\unit^*(\unit)=1$ and $\unit^*(\tau)=0$ for $\tau\in\mathcal{F}^+\setminus\{\unit\}$ is a counit in the sense that
\[(\unit^*\otimes\id_{T^+})\Delta^+=(\id_{T^+}\otimes\unit^*)\Delta^+=\id_{T^+},\]
and furthermore, the algebra homomorphism $\mathcal{A}:T^+\to T^+$ recursively defined by
\begin{align*}
&\mathcal{A}X^k=(-X)^k,\quad
\mathcal{A}\overline{\tau}=\overline{\mathcal{A}\tau},\quad
\mathcal{A}\mathcal{J}_k\tau=-\sum_l\mathcal{M}\left(\mathcal{J}_{k+l}\otimes\frac{X^l}{l!}\mathcal{A}\right)\Delta\tau,
\end{align*}
where $\mathcal{M}:T^+\otimes T^+\ni\tau\otimes\tau'\mapsto\tau\tau'\in T^+$, is an antipode of $T^+$ in the sense that
\[\mathcal{M}(\id_{T^+}\otimes\mathcal{A})\Delta^+=\unit^*=\mathcal{M}(\mathcal{A}\otimes \id_{T^+})\Delta^+.\]
The pair $(T,\Delta)$ is a comodule over $T^+$, i.e. $\Delta$ satisfies the identity
\[(\id_{T}\otimes\Delta^+)\Delta=(\Delta\otimes\id_{T^+})\Delta.\]

We denote by $G$ the set of algebra homomorphisms $g:T^+\to\CmplNum$ such that $g(\overline{\tau})=\overline{g(\tau)}$ for every $\tau\in T^+$. Then $G$ is a group with the product $\circ$ defined by
\[g\circ g'=(g\otimes g')\Delta^+,\quad g,g'\in G.\]
The inverse of $g\in G$ is given by $g^{-1}=g\mathcal{A}$. Each $g\in G$ acts on $T$ as the operator $\Gamma_g\in\mathcal{L}(T)$ defined by
\[\Gamma_g\tau=(\id_T\otimes g)\Delta\tau,\quad\tau\in T.\]

The following theorem is a modification of \cite[Theorem~8.24]{Hairer2014a}.

\begin{theorem}
Let $\alpha\in(-3,-\frac{5}{2})$ and $A=\{|\tau|_s\,;\tau\in\mathcal{F}\}$. Then $\mathcal{T}_\cgl:=(A,T,G)$ is a regularity structure which contains the polynomial structure $\mathcal{T}^\poly$ and has the complex conjugate on $U$, the abstract integration map $\mathcal{I}:W\to T$, and the products $*:U\times U\to UU$ and $*:UU\times\overline{U}\to W$, where $UU=\langle\tau\tau'\,;\tau,\tau'\in U\rangle$.
\end{theorem}

We introduce a class of suitable models associated with $\mathcal{T}_\cgl$. Let $K$ be a regularizing kernel satisfying \aref{3.1:regularizing kernel} with $r>0$. We denote by $\mathcal{T}_\cgl^{(r)}$ the regularity structure obtained by $T_\gamma=0$ for $\gamma>r$.

\begin{definition}
We say that a model $(\Pi,\Gamma)$ on $\mathcal{T}_\cgl^{(r)}$ is \emph{admissible}, if
\begin{itemize}
\item $\Pi$ realizes $K$ for $\mathcal{I}$, compatible with the complex conjugate, and satisfies \eqref{3.1:model on poly},
\item $\Gamma_{zz'}=(\Gamma_{f_z})^{-1}\Gamma_{f_{z'}}$, where $f_z\in G$ is defined by $f_z(X^k)=(-z)^k$ and
\begin{align*}
f_z(\mathcal{J}_k\tau)=-\partial^k K*(\Pi_z\tau)(z),\quad|\tau|_s+2-|k|_s>0.
\end{align*}
\end{itemize}
\end{definition}

If the model $(\Pi,\Gamma)$ is admissible, then the map $\Pi_z(\Gamma_{f_z})^{-1}:T\to\mathcal{S}'$ is independent to $z$, so that we can write $\mathbf{\Pi}=\Pi_z(\Gamma_{f_z})^{-1}$. Furthermore we have
\[(\mathbf{\Pi}X^k)(z)=z^k,\quad
\mathbf{\Pi}\mathcal{I}\tau=K*(\mathbf{\Pi}\tau),\quad
\mathbf{\Pi}\overline{\tau}=\overline{\mathbf{\Pi}\tau}.\]
Conversely, if a linear map $\mathbf{\Pi}:T\to\mathcal{S}'$ satisfies these conditions, and a family $\{f_z\,;z\in G\}$ satisfies $f_z(X^k)=(-z)^k$ and $f_z(\mathcal{J}_k\tau)=-\partial^k K*(\mathbf{\Pi}\Gamma_{f_z}\tau)(z)$ for $\mathcal{J}_k\tau$ with $|\tau|_s+2-|k|_s>0$,
then the corresponding admissible model $(\Pi,\Gamma)$ is uniquely determined.

We assume that the model is periodic in the space direction. For $n\in\Integers^3$ and $z=(t,x)\in\RealNum^4$, we write $S_nz=(t,x+n)$.

\begin{definition}
We say that a model $(\Pi,\Gamma)$ on $\mathcal{T}_\cgl^{(r)}$ is \emph{periodic} if
\[(\Pi_{S_nz}\tau)(S_nz')=(\Pi_z\tau)(z'),\quad
\Gamma_{(S_nz)(S_nz')}=\Gamma_{zz'},\]
for every $z,z'\in\RealNum^4$ and $n\in\Integers^3$.
\end{definition}

\subsection{Abstract solution map}

In the regularity structures constructed above, we can reformulate \eqref{eq:cgl} as a fixed point problem in the space $\mathcal{D}_P^{\gamma,\eta}$, by following \cite{Hairer2014a}. First, note that the fundamental solution of \eqref{3.2:fundamental solution} is given by
\[G(t,x)=\unit_{t>0}\frac{1}{\sqrt{4\pi(\mathsf{i}+\mu)t}^3}\exp\left(-\frac{|x|^2}{4(\mathsf{i}+\mu)t}\right).\]
Here, for $\lambda=re^{\mathsf{i}\theta}\in\CmplNum\ (r>0,\,\theta\in(-\frac\pi2,\frac\pi2))$, the square root is defined by $\sqrt{\lambda}=\sqrt{r}e^{\mathsf{i}\theta/2}$. Hence $G$ has the form $G(t,x)=t^{-3/2}\hat{G}(t^{-1/2}x)$ for some $\hat{G}\in\mathcal{S}(\RealNum^3,\CmplNum)$ when $t>0$, which satisfies the condition in \cite[Lemma~7.4]{Hairer2014a}.

\begin{lemma}[{\cite[Lemma~7.7]{Hairer2014a}}]
There exist a regularizing kernel $K$ and a smooth function $R$ with compact support such that
\[(G*u)(z)=(K*u)(z)+(R*u)(z)\]
holds for every periodic function $u$ supported in $\RealNum_+\times\RealNum^3$ and $z\in(-\infty,1]\times\RealNum^3$. Furthermore, $K$ and $R$ are supported in $\RealNum_+\times\RealNum^3$, and $K$ satisfies \aref{3.1:regularizing kernel} with arbitrary fixed $r>0$.
\end{lemma}

For a periodic distribution $\xi\in\mathcal{S}'$, we define the modelled distribution
\[(R_\gamma \xi)(z)=\sum_{|k|_s<\gamma}\frac{X^k}{k!}\partial^kR*\xi(z).\]



Now we reformulate \eqref{3.2:mild form} as a fixed point problem in $\mathcal{D}_P^{\gamma,\eta}$. First, for every periodic initial condition $u_0\in\mathcal{C}^\eta$ with $\eta<0$ and $\eta\notin\Integers$, the function $Gu_0$ is canonically lifted to an element of $\mathcal{D}_P^{\gamma,\eta}$ for every $\gamma>\eta$, by defining
$(Gu_0)(z)=\sum_{|k|_s<\gamma}\frac{1}{k!}X^k\partial^kGu_0(z)$ (\cite[Lemma~7.5]{Hairer2014a}). Second, note that by \pref{3.1:product}, the map $u\mapsto u^2\overline{u}$ is locally Lipschitz continuous from $\mathcal{D}_P^{\gamma,\eta}(U)$ to $\mathcal{D}_P^{\gamma+2\alpha+4,3\eta}$, if $\gamma>|2\alpha+4|$ and $\eta\le\alpha+2$. Therefore we can consider the problem
\begin{align}\label{eq:fixed point problem}
u=(\mathcal{K}_{\gamma+2\alpha+4}+R_\gamma\mathcal{R})(\unit_{t>0}F(u))+Gu_0,\quad F(u)=\Xi+\nu(1-u\overline{u})u,
\end{align}
in $u\in\mathcal{D}_P^{\gamma,\eta}$. However, $F(u)$ takes values in the sector $U$ of regularity $\alpha=|\Xi|_s<-\frac{5}{2}<-2$, so that \tref{3.1:reconstruction} is not sufficient to define $\mathcal{R}F(u)$. In order to overcome this problem, we impose the following assumption on the distribution $\xi=\mathbf{\Pi}\Xi=\Pi_z\Xi$. (Since $\Xi$ is $G$-invariant, $\Pi_z\Xi$ is independent to $z$.)

\begin{assumption}\phantomsection\label{3.3:ass on xi}
\begin{enumerate}[(1)]
\item For $\alpha<0$, we denote by $\bar{\mathcal{C}}_s^\alpha$ the completion of smooth functions under the family of norms:
\[\bbb\xi\bbb_{\alpha;K}:=\sup_{s\in\RealNum}\|\unit_{t>s}\xi\|_{\alpha;K}\]
for all compact sets $K\subset\RealNum^4$. We assume that $\xi=\mathbf{\Pi}\Xi$ belongs to $\bar{\mathcal{C}}_s^\alpha$ for $\alpha=|\Xi|_s$.
\item $K*\xi$ belongs to the space $C(\RealNum,\mathcal{C}^{\alpha+2})$.
\end{enumerate}
\end{assumption}

Under the assumption $\xi\in\bar{\mathcal{C}}_s^\alpha$, we can define $\mathcal{R}(\unit_{t>0}\Xi):=\unit_{t>0}\xi$, so that $\mathcal{K}(\unit_{t>0}\Xi)$ is also defined.

The following theorem is a modification of \cite[Theorem~7.8 and Proposition~9.8]{Hairer2014a}. We denote by $O_T=(-\infty,T]\times\RealNum^3$ and $\$\cdot\$_{\gamma,\eta;T}:=\$\cdot\$_{\gamma,\eta;O_T}$.

\begin{theorem}\label{3.3:solution map}
Let $\alpha\in(-\frac{18}{7},-\frac{5}{2})$, $\gamma>|2\alpha+4|$ and $\eta\in(-\frac{2}{3},\alpha+2)$. Assume that the regularizing kernel $K$ satisfies \aref{3.1:regularizing kernel} with $r>\gamma+2\alpha+6(>2)$. Then for every admissible and periodic model $Z=(\Pi,\Gamma)$ satisfying \aref{3.3:ass on xi} and every periodic $u_0\in\mathcal{C}^\eta$, there exists $T\in(0,\infty]$ such that the fixed point problem \eqref{eq:fixed point problem} admits a unique solution $u\in\mathcal{D}_P^{\gamma,\eta}(U)$ on $(0,T)$. The time $T$ can be chosen maximal in the sense that $\lim_{t\uparrow T}\|\mathcal{R}u(t,\cdot)\|_{\mathcal{C}^\eta}=\infty$ unless $T=\infty$.

Furthermore, the solution $u$ and the survival time $T$ depend on $(u_0,Z,\xi,K*\xi)$ locally uniformly continuously and locally uniformly lower semi-continuously, respectively, in the topology of $\mathcal{C}^\eta\times\mathcal{M}\times\bar{\mathcal{C}}_s^\alpha\times C(\RealNum,\mathcal{C}^{\alpha+2})$.
\end{theorem}

\begin{proof}
We consider $\Xi$ and $F^+(u)=\nu(1-u\overline{u})u$ separately. Here $F^+(u)$ takes values in the sector $W^+$ with regularity $|(\mathcal{I}\Xi)^2\overline{\mathcal{I}\Xi}|_s=3(\alpha+2)$. Let $\mathcal{G}_\gamma=\mathcal{K}_{\gamma+2\alpha+4}+R_\gamma\mathcal{R}$. The modelled distribution $\unit_{t>0}\Xi$ belongs to $\mathcal{D}_P^{\gamma,\alpha}$ for every $\gamma$. Hence under \aref{3.3:ass on xi}-(1), $\mathcal{G}_\gamma\Xi$ is defined as an element of $\mathcal{D}_P^{\gamma,\alpha+2}(U)$. On the other hand, $\mathcal{G}_\gamma$ maps $\mathcal{D}_P^{\gamma+2\alpha+4,3\eta}(W^+)$ into $\mathcal{D}_P^{\gamma,3\eta+2}(U)$ provided that $3\eta>-2$, as a consequence of \pref{3.1:integration against K}. Furthermore, following the arguments in \cite[Theorem~7.1]{Hairer2014a}, we have the bound
\[\$\mathcal{G}_\gamma(\unit_{t>0}(\Xi+F^+(u)))\$_{\gamma,\eta;T}
\lesssim T^\frac{\alpha+2-\eta}{2}(1+\$F^+(u)\$_{\gamma+2\alpha+4,3\eta;T})\]
for every periodic $u\in\mathcal{D}_P^{\gamma,\eta}(U)$. As in \cite[Theorem~7.8]{Hairer2014a}, this yields that there exists small $T>0$ such that \eqref{eq:fixed point problem} admits a unique solution $u\in\mathcal{D}_P^{\gamma,\eta}(U)$ on $(0,T)$.

To glue local solutions up to maximal time where the solution exists, note that $\mathcal{R}u$ belongs to the space $C((0,T),\mathcal{C}^\eta)$, even though $\eta<0$. Indeed, the solution can be written by $u=\mathcal{I}\Xi+u^+$, where $u^+$ takes values in the function-like sector $U^+$. As in \pref{3.1:funclike reconst}, $\mathcal{R}u^+$ is H\"older continuous. By \aref{3.3:ass on xi}-(2), $\mathcal{R}\mathcal{I}\Xi=K*\xi$ belongs to $C(\RealNum,\mathcal{C}^\eta)$. For $s\in(0,T)$, we start from $u_s\in\mathcal{C}^\eta$ and consider the problem
\[u=\mathcal{G}_\gamma(\unit_{t>s}(\Xi+F^+(u)))+Gu_s,\]
which is well-posed by defining $\mathcal{R}(\unit_{t>s}\Xi):=\unit_{t>s}\xi$. This can extend the time interval where the local solution exists, following \cite[Proposition~7.11]{Hairer2014a}. The existence of maximal solution and its continuity with respect to $(u_0,Z,\xi,K*\xi)$ are obtained by standard arguments in PDE theory.
\end{proof}

\subsection{Renormalization}

For each $\epsilon>0$, the noise $\xi^\epsilon$ defined in the beginning of \secref{sec_20161107073826} can be lifted to an admissible and periodic model $Z^\epsilon=(\Pi^\epsilon,\Gamma^\epsilon)$ on $\mathcal{T}_\cgl^{(r)}$, by defining the linear map $\mathbf{\Pi}^\epsilon:T\to C^\infty(\RealNum^4)$ with the additional assumptions:
\begin{align*}
\mathbf{\Pi}^\epsilon\Xi=\xi^\epsilon,\quad
\mathbf{\Pi}^\epsilon(\tau\tau')=(\mathbf{\Pi}^\epsilon\tau)(\mathbf{\Pi}^\epsilon\tau').
\end{align*}
Furthermore, $Z^\epsilon$ has the property that $\Pi_z\tau$ is a smooth function for every $\tau\in T$ and $z\in\RealNum^4$, then as a consequence, $\mathcal{R}^\epsilon f$ is also smooth and satisfies
\[(\mathcal{R}^\epsilon f)(z)=(\Pi_z^\epsilon f(z))(z)\]
for every modelled distribution $f$ (\cite[Remark~3.15]{Hairer2014a}).

We introduce a renormalization of $Z^\epsilon$ following \cite[Section~8.3]{Hairer2014a}. Let $\mathcal{F}_0\subset\mathcal{F}$ be a subset such that $\{\tau\in\mathcal{F}\,;|\tau|_s\le0\}\subset\mathcal{F}_0$, and there exists a subset $\mathcal{F}_*\subset\mathcal{F}_0$ such that $\Delta\mathcal{F}_0\subset \langle\mathcal{F}_0\rangle\otimes T_0^+$, where $T_0^+$ is the complex free commutative algebra generated by symbols
\[\mathcal{F}_0^+=\{X^k\}\cup\{\mathcal{J}_k\tau,\overline{\mathcal{J}_k\tau}\,; \tau\in\mathcal{F}_*,|k|_s<|\tau|_s+2\}.\]
Let $M:\langle\mathcal{F}_0\rangle\to \langle\mathcal{F}_0\rangle$ be a linear map such that
\[M\mathcal{I}\tau=\mathcal{I}M\tau,\quad
M\overline{\tau}=\overline{M\tau},\quad
MX^k=X^k.\]
Then two linear maps $\hat{M}:T_0^+\to T_0^+$ and $\Delta^M:\langle\mathcal{F}_0\rangle\to \langle\mathcal{F}_0\rangle\otimes T_0^+$ are uniquely determined by
\begin{align*}
&\hat{M}X^k=X^k,\quad
\hat{M}(\tau\tau')=(\hat{M}\tau)(\hat{M}\tau'),\\
&\hat{M}\mathcal{J}_k\tau=\mathcal{M}(\mathcal{J}_k\otimes\id)\Delta^M\tau,\quad
\hat{M}\overline{\mathcal{J}_k\tau}=\overline{\hat{M}\mathcal{J}_k\tau},
\end{align*}
and
\[(\id\otimes\mathcal{M})(\Delta\otimes\id)\Delta^M\tau=(M\otimes\hat{M})\Delta\tau,
\]
since $(\id\otimes\mathcal{M})(\Delta\otimes\id)$ is invertible. Furthermore, the linear map $\hat{\Delta}^M:T_0^+\to T_0^+\otimes T_0^+$ is defined by
\[(\mathcal{A}\hat{M}\mathcal{A}\otimes\hat{M})\Delta^+=(\id\otimes\mathcal{M})(\Delta^+\otimes\id)\hat{\Delta}^M,\]
since $(\id\otimes\mathcal{M})(\Delta^+\otimes\id)$ is invertible (\cite[Proposition~8.36]{Hairer2014a}).

\begin{theorem}[{\cite[Theorem~8.44]{Hairer2014a}}]\label{extension of renormalized model}
Consider $\mathcal{F}_0$ and $M$ as above. Assume that for every $\tau\in\mathcal{F}_0$ and $\hat{\tau}\in T_0^+$ we can write
\[\Delta^M\tau=\tau\otimes\unit+\sum_{|\tau^{(1)}|_s>|\tau|_s}
\tau^{(1)}\otimes\tau^{(2)},\quad
\hat{\Delta}^M\hat{\tau}=\hat{\tau}\otimes\unit
+\sum_{|\hat{\tau}^{(1)}|_s>|\hat{\tau}|_s}
\hat{\tau}^{(1)}\otimes\hat{\tau}^{(2)}.\]
Then for every admissible model $(\mathbf{\Pi},f)$ on $\mathcal{T}_\cgl^{(r)}$, the maps $\mathbf{\Pi}^M$ and $f_z^M$ defined by
\[\mathbf{\Pi}^M=\mathbf{\Pi}M,\quad f_z^M=f_z\hat{M}\]
are uniquely extended to an admissible model $Z^M=(\Pi^M,\Gamma^M)$ on $\mathcal{T}_\cgl^{(r)}$.
\end{theorem}

Now we give a renormalization map $M$ in a concrete form. In order to simplify notations, we introduce a graphical notation for the element in $\mathcal{F}$. First, we draw a circle to represent $\Xi$. For an element $\mathcal{I}\tau$, we draw a downward black line starting at the root of $\tau$. For a product $\tau\tau'$, we joint these trees at their roots. The complex conjugate $\overline{\tau}$ is denoted by changing the color black and white to each other. For example,
\[\mathcal{I}\Xi=\IX,\quad
\mathcal{I}\Xi^2=\CC,\quad
\mathcal{I}\Xi^2\overline{\mathcal{I}\Xi}=\CCD,\quad
\mathcal{I}\CCD=\ICCD,\quad
\overline{\ICCD}\CC=\CCJDDC.\]
Then we can list all of elements with negative homogeneities as follows:
\begin{center}
\begin{tabular}{cc}\hline
Homogeneity&Symbol\\\hline
$\alpha$&$\Xi$\\
$3(\alpha+2)$&$\CCD$\\
$2(\alpha+2)$&$\CC$, $\CD$\\
$5\alpha+12$&$\CDICCD$, $\CCJDDC$\\
$\alpha+2$&$\IX$, $\JX$\\
$4\alpha+10$&$\CICCD$, $\DICCD$, $\CJDDC$, $\CDICC$, $\CCJDD$, $\CDICD$, $\CCJDC$\\
$2\alpha+5$&$X_i\CC$, $X_i\CD$ ($i=1,2,3$)\\
$0$&$\unit$\\\hline
\end{tabular}
\end{center}
Since $\alpha>-\frac{18}{7}$, the element $\DHE$ has positive homogeneity $7\alpha+18>0$,
so that it does not appear here.

Considering chaos expansions of Gaussian models as in \secref{section 3.5}, we can define the renormalization map $M=M(C_1,C_{2,1},C_{2,2})$ by
\begin{align*}
M\CC&=\CC,\\
M\CD&=\CD-C_1\unit,\\
M\CCD&=\CCD-2C_1\IX,\\
M\CDICC&=\ICC(\CD-C_1\unit)=\CDICC-C_1\ICC,\\
M\CCJDD&=\CCJDD-2C_{2,1}\unit,\\
M\CDICD&=\ICD(\CD-C_1\unit)-C_{2,2}\unit=\CDICD-C_1\ICD-C_{2,2}\unit,\\
M\CCJDC&=\CCJDC,\\
M\CICCD&=(\ICCD-2C_1\IC)\IX=\CICCD-2C_1\CIC,\\
M\DICCD&=(\ICCD-2C_1\IC)\JX=\DICCD-2C_1\DIC,\\
M\CJDDC&=(\JDDC-2\overline{C_1}\JD)\IX=\CJDDC-2\overline{C_1}\CJD,\\
M\CDICCD&=(\ICCD-2C_1\IC)(\CD-C_1\unit)-2C_{2,2}\IX\\
&=\CDICCD-2C_1\CDIC-C_1(\ICCD-2C_1\IC)-2C_{2,2}\IX,\\
M\CCJDDC&=(\JDDC-2\overline{C_1}\JD)\CC-2C_{2,1}\IX=\CCJDDC-2\overline{C_1}\CCJD-2C_{2,1}\IX,\\
MX_i\CC&=X_i\CC,\\
MX_i\CD&=X_i(\CD-C_1\unit)=X_i\CD-C_1X_i
\end{align*}
for some constants $C_1,C_{2,1}$ and $C_{2,2}$. Since $M$ must be closed in the space $\langle\mathcal{F}_0\rangle$, we should choose $\mathcal{F}_0$ by
\begin{multline*}
\mathcal{F}_0=\{\Xi,\CCD,\CC,\CD,\CDICCD,\CCJDDC,\IX,\JX,\CICCD,\DICCD,\CJDDC,\CDICC,\CCJDD,\CDICD,\CCJDC,X_i\CC,X_i\CD,\unit,\\
\CDIC,\CCJD,\ICCD,\CIC,\DIC,\CJD,\ICC,\ICD,X_i\IX,X_i\JX,\IC\, ; i=1,2,3\}.
\end{multline*}
Then it turns out that we can take $\mathcal{F}_*=\{\IX,\JX,\CC,\CD,\CCD\}$. From now on, the subscript $i$ of $X_i$ runs over $\{1,2,3\}$.

\begin{lemma}
The linear map $M$ satisfies the conditions of \tref{extension of renormalized model}. Furthermore, the identity
\begin{align}\label{3.4:PiM}
(\Pi_z^M\tau)(z)=(\Pi_zM\tau)(z)
\end{align}
holds for every $\tau\in\mathcal{F}_0$ and $z\in\RealNum^4$.
\end{lemma}

\begin{proof}
Calculations of $\hat{M}$, $\Delta^M$ and $\hat{\Delta}^M$ are completely parallel to those in \cite[Section~9.2]{Hairer2014a}, so here we show only the results. Indeed we have
\[\hat{M}\mathcal{J}\tau=\mathcal{J}M\tau,\quad
\hat{M}\mathcal{J}_k\tau=\mathcal{J}_k\tau,\quad(|k|_s>0)\]
and
\begin{align*}
\Delta^M\ICCD&=M\ICCD\otimes\unit+2C_1X_i\otimes\mathcal{J}_i\IX,\\
\Delta^M\CICCD&=M\CICCD\otimes\unit+2C_1X_i\IX\otimes\mathcal{J}_i\IX,\\
\Delta^M\DICCD&=M\DICCD\otimes\unit+2C_1X_i\JX\otimes\mathcal{J}_i\IX,\\
\Delta^M\CJDDC&=M\CJDDC\otimes\unit+2\overline{C_1}X_i\IX\otimes\overline{\mathcal{J}_i\IX},\\
\Delta^M\CDICCD&=M\CDICCD\otimes\unit+2C_1X_i(\CD-C_1\unit)\otimes\mathcal{J}_i\IX,\\
\Delta^M\CCJDDC&=M\CCJDDC\otimes\unit+2\overline{C_1}X_i\CC\otimes\overline{\mathcal{J}_i\IX},\\
\Delta^M\tau&=M\tau\otimes\unit\quad(\text{otherwise}).
\end{align*}
Furthermore,
\begin{align*}
\hat{\Delta}^M\mathcal{J}\CCD&=\mathcal{J}M\CCD+2C_1(X_i\otimes\mathcal{J}_i\IX-X_i\mathcal{J}_i\IX\otimes\unit),\\
\hat{\Delta}^M\tau&=\hat{M}\tau\otimes\unit\quad(\text{otherwise}).
\end{align*}
(Here and in what follows, summation symbols over the repeated index $i$ are omitted.) Therefore, $M$ satisfies the conditions of \tref{extension of renormalized model}. The relation \eqref{3.4:PiM} is obtained by
\begin{align*}
\Pi_z^M\tau=(\Pi_z\otimes f_z)\Delta^M\tau
\end{align*}
(\cite[equation~(8.34)]{Hairer2014a}) and the fact that $\Pi_zX_i\tau(z)=0$ for every $\tau$ with $X_i\tau\in\mathcal{F}$.
\end{proof}

\begin{proposition}\label{0360:renormalized equation}
Let $Z^\epsilon=(\Pi^\epsilon,\Gamma^\epsilon)$ be a model canonically lifted from a continuous function $\xi^\epsilon$. Let $S:(u_0,Z)\mapsto u$ be the solution map given by \tref{3.3:solution map}. Given constants $C_1,C_{2,1}$ and $C_{2,2}$, denote by $\hat{Z}^\epsilon=(\hat{\Pi}^\epsilon,\hat{\Gamma}^\epsilon)$ the renormalized model given by \tref{extension of renormalized model}. Then for every periodic $u_0\in\mathcal{C}^\eta$, $\hat{u}^\epsilon=\mathcal{R}S(u_0,\hat{Z}^\epsilon)$ solves the equation
\begin{align}\label{3.4:renormalized eq}
\partial_t\hat{u}^\epsilon=(\mathsf{i}+\mu)\triangle\hat{u}^\epsilon+\nu(1-|\hat{u}^\epsilon|^2+(2C_1-2\overline{\nu}C_{2,1}-4\nu C_{2,2}))\hat{u}^\epsilon+\xi^\epsilon.
\end{align}
\end{proposition}

\begin{proof}
Since the fixed point problem \eqref{eq:fixed point problem} can be written by $u=\mathcal{I}F(u)+\cdots$, where $\cdots$ takes values in $T^\poly$, we can find functions $\varphi$ and $\{\varphi^i\}_{i=1}^3$ such that the solution $u\in\mathcal{D}_P^{\gamma,\eta}$ of \eqref{eq:fixed point problem} with $\gamma=1^+$ (greater than but sufficiently close to $1$) can be written by
\[u=\IX+\varphi\unit-\nu\ICCD-2\nu\varphi\ICD-\nu\overline{\varphi}\ICC+\varphi^iX_i.\]
In particular, since $\Pi_z\IX=\Pi_z^M\IX=K*\xi^\epsilon$ we have
\[\mathcal{R}u=\mathcal{R}^Mu=K*\xi^\epsilon+\varphi.\]
On the other hand, by \pref{3.1:integration against K}, $\hat{u}^\epsilon=\mathcal{R}^Mu$ satisfies the equation
\[\hat{u}^\epsilon(t,x)=\int_0^t\int_{\RealNum^4}G(t-s,x-y)(\mathcal{R}^MF(u))(s,y)dsdy+\int_{\RealNum^4}G(t,x-y)u_0(y)dy.\]
Hence it suffices to show that $\mathcal{R}^MF(u)$ coincides with the driving terms of \eqref{3.4:renormalized eq}. We can expand $F(u)=\Xi+\nu(1-u\overline{u})u$ up to homogeneity $0^+$ as follows.
\begin{align*}
F(u)&=\Xi-\nu\CCD-2\nu\varphi\CD-\nu\overline{\varphi}\CC+2\nu^2\CDICCD+\nu\overline{\nu}\CCJDDC\\
&\quad+\nu(1-2\varphi\overline{\varphi})\IX-\nu\varphi^2\JX+2\nu^2\overline{\varphi}\CICCD+2\nu^2\varphi\DICCD+2\nu\overline{\nu}\varphi\CJDDC\\
&\quad+2\nu^2\overline{\varphi}\CDICC+\nu\overline{\nu}\varphi\CCJDD+4\nu^2\varphi\CDICD+2\nu\overline{\nu}\overline{\varphi}\CCJDC\\
&\quad-\nu\overline{\varphi^i}X_i\CC-2\nu\varphi^iX_i\CD+\nu(\varphi-\varphi^2\overline{\varphi})\unit.
\end{align*}
Since $\mathcal{R}^Mu=\mathcal{R}Mu$ follows from \eqref{3.4:PiM}, we have
\begin{align*}
\mathcal{R}^MF(u)&=\mathcal{R}F(u)+2\nu C_1K*\xi^\epsilon+2\nu C_1\varphi\\
&\quad-4\nu^2C_{2,2}K*\xi^\epsilon-2\nu\overline{\nu}C_{2,1}K*\xi^\epsilon-2\nu\overline{\nu}C_{2,1}\varphi-4\nu^2C_{2,2}\varphi\\
&=\xi^\epsilon+\nu(1-\mathcal{R}u\overline{\mathcal{R}u})\mathcal{R}u+2\nu C_1\mathcal{R}u-2\nu\overline{\nu}C_{2,1}\mathcal{R}u-4\nu^2C_{2,2}\mathcal{R}u\\
&=\xi^\epsilon+\nu(1-|\hat{u}^\epsilon|^2+(2C_1-2\overline{\nu}C_{2,1}-4\nu C_{2,2}))\hat{u}^\epsilon.
\end{align*}
This completes the proof.
\end{proof}

\subsection{Convergence of Gaussian models}\label{section 3.5}

Our goal is to show the following renormalization result. We give its proof in the next section since it takes long.

\begin{proposition}\label{3.5:convergence of models}
If we choose $C_1^\epsilon,C_{2,1}^\epsilon$ and $C_{2,2}^\epsilon$ as in \eqref{4:subsection5:C_1C_2C_3}, then there exists a random model $\hat{Z}$ independent of the choice of $\rho$, and for every $\theta\in(0,-\frac{5}{2}-\alpha)$, $\gamma<r$, $p>1$ and every compact set $K\subset\RealNum^4$, we have the bounds
\begin{align}\label{3.5:convergence of models 1}
\boldsymbol{E}[\$\hat{Z}^\epsilon\$_{\gamma;K}^p]\lesssim1,\quad
\boldsymbol{E}[\$\hat{Z}^\epsilon;\hat{Z}\$_{\gamma;K}^p]\lesssim\epsilon^{\theta p}.
\end{align}
Furthermore, for every $T>0$ we have
\begin{align}\label{3.5:convergence of models 2}
\boldsymbol{E}[\bbb\xi^\epsilon\bbb_{\alpha;K}^p]\lesssim1,\quad \boldsymbol{E}[\bbb\xi^\epsilon-\xi\bbb_{\alpha;K}^p]\lesssim\epsilon^{\theta p},
\end{align}
and
\begin{align}\label{3.5:convergence of models 3}
\boldsymbol{E}[\|K*\xi^\epsilon\|_{C([0,T],\mathcal{C}^{\alpha+2})}^p]\lesssim1,\quad \boldsymbol{E}[\|K*\xi^\epsilon-K*\xi\|_{C([0,T],\mathcal{C}^{\alpha+2})}^p]\lesssim\epsilon^{\theta p}.
\end{align}
\end{proposition}

Combining \pref[0360:renormalized equation]{3.5:convergence of models},
we obtain \tref{0330_main result} if we choose
$C^\epsilon=2C_1^\epsilon-2\overline{\nu}C_{2,1}^\epsilon-4\nu C_{2,2}^\epsilon$.

%% file: 0380_RegStrProb.tex

\section{Proof of convergence of renormalized models}\label{sec_20161107073513}

In this section, we give a proof of \pref{3.5:convergence of models}. Since the estimates \eqref{3.5:convergence of models 2} and \eqref{3.5:convergence of models 3} are obtained in \cite[Proposition~9.5]{Hairer2014a}, we focus on the estimate \eqref{3.5:convergence of models 1}. By \cite[Theorem~10.7]{Hairer2014a}, it suffices to show that there exist $\kappa,\theta>0$ such that, for every $\tau\in\mathcal{F}$ with $|\tau|_s<0$, every test function $\varphi\in\mathcal{B}_r$ and $z\in\RealNum^4$, there exists a random variable $\langle\hat{\Pi}_z\tau,\varphi\rangle$ such that
\begin{align}\label{section4:conv of models}
\boldsymbol{E}[|\langle\hat{\Pi}_z\tau,\varphi_z^\delta\rangle|^2]\lesssim\delta^{2|\tau|_s+\kappa},\quad \boldsymbol{E}[|\langle\hat{\Pi}_z\tau-\hat{\Pi}^\epsilon_z\tau,\varphi_z^\delta\rangle|^2]\lesssim\epsilon^{2\theta}\delta^{2|\tau|_s+\kappa}.
\end{align}
We fix $z\in\RealNum^4$ throughout this section. The estimates in this section are uniform over $z$.

This section is organized as follows. In \secref{4:subsection1}, we recall the Wiener chaos decomposition of the random variable $\hat{\Pi}_z\tau$ and introduce graphical notations to describe its kernel. In \secref{4:subsection2}, we give some useful estimates to prove \eqref{section4:conv of models}. In \secref{4:subsection3}, we show the required estimate \eqref{section4:conv of models} for each symbol $\tau$. In \secref{4:subsection5}, we show the explicit forms of renormalization constants and their divergence orders.

\subsection{Wiener chaos decomposition}\label{4:subsection1}

The driving noise $\xi$ is space-time white noise on $\RealNum\times\Torus^3$, which is extended periodically to $\RealNum^4$. In precise, we are given the complex multiple Wiener integral $\mathcal{J}_{p,q}$ on $(E,m)=(\RealNum\times\Torus^3,dtdx_1dx_2dx_3)$ (see \secref{sec:itowienerintegral}) and a random distribution $\xi$ is defined by $\langle\xi,\varphi\rangle=\mathcal{J}_{1,0}(\pi\varphi)$, where $\varphi$ is a compactly supported smooth function and $\pi\varphi=\sum_{n\in\Integers^3}S_n\varphi$ is its periodic extension, where $S_n\varphi(t,x)=\varphi(t,x+n)$. Although $\mathcal{J}_{1,0}$ is an isometry from $L^2(E)$ (not $L^2(\RealNum^4)$) to $L^2(\Omega)$, when $\varphi$ is supported in $\RealNum\times[-\frac12,\frac12]^3$ (i.e. $\varphi$ and $S_n\varphi$ have disjoint supports if $n\neq0$) we have the isometry
\begin{align}\label{section4 T R}
\boldsymbol{E}[|\langle\xi,\varphi\rangle|^2]=\|\pi\varphi\|_E^2=\int_{\RealNum^4}|\varphi(z)|^2dz.
\end{align}
The approximation $\xi^\epsilon(z)=\xi*\rho^\epsilon(z)=\mathcal{J}_{1,0}(\rho^\epsilon(z-\cdot))$ belongs to the first Wiener chaos. By definition and the product formula, for each $\tau\in\mathcal{F}$ we have the Wiener chaos decomposition
\begin{align*}
(\hat{\Pi}_z^\epsilon\tau)(z')=\sum_{p,q}\mathcal{J}_{p,q}((\hat{\mathcal{W}}_z^{\epsilon,(p,q)}\tau)(z')),
\end{align*}
where $(\hat{\mathcal{W}}_z^{\epsilon,(p,q)}\tau)(z')\in L_{p,q}^2$ is the kernel function of $\mathcal{J}_{p,q}$-exponent of $(\hat{\Pi}_z^\epsilon\tau)(z')$, parametrized by $z'\in\RealNum^4$. In all these kernels mentioned below, we always assume that $(\hat{\mathcal{W}}_z^{\epsilon,(p,q)}\tau)(z')$ is supported in sufficiently small compact subset as a function of $(\RealNum^4)^{p+q}$, so that we need not distinguish integrals on $\RealNum\times\Torus^3$ and $\RealNum^4$, indeed similarly to \eqref{section4 T R}
\begin{align*}
&\boldsymbol{E}[|\langle\hat{\Pi}_z\tau,\varphi_z^\delta\rangle|^2]
\le p!q!\left\|\int\varphi_z^\delta(z')(\hat{\mathcal{W}}_z^{\epsilon,(p,q)}\tau)(z')dz'\right\|_{L_{pq}^2}^2\\
&\quad=p!q!\iint\varphi_z^\delta(z')\varphi_z^\delta(z'')\langle(\hat{\mathcal{W}}_z^{\epsilon,(p,q)}\tau)(z'),\overline{(\hat{\mathcal{W}}_z^{\epsilon,(p,q)}\tau)(z'')}\rangle_{L^2((\RealNum^4)^{p+q})} dz'dz''.
\end{align*}
This assumption is satisfied if we take the support of $K$ sufficiently small.

Following \cite[Section~10.5]{Hairer2014a}, we introduce the following graphical notations to write integrated kernels. First a dot represents a variable in $\RealNum^4$. A square dot (\begin{tikzpicture}[baseline=-0.1cm]
\putdots{A}{$(0,0)$}{below}{z'}
\end{tikzpicture}) represents a fixed variable. A gray dot (\begin{tikzpicture}[baseline=-0.1cm]
\putdotg{A}{$(0,0)$}{}{}
\end{tikzpicture}) is a variable integrated out on $\RealNum^4$, so it has no label. A variable representing the multiple Wiener integral $\mathcal{J}_{p,q}$ is written by a black dot (\begin{tikzpicture}[baseline=-0.1cm]
\putdot{A}{$(0,0)$}{}{}
\end{tikzpicture}) for a variable in $E^p$, and a white dot (\begin{tikzpicture}[baseline=-0.1cm]
\putdotw{A}{$(0,0)$}{}{}
\end{tikzpicture}) for a variable in $E^q$, respectively. Second an arrow represents a function of two variables which are represented by its vertices. We write
\begin{align*}
K(z'-z'')&=\begin{tikzpicture}[baseline=-0.1cm]
\putdots{A}{$(0,0)$}{left}{z'}
\putdots{B}{$(1,0)$}{right}{z''}
\putb{B}{A}
\end{tikzpicture},
&K^\epsilon(z'-z'')&=\begin{tikzpicture}[baseline=-0.1cm]
\putdots{A}{$(0,0)$}{left}{z'}
\putdots{B}{$(1,0)$}{right}{z''}
\putdb{B}{A}
\end{tikzpicture},\\
\overline{K}(z'-z'')&=\begin{tikzpicture}[baseline=-0.1cm]
\putdots{A}{$(0,0)$}{left}{z'}
\putdots{B}{$(1,0)$}{right}{z''}
\putw{B}{A}
\end{tikzpicture},
&\overline{K^\epsilon}(z'-z'')&=\begin{tikzpicture}[baseline=-0.1cm]
\putdots{A}{$(0,0)$}{left}{z'}
\putdots{B}{$(1,0)$}{right}{z''}
\putdw{B}{A}
\end{tikzpicture},
\end{align*}
where $K^\epsilon=K*\rho^\epsilon$. Moreover, we write
\begin{align*}
&K(z'-z'')-K(z-z'')=\begin{tikzpicture}[baseline=-0.1cm]
\putdots{A}{$(0,0)$}{left}{z'}
\putdots{B}{$(1,0)$}{right}{z''}
\dputb{B}{A}{$(0.5,0)$}
\end{tikzpicture},\\
&\overline{K}(z'-z'')-\overline{K}(z-z'')=\begin{tikzpicture}[baseline=-0.1cm]
\putdots{A}{$(0,0)$}{left}{z'}
\putdots{B}{$(1,0)$}{right}{z''}
\dputw{B}{A}{$(0.5,0)$}
\end{tikzpicture}.
\end{align*}
We note that $z\in\RealNum^4$ is fixed. We write several kernels by combining these notations. For example,
\begin{align*}
&\mathcal{J}_{3,1}(
\begin{tikzpicture}[baseline=-0.1cm]
\putdots{A}{$(0,0)$}{left}{z'}
\putdotg{E}{$(1,0)$}{left}{}
\putdot{F}{$(0.8,-0.3)$}{left}{}
\putdot{B}{$(2,0.3)$}{right}{}
\putdot{C}{$(2,0)$}{right}{}
\putdotw{D}{$(2,-0.3)$}{right}{}
\putdb{B}{E}
\putdb{C}{E}
\putdw{D}{E}
\dputb{E}{A}{$(0.5,0)$}
\putb{F}{A}
\end{tikzpicture})\\
&=\int \left\{\int(K(z'-u)-K(z-u))K^\epsilon(u-w_1)K^\epsilon(u-w_2)\overline{K^\epsilon}(u-w')du\right\}\\
&\quad\times K(z'-w_3):\xi(dw_1)\xi(dw_2)\xi(dw_3)\overline{\xi}(dw'):.
\end{align*}

\subsection{Estimates of singularity of kernels}\label{4:subsection2}

From the scaling property of $K=\sum_nK_n$, we can see that $|K(z)|\lesssim\|z\|_s^{-3}$. It is useful to consider the singularity of kernels like this. The notation $A(z'-z'')=\begin{tikzpicture}[baseline=-0.1cm]
\putdots{C}{$(2,0)$}{right}{z''}
\putdots{B}{$(1,0)$}{left}{z'}
\putl{C}{B}{$(1.3,-0.1)$}{$(1.7,0.1)$}{$(1.5,-0)$}{-\alpha}
\end{tikzpicture}$ implies that $A:\RealNum^4\setminus\{0\}\to\CmplNum$ is a smooth function supported in a ball and has the estimate $|A(z'-z'')|\lesssim\|z'-z''\|_s^{-\alpha}$. We recall some useful estimates from \cite[Section~10.3]{Hairer2014a} and \cite[Section~4.7]{Hoshino2016}.

\begin{lemma}[{\cite[Lemma~10.14]{Hairer2014a}}]
For $\alpha,\beta\in[0,5)$, we have
\begin{align*}
|\begin{tikzpicture}[baseline=-0.1cm]
\putdots{C}{$(2,0)$}{right}{z''}
\putdotg{B}{$(1,0)$}{left}{}
\putdots{A}{$(0,0)$}{left}{z'}
\putl{C}{B}{$(1.3,-0.1)$}{$(1.7,0.1)$}{$(1.5,0)$}{-\beta}
\putl{B}{A}{$(0.3,-0.1)$}{$(0.7,0.1)$}{$(0.5,0)$}{-\alpha}
\end{tikzpicture}|\lesssim
\begin{cases}
\begin{tikzpicture}[baseline=-0.1cm]
\putdots{C}{$(2,0)$}{right}{z''}
\putdots{B}{$(0,0)$}{left}{z'}
\putl{C}{B}{$(0.35,-0.1)$}{$(1.65,0.1)$}{$(1,-0)$}{-\alpha-\beta+5}
\end{tikzpicture},&\alpha+\beta>5,\\
1,&\alpha+\beta<5.
\end{cases}
\end{align*}
\end{lemma}

\begin{lemma}[{\cite[Lemma~4.31]{Hoshino2016}}]\label{section4 cut lemma}
Let $\alpha,\beta,\alpha',\beta',\gamma\in(0,5)$. If $\zeta\in(0,\alpha\wedge\beta]$ and $\eta\in(0,\alpha'\wedge\beta']$ satisfy
\begin{align*}
\alpha+\beta-5<\zeta,\quad\alpha'+\beta'-5<\eta,\quad\alpha+\beta+\alpha'+\beta'+\gamma-10<\zeta+\eta,
\end{align*}
then we have
\begin{align*}
\left|\begin{tikzpicture}[baseline=-0.1cm]
\putdots{A}{$(0,0.4)$}{left}{z'}
\putdots{A'}{$(0,-0.4)$}{left}{w'}
\putdotg{B}{$(1,0)$}{right}{}
\putdotg{C}{$(2,0)$}{right}{}
\putdots{D}{$(3,0.4)$}{right}{z''}
\putdots{D'}{$(3,-0.4)$}{right}{w''}
\putl{B}{A}{$(0.3,0.1)$}{$(0.7,0.3)$}{$(0.5,0.2)$}{-\alpha}
\putl{B}{A'}{$(0.3,-0.3)$}{$(0.7,-0.1)$}{$(0.5,-0.2)$}{-\beta}
\putcl{C}{$(1.7,0)$}{$(1.3,0)$}{B}{$(1.3,-0.1)$}{$(1.7,0.1)$}{$(1.5,-0)$}{-\gamma}
\putl{C}{D}{$(2.3,0.1)$}{$(2.7,0.3)$}{$(2.5,0.2)$}{-\alpha'}
\putl{C}{D'}{$(2.3,-0.3)$}{$(2.7,-0.1)$}{$(2.5,-0.2)$}{-\beta'}
\end{tikzpicture}\right|\lesssim
\begin{tikzpicture}[baseline=-0.1cm]
\putdots{C}{$(2,0.3)$}{right}{w'}
\putdots{B}{$(1,0.3)$}{left}{z'}
\putl{C}{B}{$(1.3,0.2)$}{$(1.7,0.4)$}{$(1.5,0.3)$}{-\zeta}
\putdots{C'}{$(2,-0.3)$}{right}{w''}
\putdots{B'}{$(1,-0.3)$}{left}{z''}
\putl{C'}{B'}{$(1.3,-0.4)$}{$(1.7,-0.2)$}{$(1.5,-0.3)$}{-\eta}
\end{tikzpicture}.
\end{align*}
\end{lemma}

For $\alpha,\beta\ge0$, we use the notation
\begin{align*}
Q_{\alpha,\beta}(z',z'')=
\begin{tikzpicture}[baseline=-0.1cm]
\putdots{A}{$(0,0)$}{left}{z'}
\putdotg{B}{$(1,0)$}{right}{}
\putcl{B}{$(0.7,0.2)$}{$(0.3,0.2)$}{A}{$(0.3,0.1)$}{$(0.7,0.3)$}{$(0.5,0.2)$}{-\alpha}
\dputcb{B}{$(0.7,-0.2)$}{$(0.3,-0.2)$}{A}{$(0.5,-0.18)$}
\putdots{A'}{$(3,0)$}{right}{z''}
\putdotg{B'}{$(2,0)$}{right}{}
\putcl{B'}{$(2.3,0.2)$}{$(2.7,0.2)$}{A'}{$(2.3,0.1)$}{$(2.7,0.3)$}{$(2.5,0.2)$}{-\alpha}
\dputcw{B'}{$(2.3,-0.2)$}{$(2.7,-0.2)$}{A'}{$(2.5,-0.18)$}
\putl{B}{B'}{$(1.3,-0.1)$}{$(1.7,0.1)$}{$(1.5,-0)$}{-\beta}
\end{tikzpicture}.
\end{align*}

\begin{lemma}\label{section4 glass lemma}
Let $\alpha,\beta\in[0,5)$. For $\theta\in(0,1\wedge(2-\alpha-\frac\beta2))$, we have
\begin{align*}
|Q_{\alpha,\beta}(z',z'')|\lesssim\|z'-z\|_s^\theta\|z''-z\|_s^\theta.
\end{align*}
\end{lemma}

\begin{proof}
Since $|\begin{tikzpicture}[baseline=-0.1cm]
\putdots{A}{$(0,0)$}{left}{z'}
\putdots{B}{$(1,0)$}{right}{u}
\dputb{B}{A}{$(0.5,0)$}
\end{tikzpicture}|\lesssim\|z'-z\|_s^\theta(\|z'-u\|_s^{-3-\theta}+\|z-u\|_s^{-3-\theta})$ by \cite[Lemma~10.18]{Hairer2014a},
\begin{align*}
|Q_{\alpha,\beta}(z',z'')|\lesssim\|z'-z\|_s^\theta\|z''-z\|_s^\theta(R(z',z'')+R(z,z'')+R(z',z)+R(z,z)),
\end{align*}
where
\begin{align*}
R(u',u'')=R_{z',z''}(u',u'')=
\begin{tikzpicture}[baseline=-0.1cm]
\putdots{A}{$(0,0.4)$}{left}{z'}
\putdots{A'}{$(0,-0.4)$}{left}{u'}
\putdotg{B}{$(1,0)$}{right}{}
\putdotg{C}{$(2,0)$}{right}{}
\putdots{D}{$(3,0.4)$}{right}{z''}
\putdots{D'}{$(3,-0.4)$}{right}{u''}
\putl{B}{A}{$(0.3,0.1)$}{$(0.7,0.3)$}{$(0.5,0.2)$}{-\alpha}
\putl{B}{A'}{$(0.1,-0.3)$}{$(0.9,-0.1)$}{$(0.5,-0.2)$}{-3-\theta}
\putcl{C}{$(1.7,0)$}{$(1.3,0)$}{B}{$(1.3,-0.1)$}{$(1.7,0.1)$}{$(1.5,-0)$}{-\beta}
\putl{C}{D}{$(2.3,0.1)$}{$(2.7,0.3)$}{$(2.5,0.2)$}{-\alpha}
\putl{C}{D'}{$(2.1,-0.3)$}{$(2.9,-0.1)$}{$(2.5,-0.2)$}{-3-\theta}
\end{tikzpicture}.
\end{align*}
It suffices to show that $R$ is bounded. By the inequality $\|z'\|_s^{-\alpha}\|z''\|_s^{-\beta}\lesssim\|z'\|_s^{-\alpha-\beta}+\|z''\|_s^{-\alpha-\beta}$ for $\alpha,\beta\ge0$, the function $R$ is bounded by the sum of functions of the form
\begin{align*}
\begin{tikzpicture}[baseline=-0.1cm]
\putdots{D}{$(4,0)$}{right}{}
\putdotg{C}{$(2,0)$}{right}{}
\putdotg{B}{$(1,0)$}{left}{}
\putdots{A}{$(-1,0)$}{left}{}
\putl{D}{C}{$(2.35,-0.1)$}{$(3.65,0.1)$}{$(3,0)$}{-\alpha-3-\theta}
\putl{C}{B}{$(1.3,-0.1)$}{$(1.7,0.1)$}{$(1.5,0)$}{-\beta}
\putl{B}{A}{$(-0.65,-0.1)$}{$(0.65,0.1)$}{$(0,0)$}{-\alpha-3-\theta}
\end{tikzpicture},
\end{align*}
which is bounded by $1$ since $2(-\alpha-3-\theta)-\beta+10>0$.
\end{proof}

\subsection{Proof of $L^2$-estimates \eqref{section4:conv of models}}\label{4:subsection3}

Now we prove the estimate \eqref{section4:conv of models} for every $\tau\in\mathcal{F}$ with $|\tau|_s<0$.

\subsubsection{$\Xi,\IX,\JX,\CC,\CD,\CCD$}\label{4:subsection3_1}

For $\tau=\Xi,\IX,\JX$, the required estimates follow from \cite[Proposition~9.5]{Hairer2014a}. We now treat $\tau=\CC,\CD,\CCD$. By definition,
\begin{align*}
&\hat{\Pi}_z^\epsilon\CC=\Pi_z^\epsilon\CC=(\Pi_z^\epsilon\IX)^2,\\
&\hat{\Pi}_z^\epsilon\CD=\Pi_z^\epsilon\CD-C_1^\epsilon=(\Pi_z^\epsilon\IX)(\Pi_z^\epsilon\JX)-C_1^\epsilon.
\end{align*}
By applying the product formula to
\begin{align*}
&\Pi_z^\epsilon\IX(z')=K*\xi^\epsilon(z')=\mathcal{J}_{1,0}(
\begin{tikzpicture}[baseline=-0.1cm]
\putdots{A}{$(0,0)$}{left}{z'}
\putdot{B}{$(1,0)$}{right}{}
\putdb{B}{A}
\end{tikzpicture}),\\
&\Pi_z^\epsilon\JX(z')=\overline{K*\xi^\epsilon}(z')=\mathcal{J}_{0,1}(
\begin{tikzpicture}[baseline=-0.1cm]
\putdots{A}{$(0,0)$}{left}{z'}
\putdotw{B}{$(1,0)$}{right}{}
\putdw{B}{A}
\end{tikzpicture}),
\end{align*}
we have
\begin{align*}
&\hat{\Pi}_z^\epsilon\CC(z')=\mathcal{J}_{2,0}(
\begin{tikzpicture}[baseline=-0.1cm]
\putdots{A}{$(0,0)$}{left}{z'}
\putdot{B}{$(1,0.2)$}{right}{}
\putdot{C}{$(1,-0.2)$}{right}{}
\putdb{B}{A}
\putdb{C}{A}
\end{tikzpicture}),\\
&\hat{\Pi}_z^\epsilon\CD(z')=\mathcal{J}_{1,1}(
\begin{tikzpicture}[baseline=-0.1cm]
\putdots{A}{$(0,0)$}{left}{z'}
\putdot{B}{$(1,0.2)$}{right}{}
\putdotw{C}{$(1,-0.2)$}{right}{}
\putdb{B}{A}
\putdw{C}{A}
\end{tikzpicture})+
\begin{tikzpicture}[baseline=-0.1cm]
\putdots{A}{$(0,0)$}{left}{z'}
\putdotg{B}{$(1,0)$}{right}{}
\putcdb{B}{$(0.7,0.2)$}{$(0.3,0.2)$}{A}
\putcdw{B}{$(0.7,-0.2)$}{$(0.3,-0.2)$}{A}
\end{tikzpicture}-C_1^\epsilon.
\end{align*}
If we choose $C_1^\epsilon=\begin{tikzpicture}[baseline=-0.1cm]
\putdots{A}{$(0,0)$}{left}{z'}
\putdotg{B}{$(1,0)$}{right}{}
\putcdb{B}{$(0.7,0.2)$}{$(0.3,0.2)$}{A}
\putcdw{B}{$(0.7,-0.2)$}{$(0.3,-0.2)$}{A}
\end{tikzpicture}=\int |K^\epsilon(z)|^2dz$, the required estimates \eqref{3.5:convergence of models 1} for $\tau=\CC,\CD$ easily follow. Indeed,
\begin{align*}
\boldsymbol{E}[|\langle\hat{\Pi}_z^\epsilon\CC,\varphi_z^\delta\rangle|^2]
&=2\iint\varphi_z^\delta(z')\varphi_z^\delta(z'')
\begin{tikzpicture}[baseline=-0.1cm]
\putdots{A}{$(0,0)$}{left}{z'}
\putdotg{B}{$(1,0.2)$}{right}{}
\putdotg{C}{$(1,-0.2)$}{right}{}
\putdots{A'}{$(2,0)$}{right}{z''}
\putdb{B}{A}
\putdb{C}{A}
\putdw{B}{A'}
\putdw{C}{A'}
\end{tikzpicture}dz'dz''\\
&\lesssim\iint\varphi_z^\delta(z')\varphi_z^\delta(z'')
\left|\begin{tikzpicture}[baseline=-0.1cm]
\putdots{A}{$(0,0)$}{left}{z'}
\putdots{B}{$(1,0)$}{right}{z''}
\putcl{B}{$(0.7,0.2)$}{$(0.3,0.2)$}{A}{$(0.3,0.1)$}{$(0.7,0.3)$}{$(0.5,0.2)$}{-1}
\putcl{B}{$(0.7,-0.2)$}{$(0.3,-0.2)$}{A}{$(0.3,-0.3)$}{$(0.7,-0.1)$}{$(0.5,-0.2)$}{-1}
\end{tikzpicture}\right|dz'dz''\\
&\lesssim\iint\varphi_z^\delta(z')\varphi_z^\delta(z'')\|z'-z''\|_s^{-2}dz'dz''\lesssim\delta^{-2}
\end{align*}
and $-2>|\CC|_s=4(\alpha+2)$. Moreover, if we choose
\begin{align*}
\hat{\Pi}_z\CC(z')=\mathcal{J}_{2,0}(
\begin{tikzpicture}[baseline=-0.1cm]
\putdots{A}{$(0,0)$}{left}{z'}
\putdot{B}{$(1,0.2)$}{right}{}
\putdot{C}{$(1,-0.2)$}{right}{}
\putb{B}{A}
\putb{C}{A}
\end{tikzpicture}),
\end{align*}
then we have
\begin{align*}
\boldsymbol{E}[|\langle\hat{\Pi}_z^\epsilon\CC-\hat{\Pi}_z\CC,\varphi_z^\delta\rangle|^2]\lesssim\epsilon^\theta\delta^{-2-\theta}
\end{align*}
for small $\theta>0$. This is obtained by similar argument, since $|(K^\epsilon-K)(z)|\lesssim\epsilon^\theta\|z\|_s^{-3-\theta}$ for $\theta\in(0,1]$ (see \cite[Lemma~10.17]{Hairer2014a}).
The case $\tau=\CD$ is similar.

For $\tau=\CCD$, by the choice of $C_1^\epsilon$ we have
\begin{align*}
\hat{\Pi}_z^\epsilon\CCD(z')&=(\Pi_z^\epsilon\IX)(z')^2(\Pi_z^\epsilon\JX)(z')-2C_1^\epsilon(\Pi_z^\epsilon\IX)(z')\\
&=\mathcal{J}_{2,1}(
\begin{tikzpicture}[baseline=-0.1cm]
\putdots{A}{$(0,0)$}{left}{z'}
\putdot{B}{$(1,0.3)$}{right}{}
\putdot{C}{$(1,0)$}{right}{}
\putdotw{D}{$(1,-0.3)$}{right}{}
\putdb{B}{A}
\putdb{C}{A}
\putdw{D}{A}
\end{tikzpicture}).
\end{align*}
Then the estimate \eqref{3.5:convergence of models 1} for $\tau=\CCD$ easily follows as above. Indeed,
\begin{align*}
\boldsymbol{E}[|\langle\hat{\Pi}_z^\epsilon\CCD,\varphi_z^\delta\rangle|^2]
&=2\iint\varphi_z^\delta(z')\varphi_z^\delta(z'')
\begin{tikzpicture}[baseline=-0.1cm]
\putdots{A}{$(0,0)$}{left}{z'}
\putdotg{B}{$(1,0.3)$}{right}{}
\putdotg{C}{$(1,0)$}{right}{}
\putdotg{D}{$(1,-0.3)$}{right}{}
\putdots{A'}{$(2,0)$}{right}{z''}
\putdb{B}{A}
\putdb{C}{A}
\putdw{D}{A}
\putdw{B}{A'}
\putdw{C}{A'}
\putdb{D}{A'}
\end{tikzpicture}dz'dz''\\
&\lesssim\iint\varphi_z^\delta(z')\varphi_z^\delta(z'')
\left|\begin{tikzpicture}[baseline=-0.1cm]
\putdots{A}{$(0,0)$}{left}{z'}
\putdots{B}{$(1,0)$}{right}{z''}
\putcl{B}{$(0.7,0.3)$}{$(0.3,0.3)$}{A}{$(0.3,0.2)$}{$(0.7,0.4)$}{$(0.5,0.3)$}{-1}
\putcl{B}{$(0.7,-0.3)$}{$(0.3,-0.3)$}{A}{$(0.3,-0.4)$}{$(0.7,-0.2)$}{$(0.5,-0.3)$}{-1}
\putcl{B}{$(0.7,0)$}{$(0.3,0)$}{A}{$(0.3,-0.1)$}{$(0.7,0.1)$}{$(0.5,0)$}{-1}
\end{tikzpicture}\right|dz'dz''\\
&\lesssim\iint\varphi_z^\delta(z')\varphi_z^\delta(z'')\|z'-z''\|_s^{-3}dz'dz''\lesssim\delta^{-3}
\end{align*}
and $-3>|\CCD|_s=6(\alpha+2)$. The estimates for $\hat{\Pi}_z^\epsilon\CCD-\hat{\Pi}_z\CCD$ are similarly obtained by choosing
\begin{align*}
\hat{\Pi}_z^\epsilon\CCD(z')=\mathcal{J}_{2,1}(
\begin{tikzpicture}[baseline=-0.1cm]
\putdots{A}{$(0,0)$}{left}{z'}
\putdot{B}{$(1,0.3)$}{right}{}
\putdot{C}{$(1,0)$}{right}{}
\putdotw{D}{$(1,-0.3)$}{right}{}
\putb{B}{A}
\putb{C}{A}
\putw{D}{A}
\end{tikzpicture}).
\end{align*}

In the subsequent computations, the estimates of $\hat{\Pi}_z^\epsilon\tau-\hat{\Pi}_z\tau$ are obtained by similar arguments to those of $\hat{\Pi}_z^\epsilon\tau$ as above by using the bound of $K^\epsilon-K$, so we  show only the uniform boundedness but not the convergence estimates explicitly. For detailed proofs, see \cite[Section~4.8]{Hoshino2016}.

\subsubsection{$X_i\CC,X_i\CD,\CDICC,\CCJDD,\CDICD,\CCJDC$}\label{4:subsection3_2}

For $\tau=X_i\CC,X_i\CD$, the corresponding estimates are easily obtained. Indeed, since $\hat{\Pi}_z^\epsilon X_i\CC(z')=(x_i'-x_i)\hat{\Pi}_z^\epsilon\CC(z')$ we have
\begin{align*}
&\boldsymbol{E}[|\langle\hat{\Pi}_z^\epsilon X_i\CC,\varphi_z^\delta\rangle|^2]\\
&=2\iint\varphi_z^\delta(z')\varphi_z^\delta(z'')(x_i'-x_i)(x_i''-x_i)
\begin{tikzpicture}[baseline=-0.1cm]
\putdots{A}{$(0,0)$}{left}{z'}
\putdotg{B}{$(1,0.2)$}{right}{}
\putdotg{C}{$(1,-0.2)$}{right}{}
\putdots{A'}{$(2,0)$}{right}{z''}
\putdb{B}{A}
\putdb{C}{A}
\putdw{B}{A'}
\putdw{C}{A'}
\end{tikzpicture}dz'dz''\\
&\lesssim\iint\varphi_z^\delta(z')\varphi_z^\delta(z'')(x_i'-x_i)(x_i''-x_i)\|z'-z''\|_s^{-2}dz'dz''\lesssim1
\end{align*}
and $0>2|X_i\CC|_s=2(2\alpha+5)$. The case $\tau=X_i\CD$ is similar.

Now we turn to $\CDICC,\CCJDD,\CDICD,\CCJDC$. In particular, we consider the renormalizations of $\CCJDD$ and $\CDICD$, since the corresponding chaos decompositions of the two other elements do not have zeroth order terms. By definition,
\begin{align*}
&\hat{\Pi}_z^\epsilon\CCJDD=\Pi_z^\epsilon\CCJDD-2C_{2,1}^\epsilon
=(\Pi_z^\epsilon\CC)(\Pi_z^\epsilon\JDD)-2C_{2,1}^\epsilon,\\
&\hat{\Pi}_z^\epsilon\CDICD=\Pi_z^\epsilon\CDICD-C_1^\epsilon\Pi_z^\epsilon\ICD-C_{2,2}^\epsilon=(\Pi_z^\epsilon\CD-C_1^\epsilon)(\Pi_z^\epsilon\ICD)-C_{2,2}^\epsilon.
\end{align*}
We note that
\begin{align*}
&\Pi_z^\epsilon\JDD(z')=\overline{K*\Pi_z^\epsilon\CC}(z')-\overline{K*\Pi_z^\epsilon\CC}(z)
=\mathcal{J}_{0,2}(\begin{tikzpicture}[baseline=-0.1cm]
\putdots{A}{$(0,0)$}{left}{z'}
\putdotg{D}{$(1,0)$}{right}{}
\putdotw{E}{$(1.8,0.2)$}{right}{}
\putdotw{F}{$(1.8,-0.2)$}{right}{}
\dputw{D}{A}{$(0.5,0)$}
\putdw{E}{D}
\putdw{F}{D}
\end{tikzpicture}),\\
&\Pi_z^\epsilon\ICD(z')=K*\Pi_z^\epsilon\CD(z')-K*\Pi_z^\epsilon\CD(z)
=\mathcal{J}_{1,1}(\begin{tikzpicture}[baseline=-0.1cm]
\putdots{A}{$(0,0)$}{left}{z'}
\putdotg{D}{$(1,0)$}{right}{}
\putdot{E}{$(1.8,0.2)$}{right}{}
\putdotw{F}{$(1.8,-0.2)$}{right}{}
\dputb{D}{A}{$(0.5,0)$}
\putdb{E}{D}
\putdw{F}{D}
\end{tikzpicture}).
\end{align*}
By applying the product formula (\tref{2:product}) we have
\begin{align*}
\hat{\Pi}_z^\epsilon\CCJDD(z')
&=\mathcal{J}_{2,2}(\begin{tikzpicture}[baseline=-0.1cm]
\putdots{A}{$(0,0)$}{left}{z'}
\putdot{B}{$(0.8,0.3)$}{right}{}
\putdot{C}{$(0.8,-0.3)$}{right}{}
\putdotg{D}{$(1,0)$}{right}{}
\putdotw{E}{$(1.8,0.2)$}{right}{}
\putdotw{F}{$(1.8,-0.2)$}{right}{}
\putdb{B}{A}
\putdb{C}{A}
\dputw{D}{A}{$(0.5,0)$}
\putdw{E}{D}
\putdw{F}{D}
\end{tikzpicture})
+4\mathcal{J}_{1,1}(\begin{tikzpicture}[baseline=-0.1cm]
\putdots{A}{$(0,0)$}{left}{z'}
\putdotg{B}{$(0.5,0.4)$}{right}{}
\putdot{C}{$(0.8,-0.3)$}{right}{}
\putdotg{D}{$(1,0)$}{right}{}
\putdotw{F}{$(1.8,0)$}{right}{}
\putdb{B}{A}
\putdb{C}{A}
\dputw{D}{A}{$(0.5,0)$}
\putdw{B}{D}
\putdw{F}{D}
\end{tikzpicture})\\
&\quad+2(\begin{tikzpicture}[baseline=-0.1cm]
\putdots{A}{$(0,0)$}{left}{z'}
\putdotg{B}{$(0.5,0.4)$}{right}{}
\putdotg{C}{$(0.5,-0.4)$}{right}{}
\putdotg{D}{$(1,0)$}{right}{}
\putdb{B}{A}
\putdb{C}{A}
\dputw{D}{A}{$(0.5,0)$}
\putdw{B}{D}
\putdw{C}{D}
\end{tikzpicture}-C_{2,1}^\epsilon)
\end{align*}
and
\begin{align*}
\hat{\Pi}_z^\epsilon\CDICD(z')
&=\mathcal{J}_{2,2}(\begin{tikzpicture}[baseline=-0.1cm]
\putdots{A}{$(0,0)$}{left}{z'}
\putdot{B}{$(0.8,0.3)$}{right}{}
\putdotw{C}{$(0.8,-0.3)$}{right}{}
\putdotg{D}{$(1,0)$}{right}{}
\putdot{E}{$(1.8,0.2)$}{right}{}
\putdotw{F}{$(1.8,-0.2)$}{right}{}
\putdb{B}{A}
\putdw{C}{A}
\dputb{D}{A}{$(0.5,0)$}
\putdb{E}{D}
\putdw{F}{D}
\end{tikzpicture})
+\mathcal{J}_{1,1}(\begin{tikzpicture}[baseline=-0.1cm]
\putdots{A}{$(0,0)$}{left}{z'}
\putdotg{B}{$(0.5,0.4)$}{right}{}
\putdotw{C}{$(0.8,-0.3)$}{right}{}
\putdotg{D}{$(1,0)$}{right}{}
\putdot{F}{$(1.8,0)$}{right}{}
\putdb{B}{A}
\putdw{C}{A}
\dputb{D}{A}{$(0.5,0)$}
\putdw{B}{D}
\putdb{F}{D}
\end{tikzpicture})\\
&\quad+\mathcal{J}_{1,1}(\begin{tikzpicture}[baseline=-0.1cm]
\putdots{A}{$(0,0)$}{left}{z'}
\putdotg{B}{$(0.5,0.4)$}{right}{}
\putdot{C}{$(0.8,-0.3)$}{right}{}
\putdotg{D}{$(1,0)$}{right}{}
\putdotw{F}{$(1.8,0)$}{right}{}
\putdw{B}{A}
\putdb{C}{A}
\dputb{D}{A}{$(0.5,0)$}
\putdb{B}{D}
\putdw{F}{D}
\end{tikzpicture})
+\begin{tikzpicture}[baseline=-0.1cm]
\putdots{A}{$(0,0)$}{left}{z'}
\putdotg{B}{$(0.5,0.4)$}{right}{}
\putdotg{C}{$(0.5,-0.4)$}{right}{}
\putdotg{D}{$(1,0)$}{right}{}
\putdb{B}{A}
\putdw{C}{A}
\dputb{D}{A}{$(0.5,0)$}
\putdw{B}{D}
\putdb{C}{D}
\end{tikzpicture}-C_{2,2}^\epsilon.
\end{align*}
Hence if we choose
\begin{align*}
C_{2,1}^\epsilon=\begin{tikzpicture}[baseline=-0.1cm]
\putdots{A}{$(0,0)$}{left}{z'}
\putdotg{B}{$(0.5,0.4)$}{right}{}
\putdotg{C}{$(0.5,-0.4)$}{right}{}
\putdotg{D}{$(1,0)$}{right}{}
\putdb{B}{A}
\putdb{C}{A}
\putw{D}{A}
\putdw{B}{D}
\putdw{C}{D}
\end{tikzpicture},\quad
C_{2,2}^\epsilon=\begin{tikzpicture}[baseline=-0.1cm]
\putdots{A}{$(0,0)$}{left}{z'}
\putdotg{B}{$(0.5,0.4)$}{right}{}
\putdotg{C}{$(0.5,-0.4)$}{right}{}
\putdotg{D}{$(1,0)$}{right}{}
\putdb{B}{A}
\putdw{C}{A}
\putb{D}{A}
\putdw{B}{D}
\putdb{C}{D}
\end{tikzpicture},
\end{align*}
we have the required bounds. Indeed, since kernels belonging to the same order chaos have the same graphs except for the difference of $K$ and $\overline{K}$, it suffices to show the bounds for one of these kernels for each order chaos. For remaining zeroth order terms, we have the bounds
\begin{align*}
|\begin{tikzpicture}[baseline=-0.1cm]
\putdots{A}{$(0,0)$}{left}{z'}
\putdotg{B}{$(0.5,0.3)$}{right}{}
\putdotg{C}{$(0.5,-0.3)$}{right}{}
\putdotg{D}{$(1,0)$}{right}{}
\putdots{A'}{$(1.8,0)$}{right}{z}
\putdb{B}{A}
\putdb{C}{A}
\putw{D}{A'}
\putdw{B}{D}
\putdw{C}{D}
\end{tikzpicture}|
\lesssim|\begin{tikzpicture}[baseline=-0.1cm]
\putdots{A}{$(0,0)$}{left}{z'}
\putdotg{B}{$(1,0)$}{right}{}
\putdots{C}{$(2,0)$}{right}{z}
\putcl{B}{$(0.7,0.2)$}{$(0.3,0.2)$}{A}{$(0.3,0.1)$}{$(0.7,0.3)$}{$(0.5,0.2)$}{-1}
\putcl{B}{$(0.7,-0.2)$}{$(0.3,-0.2)$}{A}{$(0.3,-0.3)$}{$(0.7,-0.1)$}{$(0.5,-0.2)$}{-1}
\putcl{C}{$(1.7,0)$}{$(1.3,0)$}{B}{$(1.3,-0.1)$}{$(1.7,0.1)$}{$(1.5,0)$}{-3}
\end{tikzpicture}|\lesssim\|z'-z\|_s^{-\kappa}
\end{align*}
for an arbitrary small $\kappa>0$. For the second order terms, by \lref{section4 glass lemma} we have
\begin{align*}
|\begin{tikzpicture}[baseline=0.4cm]
\putdots{A}{$(0,0)$}{left}{z'}
\putdotg{B}{$(-0.4,0.5)$}{right}{}
\putdotg{C}{$(0.7,0)$}{right}{}
\putdotg{D}{$(0,1)$}{right}{}
\putdotg{F}{$(0.7,1)$}{right}{}
\putdots{A'}{$(1.4,0)$}{right}{z''}
\putdotg{B'}{$(1.8,0.5)$}{right}{}
\putdotg{D'}{$(1.4,1)$}{right}{}
\putdb{B}{A}
\putdb{C}{A}
\hputw{D}{A}{$(0,0.5)$}
\putdw{B}{D}
\putdw{F}{D}
\putdw{B'}{A'}
\putdw{C}{A'}
\hputb{D'}{A'}{$(1.4,0.5)$}
\putdb{B'}{D'}
\putdb{F}{D'}
\end{tikzpicture}|&=
|\begin{tikzpicture}[baseline=-0.1cm]
\putdots{A}{$(0,0)$}{left}{z'}
\putdotg{C}{$(0.7,0)$}{right}{}
\putdots{A'}{$(1.4,0)$}{right}{z''}
\putdb{C}{A}
\putdw{C}{A'}
\end{tikzpicture}|\times|Q_{1,1}(z',z'')|\\
&\lesssim\|z'-z\|_s^{\frac12-\kappa}\|z''-z\|_s^{\frac12-\kappa}\|z'-z''\|_s^{-1}
\end{align*}
for small $\kappa>0$. Similarly, for the fourth order terms, we have
\begin{align*}
|\begin{tikzpicture}[baseline=0.4cm]
\putdots{A}{$(0,0)$}{left}{z'}
\putdotg{C}{$(0.7,0.2)$}{right}{}
\putdotg{C'}{$(0.7,-0.2)$}{right}{}
\putdotg{D}{$(0,1)$}{right}{}
\putdotg{F}{$(0.7,0.8)$}{right}{}
\putdotg{F'}{$(0.7,1.2)$}{right}{}
\putdots{A'}{$(1.4,0)$}{right}{z''}
\putdotg{D'}{$(1.4,1)$}{right}{}
\putdb{C}{A}
\putdb{C'}{A}
\hputw{D}{A}{$(0,0.5)$}
\putdw{F}{D}
\putdw{F'}{D}
\putdw{C}{A'}
\putdw{C'}{A'}
\hputb{D'}{A'}{$(1.4,0.5)$}
\putdb{F}{D'}
\putdb{F'}{D'}
\end{tikzpicture}|&=
|\begin{tikzpicture}[baseline=-0.1cm]
\putdots{A}{$(0,0)$}{left}{z'}
\putdotg{C}{$(0.7,0.2)$}{right}{}
\putdotg{C'}{$(0.7,-0.2)$}{right}{}
\putdots{A'}{$(1.4,0)$}{right}{z''}
\putdb{C}{A}
\putdw{C}{A'}
\putdb{C'}{A}
\putdw{C'}{A'}
\end{tikzpicture}|\times|Q_{0,2}(z',z'')|\\
&\lesssim\|z'-z\|_s^{1-\kappa}\|z''-z\|_s^{1-\kappa}\|z'-z''\|_s^{-2}
\end{align*}
for small $\kappa>0$. As a consequence, we have
\begin{align*}
E|\langle\hat{\Pi}_z^\epsilon\CCJDD,\varphi_z^\delta\rangle|^2&\lesssim\iint\varphi_z^\delta(z')\varphi_z^\delta(z'')\{\|z'-z\|_s^{-\kappa}\|z''-z\|_s^{-\kappa}\\
&\quad+\|z'-z\|_s^{\frac12-\kappa}\|z''-z\|_s^{\frac12-\kappa}\|z'-z''\|_s^{-1}\\
&\quad+\|z'-z\|_s^{1-\kappa}\|z''-z\|_s^{1-\kappa}\|z'-z''\|_s^{-2}\}dz'dz''\lesssim\delta^{-2\kappa}
\end{align*}
for an arbitrary small $\kappa>0$. The cases $\tau=\CDICC,\CDICD,\CCJDC$ are similar.

\subsubsection{$\CICCD,\DICCD,\CJDDC,\CDICCD,\CCJDDC$}

We treat the case $\tau=\CICCD$. The cases $\tau=\DICCD,\CJDDC$ are similar. By definition,
\begin{align*}
\hat{\Pi}_z^\epsilon\CICCD(z')&=\Pi_z^\epsilon\CICCD(z')-2C_1^\epsilon(\Pi_z^\epsilon\CIC(z')-f_z^\epsilon(\mathcal{J}_i\IX)(\Pi_z^\epsilon X_i\IX)(z'))\\
&=\Pi_z^\epsilon\IX(z')\{K*\Pi_z^\epsilon\CCD(z')-K*\Pi_z^\epsilon\CCD(z)\\
&\qquad-2C_1^\epsilon(K*\Pi_z^\epsilon\IX(z')-K*\Pi_z^\epsilon\IX(z))\}\\
&=\mathcal{J}_{1,0}(
\begin{tikzpicture}[baseline=-0.1cm]
\putdots{A}{$(0,0)$}{left}{z'}
\putdot{B}{$(1,0)$}{right}{}
\putdb{B}{A}
\end{tikzpicture})
\mathcal{J}_{2,1}(
\begin{tikzpicture}[baseline=-0.1cm]
\putdots{A}{$(0,0)$}{left}{z'}
\putdotg{E}{$(1,0)$}{left}{}
\putdot{B}{$(2,0.3)$}{right}{}
\putdot{C}{$(2,0)$}{right}{}
\putdotw{D}{$(2,-0.3)$}{right}{}
\putdb{B}{E}
\putdb{C}{E}
\putdw{D}{E}
\dputb{E}{A}{$(0.5,0)$}
\end{tikzpicture})\\
&=\mathcal{J}_{3,1}(
\begin{tikzpicture}[baseline=-0.1cm]
\putdots{A}{$(0,0)$}{left}{z'}
\putdotg{E}{$(1,0)$}{left}{}
\putdot{F}{$(0.8,-0.3)$}{left}{}
\putdot{B}{$(2,0.3)$}{right}{}
\putdot{C}{$(2,0)$}{right}{}
\putdotw{D}{$(2,-0.3)$}{right}{}
\putdb{B}{E}
\putdb{C}{E}
\putdw{D}{E}
\dputb{E}{A}{$(0.5,0)$}
\putdb{F}{A}
\end{tikzpicture})
+\mathcal{J}_{2,0}(
\begin{tikzpicture}[baseline=-0.1cm]
\putdots{A}{$(0,0)$}{left}{z'}
\putdotg{E}{$(1,0)$}{left}{}
\putdot{B}{$(2,0.2)$}{right}{}
\putdot{C}{$(2,-0.2)$}{right}{}
\putdotg{D}{$(0.5,0.4)$}{right}{}
\putdb{B}{E}
\putdb{C}{E}
\putdw{D}{E}
\dputb{E}{A}{$(0.5,0)$}
\putdb{D}{A}
\end{tikzpicture}).
\end{align*}
The summation symbol over $i=1,2,3$ is omitted again. For the fourth order term, by \lref{section4 glass lemma} we have
\begin{align*}
|\begin{tikzpicture}[baseline=0.4cm]
\putdots{A}{$(0,0)$}{left}{z'}
\putdotg{E}{$(0,1)$}{left}{}
\putdotg{F}{$(1,0)$}{left}{}
\putdotg{B}{$(1,1.3)$}{right}{}
\putdotg{C}{$(1,1)$}{right}{}
\putdotg{D}{$(1,0.7)$}{right}{}
\putdots{A'}{$(2,0)$}{right}{z''}
\putdotg{E'}{$(2,1)$}{left}{}
\putdb{B}{E}
\putdb{C}{E}
\putdw{D}{E}
\hputb{E}{A}{$(0,0.5)$}
\putdb{F}{A}
\putdw{B}{E'}
\putdw{C}{E'}
\putdb{D}{E'}
\hputw{E'}{A'}{$(2,0.5)$}
\putdw{F}{A'}
\end{tikzpicture}|&=
|\begin{tikzpicture}[baseline=-0.1cm]
\putdots{A}{$(0,0)$}{left}{z'}
\putdotg{C}{$(0.7,0)$}{right}{}
\putdots{A'}{$(1.4,0)$}{right}{z''}
\putdb{C}{A}
\putdw{C}{A'}
\end{tikzpicture}|\times|Q_{0,3}(z',z'')|\\
&\lesssim\|z'-z\|_s^{\frac12-\kappa}\|z''-z\|_s^{\frac12-\kappa}\|z'-z''\|_s^{-1}
\end{align*}
for small $\kappa>0$. For the second order term, we decompose it as
\begin{align}\label{section4_4 second of CICCD}
\begin{tikzpicture}[baseline=-0.1cm]
\putdots{A}{$(0,0)$}{left}{z'}
\putdotg{E}{$(1,0)$}{left}{}
\putdot{B}{$(2,0.2)$}{right}{}
\putdot{C}{$(2,-0.2)$}{right}{}
\putdotg{D}{$(0.5,0.4)$}{right}{}
\putdb{B}{E}
\putdb{C}{E}
\putdw{D}{E}
\dputb{E}{A}{$(0.5,0)$}
\putdb{D}{A}
\end{tikzpicture}=
\begin{tikzpicture}[baseline=-0.1cm]
\putdots{A}{$(0,0)$}{left}{z'}
\putdotg{E}{$(1,0)$}{left}{}
\putdot{B}{$(2,0.2)$}{right}{}
\putdot{C}{$(2,-0.2)$}{right}{}
\putdotg{D}{$(0.5,0.4)$}{right}{}
\putdb{B}{E}
\putdb{C}{E}
\putdw{D}{E}
\putb{E}{A}
\putdb{D}{A}
\end{tikzpicture}-
\begin{tikzpicture}[baseline=-0.1cm]
\putdots{A}{$(0,0)$}{left}{z'}
\putdotg{E}{$(1,0)$}{left}{}
\putdot{B}{$(2,0.2)$}{right}{}
\putdot{C}{$(2,-0.2)$}{right}{}
\putdotg{D}{$(0.5,0.4)$}{right}{}
\putdots{F}{$(0.2,-0.4)$}{left}{z}
\putdb{B}{E}
\putdb{C}{E}
\putdw{D}{E}
\putb{E}{F}
\putdb{D}{A}
\end{tikzpicture}.
\end{align}
By Schwarz's inequality, it suffices to consider the bound for each term. For the first term, we have
\begin{align*}
|\begin{tikzpicture}[baseline=-0.1cm]
\putdots{A}{$(0,0)$}{left}{z'}
\putdotg{E}{$(0.8,0)$}{left}{}
\putdotg{B}{$(1.6,0.2)$}{right}{}
\putdotg{C}{$(1.6,-0.2)$}{right}{}
\putdotg{D}{$(0.4,0.4)$}{right}{}
\putdots{A'}{$(3.2,0)$}{right}{z''}
\putdotg{E'}{$(2.4,0)$}{left}{}
\putdotg{D'}{$(2.8,0.4)$}{right}{}
\putdb{B}{E}
\putdb{C}{E}
\putdw{D}{E}
\putb{E}{A}
\putdb{D}{A}
\putdw{B}{E'}
\putdw{C}{E'}
\putdb{D'}{E'}
\putw{E'}{A'}
\putdw{D'}{A'}
\end{tikzpicture}|\lesssim
|\begin{tikzpicture}[baseline=-0.1cm]
\putdots{A}{$(0,0)$}{left}{z'}
\putdotg{B}{$(1,0)$}{right}{}
\putdotg{C}{$(2,0)$}{right}{}
\putdots{D}{$(3,0)$}{right}{z''}
\putcl{B}{$(0.7,0)$}{$(0.3,0)$}{A}{$(0.3,-0.1)$}{$(0.7,0.1)$}{$(0.5,0)$}{-4}
\putcl{C}{$(1.7,0)$}{$(1.3,0)$}{B}{$(1.3,-0.1)$}{$(1.7,0.1)$}{$(1.5,-0)$}{-2}
\putcl{D}{$(2.7,0)$}{$(2.3,0)$}{C}{$(2.3,-0.1)$}{$(2.7,0.1)$}{$(2.5,0)$}{-4}
\end{tikzpicture}|\lesssim\|z'-z''\|_s^{-\kappa}
\end{align*}
for small $\kappa>0$. For the second term, by \lref{section4 cut lemma} we have
\begin{align*}
|\begin{tikzpicture}[baseline=-0.1cm]
\putdots{A}{$(0,0)$}{left}{z'}
\putdotg{E}{$(0.8,0)$}{left}{}
\putdotg{B}{$(1.6,0.2)$}{right}{}
\putdotg{C}{$(1.6,-0.2)$}{right}{}
\putdotg{D}{$(0.4,0.4)$}{right}{}
\putdots{F}{$(0.16,-0.4)$}{left}{z}
\putdots{A'}{$(3.2,0)$}{right}{z''}
\putdotg{E'}{$(2.4,0)$}{left}{}
\putdotg{D'}{$(2.8,0.4)$}{right}{}
\putdots{F'}{$(3.04,-0.4)$}{right}{z}
\putdb{B}{E}
\putdb{C}{E}
\putdw{D}{E}
\putb{E}{F}
\putdb{D}{A}
\putdw{B}{E'}
\putdw{C}{E'}
\putdb{D'}{E'}
\putw{E'}{F'}
\putdw{D'}{A'}
\end{tikzpicture}|&\lesssim
|\begin{tikzpicture}[baseline=-0.1cm]
\putdots{A}{$(0,0.4)$}{left}{z'}
\putdots{A'}{$(0,-0.4)$}{left}{z}
\putdotg{B}{$(1,0)$}{right}{}
\putdotg{C}{$(2,0)$}{right}{}
\putdots{D}{$(3,0.4)$}{right}{z''}
\putdots{D'}{$(3,-0.4)$}{right}{z}
\putl{B}{A}{$(0.3,0.1)$}{$(0.7,0.3)$}{$(0.5,0.2)$}{-1}
\putl{B}{A'}{$(0.3,-0.3)$}{$(0.7,-0.1)$}{$(0.5,-0.2)$}{-3}
\putcl{C}{$(1.7,0)$}{$(1.3,0)$}{B}{$(1.3,-0.1)$}{$(1.7,0.1)$}{$(1.5,-0)$}{-2}
\putl{C}{D}{$(2.3,0.1)$}{$(2.7,0.3)$}{$(2.5,0.2)$}{-1}
\putl{C}{D'}{$(2.3,-0.3)$}{$(2.7,-0.1)$}{$(2.5,-0.2)$}{-3}
\end{tikzpicture}|\\
&\lesssim\|z'-z\|_s^{-\kappa}\|z''-z\|_s^{-\kappa}
\end{align*}
for small $\kappa>0$. As a consequence, we have
\begin{align*}
E|\langle\hat{\Pi}_z^\epsilon\CICCD,\varphi_z^\delta\rangle|^2\lesssim\delta^{-2\kappa}.
\end{align*}

Finally we treat $\tau=\CDICCD$. The other one is similar. By definition,
\begin{align*}
\hat{\Pi}_z^\epsilon\CDICCD(z')
&=\Pi_z^\epsilon\{(\ICCD-2C_1^\epsilon\IC)(\CD-C_1^\epsilon\unit)-2C_{2,2}^\epsilon\IX\}(z')\\
&\quad+2C_1^\epsilon f_z^\epsilon(\mathcal{J}_i\IX)\Pi_z^\epsilon\{X_i(\CD-C_1^\epsilon\unit)\}(z')\\
&=\Pi_z^\epsilon(\CD-C_1^\epsilon)(z')\{K*\Pi_z^\epsilon\CCD(z')-K*\Pi_z^\epsilon\CCD(z)\\
&\quad-2C_1^\epsilon(K*\Pi_z^\epsilon\IX(z')-K*\Pi_z^\epsilon\IX(z))\}(z')-2C_{2,2}^\epsilon\Pi_z^\epsilon\IX(z')\\
&=\mathcal{J}_{1,1}(
\begin{tikzpicture}[baseline=-0.1cm]
\putdots{A}{$(0,0)$}{left}{z'}
\putdot{B}{$(1,0.2)$}{right}{}
\putdotw{C}{$(1,-0.2)$}{right}{}
\putdb{B}{A}
\putdw{C}{A}
\end{tikzpicture})
\mathcal{J}_{2,1}(
\begin{tikzpicture}[baseline=-0.1cm]
\putdots{A}{$(0,0)$}{left}{z'}
\putdotg{E}{$(1,0)$}{left}{}
\putdot{B}{$(2,0.3)$}{right}{}
\putdot{C}{$(2,0)$}{right}{}
\putdotw{D}{$(2,-0.3)$}{right}{}
\putdb{B}{E}
\putdb{C}{E}
\putdw{D}{E}
\dputb{E}{A}{$(0.5,0)$}
\end{tikzpicture})\\
&\quad
-2\begin{tikzpicture}[baseline=-0.1cm]
\putdots{A}{$(0,0)$}{left}{z'}
\putdotg{B}{$(0.5,0.4)$}{right}{}
\putdotg{C}{$(0.5,-0.4)$}{right}{}
\putdotg{D}{$(1,0)$}{right}{}
\putdb{B}{A}
\putdw{C}{A}
\putb{D}{A}
\putdw{B}{D}
\putdb{C}{D}
\end{tikzpicture}
\mathcal{J}_{1,0}(
\begin{tikzpicture}[baseline=-0.1cm]
\putdots{A}{$(0,0)$}{left}{z'}
\putdot{B}{$(1,0)$}{right}{}
\putdb{B}{A}
\end{tikzpicture})\\
&=\mathcal{J}_{3,2}(\begin{tikzpicture}[baseline=-0.1cm]
\putdots{A}{$(0,0)$}{left}{z'}
\putdotg{E}{$(1,0)$}{left}{}
\putdot{B}{$(2,0.3)$}{right}{}
\putdot{C}{$(2,0)$}{right}{}
\putdotw{D}{$(2,-0.3)$}{right}{}
\putdot{F}{$(0.8,0.3)$}{right}{}
\putdotw{G}{$(0.8,-0.3)$}{right}{}
\putdb{B}{E}
\putdb{C}{E}
\putdw{D}{E}
\dputb{E}{A}{$(0.5,0)$}
\putdb{F}{A}
\putdw{G}{A}
\end{tikzpicture})
+\mathcal{J}_{2,1}(
\begin{tikzpicture}[baseline=-0.1cm]
\putdots{A}{$(0,0)$}{left}{z'}
\putdotg{E}{$(1,0)$}{left}{}
\putdot{B}{$(2,0.2)$}{right}{}
\putdot{C}{$(2,-0.2)$}{right}{}
\putdotg{D}{$(0.5,0.4)$}{right}{}
\putdotw{F}{$(0.8,-0.3)$}{right}{}
\putdb{B}{E}
\putdb{C}{E}
\putdw{D}{E}
\dputb{E}{A}{$(0.5,0)$}
\putdb{D}{A}
\putdw{F}{A}
\end{tikzpicture})\\
&\quad+2\mathcal{J}_{2,1}(
\begin{tikzpicture}[baseline=-0.1cm]
\putdots{A}{$(0,0)$}{left}{z'}
\putdotg{E}{$(1,0)$}{left}{}
\putdot{B}{$(2,0.2)$}{right}{}
\putdotw{C}{$(2,-0.2)$}{right}{}
\putdotg{D}{$(0.5,0.4)$}{right}{}
\putdot{F}{$(0.8,-0.3)$}{right}{}
\putdb{B}{E}
\putdw{C}{E}
\putdb{D}{E}
\dputb{E}{A}{$(0.5,0)$}
\putdw{D}{A}
\putdb{F}{A}
\end{tikzpicture})\\
&\quad+2\{\mathcal{J}_{1,0}(\begin{tikzpicture}[baseline=-0.1cm]
\putdots{A}{$(0,0)$}{left}{z'}
\putdotg{B}{$(0.5,0.4)$}{right}{}
\putdotg{C}{$(0.5,-0.4)$}{right}{}
\putdotg{D}{$(1,0)$}{right}{}
\putdot{E}{$(2,0)$}{right}{}
\putdb{B}{A}
\putdw{C}{A}
\dputb{D}{A}{$(0.5,0)$}
\putdw{B}{D}
\putdb{C}{D}
\putdb{E}{D}
\end{tikzpicture})
-\begin{tikzpicture}[baseline=-0.1cm]
\putdots{A}{$(0,0)$}{left}{z'}
\putdotg{B}{$(0.5,0.4)$}{right}{}
\putdotg{C}{$(0.5,-0.4)$}{right}{}
\putdotg{D}{$(1,0)$}{right}{}
\putdb{B}{A}
\putdw{C}{A}
\putb{D}{A}
\putdw{B}{D}
\putdb{C}{D}
\end{tikzpicture}
\mathcal{J}_{1,0}(
\begin{tikzpicture}[baseline=-0.1cm]
\putdots{A}{$(0,0)$}{left}{z'}
\putdot{B}{$(1,0)$}{right}{}
\putdb{B}{A}
\end{tikzpicture})\}.
\end{align*}
For the fifth order term, we have
\begin{align*}
|\begin{tikzpicture}[baseline=0.4cm]
\putdots{A}{$(0,0)$}{left}{z'}
\putdotg{E}{$(0,1)$}{left}{}
\putdotg{F}{$(1,0.2)$}{left}{}
\putdotg{G}{$(1,-0.2)$}{left}{}
\putdotg{B}{$(1,1.3)$}{right}{}
\putdotg{C}{$(1,1)$}{right}{}
\putdotg{D}{$(1,0.7)$}{right}{}
\putdots{A'}{$(2,0)$}{right}{z''}
\putdotg{E'}{$(2,1)$}{left}{}
\putdb{B}{E}
\putdb{C}{E}
\putdw{D}{E}
\hputb{E}{A}{$(0,0.5)$}
\putdb{F}{A}
\putdw{G}{A}
\putdw{B}{E'}
\putdw{C}{E'}
\putdb{D}{E'}
\hputw{E'}{A'}{$(2,0.5)$}
\putdw{F}{A'}
\putdb{G}{A'}
\end{tikzpicture}|&=
|\begin{tikzpicture}[baseline=-0.1cm]
\putdots{A}{$(0,0)$}{left}{z'}
\putdotg{C}{$(0.7,0.2)$}{right}{}
\putdotg{C'}{$(0.7,-0.2)$}{right}{}
\putdots{A'}{$(1.4,0)$}{right}{z''}
\putdb{C}{A}
\putdw{C}{A'}
\putdw{C'}{A}
\putdb{C'}{A'}
\end{tikzpicture}|\times|Q_{0,3}(z',z'')|\\
&\lesssim\|z'-z\|_s^{\frac12-\kappa}\|z''-z\|_s^{\frac12-\kappa}\|z'-z''\|_s^{-2}
\end{align*}
for small $\kappa>0$. For the third order terms, we note that the required bounds are obtained by multiplying \begin{tikzpicture}[baseline=-0.1cm]
\putdots{A}{$(0,0)$}{left}{z'}
\putdotg{C}{$(0.7,0)$}{right}{}
\putdots{A'}{$(1.4,0)$}{right}{z''}
\putdb{C}{A}
\putdw{C}{A'}
\end{tikzpicture} to \eqref{section4_4 second of CICCD}, so we have the bound
\begin{align*}
(\|z'-z''\|_s^{-2\kappa}+\|z'-z\|_s^{-\kappa}\|z''-z\|_s^{-\kappa})\|z'-z''\|_s^{-1}.
\end{align*}
For the first order terms, we need to introduce the renormalization
\begin{align*}
&\begin{tikzpicture}[baseline=-0.1cm]
\putdots{A}{$(0,0)$}{left}{z'}
\putdotg{B}{$(0.5,0.4)$}{right}{}
\putdotg{C}{$(0.5,-0.4)$}{right}{}
\putdotg{D}{$(1,0)$}{right}{}
\putdot{E}{$(2,0)$}{right}{}
\putdb{B}{A}
\putdw{C}{A}
\dputb{D}{A}{$(0.5,0)$}
\putdw{B}{D}
\putdb{C}{D}
\putdb{E}{D}
\end{tikzpicture}
-\begin{tikzpicture}[baseline=-0.1cm]
\putdots{A}{$(0,0)$}{left}{z'}
\putdotg{B}{$(0.5,0.4)$}{right}{}
\putdotg{C}{$(0.5,-0.4)$}{right}{}
\putdotg{D}{$(1,0)$}{right}{}
\putdb{B}{A}
\putdw{C}{A}
\putb{D}{A}
\putdw{B}{D}
\putdb{C}{D}
\end{tikzpicture}
\begin{tikzpicture}[baseline=0cm]
\putdots{A}{$(0,0)$}{left}{z'}
\putdot{B}{$(1,0)$}{right}{}
\putdb{B}{A}
\end{tikzpicture}\\
&=\mathfrak{R}L^\epsilon*K^\epsilon(z'-
\begin{tikzpicture}[baseline=-0.1cm]
\putdot{A}{$(0,0)$}{left}{}
\end{tikzpicture})-
\begin{tikzpicture}[baseline=-0.1cm]
\putdots{A}{$(0,0)$}{left}{z'}
\putdotg{B}{$(0.5,0.4)$}{right}{}
\putdotg{C}{$(0.5,-0.4)$}{right}{}
\putdotg{D}{$(1,0)$}{right}{}
\putdot{E}{$(2,0.3)$}{right}{}
\putdots{F}{$(2,-0.3)$}{right}{z}
\putdb{B}{A}
\putdw{C}{A}
\putb{D}{F}
\putdw{B}{D}
\putdb{C}{D}
\putdb{E}{D}
\end{tikzpicture},
\end{align*}
where $L^\epsilon(z'-w)=\begin{tikzpicture}[baseline=-0.1cm]
\putdots{A}{$(0,0)$}{left}{z'}
\putdotg{B}{$(0.5,0.4)$}{right}{}
\putdotg{C}{$(0.5,-0.4)$}{right}{}
\putdots{D}{$(1,0)$}{right}{w}
\putdb{B}{A}
\putdw{C}{A}
\putb{D}{A}
\putdw{B}{D}
\putdb{C}{D}
\end{tikzpicture}$ and $\mathfrak{R}L^\epsilon$ is the distribution defined by
\begin{align*}
\langle\mathfrak{R}L^\epsilon,\varphi\rangle=\int L^\epsilon(z)(\varphi(z)-\varphi(0))dz
\end{align*}
for test function $\varphi$, see \cite[Definition~10.15]{Hairer2014a}.
By \cite[Lemma~10.16]{Hairer2014a}, we have the bound
\begin{align*}
|\mathfrak{R}L^\epsilon*K^\epsilon(z)|\lesssim\|z\|_s^{-3}.
\end{align*}
For the remaining term, we have
\begin{align*}
|\begin{tikzpicture}[baseline=-0.1cm]
\putdots{A}{$(0,0)$}{left}{z'}
\putdotg{B}{$(0.5,0.4)$}{right}{}
\putdotg{C}{$(0.5,-0.4)$}{right}{}
\putdotg{D}{$(1,0)$}{right}{}
\putdotg{E}{$(2,0.3)$}{right}{}
\putdots{F}{$(1.8,-0.3)$}{below}{z}
\putdots{A'}{$(4,0)$}{right}{z''}
\putdotg{B'}{$(3.5,0.4)$}{right}{}
\putdotg{C'}{$(3.5,-0.4)$}{right}{}
\putdotg{D'}{$(3,0)$}{right}{}
\putdots{F'}{$(2.2,-0.3)$}{below}{z}
\putdb{B}{A}
\putdw{C}{A}
\putb{D}{F}
\putdw{B}{D}
\putdb{C}{D}
\putdb{E}{D}
\putdw{B'}{A'}
\putdb{C'}{A'}
\putw{D'}{F'}
\putdb{B'}{D'}
\putdw{C'}{D'}
\putdw{E}{D'}
\end{tikzpicture}|&\lesssim
|\begin{tikzpicture}[baseline=-0.1cm]
\putdots{A}{$(0,0.4)$}{left}{z'}
\putdots{A'}{$(0,-0.4)$}{left}{z}
\putdotg{B}{$(1,0)$}{right}{}
\putdotg{C}{$(2,0)$}{right}{}
\putdots{D}{$(3,0.4)$}{right}{z''}
\putdots{D'}{$(3,-0.4)$}{right}{z}
\putl{B}{A}{$(0.3,0.1)$}{$(0.7,0.3)$}{$(0.5,0.2)$}{-2}
\putl{B}{A'}{$(0.3,-0.3)$}{$(0.7,-0.1)$}{$(0.5,-0.2)$}{-3}
\putcl{C}{$(1.7,0)$}{$(1.3,0)$}{B}{$(1.3,-0.1)$}{$(1.7,0.1)$}{$(1.5,-0)$}{-1}
\putl{C}{D}{$(2.3,0.1)$}{$(2.7,0.3)$}{$(2.5,0.2)$}{-2}
\putl{C}{D'}{$(2.3,-0.3)$}{$(2.7,-0.1)$}{$(2.5,-0.2)$}{-3}
\end{tikzpicture}|\\
&\lesssim\|z'-z\|_s^{-\frac12-\kappa}\|z''-z\|_s^{-\frac12-\kappa}.
\end{align*}
As a consequence, we have
\begin{align*}
\boldsymbol{E}[|\langle\hat{\Pi}_z^\epsilon\CDICCD,\varphi_z^\delta\rangle|^2]\lesssim\delta^{-1-2\kappa}.
\end{align*}

\subsection{Behaviors of renormalization constants}\label{4:subsection5}

In \secref[4:subsection3_1]{4:subsection3_2}, we obtained renormalization constants
\begin{equation}\label{4:subsection5:C_1C_2C_3}
\begin{aligned}
C_1^\epsilon&=\begin{tikzpicture}[baseline=-0.1cm]
\putdots{A}{$(0,0)$}{left}{z'}
\putdotg{B}{$(1,0)$}{right}{}
\putcdb{B}{$(0.7,0.2)$}{$(0.3,0.2)$}{A}
\putcdw{B}{$(0.7,-0.2)$}{$(0.3,-0.2)$}{A}
\end{tikzpicture}=\int|K^\epsilon(z)|^2dz,\\
C_{2,1}^\epsilon&=\begin{tikzpicture}[baseline=-0.1cm]
\putdots{A}{$(0,0)$}{left}{z'}
\putdotg{B}{$(0.5,0.4)$}{right}{}
\putdotg{C}{$(0.5,-0.4)$}{right}{}
\putdotg{D}{$(1,0)$}{right}{}
\putdb{B}{A}
\putdb{C}{A}
\putw{D}{A}
\putdw{B}{D}
\putdw{C}{D}
\end{tikzpicture}=\int\overline{K}(z)(Q^\epsilon(z))^2dz,\\
C_{2,2}^\epsilon&=\begin{tikzpicture}[baseline=-0.1cm]
\putdots{A}{$(0,0)$}{left}{z'}
\putdotg{B}{$(0.5,0.4)$}{right}{}
\putdotg{C}{$(0.5,-0.4)$}{right}{}
\putdotg{D}{$(1,0)$}{right}{}
\putdb{B}{A}
\putdw{C}{A}
\putb{D}{A}
\putdw{B}{D}
\putdb{C}{D}
\end{tikzpicture}=\int K(z)|Q^\epsilon(z)|^2dz,
\end{aligned}
\end{equation}
where
\[
Q^\epsilon(z)=\int K^\epsilon(z-w)\overline{K^\epsilon}(-w)dw.
\]
Note that $Q^\epsilon=Q*\pi^\epsilon$, where $\pi(t,x)=\epsilon^{-5}\pi(\epsilon^{-2}t,\epsilon^{-1}x)$ and
\begin{align*}
Q(z)=\int K(z-w)\overline{K}(-w)dw,\quad
\pi(z)=\int\rho(z-w)\rho(-w)dw,
\end{align*}

\begin{proposition}\label{section4_5:estimate of C}
There exist constants $C_1,C_{2,1}$ and $C_{2,2}$ independent of $\epsilon$ such that
\[
C_1^\epsilon\sim C_1\epsilon^{-1},\quad
C_{2,1}^\epsilon\sim C_{2,1}\log\epsilon^{-1},\quad
C_{2,1}^\epsilon\sim C_{2,2}\log\epsilon^{-1}
\]
as $\epsilon\downarrow0$. Here for two functions $A_\epsilon$ and $B_\epsilon$ of $\epsilon$, we write $A_\epsilon\sim B_\epsilon$ if there exists a constant $C$ independent of $\epsilon$ and $|A_\epsilon-B_\epsilon|\le C$ holds.
\end{proposition}

In order to prove the above estimates, we prepare some notations. For $\alpha>0$ and a compactly supported function $A\in C^\infty(\RealNum^4\setminus\{0\},\CmplNum)$, we say that $A\in\mathfrak{S}_\alpha$ if there exists a function $\widetilde{A}\in C^\infty(\RealNum^4\setminus\{0\},\CmplNum)$ and $A_-\in L^\infty(\RealNum^4,\CmplNum)$ such that
\begin{itemize}
\item $\widetilde{A}=A+A_-$ on $\RealNum^4\setminus\{0\}$,
\item $\widetilde{A}(\lambda^2t,\lambda x)=\lambda^{-\alpha}\widetilde{A}(t,x)$ for every $\lambda>0$ and $(t,x)\neq0$,
\item $|A_-(z)|\lesssim\|z\|_s^{-\alpha}$ for every $z\in\RealNum^4$.
\end{itemize}
The second scaling property of $\widetilde{A}$ ensures that $|\widetilde{A}(z)|\lesssim\|z\|_s^{-\alpha}$ for every $z\in\RealNum^4$ (see \cite[Lemma~5.5]{Hairer2014a}).

\begin{proposition}
Let $\alpha,\beta\in(0,5)$.
\begin{enumerate}
\item[(1)] If $A\in\mathfrak{S}_\alpha$, $B\in\mathfrak{S}_\beta$ and $\alpha+\beta>5$, then $A*B\in\mathfrak{S}_{\alpha+\beta-5}$.
\item[(2)] If $A\in\mathfrak{S}_\alpha$ and $B\in\mathfrak{S}_\beta$, then $AB\in\mathfrak{S}_{\alpha+\beta}$.
\end{enumerate}
\end{proposition}

\begin{proof}
For (1), in the decomposition
\[
A*B=\widetilde{A}*\widetilde{B}-A_-*\widetilde{B}-\widetilde{A}*B_-+A_-*B_-,
\]
we see that the last three terms are bounded by $\|z\|_s^{5-\alpha-\beta}$
by using \cite[Lemma~4.14]{Hoshino2016}. Hence it suffices to set $\widetilde{A*B}=\widetilde{A}*\widetilde{B}$.

The assertion (2) is similarly obtained.
\end{proof}

\begin{proof}[Proof of \pref{section4_5:estimate of C}]
First we show the estimate of $C_1^\epsilon$. Since $K,\overline{K}\in\mathfrak{S}_3$, we have $Q\in\mathfrak{S}_1$. Hence we have
\begin{align*}
C_1^\epsilon&=Q^\epsilon(0)=\int Q(-z)\pi^\epsilon(z)dz\\
&=\int \widetilde{Q}(-z)\pi^\epsilon(z)dz-\int Q_-(-z)\pi^\epsilon(z)dz\\
&=\epsilon^{-1}\int\widetilde{Q}(-z)\pi(z)dz+\mathcal{O}(1).
\end{align*}
The last equality follows from the scaling property of $\widetilde{Q}$ and the boundedness of $Q_-$.

Next we show the estimate of $C_{2,1}^\epsilon$. Note that
\begin{align*}
C_{2,1}^\epsilon&=\int\overline{K}(z)(Q^\epsilon(z))^2dz\sim\int_{\|z\|_s>\epsilon}\overline{K}(z)(Q(z))^2dz.
\end{align*}
Indeed, since $|Q^\epsilon(z)|\lesssim\epsilon^{-1}$ and $|Q(z)-Q^\epsilon(z)|\lesssim\epsilon^\theta\|z\|_s^{-1-\theta}$ for every $\theta\in(0,1]$ (see \cite[Lemma~10.17]{Hairer2014a}), we have
\begin{align*}
\left|\int_{\|z\|_s\le\epsilon}\overline{K}(z)(Q^\epsilon(z))^2dz\right|\lesssim\epsilon^{-2}\int_{\|z\|_s\le\epsilon}\|z\|_s^{-3}dz\lesssim1
\end{align*}
and
\begin{align*}
&\left|\int_{\|z\|_s>\epsilon}\overline{K}(z)\{(Q^\epsilon(z))^2-(Q(z))^2\}dz\right|\lesssim\epsilon^\theta\int_{\|z\|_s>\epsilon}\|z\|_s^{-5-\theta}dz\lesssim1.
\end{align*}
Hence it suffices to consider $\int_{\|z\|_s>\epsilon}R(z)dz$, where $R=\overline{K}Q^2\in\mathfrak{S}_5$. However, we replace $R$ by a function $S_+$ defined below. Let $\varphi$ be a smooth and nonnegative function such that $\supp(\varphi)\subset\{\frac12\le\|z\|_s\le2\}$ and $\sum_{n=-\infty}^\infty\varphi(2^{2n}t,2^nx)=1$ for all $(t,x)\neq0$. Define
\[
S_+=\sum_{n=0}^\infty\widetilde{R}\varphi_n,\quad
S_-=\sum_{n=-\infty}^{-1}\widetilde{R}\varphi_n,
\]
where $\varphi_n(t,x)=\varphi(2^{2n}t,2^nx)$. Note that $R+R_-=S_++S_-$. Since $\supp(S_--R_-)$ is compact, we have
\[
\int_{\|z\|_s>\epsilon}R(z)dz=\int_{\|z\|_s>\epsilon}S_+(z)dz+\mathcal{O}(1).
\]
By the scaling property of $\widetilde{R}$, we have $S_+=\sum_{n=0}^\infty S_n$, where
\[
S_n(t,x)=2^{5n}S_0(2^{2n},2^nx),\quad
S_0=\widetilde{R}\varphi.
\]
Since $S_0\in C_0^\infty$, we have
\begin{align*}
\int_{\|z\|_s>\epsilon}S_+(z)dz
&=\sum_{n=0}^\infty\int_{\|z\|_s>\epsilon}S_n(z)dz\\
&=\sum_{n=0}^\infty\int_{\|z\|_s>\epsilon2^n}S_0(z)dz
=\sum_{n=0}^{N(\epsilon)}\int_{\|z\|_s>\epsilon2^n}S_0(z)dz,
\end{align*}
where $N(\epsilon)\in\NaturalNum$ is the largest number such that
$\{\|z\|_s>\epsilon2^{N(\epsilon)}\}\cap\supp(S_0)\neq\emptyset$,
so there exists a constant $C>0$ such that $N(\epsilon)\sim C\log\epsilon^{-1}$. Since
\begin{align*}
&\sum_{n=0}^{N(\epsilon)}\left|\int_{\|z\|_s>\epsilon2^n}S_0(z)dz-\int S_0(z)dz\right|\\
&\lesssim\sum_{n=0}^{N(\epsilon)}\int_{\|z\|_\epsilon\le\epsilon2^n}dz
\lesssim\sum_{n=0}^{N(\epsilon)}\epsilon^52^{5n}\lesssim\epsilon^52^{5N(\epsilon)}\lesssim1,
\end{align*}
we have the estimate
\begin{align*}
C_{2,1}^\epsilon=N(\epsilon)\int S_0(z)dz+\mathcal{O}(1)\sim C\int S_0(z)dz\log\epsilon^{-1}.
\end{align*}

The estimate of $C_{2,2}^\epsilon$ is similar.
\end{proof}

%% file: 0550_WellposednessOfCGLByPara.tex

\section{CGL by the theory of paracontrolled distributions}\label{sec_20161107073643}

In \secref[sec_20161107073643]{sec_20161107073600}, we study well-posedness of CGL \eqref{eq:cgl}
by using the paracontrolled distribution theory introduced by \cite{GubinelliImkellerPerkowski2015}.
In that paper, they studied some problems such as differential equations driven by fractional Brownian motion,
a Burgers-type stochastic PDE, and a nonlinear version of the parabolic Anderson model.
After that Catellier-Chouk \cite{CatellierChouk2013Arxiv} showed local well-posedness of
the three-dimensional stochastic quantization equation
(the dynamic $\Phi^4_3$ model), which is an $\RealNum$-valued version of CGL.

Our proof of the local well-posedness of CGL consists of two parts:
a deterministic and a probabilistic part.

In \secref{sec_20161107073643}, we deal with a deterministic version of CGL
and construct a solution map from a space of driving vectors to a space of solutions.
We also see that the solution map is continuous.
In this section, $\whiteNoise$ is a deterministic distribution
which takes values in the H\"older-Besov space $\HolBesSp{-\frac{5}{2}-\kappa}$ for any $\kappa>0$ small enough.
To construct the solution map, we rely on the method introduced by Mourrat-Weber \cite{MourratWeber2016Arxiv}.
We state the precise assertion concerning the well-posedness in \tref{thm_20160920085605}.
In \tref{thm_20160920094302}, we see that the solution obtained in \tref{thm_20160920085605} solves
the renormalized equation \eqref{eq:rcgl} in the usual sense.

\secref{sec_20161107073600} is the probabilistic part
and devoted to constructing a driving vector $X$ associated to the space-time white noise $\xi$ defined on $\RealNum\times\Torus^3$.
We follow the approach as in \cite{GubinelliPerkowski2017}
and obtain the  driving vector in \tref{thm_20161128051824}.
Here we explain how to mollify the white noise $\xi$.
Let $\chi$ be a smooth real-valued function defined on $\RealNum^3$ such that
(1) $\supp\chi\subset B(0,1)$, where $B(x,r)$ denotes the open ball of radius $r>0$ and center $x\in\RealNum^3$,
(2) $\chi(0)=1$.
We set $\chi^\epsilon(k)=\chi(\epsilon k)$ for every $k\in\Integers^3$.
Define $\FourierBase[k](x)=e^{2\pi\ImUnit k\cdot x}$ for every $k\in\Integers^3$ and $x\in\Torus^3$.
Here, the dot $\cdot$ denotes the usual inner product.
We define $\xi^\epsilon$ by
\begin{align}\label{eq_20161201075726}
 	\whiteNoise^\epsilon
 	=
 		\sum_{k\in\Integers^3}
 			\chi^\epsilon(k)
 			\FourierCoeff{\whiteNoise}(k)
 			\FourierBase[k].
\end{align}
Here, $\{\FourierCoeff{\whiteNoise}(k)\}_{k\in\Integers^3}$ denotes the Fourier transform of $\whiteNoise$
and it has the same law with independent copies of the complex white noise on $\RealNum$.
We see that $\whiteNoise^\epsilon\to\whiteNoise$ in an appropriate topology.
For the smeared noise $\whiteNoise^\epsilon$, we define a family of processes $\{X^\epsilon\}_{0<\epsilon<1}$.
In this definition of $X^\epsilon$, we will use the dyadic partition of unity $\{\DyaPartOfUnit[m]\}_{m=-1}^\infty$ via the resonant $\reso$
and renormalization constants $\constRe[1]{\epsilon}$, $\constRe[2,1]{\epsilon}$ and $\constRe[2,2]{\epsilon}$;
see \secref{sec_1501991661} for the definitions of $\{\DyaPartOfUnit[m]\}_{m=-1}^\infty$ and $\reso$
and see \eqref{eq_20161015073329} for the renormalization constants.
We obtain the driving vector $X$ as a limit of $\{X^\epsilon\}_{0<\epsilon<1}$.
By setting
$
  \constRe{\epsilon}
  =
    2
    (
      \constRe[1]{\epsilon}
      -\CmplConj{\nu} \CmplConj{\constRe[2,1]{\epsilon}}
      -2\nu \constRe[2,2]{\epsilon}
    )
$,
we have $|\constRe{\epsilon}|\to\infty$ as $\epsilon\to 0$.

By combining \trefs{thm_20160920085605}{thm_20160920094302}{thm_20161128051824},
we obtain the following main theorem in \secref[sec_20161107073643]{sec_20161107073600}:
\begin{theorem}\label{thm_2017013154042}
  Let $0<\kappa'<1/18$ and $u_0\in\HolBesSp{-\frac{2}{3}+\kappa'}$.
	Consider the renormalized equation \eqref{eq:rcgl} with $C^\epsilon=\constRe{\epsilon}$.
  Then, for every $0<\epsilon<1$,
  there exist a unique process $u^\epsilon$ and a random time $T_\ast^\epsilon\in(0,1]$ such that
	\begin{itemize}
		\item	$u^\epsilon$ solves \eqref{eq:rcgl} on $[0,T_\ast^\epsilon]\times\Torus^3$,
		\item	$T_\ast^\epsilon$ converges to some a.s.~positive random time $T_\ast$ in probability,
		\item	$u^\epsilon$ converges to some process $u$ defined on $[0,T_\ast)\times\Torus^3$ in the following sense:
				\begin{align*}
					\lim_{\varepsilon\downarrow 0}
					\sup_{0\leq s\leq T_\ast/2}
						\|u^\epsilon_s-u_s\|_{\HolBesSp{-\frac{2}{3}+\kappa'}}
					=
						0
				\end{align*}
				in probability.
        Here, we set
        $
          \sup_{0\leq s\leq T_\ast/2}
            \|u^\epsilon_s-u_s\|_{\HolBesSp{-\frac{2}{3}+\kappa'}}
          =
            \infty
        $
        on the event $\{T_\ast^\epsilon<T_\ast/2\}$.
				Furthermore, $u$ is independent of the choice of $\{\DyaPartOfUnit[m]\}_{m=-1}^\infty$ and $\chi$.
	\end{itemize}
\end{theorem}
Here, we will make comments on this theorem.
Note that the process $u^\epsilon$ and $u$ are obtained by substituting $X^\epsilon$ and $X$ into the solution map, respectively.
Since $X^\epsilon$ converges to $X$ and the solution map is continuous, we see that $u^\epsilon$ converges to $u$.
In addition, $u^\epsilon$ solves \eqref{eq:rcgl} in the usual sense, hence we see the theorem.
We need to pay attention to the assertion that $u$ is independent of the choice of $\whiteNoise^\epsilon$.
Recall that $X^\epsilon$ depends on $\{\DyaPartOfUnit[m]\}_{m=-1}^\infty$.
Hence $u^\epsilon$ may, too.
However, we see that $u^\epsilon$ does not.
In fact, we obtain an expression of the renormalization constant $\constRe{\epsilon}$
which does not depend on $\{\DyaPartOfUnit[m]\}_{m=-1}^\infty$ in \pref{prop_20161107071120}.
Hence, \eqref{eq:rcgl} is independent of $\{\DyaPartOfUnit[m]\}_{m=-1}^\infty$ and so is the solution $u^\epsilon$.
As a consequence, the limit $u$ is independent of $\{\DyaPartOfUnit[m]\}_{m=-1}^\infty$.
In addition, the limit $u$ is independent of $\chi$ because the driving vector $X$ is independent of $\chi$ (\tref{thm_20161128051824}).
Hence we see the solution does not depend on $\{\DyaPartOfUnit[m]\}_{m=-1}^\infty$ or $\chi$.

\begin{remark}\label{0550 tx regularize is ok}
As stated in \secref{sec_1503967525},
we can choose common approximation noise for the renormalized equation \eqref{eq:rcgl}
to obtain the solutions in \tref[0330_main result]{thm_2017013154042}.
In this sense, the solutions in \tref[0330_main result]{thm_2017013154042} "essentially coincide," or at least look very similar.

In \tref{thm_2017013154042}, the noise $\whiteNoise$ is smeared only in spatial direction.
However, we can consider the case that the noise is smeared both in temporal and spatial directions.
For a non-negative Schwartz function $\varrho$ on $\RealNum^4$ such that $\int\varrho=1$,
we consider the scaling $\varrho^\epsilon(t,x)=\epsilon^{-5}\varrho(\epsilon^{-2}t,\epsilon^{-1}x)$,
which is the mollifier considered in \tref{0330_main result},
and replace $\whiteNoise$ by smooth noise
\begin{align}\label{eq_1503279335}
    \tilde{\whiteNoise}^\epsilon(t,x)
    =
        \int_{\RealNum\times\RealNum^3}
            \varrho^\epsilon(t-s,x-y)
            \xi(s,y)\,
            dsdy.
\end{align}
Then the same claim as \tref{thm_2017013154042} holds for the renormalized equation \eqref{eq:rcgl} with $\tilde{\xi}^\epsilon$,
under well-adjusted choice of $C^\epsilon$. Moreover, the limit process coincides with that in \tref{thm_2017013154042}.
This is because the limit driving vector dose not change
under the different choice of approximations (\rref{0550 X is invariant under diff approx}).
\end{remark}

Finally, we should note that Hoshino showed
the global-in-time well-posedness of CGL \eqref{eq:cgl}
in the case that $\mu>1/\sqrt{8}$ and $\Re \nu>0$ \cite{Hoshino2017b}.

\secref[sec_20161107073643]{sec_20161107073600} are independent of \secref[sec_20161107073826]{sec_20161107073513}.
We do not use the symbols introduced in \secref[sec_20161107073826]{sec_20161107073513}.

\subsection{Besov-H\"older spaces and paradifferential calculus}\label{sec_1501991661}
In this section, we introduce the Besov-H\"older spaces and paradifferential calculus.
The results in this section can be found in \cite{GubinelliImkellerPerkowski2015,BahouriCheminDanchin2011}
or follow from them easily.

\subsubsection{Besov spaces}
We introduce the Besov spaces and recall their basic properties.
Let $\cD\equiv\cD(\Torus^3,\CmplNum)$ be
the space of all smooth $\CmplNum$-valued functions on $\Torus^3$
and $\cD'$ its dual of $\cD$.
We set $\FourierBase[k](x)=e^{2\pi\ImUnit k\cdot x}$ for every $k\in\Integers^3$ and $x\in\Torus^3$.
The Fourier transform $\FourierTrans{f}$ for $f\in\cD$ is defined by
$
	\FourierTrans f(k)
	=
		\int_{\Torus^3}
			\FourierBase[-k](x)
			f(x)\,
			dx
$
and its inverse $\FourierTrans^{-1}g$ for a rapidly decreasing sequence $\{g(k)\}_{k\in\Integers^3}$ is defined by
$
	\FourierTrans^{-1} g(x)
	=
		\sum_{k\in\Integers^3}
			g(k)
			\FourierBase[k](x)
$.
For every rapidly decreasing smooth function $\phi$, we set
$
	\phi(D)f
	=
		\FourierTrans^{-1}\phi\FourierTrans f
	=
		\sum_{k\in\Integers^3}
			\phi(k)
			\FourierCoeff{f}(k)
			\FourierBase[k]
$.

We denote by $\{\DyaPartOfUnit[m]\}_{m=-1}^\infty$ a dyadic partition of unity, that is,
it satisfies the following:
(1) $\DyaPartOfUnit[m]\colon\RealNum^3\to[0,1]$ is radial and smooth,
(2) $
		\supp(\DyaPartOfUnit[-1])
		\subset
			B(0,\frac{4}{3})
	$,
	$
		\supp(\DyaPartOfUnit[0])
		\subset
			B(0,\frac{8}{3})\setminus B(0,\frac{3}{4})
	$,
(3) $\DyaPartOfUnit[m](\cdot)=\DyaPartOfUnit[0](2^{-m}\cdot)$ for $m\geq 0$,
(4) $
		\sum_{m=-1}^\infty
			\DyaPartOfUnit[m](\cdot)
		=
			1
	$.
Here $B(0,r)=\{x\in\RealNum^3;|x|<r\}$.
The Littlewood-Paley blocks $\{\LPBlock[m]\}_{m=-1}^\infty$ are defined by $\LPBlock[m]=\DyaPartOfUnit[m](D)$.

We are ready to define Besov space $\HolBesSp{\alpha}=\HolBesSp{\alpha}(\Torus^3,\CmplNum)$
for $\alpha\in\RealNum$.
It is defined as the completion of $\cD$ under the norm
\begin{align*}
	\|f\|_{\HolBesSp{\alpha}}
	=
		\sup_{m\geq -1}
			2^{m\alpha}
			\|\LPBlock[m] f\|_{L^\infty}.
\end{align*}
The next is frequently used results on Besov spaces:
\begin{proposition}[{\cite[Theorem~2.80]{BahouriCheminDanchin2011}}]\label{prop_20160928054620}
	We have the following:
	\begin{itemize}
		\item	$
					\|f\|_{\HolBesSp{\alpha}}
					\lesssim
						\|f\|_{\HolBesSp{\beta}}
				$
				if $\alpha<\beta$.
		\item	Let $\alpha,\alpha_1,\alpha_2\in\RealNum$ satisfy
				$
					\alpha
					=
						(1-\theta)\alpha_1+\theta\alpha_2
				$
				for some $0<\theta<1$. Then, we have
				\begin{align*}
					\|f\|_{\HolBesSp{\alpha}}
					\leq
						\|f\|_{\HolBesSp{\alpha_1}}^{1-\theta}
						\|f\|_{\HolBesSp{\alpha_2}}^\theta.
				\end{align*}
	\end{itemize}
\end{proposition}

\subsubsection{Paraproducts and Commutator estimates}
For every $f\in\HolBesSp{\alpha}$, $g\in\HolBesSp{\beta}$, we set
\begin{gather*}
	\begin{aligned}
		f\lpara g
		&=
			\sum_{m_1\geq m_2+2}
				\LPBlock[m_1]f
				\LPBlock[m_2]g,
		&
		f\rpara g
		&=
			\sum_{m_1+2\leq m_2}
				\LPBlock[m_1]f
				\LPBlock[m_2]g,
	\end{aligned}\\
	f\reso g
	=
		\sum_{|m_1-m_2|\leq 1}
			\LPBlock[m_1]f
			\LPBlock[m_2]g.
\end{gather*}

The following are properties of paraproduct.
\begin{proposition}[Paraproduct and resonant estimate {\cite[Theorem~2.82 and 2.85]{BahouriCheminDanchin2011}}]\label{prop_20160926062106}
	We have the following:
	\begin{enumerate}[(1)]
		\item	For every $\beta\in\RealNum$,
				$
					\|f\rpara g\|_{\HolBesSp{\beta}}
					\lesssim
						\|f\|_{L^\infty}
						\|g\|_{\HolBesSp{\beta}}
				$.
		\item	For every $\alpha <0$ and $\beta\in\RealNum$,
				$
					\|f\rpara g\|_{\HolBesSp{\alpha+\beta}}
					\lesssim
						\|f\|_{\HolBesSp{\alpha}}
						\|g\|_{\HolBesSp{\beta}}
				$.
		\item	If $\alpha+\beta>0$, then
				$
					\|f\reso g\|_{\HolBesSp{\alpha+\beta}}
					\lesssim
						\|f\|_{\HolBesSp{\alpha}}
						\|g\|_{\HolBesSp{\alpha}}
				$.
	\end{enumerate}
\end{proposition}
\begin{remark}
	Let $f\in\HolBesSp{\alpha}$ and $g\in\HolBesSp{\beta}$ for $\alpha<0$ and $\beta>0$ with $\alpha+\beta>0$.
	Then the product $fg$ is well-defined as an element in $\HolBesSp{\alpha}$.
\end{remark}

\begin{proposition}[Commutator estimates, {\cite[Lemma~2.4]{GubinelliImkellerPerkowski2015}}]\label{prop_comm_20160919055939}
	Let $0<\alpha<1$, $\beta,\gamma\in\RealNum$ satisfy $\beta+\gamma<0$ and $\alpha+\beta+\gamma>0$.
	Define the map $\comR$ by
	\begin{align*}
		\comR(f,g,h)
		=
			(f\rpara g)\reso h-f(g\reso h)
	\end{align*}
	for $f,g,h\in\SmoothFunc{\infty}{\Torus^3}{\CmplNum}$.
	Then $\comR$ is uniquely extended to a continuous trilinear map from
	$
		\HolBesSp{\alpha}
		\times\HolBesSp{\beta}
		\times\HolBesSp{\gamma}
	$
	to
	$\HolBesSp{\alpha+\beta+\gamma}$.
\end{proposition}

\subsubsection{Regularity of $\HolBesSp{\alpha}$-valued functions}
Here we consider $\HolBesSp{\alpha}$-valued functions and introduce several classes of them.
Let $0<\delta\leq 1$ and $\eta\geq 0$ and define these classes as follows:
\begin{itemize}
	\item	$C_T\HolBesSp{\alpha}$ is the space of all continuous functions from $[0,T]$ to $\HolBesSp{\alpha}$
			which is equipped with the supremum norm
			\begin{align*}
				\|u\|_{C_T\HolBesSp{\alpha}}
				=
					\sup_{0\leq t\leq T}
						\|u_t\|_{\HolBesSp{\alpha}},
			\end{align*}
	\item	$C_T^\delta\HolBesSp{\alpha}$ is the space of all $\delta$-H\"older continuous functions
			from $[0,T]$ to $\HolBesSp{\alpha}$ which is equipped with the seminorm
			\begin{align*}
				\|u\|_{C_T^\delta\HolBesSp{\alpha}}
				=
					\sup_{0\leq s<t\leq T}
						\frac{\|u_t-u_s\|_{\HolBesSp{\alpha}}}{|t-s|^\delta},
			\end{align*}
	\item	$
				\cE_T^\eta\HolBesSp{\alpha}
				=
					\{
						u\in \SmoothFunc{}{(0,T]}{\HolBesSp{\alpha}};
						\|u\|_{\cE_T^\eta\HolBesSp{\alpha}}<\infty
					\}
			$,
			where
			\begin{align*}
				\|u\|_{\cE_T^\eta\HolBesSp{\alpha}}
				=
					\sup_{0<t\leq T}
						t^\eta\|u_t\|_{\HolBesSp{\alpha}},
			\end{align*}
	\item	$
				\cE_T^{\eta,\delta}\HolBesSp{\alpha}
				=
					\{
						u\in \SmoothFunc{}{(0,T]}{\HolBesSp{\alpha}};
						\|u\|_{\cE_T^{\eta,\delta}\HolBesSp{\alpha}}<\infty
					\}
			$,
			where
			\begin{align*}
				\|u\|_{\cE_T^{\eta,\delta}\HolBesSp{\alpha}}
				=
					\sup_{0<s<t\leq T}
						s^\eta
						\frac{\|u_t-u_s\|_{\HolBesSp{\alpha}}}{|t-s|^\delta},
			\end{align*}
	\item	$
				\cL_T^{\alpha,\delta}
				=
					C_T\HolBesSp{\alpha}
					\cap
					C_T^\delta\HolBesSp{\alpha-2\delta}
			$,
	\item	$
				\cL_T^{\eta,\alpha,\delta}
				=
					\cE_T^\eta\HolBesSp{\alpha}
					\cap
					C_T\HolBesSp{\alpha-2\eta}
					\cap
					\cE_T^{\eta,\delta}\HolBesSp{\alpha-2\delta}
			$.
\end{itemize}

\begin{remark}
	We introduced the norms on the spaces $\cE_T^\eta\HolBesSp{\alpha}$ and $\cE_T^{\eta,\delta}\HolBesSp{\alpha}$
	in order to control explosion at $t=0$.
	The definition of $\cL_T^{\alpha,\delta}$ is natural from the time-space scaling of CGL.
\end{remark}

For $\mu>0$, we set $\mathcal{L}^1=\partial_t-\{(\ImUnit+\mu)\LaplaceOp-1\}$, $P^1_t=e^{t\{(\ImUnit+\mu)\LaplaceOp-1\}}$.
We present results on smoothing effects of semigroup $\{P^1_t\}_{t\geq 0}$.
\begin{proposition}[Effects of heat semigroup]\label{prop_20160927051055}
	Let $\alpha\in\RealNum$.
	\begin{enumerate}[(1)]
		\item	For every $\delta>0$,
				$
					\|P^1_t u\|_{\HolBesSp{\alpha+2\delta}}
					\lesssim
						t^{-\delta}
						\|u\|_{\HolBesSp{\alpha}}
				$
				uniformly in $t>0$.
		\item	For every $\delta\in[0,1]$,
				$
					\|(P^1_t-1)u\|_{\HolBesSp{\alpha-2\delta}}
					\lesssim
						t^\delta
						\|u\|_{\HolBesSp{\alpha}}
				$
				uniformly in $t>0$.
	\end{enumerate}

	\begin{proposition}[Schauder estimates]\label{prop_20160930075129}
		Let $T>0$.
		We see the following:
		\begin{enumerate}[(1)]
			\item	Let $u\in\HolBesSp{\alpha}$.
					For every $\beta\geq \alpha$ and $\delta\in[0,1]$, we have
					\begin{align*}
						\|
							(t\mapsto P^1_t u)_{t\geq 0}
						\|_{\cL_T^{\frac{\beta-\alpha}{2},\beta,\delta}}
						\lesssim
							\|u\|_{\HolBesSp{\alpha}}.
					\end{align*}
			\item	Let $\alpha\neq\beta$.
					Let $u\in\cE_T^\eta\HolBesSp{\alpha}$ for $\eta\in[0,1)$ and set
					\begin{align*}
						U_t
						=
							\int_0^t
								P^1_{t-s}
								u_s\,
								ds.
					\end{align*}
					Then for every $\gamma\in[\alpha,\alpha-2\eta+2)$,
					$\beta\in[\gamma,\alpha+2)$ and $\delta\in(0,\frac{\beta-\alpha}{2}]$,
					we have
					\begin{align*}
						\|U\|_{\cL_T^{\frac{\beta-\gamma}{2},\beta,\delta}}
						\lesssim
							T^{\frac{\alpha-2\eta+2-\gamma}{2}}
							\|u\|_{\cE_T^\eta\HolBesSp{\alpha}}.
					\end{align*}
		\end{enumerate}
	\end{proposition}

	\begin{proposition}[Commutation between paraproduct and heat semigroup]
		\label{prop_20160927051159}
		Let $\alpha<1$, $\beta\in\RealNum$, $\delta\geq 0$.
		Define
		\begin{align*}
			[P^1_t,u\rpara]v
			=
				P^1_t (u\rpara v)
				-
				u\rpara P^1_t v.
		\end{align*}
		Then we have
		\begin{align*}
			\|
				[P^1_t,u\rpara]v
			\|_{\HolBesSp{\alpha+\beta+2\delta}}
			\lesssim
				t^{-\delta}
				\|u\|_{\HolBesSp{\alpha}}
				\|v\|_{\HolBesSp{\beta}}
		\end{align*}
		uniformly over $t>0$.
	\end{proposition}
\end{proposition}

We can show the above results in a similar way as
\cite[Corollary~2.6, Proposition~2.8]{Hoshino2016arXiv}
and \cite[Propsosition~A.15]{MourratWeber2016Arxiv},
because $\mu>0$.

\subsection{Definitions of driving vectors and solutions}

First of all, we give the definition of a driving vector.
We set
\begin{gather}
	\label{eq_20161010092955}
	I(u_0,v)_t
	=
		P^1_t u_0
		+
		\int_0^t
			P^1_{t-s} v_s\,
			ds,\\
	\label{eq_20161006075406}
	I(v)_t
	=
		\int_{-\infty}^t
			P^1_{t-s} v_s\,
			ds,
\end{gather}
whenever they are well-defined.
Note that if we can choose
$
	u_0
	=
		\int_{-\infty}^0
			P^1_{-s}v_s\,
			ds
$,
then $I(u_0,v)=I(v)$.

Let $0<\kappa<\kappa'<1/18$ and $T>0$.
The following is the definition of a driving vector.
\begin{definition}
	We call a vector of space-time distributions
	\begin{multline*}
		X
		=
			(
				X^\A,
				X^\AA,
				X^\AB,
				X^\IAA,
				X^\IAB,
				X^\IAAB,\\
				X^\IAABoA,
				X^\IAABoB,
				X^\IAAoAB,
				X^\IAAoBB,
				X^\IABoAB,
				X^\IABoBB,
				X^\IAABoAB,
				X^\IAABoBB
		)\\
		\in
			C_T\HolBesSp{-\frac{1}{2}-\kappa}
			\times
			(C_T\HolBesSp{-1-\kappa})^2
			\times
			(C_T\HolBesSp{1-\kappa})^2
			\times
			\cL_T^{\frac{1}{2}-\kappa,\frac{1}{4}-\frac{1}{2}\kappa}
			\times
			(C_T\HolBesSp{-\kappa})^6
			\times
			(C_T\HolBesSp{-\frac{1}{2}-\kappa})^2
	\end{multline*}
	which satisfies
	$
		I(X^\IAA_0,X^\AA)=X^\IAA
	$
	and
	$
		I(X^\IAB_0,X^\AB)=X^\IAB
	$
	a \emph{driving vector} of CGL.
	We denote by $\drivers{T}{\kappa}$ the set of all driving vectors.
	We define the norm $\|\cdot\|_{\drivers{T}{\kappa}}$ by the sum of the norm of each component.
\end{definition}
Note that we assume that the component $X^\IAAB$ has H\"older continuity
and it belongs to $\cL_T^{\frac{1}{2}-\kappa,\frac{1}{4}-\frac{1}{2}\kappa}$.
We easily see that the space $\drivers{T}{\kappa}$ is a closed set of the product Banach spaces.
Next we define the space of solutions and give the notion of a solution.

We describe the tree-like symbols $\A$, $\AA$, $\AB$, $\IAA$,\dots in the definition.
The dot and the line denote the white noise and the operation $I$, respectively.
Hence, $\A$ represents $I(\whiteNoise)=Z$.
The symbols $\B$ and $\AB$ stand for the complex conjugate of $Z$ and the product $Z\CmplConj{Z}$, respectively.
So $\IAAB$ means $I(Z^2\CmplConj{Z})$.
Finally, $\IAABoA$ denotes the resonance term of $I(Z^2\CmplConj{Z})$ and $Z$.
\begin{definition}
	We set
	\begin{align*}
		\sols{T}{\kappa}{\kappa'}
		=
			\cL_T^{\frac{5}{6}-\kappa',1-\kappa',1-\frac{1}{2}\kappa'}
			\times
			\cL_T^{1-\kappa'+\kappa,\frac{3}{2}-2\kappa',1-\kappa'}.
	\end{align*}
\end{definition}

Next, we fix $X\in\drivers{T}{\kappa}$ and set $Z=X^\A$ and $W=X^\IAAB$.
Define $F$ and $G$ on $\sols{T}{\kappa}{\kappa'}$ by
\begin{align}
	\label{eq_20161002045615}
	F(v,w)
	&=
		-
		\nu
		\{
			2(-\nu X^{\IAAB}+v+w)\rpara X^\AB
			+(-\CmplConj{\nu}\CmplConj{X^{\IAAB}}+\CmplConj{v}+\CmplConj{w})\rpara X^\AA
		\},\\
	\label{eq_20161002050746}
	G(v,w)
	&=
		G_1(v,w)
		+
		\cdots
		+
		G_8(v,w).
\end{align}
Here $G_1(v,w),\dots,G_8(v,w)$ will be defined shortly.

Since $Z_t\in\HolBesSp{-\frac{1}{2}-\kappa}$ and $W_t\in\HolBesSp{\frac{1}{2}-\kappa}$,
the product $W_tZ_t$ and $W_t\CmplConj{Z}_t$ are not defined a priori.
We define them by
\begin{align*}
	WZ
	&=
		W(\rpara+\lpara)Z+X^\IAABoA,
	&
	W\CmplConj{Z}
	&=
		W(\rpara+\lpara)Z+X^\IAABoB.
\end{align*}
The products $W^2\CmplConj{Z}$ and $W\CmplConj{W}Z$ are also defined by
\begin{gather*}
	\begin{aligned}
		W^2\CmplConj{Z}
		&=
			2WX^\IAABoB
			+\comR(W,\CmplConj{Z},W)\\
		&\phantom{=}\quad\qquad\qquad
			+(W\rpara\CmplConj{Z})(\rpara+\lpara)W
			+W(W\lpara\CmplConj{Z}),
	\end{aligned}\\
	\begin{aligned}
		W\CmplConj{W}Z
		&=
			\CmplConj{W}X^\IAABoA
			+W\CmplConj{X^\IAABoB}
			+\comR(\CmplConj{W},Z,W)\\
		&\phantom{=}\quad\qquad\qquad
			+(\CmplConj{W}\rpara Z)(\rpara+\lpara)W
			+W(\CmplConj{W}\lpara Z).
	\end{aligned}
\end{gather*}
It follows from \pref{prop_20160926062106} that
$
	(WZ)_t,
	(W\CmplConj{Z})_t,
	(W^2\CmplConj{Z})_t,
	(W\CmplConj{W}Z)_t
	\in
		\HolBesSp{-\frac{1}{2}-\kappa}
$
hold.

In order to define $G_6(v,w)$, we use $\com(v,w)$ defined as follows.
For every $v_0\in\HolBesSp{-\frac{2}{3}+\kappa'}$ and $(v,w)\in\sols{T}{\kappa}{\kappa'}$, we set
\begin{align*}
	\FourierCoeff{v}_t
	&=
		P^1_tv_0
		+
		\int_0^t
			P^1_{t-s}
			[F(v,w)(s)]\,
			ds.
\end{align*}
Define
\begin{align}\label{eq_20161129003748}
	\com(v,w)
	=
		\FourierCoeff{v}
		+
		\nu
		\{
			2(-\nu W+v+w)\rpara X^\IAB
			+(-\CmplConj{\nu}\CmplConj{W}+\CmplConj{v}+\CmplConj{w})\rpara X^\IAA
		\}.
\end{align}
From \lref{lem_20161003140013}, we see that $\com(v,w)\reso X^\AB$
and $\CmplConj{\com(v,w)}\reso X^\AA$ are well-defined.
Roughly speaking, $\com(v,w)_t$ is something like
\begin{align*}
  [[I(\cdot),-2\nu(-\nu W+v+w)\rpara]X^\AB]_t
  +[[I(\cdot),-\nu(-\CmplConj{\nu}\CmplConj{W}+\CmplConj{v}+\CmplConj{w})\rpara]X^\AA]_t.
\end{align*}
Here, $[[I(\cdot),u\rpara] v]_t=I(u\rpara v)_t-u_t\rpara I(v)_t$.

We are in a position to define $G_1$,\dots,$G_8$.
We write $u_2=v+w$ and set
\begin{align*}
	G_1(v,w)
	&=
		-\nu u_2^2\CmplConj{u_2},\\
	G_2(v,w)
	&=
		-
		\nu
		\{
			u_2^2
			(\CmplConj{Z}-\CmplConj{\nu} \CmplConj{W})
			+
			2
			u_2
			\CmplConj{u_2}
			(Z-\nu W)
		\},\\
	G_3(v,w)
	&=
		-
		\nu
		\big\{
			u_2
			(
				2\nu\CmplConj{\nu}W\CmplConj{W}
				-2\nu W\CmplConj{Z}
				-2\CmplConj{\nu}\CmplConj{W}Z
				-4\nu X^\IABoAB
				-\CmplConj{\nu} \CmplConj{X^\IAAoBB}
			)\\
	&\phantom{=}\qquad\qquad\qquad\quad
			+
			\CmplConj{u_2}
			(
				\nu^2W^2
				-2\nu WZ
				-2\CmplConj{\nu} \CmplConj{X^\IABoBB}
				-2\nu X^\IAAoAB
			)
		\big\}\\
	&\phantom{=}\quad
		+
		(\nu+1)u_2,\\
	G_4(v,w)
	&=
		-
		\nu
		\big\{
			-\nu^2\CmplConj{\nu}W^2\CmplConj{W}
			+\nu^2W^2\CmplConj{Z}
			+2\nu\CmplConj{\nu}W\CmplConj{W}Z\\
	&\phantom{=}\quad\qquad
			+4\nu^2 WX^\IABoAB
			+4\nu^2 \comR(W,X^\IAB,X^\AB)\\
	&\phantom{=}\quad\qquad
			+2\CmplConj{\nu}^2 \CmplConj{W}\CmplConj{X^\IABoBB}
			+2\CmplConj{\nu}^2 \comR(\CmplConj{W},\CmplConj{X^\IAB},X^\AA)\\
		&\phantom{=}\quad\qquad
			+2\nu\CmplConj{\nu} \CmplConj{W}X^\IAAoAB
			+2\nu\CmplConj{\nu} \comR(\CmplConj{W},X^\IAA,X^\AB)\\
	&\phantom{=}\quad\qquad
			+\nu\CmplConj{\nu} W\CmplConj{X^\IAAoBB}
			+\nu\CmplConj{\nu} \comR(W,\CmplConj{X^\IAA},X^\AA)\\
	&\phantom{=}\quad\qquad
			-2\nu X^\IAABoAB
			-2\nu W\lpara X^{\AB}\\
	&\phantom{=}\quad\qquad
			-\CmplConj{\nu}\CmplConj{X^\IAABoBB}
			-\CmplConj{\nu}\CmplConj{W}\lpara X^{\AA}
		\big\}\\
	&\phantom{=}\quad
		+(\nu+1)(Z-\nu W),\\
	G_5(v,w)
	&=
		-
		\nu
		\big\{
			-4\nu \comR(u_2,X^\IAB,X^\AB)
			-2\nu \comR(\CmplConj{u_2},X^\IAA,X^\AB)\\
	&\phantom{=}\quad\qquad\qquad
			-2\CmplConj{\nu}\comR(\CmplConj{u_2},\CmplConj{X^\IAB},X^\AA)
			-\CmplConj{\nu}\comR(u_2,\CmplConj{X^\IAA},X^\AA)
		\big\},\\
	G_6(v,w)
	&=
		-\nu
		\{
			2\com(v,w)\reso X^\AB
			+
			\CmplConj{\com(v,w)}\reso X^\AA
		\},\\
	G_7(v,w)
	&=
		-\nu
		\{
			2w\reso X^\AB
			+
			\CmplConj{w}\reso X^\AA
		\},\\
	G_8(v,w)
	&=
		-\nu
		\{
			2u_2\lpara  X^\AB
			+
			\CmplConj{u_2}\lpara  X^\AA
		\}.
\end{align*}

The map
$
	\mathcal{M}
	=
		(\mathcal{M}^1,\mathcal{M}^2)
$
is defined on $\sols{T}{\kappa}{\kappa'}$ by
\begin{align}
	\label{eq_20161003034438}
	[\mathcal{M}^1(v,w)](t)
	&=
		P^1_tv_0
		+
		\int_0^t
			P^1_{t-s}
			F(v,w)(s)\,
			ds,\\
	\label{eq_20161003034500}
	[\mathcal{M}^2(v,w)](t)
	&=
		P^1_t w_0
		+
		\int_0^t
			P^1_{t-s}
			G(v,w)(s)\,
			ds
\end{align}
for every
$
	(v_0,w_0)
	\in
		\HolBesSp{-\frac{2}{3}+\kappa'}
		\times
		\HolBesSp{-\frac{1}{2}-2\kappa}
$.
We will use \pref{prop_20160930075129} to check that
the map is well-defined map from $\sols{T}{\kappa}{\kappa'}$ to itself
and has good property.

\begin{definition}\label{eq_20160921004656}
	For every
	$
		(v_0,w_0)
		\in
			\HolBesSp{-\frac{2}{3}+\kappa'}
			\times
			\HolBesSp{-\frac{1}{2}-2\kappa}
	$
	and
	$
		X\in\drivers{T}{\kappa}
	$,
	we consider the system
	\begin{gather}\label{eq_20160920090410}
		\left\{
			\begin{aligned}
				v_t
				&=
					[\mathcal{M}^1(v,w)](t),\\
				w_t
				&=
					[\mathcal{M}^2(v,w)](t).
			\end{aligned}
		\right.
	\end{gather}

	If there exists $(v,w)\in\sols{T_\ast}{\kappa}{\kappa'}$ for some $T_\ast>0$,
	then we call $(v,w)$ the solution to \eqref{eq:cgl} on $[0,T_\ast]$.
\end{definition}

In \tref{thm_20160920094302}, we see that the solution obtained in the sense of this definition solves
the renormalized equation \eqref{eq:rcgl} in the usual sense.
Hence, this definition is proper.

We interpret \eqref{eq_20160920090410} as a fixed point problem $\mathcal{M}:(v,w)\mapsto(\mathcal{M}^1(v,w),\mathcal{M}^2(v,w))$.
We show that the map $\mathcal{M}$ is well-defined
and a contraction in \secref{sec_20161002045325}.
\secref{sec_20160920093609} is devoted to the construction and the uniqueness of the solution.
We show that the solution to \eqref{eq:cgl} satisfies a renormalized equation in \secref{sec_20160920094017}.
In that section, we see the validity of the notion of the solution to CGL.

Before starting our discussion,
we will remark on the function spaces
we have just introduced.
\begin{remark}
	We make several comments on $\cL_T^{\eta,\alpha}$ and $\cL_T^{\eta,\alpha,\delta}$.
	\begin{itemize}
		\item	The inclusion
				$
					\cL_T^{\alpha,\delta}
					\subset
					\cL_T^{\alpha,\delta'}
				$
				holds for every $0<\delta'\leq\delta$.
				To prove this assertion, we use \pref{prop_20160928054620}.
				Set $\theta=\delta'/\delta$.
				Then $\alpha-2\delta'=(\alpha-2\delta)\theta+\alpha(1-\theta)$.
				For every $W\in\cL_T^{\alpha,\delta}$, we see
				\begin{align*}
					\|W_t-W_s\|_{\HolBesSp{\alpha-2\delta'}}
					&\leq
						\|W_t-W_s\|_{\HolBesSp{\alpha-2\delta}}^\theta
						\|W_t-W_s\|_{\HolBesSp{\alpha}}^{1-\theta}\\
					&\lesssim
						\{(t-s)^\delta\|W\|_{C_T^\delta\HolBesSp{\alpha-2\delta}}\}^\theta
						\|W\|_{C_T\HolBesSp{\alpha}}^{1-\theta}\\
					&\lesssim
						(t-s)^{\delta'}
						\|W\|_{\cL_T^{\alpha,\delta}}.
				\end{align*}
		\item	The inclusion
				$
					\cL_T^{\eta,\alpha,\delta}
					\subset
					\cL_T^{\eta,\alpha,\delta'}
				$
				holds for $0<\delta'\leq\delta$.
				Indeed, for every $v\in\cL_T^{\eta,\alpha,\delta}$ and $0<s<t\leq T$, we have
				\begin{gather*}
					\|v_t-v_s\|_{\HolBesSp{\alpha-2\delta}}
					\leq
						s^{-\eta}|t-s|^\delta
						\|v\|_{\cE_T^{\eta,\delta}\HolBesSp{\alpha-2\delta}},\\
					\|v_t-v_s\|_{\HolBesSp{\alpha}}
					\leq
						\|v_t\|_{\HolBesSp{\alpha}}
						+\|v_s\|_{\HolBesSp{\alpha}}
					\leq
						t^{-\eta}\|v\|_{\cE_T^\eta\HolBesSp{\alpha}}
						+s^{-\eta}\|v\|_{\cE_T^\eta\HolBesSp{\alpha}}.
				\end{gather*}
				Hence, for $\theta=\delta'/\delta$, we see
				\begin{align*}
					\|v_t-v_s\|_{\HolBesSp{\alpha-2\delta'}}
					&\leq
						\|v_t-v_s\|_{\HolBesSp{\alpha-2\delta}}^\theta
						\|v_t-v_s\|_{\HolBesSp{\alpha}}^{1-\theta}\\
					&\lesssim
						\{
							s^{-\eta}
							|t-s|^\delta
							\|v\|_{\cE_T^{\eta,\delta}\HolBesSp{\alpha-2\delta}}
						\}^\theta
						\{
							s^{-\eta}
							\|v\|_{\cE_T^\eta\HolBesSp{\alpha}}
						\}^{1-\theta}\\
					&\lesssim
						s^{-\eta}
						|t-s|^{\delta'}
						\|v\|_{\cL_T^{\eta,\alpha,\delta}},
				\end{align*}
				which implies $v\in\cL_T^{\eta,\alpha,\delta'}$.
		\item	For every $v\in\cL_T^{\eta,\alpha,\delta}$
				and $\alpha-2\eta\leq\gamma\leq\alpha$,
				we have $v_t\in\HolBesSp{\gamma}$ and
				\begin{align*}
					\|v_t\|_{\HolBesSp{\gamma}}
					\leq
						t^{-\frac{1}{2}\left(\gamma-(\alpha-2\eta)\right)}
						\|v\|_{\cL_T^{\eta,\alpha,\delta}}
				\end{align*}
				for any $0<t\leq T$.
				Since $v_t\in\HolBesSp{\alpha-2\eta}\cap\HolBesSp{\alpha}$,
				we take $\theta$ such that
				$
					\gamma
					=
						(\alpha-2\eta)(1-\theta)
						+\alpha\theta
				$
				and	use \pref{prop_20160928054620} to obtain
				\begin{align*}
					\|v_t\|_{\HolBesSp{\gamma}}
					\leq
						\|v_t\|_{\HolBesSp{\alpha-2\eta}}^{1-\theta}
						\|v_t\|_{\HolBesSp{\alpha}}^\theta
					\leq
						\|v\|_{C_T\HolBesSp{\alpha-2\eta}}^{1-\theta}
						\{t^{-\eta}\|v\|_{\cE_T^\eta\HolBesSp{\alpha}}\}^\theta
					\leq
						t^{-\eta\theta}
						\|v\|_{\cL_T^{\eta,\alpha,\delta}}.
				\end{align*}
				Combining this with
				$
					\eta\theta
					=
						\frac{1}{2}\left(\gamma-(\alpha-2\eta)\right)
				$,
				we see the assertion.
	\end{itemize}
\end{remark}
\begin{remark}\label{rem_20160921064624}
	We make several comments on
	$
		\cL_T^{\frac{1}{2}-\kappa,\frac{1}{4}-\frac{1}{2}\kappa}
	$
	and
	$
		\sols{T}{\kappa}{\kappa'}
	$.
	Recall that
	\begin{align*}
		\cL_T^{\frac{1}{2}-\kappa,\frac{1}{4}-\frac{1}{2}\kappa}
		&=
			C_T\HolBesSp{\frac{1}{2}-\kappa}
			\cap
			C_T^{\frac{1}{4}-\frac{1}{2}\kappa}\HolBesSp{0},\\
		\cL_T^{\frac{5}{6}-\kappa',1-\kappa',1-\frac{1}{2}\kappa'}
		&=
			\mathcal{E}_T^{\frac{5}{6}-\kappa'}\HolBesSp{1-\kappa'}
			\cap
			C_T\HolBesSp{-\frac{2}{3}+\kappa'}
			\cap
			\mathcal{E}_T^{\frac{5}{6}-\kappa',1-\frac{1}{2}\kappa'}\HolBesSp{-1},\\
		\cL_T^{1-\kappa'+\kappa,\frac{3}{2}-2\kappa',1-\kappa'}
		&=
			\mathcal{E}_T^{1-\kappa'+\kappa}\HolBesSp{\frac{3}{2}-2\kappa'}
			\cap
			C_T\HolBesSp{-\frac{1}{2}-2\kappa}
			\cap
			\mathcal{E}_T^{1-\kappa'+\kappa,1-\kappa'}\HolBesSp{-\frac{1}{2}}.
	\end{align*}

	\begin{itemize}
		\item	For every $0<\kappa<\kappa'$, we have
				$
					\cL_T^{\frac{1}{2}-\kappa,\frac{1}{4}-\frac{1}{2}\kappa}
					\subset
					\cL_T^{\frac{1}{2}-\kappa,\frac{1}{4}-\frac{1}{2}\kappa'}
					=
						C_T\HolBesSp{\frac{1}{2}-\kappa}
						\cap
						C_T^{\frac{1}{4}-\frac{1}{2}\kappa'}\HolBesSp{\kappa'-\kappa}
				$.
				This inclusion implies
				\begin{align*}
					\|X^\IAAB_t-X^\IAAB_s\|_{\HolBesSp{\kappa'-\kappa}}
					\lesssim
						(t-s)^{\frac{1}{4}-\frac{1}{2}\kappa'}
						\|X^\IAAB\|_{C_T^{\frac{1}{4}-\frac{1}{2}\kappa}\HolBesSp{\kappa'-\kappa}}.
				\end{align*}
		\item	Note that
				\begin{align*}
					\sols{T}{\kappa}{\kappa'}
					&=
						\cL_T^{\frac{5}{6}-\kappa',1-\kappa',1-\frac{1}{2}\kappa'}
						\times
						\cL_T^{1-\kappa'+\kappa,\frac{3}{2}-2\kappa',1-\kappa'}\\
					&\subset
						\cL_T^{\frac{5}{6}-\kappa',1-\kappa',\frac{1}{2}-\kappa'}
						\times
						\cL_T^{1-\kappa'+\kappa,\frac{3}{2}-2\kappa',\frac{1}{2}-\kappa'}
				\end{align*}
				holds and,
				for every $(v,w)\in\sols{T}{\kappa}{\kappa'}$, we have
				\begin{gather*}
					\|v_t-v_s\|_{L^\infty}
					\lesssim
						\|v_t-v_s\|_{\HolBesSp{\kappa'}}
					\lesssim
						s^{-\left(\frac{5}{6}-\kappa'\right)}
						(t-s)^{\frac{1}{2}-\kappa'}
						\|v\|_{\cL_T^{\frac{5}{6}-\kappa',1-\kappa',1-\frac{1}{2}\kappa'}},\\
					\|w_t-w_s\|_{L^\infty}
					\lesssim
						\|w_t-w_s\|_{\HolBesSp{\kappa'}}
					\lesssim
						s^{-(1-\kappa'+\kappa)}
						(t-s)^{\frac{1}{2}-\kappa'}
						\|w\|_{\cL_T^{1-\kappa'+\kappa,\frac{3}{2}-2\kappa',1-\kappa'}}.
				\end{gather*}
		\item	For every $(v,w)\in\sols{T}{\kappa}{\kappa'}$,
				we have
				\begin{gather*}
					\|v_t\|_{\HolBesSp{\alpha}}
					\leq
						t^{
							-\frac{1}{2}
							\left(
								\alpha+\frac{2}{3}-\kappa'
							\right)
						}
						\|v\|_{\cL_T^{\frac{5}{6}-\kappa',1-\kappa',1-\frac{1}{2}\kappa'}},\\
					\|w_t\|_{\HolBesSp{\beta}}
					\leq
						t^{
							-\frac{1}{2}
							\left(
								\beta+\frac{1}{2}+2\kappa
							\right)
						}
						\|w\|_{\cL_T^{1-\kappa'+\kappa,\frac{3}{2}-2\kappa',1-\kappa'}},
				\end{gather*}
				where $\alpha$ and $\beta$ satisfy
				$
					-\frac{2}{3}+\kappa'
					\leq
						\alpha
					\leq
						1-\kappa'
				$
				and
				$
					-\frac{1}{2}-2\kappa
					\leq
						\beta
					\leq
						\frac{3}{2}-2\kappa'
				$.

				In particular, for $\alpha=\beta=\kappa'-\kappa$, we have
				\begin{gather*}
					\|v_t\|_{L^\infty}
					\leq
						\|v_t\|_{\HolBesSp{\kappa'-\kappa}}
					\leq
						t^{-\frac{2-3\kappa}{6}}
						\|v\|_{\cL_T^{\frac{5}{6}-\kappa',1-\kappa',1-\frac{1}{2}\kappa'}},\\
					\|w_t\|_{L^\infty}
					\leq
						\|w_t\|_{\HolBesSp{\kappa'-\kappa}}
					\leq
						t^{
							-\frac{1}{2}
							\left(
								\frac{1}{2}+\kappa'+\kappa
							\right)
						}
						\|w\|_{\cL_T^{1-\kappa'+\kappa,\frac{3}{2}-2\kappa',1-\kappa'}},\\
					\|v_t+w_t\|_{L^\infty}
					\leq
						\|v_t\|_{L^\infty}+\|w_t\|_{L^\infty}
					\lesssim
						t^{-\frac{2-3\kappa}{6}}
						\|(v,w)\|_{\sols{T}{\kappa}{\kappa'}}.
				\end{gather*}
				In the last estimate, we used $0<\kappa<\kappa'<1/18$
				and $0<t\leq T$.
				We also see
				\begin{align*}
					\|v_t+w_t\|_{\HolBesSp{\gamma}}
					\lesssim
						t^{
							-\frac{1}{2}
							\left(
								\gamma+\frac{2}{3}-\kappa'
							\right)
						}
						\|(v,w)\|_{\sols{T}{\kappa}{\kappa'}}
				\end{align*}
				if
				$
					-\frac{1}{2}-2\kappa
					\leq
						\gamma
					\leq
						1-\kappa'
				$.
	\end{itemize}
\end{remark}

\subsection{Properties of the integration map}\label{sec_20161002045325}

Let $0<\kappa<\kappa'<1/18$ and  $0<T\leq 1$.
We fix $X,X^{(1)},X^{(2)}\in\drivers{1}{\kappa}$
and set $Z=X^\A$, $W=X^{\IAAB}$, $Z^{(i)}=X^{\A,(i)}$ and $W^{(i)}=X^{(i),\IAAB}$ for $i=1,2$.
We sometimes use the symbol $F_X$, $G_X$ and
$
	\mathcal{M}_{(v_0,w_0),X}
	=
		(\mathcal{M}_{(v_0,w_0),X}^1,\mathcal{M}_{(v_0,w_0),X}^2)
$
to indicate the dependence on the driving vector $X$ and the initial data $(v_0,w_0)$.

\subsubsection{Properties of $\mathcal{M}^1$}
Let us start our discussion with $F$.
\begin{lemma}\label{lem_20161002130857}
	For any $(v,w)\in\sols{T}{\kappa}{\kappa'}$ and $0<t\leq T$,
	we have $F(v,w)(t)\in\HolBesSp{-1-\kappa}$ and
	\begin{align*}
		\|F(v,w)(t)\|_{\HolBesSp{-1-\kappa}}
		\leq
			\const
			(
				\|X\|_{\drivers{1}{\kappa}}
				+
				t^{-\frac{2-3\kappa}{6}}
				\|(v,w)\|_{\sols{T}{\kappa}{\kappa'}}
			)
			\|X\|_{\drivers{1}{\kappa}},
	\end{align*}
	where $\const$ is a positive constant depending only on $\kappa$, $\kappa'$, $\mu$ and $\nu$.
\end{lemma}
\begin{proof}
	Set	$\Phi=-\nu W+v+w$.
	From $W_t\in\HolBesSp{\frac{1}{2}-\kappa}$,
	$v_t\in\HolBesSp{1-\kappa'}$
	and $w_t\in\HolBesSp{\frac{3}{2}-2\kappa'}$,
	we see that	$\Phi_t\in\HolBesSp{\frac{1}{2}-\kappa}$
	and
	$
		\|\Phi_t\|_{L^\infty}
		\leq
			\|W_t\|_{L^\infty}
			+\|v_t\|_{L^\infty}
			+\|w_t\|_{L^\infty}
	$
	hold for every $(v,w)\in\sols{T}{\kappa}{\kappa'}$.
	Note that
	$
		\|W_t\|_{L^\infty}
		\lesssim
			\|W_t\|_{\HolBesSp{\frac{1}{2}-\kappa}}
		\leq
			\|X\|_{\drivers{T}{\kappa}}
	$
	holds from \pref{prop_20160928054620}.
	From \rref{rem_20160921064624}, we see
	\begin{gather*}
		\|v_t\|_{L^\infty}
		\lesssim
			t^{-\frac{2-3\kappa}{6}}
			\|v\|_{\cL_T^{\frac{5}{6}-\kappa',1-\kappa',1-\frac{1}{2}\kappa'}}
		\leq
			t^{-\frac{2-3\kappa}{6}}
			\|(v,w)\|_{\sols{T}{\kappa}{\kappa'}},\\
		\|w_t\|_{L^\infty}
		\lesssim
			t^{
				-\frac{1}{2}
				\left(
					\frac{1}{2}+\kappa'+\kappa
				\right)
			}
			\|w\|_{\cL_T^{1-\kappa'+\kappa,\frac{3}{2}-2\kappa',1-\kappa'}}
		\leq
			t^{-\frac{2-3\kappa}{6}}
			\|(v,w)\|_{\sols{T}{\kappa}{\kappa'}}.
	\end{gather*}
	Combining this with $X^\AB_t\in\HolBesSp{-1-\kappa}$ and
	using \pref{prop_20160926062106},
	we see
	$
		\Phi_t
		\rpara
		X^\AB_t
		\in
			\HolBesSp{-1-\kappa}
	$
	and
	\begin{align*}
		\|
			\Phi_t
			\rpara
			X^\AB_t
		\|_{\HolBesSp{-1-\kappa}}
		&\lesssim
			\|\Phi_t\|_{L^\infty}
			\|X^\AB_t\|_{\HolBesSp{-1-\kappa}}\\
		&\leq
			(
				\|X\|_{\drivers{1}{\kappa}}
				+
				t^{-\frac{2-3\kappa}{6}}
				\|(v,w)\|_{\sols{T}{\kappa}{\kappa'}}
			)
			\|X\|_{\drivers{1}{\kappa}}.
	\end{align*}
	The term
	$
		\CmplConj{\Phi_t}
		\rpara
		X^\AA_t
	$
	also has a similar bound.
	From the defintion of $F(v,w)$, we see the assertion.
\end{proof}

\begin{proposition}\label{prop_20161003143833}
	The map
	$
		\mathcal{M}^1:
			\sols{T}{\kappa}{\kappa'}
			\to
			\cL_T^{\frac{5}{6}-\kappa',1-\kappa',1-\frac{1}{2}\kappa'}
	$
	is well-defined and,
	for any
	$
		(v,w)\in\sols{T}{\kappa}{\kappa'}
	$,
	we have
	\begin{align*}
		\|\mathcal{M}^1(v,w)\|_{\cL_T^{\frac{5}{6}-\kappa',1-\kappa',1-\frac{1}{2}\kappa'}}
		\leq
			\const[1]
			(
				1+\|v_0\|_{\HolBesSp{-\frac{2}{3}+\kappa'}}
			)
			+
			\const[2]
			T^{\frac{1-\kappa'}{2}}
			\|(v,w)\|_{\sols{T}{\kappa}{\kappa'}}.
	\end{align*}
	Here, $\const[1]$ and $\const[2]$ are positive constants depending only on $\kappa$, $\kappa'$, $\mu$, $\nu$ and $\|X\|_{\drivers{1}{\kappa}}$.
	In particular, they are given by at most second-order polynomials in $\|X\|_{\drivers{1}{\kappa}}$.
\end{proposition}
\begin{proof}
	Applying the first assertion of \pref{prop_20160930075129}
	with $\alpha=-\frac{2}{3}+\kappa'$, $\beta=1-\kappa'$ and $\delta=1-\frac{1}{2}\kappa'$
	to $v_0\in\HolBesSp{-\frac{2}{3}+\kappa'}$, we see
	\begin{align*}
		\|
			(t\mapsto P^1_tv_0)_{t\geq 0}
		\|_{\cL_T^{\frac{5}{6}-\kappa',1-\kappa',1-\frac{1}{2}\kappa'}}
		\lesssim
			\|v_0\|_{\HolBesSp{-\frac{2}{3}+\kappa'}}.
	\end{align*}

	From \lref{lem_20161002130857}, we see
	$F(v,w)\in\cE_T^{\frac{2-3\kappa}{6}}\HolBesSp{-1-\kappa}$
	and its norm has an upper bound
	$
		\const[1]
		(1+\|(v,w)\|_{\sols{T}{\kappa}{\kappa'}})
	$.
	Here, $\const[1]$ is positive and is given by a polynomial in $\|X\|_{\drivers{1}{\kappa}}$.
	Applying the second assertion in \pref{prop_20160930075129}
	with $\alpha=-1-\kappa$, $\beta=1-\kappa'$,
	$\gamma=-\frac{2}{3}+\kappa'$,
    $\delta=1-\frac{1}{2}\kappa'$
	and $\eta=\frac{2-3\kappa}{6}$
	to $F(v,w)\in\cE_T^{\frac{2-3\kappa}{6}}\HolBesSp{-1-\kappa}$, we see
	\begin{multline*}
		\left\|
			{
				(
					t
					\mapsto
						\int_0^t
							P^1_{t-s} F(v,w)(s)\,
							ds
				)_{t\geq 0}
			}
		\right\|_{\cL_T^{\frac{5}{6}-\kappa',1-\kappa',1-\frac{1}{2}\kappa'}}\\
		\lesssim
			T^{\frac{1-\kappa'}{2}}
			\|F(v,w)\|_{\cE_T^{\frac{2-3\kappa}{6}}\HolBesSp{-1-\kappa}}
		\leq
			T^{\frac{1-\kappa'}{2}}
			\const[2]
			(1+\|(v,w)\|_{\sols{T}{\kappa}{\kappa'}}).
	\end{multline*}
	The proof is completed.
\end{proof}

Next we show that $\mathcal{M}^1$ is Lipschitz continuous.
\begin{lemma}\label{lem_20161003043650}
	For any
	$
		(v^{(1)},w^{(1)}),(v^{(2)},w^{(2)})\in\sols{T}{\kappa}{\kappa'}
	$
	and
	$0<t\leq T$,
	we have
	\begin{multline*}
		\|
			F_{X^{(1)}}(v^{(1)},w^{(1)})(t)
			-
			F_{X^{(2)}}(v^{(2)},w^{(2)})(t)
		\|_{\HolBesSp{-1-\kappa}}\\
		\leq
			\const
			(1+t^{-\frac{2-3\kappa}{6}})
			\left\{
				\|X^{(1)}-X^{(2)}\|_{\drivers{1}{\kappa}}
				+\|(v^{(1)},w^{(1)})-(v^{(2)},w^{(2)})\|_{\sols{T}{\kappa}{\kappa'}}
			\right\}.
	\end{multline*}
	Here, $\const$ is a positive constant depending only on $\kappa$, $\kappa'$, $\mu$, $\nu$,
	$\|X^{(i)}\|_{\drivers{1}{\kappa}}$ and	$\|(v^{(i)},w^{(i)})\|_{\sols{T}{\kappa}{\kappa'}}$.
	In particular, it is given by a first-order polynomial in
	$\|X^{(i)}\|_{\drivers{1}{\kappa}}$ and	$\|(v^{(i)},w^{(i)})\|_{\sols{T}{\kappa}{\kappa'}}$.
\end{lemma}
\begin{proof}
	Set	$\Phi^{(i)}=-\nu W^{(i)}+v^{(i)}+w^{(i)}$ for $i=1,2$.
	Then
	\begin{multline*}
		\|
			\Phi^{(1)}_t\rpara X^{(1),\AB}_t
			-
			\Phi^{(2)}_t\rpara X^{(2),\AB}_t
		\|_{\HolBesSp{-1-\kappa}}\\
		\begin{aligned}
			&=
				\frac{1}{2}
				\|
					(\Phi^{(1)}_t+\Phi^{(2)}_t)
					\rpara
					(X^{(1),\AB}_t-X^{(2),\AB}_t)\\
			&\phantom{=}\quad\qquad\qquad\qquad
					+
					(\Phi^{(1)}_t-\Phi^{(2)}_t)
					\rpara
					(X^{(1),\AB}_t+X^{(2),\AB}_t)
				\|_{\HolBesSp{-1-\kappa}}\\
			&\lesssim
				\|\Phi^{(1)}_t+\Phi^{(2)}_t\|_{L^\infty}
				\|X^{(1),\AB}_t-X^{(2),\AB}_t\|_{\HolBesSp{-1-\kappa}}\\
			&\phantom{=}\quad\qquad\qquad\qquad
				+
				\|\Phi^{(1)}_t-\Phi^{(2)}_t\|_{L^\infty}
				\|X^{(1),\AB}_t+X^{(2),\AB}_t\|_{\HolBesSp{-1-\kappa}}.
		\end{aligned}
	\end{multline*}
	The term $\|\Phi^{(1)}_t+\Phi^{(2)}_t\|_{L^\infty}$ is dominated as follows:
	\begin{align*}
		\|\Phi^{(1)}_t+\Phi^{(2)}_t\|_{L^\infty}
		&\lesssim
			\|\Phi^{(1)}_t\|_{L^\infty}+\|\Phi^{(2)}_t\|_{L^\infty}\\
		&\lesssim
			\sum_{i=1,2}
				(
					\|X^{(i)}\|_{\drivers{T}{\kappa}}
					+
					t^{-\frac{2-3\kappa}{6}}
					\|(v,w)^{(i)}\|_{\sols{T}{\kappa}{\kappa'}}
				)
				\|X^{(i)}\|_{\drivers{T}{\kappa}}\\
		&=
			\const[1]+t^{-\frac{2-3\kappa}{6}}\const[2],
	\end{align*}
	where $\const[1]$ and $\const[2]$ are positive constants given by
	\begin{align*}
		\const[1]
		&=
			\|X^{(1)}\|_{\drivers{1}{\kappa}}^2
			+\|X^{(2)}\|_{\drivers{1}{\kappa}}^2,\\
		\const[2]
		&=
			\|(v^{(1)},w^{(1)})\|_{\sols{T}{\kappa}{\kappa'}}
			\|X^{(1)}\|_{\drivers{1}{\kappa}}
			+
			\|(v^{(2)},w^{(2)})\|_{\sols{T}{\kappa}{\kappa'}}
			\|X^{(2)}\|_{\drivers{1}{\kappa}}.
	\end{align*}
	The term $\|\Phi^{(1)}_t-\Phi^{(2)}_t\|_{L^\infty}$ is dominated as follows:
	\begin{align*}
		\|\Phi^{(1)}_t-\Phi^{(2)}_t\|_{L^\infty}
		\lesssim
			\|X^{(1)}-X^{(2)}\|_{\drivers{T}{\kappa}}
			+
			t^{-\frac{2-3\kappa}{6}}
			\|(v^{(1)},w^{(1)})-(v^{(2)},w^{(2)})\|_{\sols{T}{\kappa}{\kappa'}}.
	\end{align*}
	Setting
	$
		\const[3]
		=
			\|X^{(1)}\|_{\drivers{1}{\kappa}}
			+\|X^{(2)}\|_{\drivers{1}{\kappa}}
	$,
	we see
	\begin{multline*}
		\|
			\Phi^{(1)}_t\rpara X^{(1),\AB}_t
			-
			\Phi^{(2)}_t\rpara X^{(2),\AB}_t
		\|_{\HolBesSp{-1-\kappa}}\\
		\begin{aligned}
			&\lesssim
				(\const[1]+t^{-\frac{2-3\kappa}{6}}\const[2])
				\|X^{(1)}-X^{(2)}\|_{\drivers{T}{\kappa}}\\
			&\phantom{=}\quad
				+
				\bigg\{
					\|X^{(1)}-X^{(2)}\|_{\drivers{T}{\kappa}}
					+
					t^{-\frac{2-3\kappa}{6}}
					\|(v^{(1)},w^{(1)})-(v^{(2)},w^{(2)})\|_{\sols{T}{\kappa}{\kappa'}}
				\bigg\}
				\const[3]\\
			&\lesssim
				(\const[1]+\const[2]+\const[3])
				(1+t^{-\frac{2-3\kappa}{6}})\\
			&\phantom{=}\quad
				\times
				\left\{
					\|X^{(1)}-X^{(2)}\|_{\drivers{T}{\kappa}}
					+\|(v^{(1)},w^{(1)})-(v^{(2)},w^{(2)})\|_{\sols{T}{\kappa}{\kappa'}}
				\right\},
		\end{aligned}
	\end{multline*}
	which implies the conclusion.
\end{proof}

\begin{proposition}\label{prop_20161005032251}
	For any
	$
		(v^{(1)},w^{(1)}),(v^{(2)},w^{(2)})\in\sols{T}{\kappa}{\kappa'}
	$,
	we have
	\begin{multline*}
		\|
			\mathcal{M}_{(v^{(1)}_0,w^{(1)}_0),X^{(1)}}^1(v^{(1)},w^{(1)})
			-
			\mathcal{M}_{(v^{(2)}_0,w^{(2)}_0),X^{(2)}}^1(v^{(2)},w^{(2)})
		\|_{\cL_T^{\frac{5}{6}-\kappa',1-\kappa',1-\frac{1}{2}\kappa'}}\\
    \begin{aligned}
  		&\leq
  			\const[3]
  			\|v^{(1)}_0-v^{(2)}_0\|_{\HolBesSp{-\frac{2}{3}+\kappa'}}\\
      &
  			+
  			\const[4]
  			T^{\frac{1-\kappa'}{2}}
  			\Big(
  				\|X^{(1)}-X^{(2)}\|_{\drivers{1}{\kappa}}
  				+\|(v^{(1)},w^{(1)})-(v^{(2)},w^{(2)})\|_{\sols{T}{\kappa}{\kappa'}}
  			\Big)
    \end{aligned}
	\end{multline*}
	Here, $\const[3]$ and $\const[4]$ are positive constants depending only on $\kappa$, $\kappa'$, $\mu$, $\nu$,
	$\|X^{(i)}\|_{\drivers{1}{\kappa}}$ and	$\|(v^{(i)},w^{(i)})\|_{\sols{T}{\kappa}{\kappa'}}$.
	In particular, they are given by at most first-order polynomials in
	$\|X^{(i)}\|_{\drivers{1}{\kappa}}$ and	$\|(v^{(i)},w^{(i)})\|_{\sols{T}{\kappa}{\kappa'}}$.
\end{proposition}
\begin{proof}
	The assertion follows from \lref{lem_20161003043650} and the fact that
	\begin{multline*}
		\mathcal{M}_{X^{(1)}}^1(v^{(1)},w^{(1)})
		-
		\mathcal{M}_{X^{(2)}}^1(v^{(2)},w^{(2)})\\
		=
			P^1_t
			(v^{(1)}_0-v^{(2)}_0)
			+
			\int_0^t
				P^1_{t-s}
				\{F_{X^{(1)}}(v^{(1)},w^{(1)})(s)-F_{X^{(2)}}(v^{(2)},w^{(2)})(s)\}\,
				ds.
	\end{multline*}
	By a similar argument to the proof of \pref{prop_20161003143833},
	we see the conclusion.
\end{proof}

\subsubsection{Properties of $\mathcal{M}^2$}\label{sec_20160920093505}
Here, we consider properties of $\mathcal{M}^2$.
Let $0<T\leq 1$.
We fix $X\in\drivers{1}{\kappa}$ and write $Z=X^\A$ and $W=X^\IAAB$.
We denote by $\delta_{st}$ the difference operator, that is, $\delta_{st}f=f(t)-f(s)$.

First of all, we study $\com(v,w)$ defined by \eqref{eq_20161129003748}.
Let $v_0\in\HolBesSp{-\frac{2}{3}+\kappa'}$ and $(v,w)\in\sols{T}{\kappa}{\kappa'}$.
For notational simplicity, we set $\Phi^1=-2\nu(-\nu W+v+w)$,
$\Phi^2=-\nu(-\CmplConj{\nu}\CmplConj{W}+\CmplConj{v}+\CmplConj{w})$,
$\Psi^1=X^\AB$, $\Psi^2=X^\AA$ and
\begin{align*}
	U_t
	=
		\int_0^t
	 		P^1_{t-s}[F(v,w)(s)]\,
	 		ds
	 	-
 		\Phi^1_t
 		\rpara
 		X^\IAB_t
		-
 		\Phi^2_t
 		\rpara
 		X^\IAA_t.
\end{align*}

\begin{remark}
	The implicit constants which will appear in \lrefs{lem_20161005024632}{lem_20160822041507}{lem_20161003140013}
	depend only on $\kappa$, $\kappa'$, $\mu$, $\nu$ and $\|X\|_{\drivers{1}{\kappa}}$.
	In particular, the constants are given by an at most first-order polynomials in $\|X\|_{\drivers{1}{\kappa}}$.
\end{remark}

\begin{lemma}\label{lem_20161005024632}
	For every $0<t\leq T$, we have the following:
	\begin{enumerate}[(1)]
		\item	We have
				\begin{multline*}
					U_t
				 	=
				 		-\Phi^1_t\rpara P^1_tX^\IAB_0
				 		-\Phi^2_t\rpara P^1_tX^\IAA_0\\
				 		+
				 		\sum_{i=1,2}
				 			\left\{
				 				\int_0^t
									\delta_{st}\Phi^i
									\rpara
									P^1_{t-s}\Psi^i_s\,
									ds
								-
								\int_0^t
									[P^1_{t-s},\Phi^i_s\rpara]\Psi^i_s\,
									ds
				 			\right\}.
				\end{multline*}
		\item	We have $U_t\in\HolBesSp{1+\kappa'}$ and
				\begin{multline*}
					\|U_t\|_{\HolBesSp{1+\kappa'}}
					\lesssim
						1
						+
						t^{-\kappa'}
				 		\{1+\|v_t\|_{L^\infty}+\|w_t\|_{L^\infty}\}\\
						+
						\int_0^t
							(t-s)^{-\frac{3+2\kappa}{4}}
							\|v_s\|_{\HolBesSp{\frac{1}{2}+\kappa'}}\,
							ds
						+
						\int_0^t
							(t-s)^{-\frac{1+2\kappa'}{2}}
							\|w_s\|_{\HolBesSp{1+2\kappa'}}\,
							ds\\
						+
				 		\int_0^t
							(t-s)^{-\frac{2+\kappa+\kappa'}{2}}
							\{
								\|\delta_{st}v\|_{L^\infty}
								+\|\delta_{st}w\|_{L^\infty}
							\}\,
							ds.
				\end{multline*}
	\end{enumerate}
\end{lemma}

\begin{proof}
	We show the first assertion.
	For $(\Phi,\Psi)=(\Phi^1,\Psi^1)$, \pref{prop_20160927051159} implies
	\begin{align*}
		P^1_{t-s}[\Phi\rpara\Psi](s)
		&=
			P^1_{t-s}[\Phi_s\rpara\Psi_s]\\
		&=
			\Phi_s\rpara P^1_{t-s}\Psi_s
			+[P^1_{t-s},\Phi_s\rpara]\Psi_s\\
		&=
			\Phi_t\rpara P^1_{t-s}\Psi_s
			-\delta_{st}\Phi\rpara P^1_{t-s}\Psi_s
			+[P^1_{t-s},\Phi_s\rpara]\Psi_s.
	\end{align*}
	Hence
	\begin{multline*}
	 	\int_0^t
	 		P^1_{t-s}[\Phi\rpara\Psi](s)\,
	 		ds\\
		 =
		 	\Phi_t
		 	\rpara
		 	\int_0^t
		 		P^1_{t-s}\Psi_s\,
		 		ds
			-
			\int_0^t
				\delta_{st}\Phi\rpara P^1_{t-s}\Psi_s\,
				ds
			+
			\int_0^t
				[P^1_{t-s},\Phi_s\rpara]\Psi_s\,
				ds.
	\end{multline*}
	Substituting
	$
		\int_0^t
	 		P^1_{t-s}\Psi_s\,
	 		ds
 		=
			X^\IAB_t
	 		-
	 		P^1_t
			X^\IAB_0
	$
	to the first term in the above,
	we see
	\begin{multline}\label{eq_20161005022939}
		\int_0^t
	 		P^1_{t-s}[\Phi\rpara\Psi](s)\,
	 		ds
	 	-
	 	\Phi_t
	 	\rpara
	 	X^\IAB_t\\
	 	=
	 		-
		 	\Phi_t
		 	\rpara
		 	P^1_tX^\IAB_0
		 	-
			\int_0^t
				\delta_{st}\Phi\rpara P^1_{t-s}\Psi_s\,
				ds
			+
			\int_0^t
				[P^1_{t-s},\Phi_s\rpara]\Psi_s\,
				ds.
	\end{multline}
	Since a similar equality holds for
	$(\Phi,\Psi)=(\Phi^2,\Psi^2)$,
	we have verified the first assertion.

	For the second assertion, we estimate the terms in \eqref{eq_20161005022939}.
	For the first term in \eqref{eq_20161005022939},
	we use $1+\kappa'=1-\kappa+2\cdot\frac{\kappa+\kappa'}{2}$ and obtain
	\begin{align*}
		\|
			\Phi^1_t
	 		\rpara
	 		P^1_tX^\IAB_0
	 	\|_{\HolBesSp{1+\kappa'}}
		 &\lesssim
	 		\|\Phi^1_t\|_{L^\infty}
	 		\|P^1_tX^\IAB_0\|_{\HolBesSp{1+\kappa'}}\\
	 	&\lesssim
	 		\|-\nu W_t+v_t+w_t\|_{L^\infty}
	 		\cdot
	 		t^{-\frac{\kappa+\kappa'}{2}}
	 		\|X^\IAB_0\|_{\HolBesSp{1-\kappa}}\\
	 	&\lesssim
	 		t^{-\kappa'}
	 		\{1+\|v_t\|_{L^\infty}+\|w_t\|_{L^\infty}\}.
	\end{align*}

	We estimate the second term in \eqref{eq_20161005022939}.
	From
	$
		1+\kappa'
		=
			-1-\kappa
			+
			2\cdot\frac{2+\kappa+\kappa'}{2}
	$,
	we have
	\begin{align*}
		\|
			\delta_{st}\Phi
			\rpara
			P^1_{t-s}\Psi_s
		\|_{\HolBesSp{1+\kappa'}}
		&\lesssim
			\|\delta_{st}\Phi\|_{L^\infty}
			\|P^1_{t-s}\Psi_s\|_{\HolBesSp{1+\kappa'}}\\
		&\lesssim
			\|\delta_{st}\Phi\|_{L^\infty}
			\cdot
			(t-s)^{-\frac{2+\kappa+\kappa'}{2}}
			\|\Psi_s\|_{\HolBesSp{-1-\kappa}}\\
		&\lesssim
			(t-s)^{-\frac{2+\kappa+\kappa'}{2}}
			\|\delta_{st}\Phi\|_{L^\infty}.
	\end{align*}
	Note
	\begin{gather*}
		\|\delta_{st}\Phi\|_{L^\infty}
		\lesssim
			\|\delta_{st}v\|_{L^\infty}
			+\|\delta_{st}w\|_{L^\infty}
			+\|\delta_{st}W\|_{L^\infty},\\
		\|\delta_{st}W\|_{L^\infty}
		\lesssim
			\|\delta_{st}W\|_{\HolBesSp{\kappa'-\kappa}}
		\lesssim
			(t-s)^{\frac{1}{4}-\frac{1}{2}\kappa'}
			\|W\|_{C_T^{\frac{1}{4}-\frac{1}{2}\kappa'}\HolBesSp{\kappa'-\kappa}}.
	\end{gather*}
	For the latter estimate, see \rref{rem_20160921064624}.
	From them, we have
	\begin{multline*}
		\left\|
			\int_0^t
				\delta_{st}\Phi
				\rpara
				P^1_{t-s}\Psi_s\,
				ds
		\right\|_{\HolBesSp{1+\kappa'}}\\
		\begin{aligned}
			&\lesssim
				\int_0^t
					(t-s)^{-\frac{2+\kappa+\kappa'}{2}}
					\{
						\|\delta_{st}v\|_{L^\infty}
						+\|\delta_{st}w\|_{L^\infty}
						+(t-s)^{\frac{1}{4}-\frac{1}{2}\kappa'}
					\}\,
					ds\\
			&\lesssim
				\int_0^t
					(t-s)^{-\frac{2+\kappa+\kappa'}{2}}
					\{
						\|\delta_{st}v\|_{L^\infty}
						+\|\delta_{st}w\|_{L^\infty}
					\}\,
					ds
				+
				1.
		\end{aligned}
	\end{multline*}
	The last inequality follows from
	$
		\int_0^t
			(t-s)^{-1-\frac{\kappa+\kappa'}{2}+\frac{1}{4}-\frac{1}{2}\kappa'}\,
			ds
		<\infty
	$.
	The estimate of the second term has finished.

	Lastly, we estimate the third term.
	We consider the contribution of $W$, $v$ and $w$ separably.
	In the proof, we use \pref{prop_20160927051159}. Note
	\begin{align*}
		1+\kappa'
		&=
			\left(
				\frac{1}{2}-\kappa
			\right)
			+
			(-1-\kappa)
			+
			2
			\cdot
			\frac{3+4\kappa+2\kappa'}{4}\\
		&=
			\left(
				\frac{1}{2}+\kappa'
			\right)
			+
			(-1-\kappa)
			+
			2
			\cdot
			\frac{3+2\kappa}{4}\\
		&=
			(1-\kappa'+\kappa)
			+(-1-\kappa)
			+
			2
			\cdot
			\frac{1+2\kappa'}{2}.
	\end{align*}
	We also use $\|\Psi_s\|_{\HolBesSp{-1-\kappa}}\lesssim 1$.
	For $W$, we see
	\begin{align*}
		\left\|
			\int_0^t
				[P^1_{t-s},W_s\rpara]\Psi_s\,
				ds
		\right\|_{\HolBesSp{1+\kappa'}}
		&\lesssim
			\int_0^t
				(t-s)^{-\frac{3+4\kappa+2\kappa'}{4}}
				\|W_s\|_{\HolBesSp{\frac{1}{2}-\kappa}}
				\|\Psi_s\|_{\HolBesSp{-1-\kappa}}\,
				ds\\
		&\lesssim
			1.
	\end{align*}
	For $v$ and $w$, we have
	\begin{gather*}
		\left\|
			\int_0^t
				[P^1_{t-s},v_s\rpara]\Psi_s\,
				ds
		\right\|_{\HolBesSp{1+\kappa'}}
		\lesssim
			\int_0^t
				(t-s)^{-\frac{3+2\kappa}{4}}
				\|v_s\|_{\HolBesSp{\frac{1}{2}+\kappa'}}\,
				ds,\\
		\begin{aligned}
			\left\|
				\int_0^t
					[P^1_{t-s},w_s\rpara]\Psi_s\,
					ds
			\right\|_{\HolBesSp{1+\kappa'}}
			&\lesssim
				\int_0^t
					(t-s)^{-\frac{1+2\kappa'}{2}}
					\|w_s\|_{\HolBesSp{1-\kappa'+\kappa}}\,
					ds\\
			&\lesssim
				\int_0^t
					(t-s)^{-\frac{1+2\kappa'}{2}}
					\|w_s\|_{\HolBesSp{1+2\kappa'}}\,
					ds.
		\end{aligned}
	\end{gather*}
	The proof is completed.
\end{proof}

\begin{lemma}\label{lem_20160822041507}
	For any $(v,w)\in\sols{T}{\kappa}{\kappa'}$ and $0<t\leq T$,
	we have $\com(v,w)(t)\in\HolBesSp{1+\kappa'}$ and
	\begin{multline*}
		\|\com(v,w)(t)\|_{\HolBesSp{1+\kappa'}}\\
		\begin{aligned}
			&\lesssim
				1
				+t^{-\frac{5}{6}}\|v_0\|_{\HolBesSp{-\frac{2}{3}+\kappa'}}
				+t^{-\kappa'}(1+\|v_t\|_{L^\infty}+\|w_t\|_{L^\infty})\\
			&\phantom{=}\quad
				+
					\int_0^t
						(t-s)^{-\frac{3+2\kappa'}{4}}
						\|v_s\|_{\HolBesSp{\frac{1}{2}+\kappa'}}\,
						ds
				+
					\int_0^t
						(t-s)^{-\frac{1+2\kappa'}{2}}
						\|w_s\|_{\HolBesSp{1+2\kappa'}}\,
						ds\\
			&\phantom{=}\quad
				+
					\int_0^t
						(t-s)^{-1-\frac{\kappa+\kappa'}{2}}
						(\|\delta_{st} v\|_{L^\infty}+\|\delta_{st} w\|_{L^\infty})\,
						ds.
		\end{aligned}
	\end{multline*}
\end{lemma}

\begin{proof}
	From definition \eqref{eq_20161129003748}, we have
	$
		\com(v,w)(t)
		=
			P^1_tv_0
			+U_t
	$.

	Noting
	$
		1+\kappa'
		=
			\left(
				-\frac{2}{3}+\kappa'
			\right)
			+
			2
			\cdot
			\frac{5}{6}
		=
			\frac{1}{2}+\kappa'
			+2\cdot\frac{1}{4}
	$
	and using \pref{prop_20160927051055}, we see
	\begin{gather*}
		\|P^1_tv_0\|_{\HolBesSp{1+\kappa'}}
		\lesssim
			t^{-\frac{5}{6}}
			\|v_0\|_{\HolBesSp{-\frac{2}{3}+\kappa'}},\\
		\begin{aligned}
			\left\|
				\int_0^t
					P^1_{t-s}
					v_s\,
					ds
			\right\|_{\HolBesSp{1+\kappa'}}
			\leq
				\int_0^t
					\|
						P^1_{t-s}
						v_s
					\|_{\HolBesSp{1+\kappa'}}\,
					ds
			\lesssim
				\int_0^t
					(t-s)^{-\frac{1}{4}}
					\|v_s\|_{\HolBesSp{\frac{1}{2}+\kappa'}}\,
					ds.
		\end{aligned}
	\end{gather*}
	Note that the last term smaller than or equal to
	$
		\int_0^t
			(t-s)^{-\frac{3+2\kappa}{4}}
			\|v_s\|_{\HolBesSp{\frac{1}{2}+\kappa'}}\,
			ds
	$.
	Combining these and \lref{lem_20161005024632}, we see the conclusion.
\end{proof}

\begin{lemma}\label{lem_20161003140013}
	For any
	$
		(v,w)\in\sols{T}{\kappa}{\kappa'}
	$
	and
	$0<t\leq T$,
	we have
	\begin{align*}
		\|\com(v,w)(t)\|_{\HolBesSp{1+\kappa'}}
		\lesssim
			\const
			(
				1
				+
				t^{-\frac{5}{6}}\|v_0\|_{\HolBesSp{-\frac{2}{3}+\kappa'}}
				+
				t^{-\frac{1+2\kappa+2\kappa'}{2}}
				\|(v,w)\|_{\sols{T}{\kappa}{\kappa'}}
			).
	\end{align*}
\end{lemma}
\begin{proof}
	We estimate each term in the upper bound of $\|\com(v,w)(t)\|_{\HolBesSp{1+\kappa'}}$
	in \lref{lem_20160822041507} by using \rref{rem_20160921064624}.
	The first three terms are estimated as follows:
	\begin{multline*}
		1
		+t^{-\frac{5}{6}}\|v_0\|_{\HolBesSp{-\frac{2}{3}+\kappa'}}
		+t^{-\kappa'}(1+\|v(t)\|_{L^\infty}+\|w(t)\|_{L^\infty})\\
		\lesssim
			1
			+t^{-\frac{5}{6}}\|v_0\|_{\HolBesSp{-\frac{2}{3}+\kappa'}}
			+t^{
				-
				\left(
					\kappa'+\frac{2-3\kappa}{6}
				\right)
				}
			(1+\|(v,w)\|_{\sols{T}{\kappa}{\kappa'}}).
	\end{multline*}
	To estimate other terms, we use the fact that the inequality
	\begin{align*}
		\int_0^t
			(t-s)^{-\theta_1}
			s^{-\theta_2}\,
			ds
		\lesssim
			t^{1-\theta_1-\theta_2}
	\end{align*}
	holds for $0<\theta_1,\theta_2<1$ and $t>0$.
	From \rref{rem_20160921064624}, we see
	\begin{multline*}
		\int_0^t
			(t-s)^{-\frac{3+2\kappa'}{4}}
			\|v_s\|_{\HolBesSp{\frac{1}{2}+\kappa'}}\,
			ds
		+
		\int_0^t
			(t-s)^{-\frac{1+2\kappa'}{2}}
			\|w_s\|_{\HolBesSp{1+2\kappa'}}\,
			ds\\
		\lesssim
			t^{
				-
				\left(
					\frac{1}{3}+\frac{1}{2}\kappa'
				\right)
				}
			\|v\|_{\cL_T^{\frac{5}{6}-\kappa',1-\kappa',1-\frac{1}{2}\kappa'}}
			+
			t^{
				-
				\left(
					\frac{1}{4}+2\kappa'+\kappa
				\right)
			}
			\|w\|_{\cL_T^{1-\kappa'+\kappa,\frac{3}{2}-2\kappa',1-\kappa'}}
	\end{multline*}
	and
	\begin{multline*}
		\int_0^t
			(t-s)^{-\left(1+\frac{\kappa+\kappa'}{2}\right)}
			(\|\delta_{st} v\|_{L^\infty}+\|\delta_{st} w\|_{L^\infty})\,
			ds\\
		\lesssim
			t^{-\frac{2+3\kappa+3\kappa'}{6}}
			\|v\|_{\cL_T^{\frac{5}{6}-\kappa',1-\kappa',1-\frac{1}{2}\kappa'}}
			+
			t^{-\frac{1+3\kappa+\kappa'}{2}}
			\|w\|_{\cL_T^{1-\kappa'+\kappa,\frac{3}{2}-2\kappa',1-\kappa'}}.
	\end{multline*}
	Combining them, we see the estimate of $\|\com(v,w)(t)\|_{\HolBesSp{1+\kappa'}}$.
\end{proof}

\begin{lemma}\label{lem_20160930130701}
	For any $(v,w)\in\sols{T}{\kappa}{\kappa'}$ and $0<t\leq T$, we have
	\begin{multline*}
		\|G(v,w)(t)\|_{\HolBesSp{-\frac{1}{2}-2\kappa}}\\
		\leq
			\const
			\Big(
				1
				+t^{-\frac{5}{6}}\|v_0\|_{\HolBesSp{-\frac{2}{3}+\kappa'}}
				+
				t^{-\frac{2-3\kappa}{2}}
				\Big(
					\|(v,w)\|_{\sols{T}{\kappa}{\kappa'}}^3
					+\|(v,w)\|_{\sols{T}{\kappa}{\kappa'}}
				\Big)
			\Big).
	\end{multline*}
	Here, $\const$ is a positive constant depending only on $\kappa$, $\kappa'$, $\mu$, $\nu$ and $\|X\|_{\drivers{1}{\kappa}}$
	and it is given by a third-order polynomial in $\|X\|_{\drivers{1}{\kappa}}$.
\end{lemma}
\begin{proof}
	We write $u_2=v+w$.
	It follows from \rref{rem_20160921064624} that
	\begin{align*}
		\|G_1(v,w)(t)\|_{L^\infty}
		\lesssim
			\|v_t+w_t\|_{L^\infty}^3
		\lesssim
			t^{-\frac{2-3\kappa}{2}}
			\|(v,w)\|_{\sols{T}{\kappa}{\kappa'}}^3.
	\end{align*}

	To estimate $G_2(v,w)(t)$, we use the Bony decomposition.
	Applying the decomposition to $u_2(t)\in\HolBesSp{\frac{1}{2}+\kappa'}$,
	we see
	\begin{align*}
		\|u_2(t)^2\|_{\HolBesSp{\frac{1}{2}+\kappa'}}
		&=
			\|
				u_2(t)\reso u_2(t)
				+2u_2(t)\rpara u_2(t)
			\|_{\HolBesSp{\frac{1}{2}+\kappa'}}\\
		&\lesssim
			\|u_2(t)\|_{\HolBesSp{\frac{1}{2}\left(\frac{1}{2}+\kappa'\right)}}^2
			+2\|u_2(t)\|_{L^\infty}\|u_2(t)\|_{\HolBesSp{\frac{1}{2}+\kappa'}}\\
		&\lesssim
			t^{-\frac{11-6\kappa'}{12}}
			\|(v,w)\|_{\sols{T}{\kappa}{\kappa'}}^2
			+
			2
			t^{-\frac{2-3\kappa}{6}}
			\|(v,w)\|_{\sols{T}{\kappa}{\kappa'}}
			\cdot
			t^{-\frac{7}{12}}
			\|(v,w)\|_{\sols{T}{\kappa}{\kappa'}}
			\\
		&=
			t^{-\frac{11-6\kappa'}{12}}
			\|(v,w)\|_{\sols{T}{\kappa}{\kappa'}}^2
			+
			2
			t^{-\frac{11-6\kappa}{12}}
			\|(v,w)\|_{\sols{T}{\kappa}{\kappa'}}^2\\
		&\lesssim
			t^{-\frac{11-6\kappa}{12}}
			\|(v,w)\|_{\sols{T}{\kappa}{\kappa'}}^2.
	\end{align*}
	In these estimate, we used $0<\kappa<\kappa'<1/18$ (see \rref{rem_20160921064624}).
	The term $\|u_2(t)\CmplConj{u_2(t)}\|_{\HolBesSp{\frac{1}{2}+\kappa'}}$ has the same bound.
	Since
	$
		G_2(v,w)(t)
		=
			a_1
			u_2(t)^2
			+
			a_2
			u_2(t)
			\CmplConj{u_2(t)}
	$
	for some $a_1,a_2\in\HolBesSp{-\frac{1}{2}-\kappa}$
	and $u_2(t)\in\HolBesSp{\frac{1}{2}+\kappa'}$,
	we have
	\begin{align*}
		\|G_2(v,w)(t)\|_{\HolBesSp{-\frac{1}{2}-\kappa}}
		&\lesssim
			\|u_2(t)^2\|_{\HolBesSp{\frac{1}{2}+\kappa'}}
			+\|u_2(t)\CmplConj{u_2(t)}\|_{\HolBesSp{\frac{1}{2}+\kappa'}}\\
		&\lesssim
			t^{-\frac{11-6\kappa}{12}}
			\|(v,w)\|_{\sols{T}{\kappa}{\kappa'}}^2.
	\end{align*}
	Noting
	$
		\|(v,w)\|_{\sols{T}{\kappa}{\kappa'}}^2
		\leq
			\|(v,w)\|_{\sols{T}{\kappa}{\kappa'}}^3
			+
			\|(v,w)\|_{\sols{T}{\kappa}{\kappa'}}
	$,
	we see the estimate.

	Since
	$
		G_3(v,w)(t)
		=
			b_1u_2(t)
			+b_2\CmplConj{u_2(t)}
			+(\nu+1)u_2(t)
	$
	for some $b_1,b_2\in\HolBesSp{-\frac{1}{2}-\kappa}$, we have
	\begin{align*}
		\|G_3(v,w)(t)\|_{\HolBesSp{-\frac{1}{2}-\kappa}}
		\lesssim
			\|u_2(t)\|_{\HolBesSp{\frac{1}{2}+\kappa'}}
		\lesssim
			t^{-\frac{7}{12}}
			\|(v,w)\|_{\sols{T}{\kappa}{\kappa'}}.
	\end{align*}

	The estimates of $G_4(v,w)(t)$ and $G_5(v,w)(t)$ are obtained easily.
	The terms which admits the lowest regularity in the defintions of $G_4(v,w)(t)$
	are $W_t\lpara X^\AB$ and $\CmplConj{W_t}\lpara X^\AA$
	and their regularity is $-\frac{1}{2}-2\kappa$.
	Therefore we obtain $\|G_4(v,w)(t)\|_{\HolBesSp{-\frac{1}{2}-2\kappa}}\lesssim 1$.
	From \pref{prop_comm_20160919055939}, we see
	\begin{align*}
		\|G_5(v,w)(t)\|_{\HolBesSp{\frac{1}{2}+\kappa'-2\kappa}}
		\lesssim
			\|u_2\|_{\HolBesSp{\frac{1}{2}+\kappa'}}
		\lesssim
			t^{-\frac{7}{12}}
			\|(v,w)\|_{\sols{T}{\kappa}{\kappa'}}.
	\end{align*}

	From the definition of $G_6(v,w)(t)$, we have
	\begin{align*}
		\|G_6(v,w)(t)\|_{\HolBesSp{\kappa'-\kappa}}
		&\lesssim
			\|\com(v,w)(t)\reso X^\AB_t\|_{\HolBesSp{1+\kappa'+(-1-\kappa)}}\\
		&\phantom{\lesssim}\quad\qquad
			+\|\CmplConj{\com(v,w)(t)}\reso X^\AA_t\|_{\HolBesSp{1+\kappa'+(-1-\kappa)}}\\
		&\lesssim
			\|\com(v,w)(t)\|_{\HolBesSp{1+\kappa'}}\\
		&\lesssim
			\const
				(
					1
					+
					t^{-\frac{5}{6}}
					\|v_0\|_{\HolBesSp{-\frac{2}{3}+\kappa'}}
					+
					t^{-\frac{1+2\kappa+2\kappa'}{2}}
					\|(v,w)\|_{\sols{T}{\kappa}{\kappa'}}
				).
	\end{align*}
	In the last line, we used \lref{lem_20161003140013}.

	For $\tau=\AB,\AA$, we see
	\begin{gather*}
		\|w_t\reso X^\tau_t\|_{\HolBesSp{(1+\kappa')+(-1-\kappa)}}
		\lesssim
			\|w_t\|_{\HolBesSp{1+\kappa'}}
			\|X^\tau_t\|_{\HolBesSp{-1-\kappa}}
		\lesssim
			t^{-\frac{3+2\kappa+\kappa'}{4}}
			\|(v,w)\|_{\sols{T}{\kappa}{\kappa'}},\\
		\|u_2(t)\lpara X^\tau_t\|_{\HolBesSp{\left(\frac{1}{2}+\kappa'\right)+(-1-\kappa)}}
		\lesssim
			\|u_2(t)\|_{\HolBesSp{\frac{1}{2}+\kappa'}}
			\|X^\tau_t\|_{\HolBesSp{-1-\kappa}}
		\lesssim
			t^{-\frac{7}{12}}
			\|(v,w)\|_{\sols{T}{\kappa}{\kappa'}}.
	\end{gather*}
	In these estimates, we used \rref{rem_20160921064624} .
	We obtain
	\begin{gather*}
		\|G_7(v,w)(t)\|_{\HolBesSp{\kappa'-\kappa}}
		\lesssim
			t^{-\frac{3+4\kappa+2\kappa'}{4}}
			\|(v,w)\|_{\sols{T}{\kappa}{\kappa'}},\\
		\|G_8(v,w)(t)\|_{\HolBesSp{-\frac{1}{2}+\kappa'-\kappa}}
		\lesssim
			t^{-\frac{7}{12}}
			\|(v,w)\|_{\sols{T}{\kappa}{\kappa'}}.
	\end{gather*}
	The proof is completed.
\end{proof}

\begin{proposition}\label{prop_20161003143917}
	The map
	$
		\mathcal{M}^2:
			\sols{T}{\kappa}{\kappa'}
			\to
			\cL_T^{1-\kappa'+\kappa,\frac{3}{2}-2\kappa',1-\kappa'}
	$
	is well-defined and,
	for any
	$
		(v,w)\in\sols{T}{\kappa}{\kappa'}
	$,
	we have
	\begin{multline*}
		\|\mathcal{M}^2(v,w)\|_{\cL_T^{1-\kappa'+\kappa,\frac{3}{2}-2\kappa',1-\kappa'}}
		\leq
			\const[1]
			(
				1
				+
				\|v_0\|_{\HolBesSp{-\frac{2}{3}+\kappa'}}
				+
				\|w_0\|_{\HolBesSp{-\frac{1}{2}-2\kappa}}
			)\\
			+
			\const[2]
			T^{\frac{3}{2}\kappa}
			\left(
				\|(v,w)\|_{\sols{T}{\kappa}{\kappa'}}^3
				+\|(v,w)\|_{\sols{T}{\kappa}{\kappa'}}
			\right).
	\end{multline*}
	Here, $\const[1]$ and $\const[2]$ are positive constants depending only on $\kappa$, $\kappa'$, $\mu$, $\nu$ and $\|X\|_{\drivers{1}{\kappa}}$.
	They are given by at most third-order polynomials in $\|X\|_{\drivers{1}{\kappa}}$.
\end{proposition}
\begin{proof}
	Recall \eqref{eq_20161003034500}.
	It follows from \pref{prop_20160930075129} that
	\begin{align*}
		\|
			(t\mapsto P^1_t w_0)_{t\geq 0}
		\|_{\cL_T^{1-\kappa'+\kappa,\frac{3}{2}-2\kappa',1-\kappa'}}
		\lesssim
			\|w_0\|_{\HolBesSp{-\frac{1}{2}-2\kappa}}.
	\end{align*}
	\lref{lem_20160930130701} implies
	$G(v,w)\in\cE_T^{\frac{2-3\kappa}{2}}\HolBesSp{-\frac{1}{2}-2\kappa}$
	and
	\begin{align*}
		\|G(v,w)\|_{\cE_T^{\frac{2-3\kappa}{2}}\HolBesSp{-\frac{1}{2}-2\kappa}}
		\leq
			\const
			\Big(
				1
				+\|v_0\|_{\HolBesSp{-\frac{2}{3}+\kappa'}}
				+
					\|(v,w)\|_{\sols{T}{\kappa}{\kappa'}}^3
					+\|(v,w)\|_{\sols{T}{\kappa}{\kappa'}}
			\Big).
	\end{align*}
	\pref{prop_20160930075129} implies
	\begin{multline*}
		\left\|
			(
				t
				\mapsto
					\int_0^t
						P^1_{t-s}
						G(v,w)(s)\,
						ds
			)_{t\geq 0}
		\right\|_{\cL_T^{1-\kappa'+\kappa,\frac{3}{2}-2\kappa',1-\kappa'}}\\
		\lesssim
			T^{\frac{3}{2}\kappa}
			\|G(v,w)\|_{\cE_T^{\frac{2-3\kappa}{2}}\HolBesSp{-\frac{1}{2}-2\kappa}}.
	\end{multline*}
	Combining these, we have shown the assertion.
\end{proof}

We also have local Lipschitz continuity of $\mathcal{M}_2$.
\begin{proposition}\label{prop_20161005032416}
	For any
	$
		(v^{(1)},w^{(1)}),(v^{(2)},w^{(2)})\in\sols{T}{\kappa}{\kappa'}
	$,
	we have
	\begin{multline*}
		\|
			\mathcal{M}_{(v^{(1)}_0,w^{(1)}_0),X^{(1)}}^2(v^{(1)},w^{(1)})
			-
			\mathcal{M}_{(v^{(2)}_0,w^{(2)}_0),X^{(2)}}^2(v^{(2)},w^{(2)})
		\|_{\cL_T^{\frac{5}{6}-\kappa',1-\kappa',1-\frac{1}{2}\kappa'}}\\
		\leq
			\const[3]
			\left(
				\|v^{(1)}_0-v^{(2)}_0\|_{\HolBesSp{-\frac{2}{3}+\kappa'}}
				+
				\|w^{(1)}_0-w^{(2)}_0\|_{\HolBesSp{-\frac{1}{2}-2\kappa}}
			\right)\\
			+
			\const[4]
			T^{\frac{3}{2}\kappa}
			\Big(
				\|X^{(1)}-X^{(2)}\|_{\drivers{T}{\kappa}}
				+\|(v^{(1)},w^{(1)})-(v^{(2)},w^{(2)})\|_{\sols{T}{\kappa}{\kappa'}}
			\Big)
	\end{multline*}
	Here, $\const[3]$ and $\const[4]$ are positive constants depending only on $\kappa$, $\kappa'$, $\mu$, $\nu$,
	$\|X^{(i)}\|_{\drivers{1}{\kappa}}$ and	$\|(v^{(i)},w^{(i)})\|_{\sols{T}{\kappa}{\kappa'}}$.
	In particular, they are given by at most second-order polynomials in
	$\|X^{(i)}\|_{\drivers{1}{\kappa}}$ and	$\|(v^{(i)},w^{(i)})\|_{\sols{T}{\kappa}{\kappa'}}$.
\end{proposition}

\begin{proof}
	We can show the assertion by a similar way as \pref{prop_20161003143917}.
\end{proof}

\subsection{Local existence and uniqueness}\label{sec_20160920093609}
We show local well-posedness of CGL \eqref{eq:cgl}.
This is the most important theorem in this section.
\begin{theorem}\label{thm_20160920085605}
  Let $0<\kappa<\kappa'<1/18$.
  There exists a continuous function
  $
    \tilde{T}_\ast:
    \HolBesSp{-\frac{2}{3}+\kappa'}
    \times
    \HolBesSp{-\frac{1}{2}-2\kappa}
    \times
    \drivers{1}{\kappa}
    \to
      (0,1]
  $
  such that the following (1) and (2) hold:
	\begin{enumerate}[(1)]
		\item	For every
				$
					(v_0,w_0)
					\in
						\HolBesSp{-\frac{2}{3}+\kappa'}
						\times
						\HolBesSp{-\frac{1}{2}-2\kappa}
				$
				and
				$
					X\in\drivers{1}{\kappa}
				$,
        set $T_\ast=\tilde{T}_\ast(v_0,w_0,X)$.
        Then, the system \eqref{eq_20160920090410} admits a unique solution $(v,w)\in\sols{T_\ast}{\kappa}{\kappa'}$
				and there is a positive constant $\const$ depending only on $\mu$, $\nu$, $\kappa$, $\kappa'$, $T_\ast$ and
				$\|X\|_{\drivers{1}{\kappa}}$
				such that
				\begin{align*}
					\|(v,w)\|_{\sols{T_\ast}{\kappa}{\kappa'}}
					\leq
						\const
						\left(
							1
							+\|v_0\|_{\HolBesSp{-\frac{2}{3}+\kappa'}}
							+\|w_0\|_{\HolBesSp{-\frac{1}{2}-2\kappa}}
						\right).
				\end{align*}
		\item	Let $\{(v_0^{(n)},w_0^{(n)})\}_{n=1}^\infty$ and $\{X^{(n)}\}_{n=1}^\infty$
        converge to $(v_0,w_0)$ in
				$
						\HolBesSp{-\frac{2}{3}+\kappa'}
						\times
						\HolBesSp{-\frac{1}{2}-2\kappa}
				$
				and
        $X$ in $\drivers{1}{\kappa}$,
        respectively.
        Set $T_\ast^{(n)}=\tilde{T}_\ast(v_0^{(n)},w_0^{(n)},X^{(n)})$
        and let $(v^{(n)},w^{(n)})$ be a unique solution on $[0,T^{(n)}_\ast]$
        to the system \eqref{eq_20160920090410} with the initial condition $(v_0^{(n)},w_0^{(n)})$ driven by $X^{(n)}$.
        Then, for every $0<t<T_\ast$, we have
        \begin{align*}
          \lim_{n\to\infty}
            \|(v^{(n)},w^{(n)})-(v,w)\|_{\sols{t}{\kappa}{\kappa'}}
          =
            0.
        \end{align*}
	\end{enumerate}
\end{theorem}

In the proof the function $\tilde{T}_\ast$ is concretely given by
$
  \tilde{T}_\ast(v_0,w_0,X)
  =
    T_\ast
$,
where $T_\ast$ is defined by \eqref{eq_20161005054253} and \eqref{eq_20161005054259}.
We prove the theorem by using the properties of $\mathcal{M}$ we have just shown.

For every $0<T\leq 1$ and $M>0$, we define
\begin{align*}
	\mathbf{B}_{T,M}
	=
		\{
			(v,w)\in\sols{T}{\kappa}{\kappa'};
			\|(v,w)\|_{\sols{T}{\kappa}{\kappa'}}\leq M
		\}.
\end{align*}
\pref[prop_20161003143833]{prop_20161003143917} imply
\begin{multline*}
	\|\mathcal{M}(v,w)\|_{\sols{T}{\kappa}{\kappa'}}
	\leq
		\const[1]
		\big(
			1
			+
			\|v_0\|_{\HolBesSp{-\frac{2}{3}+\kappa'}}
			+
			\|w_0\|_{\HolBesSp{-\frac{1}{2}-2\kappa}}
		\big)\\
		+
		\const[2]
		T^{\frac{3}{2}\kappa}
		(
			\|(v,w)\|_{\sols{T}{\kappa}{\kappa'}}^3
			+\|(v,w)\|_{\sols{T}{\kappa}{\kappa'}}
		).
\end{multline*}
Here, $\const[1]$ and $\const[2]$ are positive constants depending only on $\kappa$, $\kappa'$, $\mu$, $\nu$ and $\|X\|_{\drivers{1}{\kappa}}$.
In particular, they are given by at most third-order polynomials in $\|X\|_{\drivers{1}{\kappa}}$.
\pref[prop_20161005032251]{prop_20161005032416} imply
\begin{multline}\label{eq_20161005061121}
	\|
		\mathcal{M}_{(v^{(1)}_0,w^{(1)}_0),X^{(1)}}(v^{(1)},w^{(1)})
		-
		\mathcal{M}_{(v^{(2)}_0,w^{(2)}_0),X^{(2)}}(v^{(2)},w^{(2)})
	\|_{\sols{T}{\kappa}{\kappa'}}\\
	\leq
		\const[3]
		(
			\|v^{(1)}_0-v^{(2)}_0\|_{\HolBesSp{-\frac{2}{3}+\kappa'}}
			+
			\|w^{(1)}_0-w^{(2)}_0\|_{\HolBesSp{-\frac{1}{2}-2\kappa}}
		)\\
		+
		\const[4]
		T^{\frac{3}{2}\kappa}
		(
			\|X^{(1)}-X^{(2)}\|_{\drivers{T}{\kappa}}
			+\|(v^{(1)},w^{(1)})-(v^{(2)},w^{(2)})\|_{\sols{T}{\kappa}{\kappa'}}
		).
\end{multline}
Here, $\const[3]$ and $\const[4]$ are positive constants depending only on $\kappa$, $\kappa'$, $\mu$, $\nu$,
$\|X^{(i)}\|_{\drivers{1}{\kappa}}$ and	$\|(v^{(i)},w^{(i)})\|_{\sols{T}{\kappa}{\kappa'}}$.
In particular, they are given by at most second-order polynomials in
$\|X^{(i)}\|_{\drivers{1}{\kappa}}$ and	$\|(v^{(i)},w^{(i)})\|_{\sols{T}{\kappa}{\kappa'}}$.

\begin{proof}[Proof of \tref{thm_20160920085605}]
	For the proof of existence, we use \pref[prop_20161003143833]{prop_20161003143917}.
	We will show the map $\mathcal{M}$ is contraction from $\mathbf{B}_{T,M}$ to itself for small $T>0$ and suitable $M>0$
	and obtain the existence of solution by the fixed point theorem.
	Let $M\geq 1$.
	For any $(v,w)\in\sols{T}{\kappa}{\kappa'}$, we have
	\begin{align*}
		\|\mathcal{M}(v,w)\|_{\sols{T}{\kappa}{\kappa'}}
		&\leq
			(\const[1]+\const[2])
			(
				1
				+\|v_0\|_{\HolBesSp{-\frac{2}{3}+\kappa'}}
				+\|w_0\|_{\HolBesSp{-\frac{1}{2}-2\kappa}}
				+
				T^{\frac{3}{2}\kappa}
				M^3
			)\\
		&\leq
			(\const[1]+\const[2])
			(
				1
				+\|v_0\|_{\HolBesSp{-\frac{2}{3}+\kappa'}}
				+\|w_0\|_{\HolBesSp{-\frac{1}{2}-2\kappa}}
			)
			(
				1
				+
				T^{\frac{3}{2}\kappa}
				M^3
			).
	\end{align*}
	In a similar way, we see
	\begin{multline*}
		\|\mathcal{M}(v^{(1)},w^{(1)})-\mathcal{M}(v^{(2)},w^{(2)})\|_{\sols{T}{\kappa}{\kappa'}}\\
		\begin{aligned}
			&\leq
				\const[\|X\|_{\drivers{1}{\kappa}}]
				T^{\frac{3}{2}\kappa}
				(1+M^2)
				\|(v^{(1)},w^{(1)})-(v^{(2)},w^{(2)})\|_{\sols{T}{\kappa}{\kappa'}}.
		\end{aligned}
	\end{multline*}
	Here $\const[\|X\|_{\drivers{1}{\kappa}}]>0$
	is given by a second-order polynomial with respect to $\|X\|_{\drivers{T}{\kappa}}$.
	Set
	\begin{gather}
		\label{eq_20161005054253}
		M_\ast
		=
			2
			(\const[1]+\const[2])
			(
				1
				+\|v_0\|_{\HolBesSp{-\frac{2}{3}+\kappa'}}
				+\|w_0\|_{\HolBesSp{-\frac{1}{2}-2\kappa}}
			)
			\vee
			1,\\
		\label{eq_20161005054259}
		T_\ast^{\frac{3}{2}\kappa}
		=
			M_\ast^{-3}
			\wedge
			\{2\const[\|X\|_{\drivers{1}{\kappa}}](1+M_\ast^2)\}^{-1}
			\wedge
			1.
	\end{gather}
	Then
	\begin{gather}
		\label{eq_20161005060055}
		\|\mathcal{M}(v,w)\|_{\sols{T_\ast}{\kappa}{\kappa'}}
		\leq
			M_\ast,\\
		\|\mathcal{M}(v^{(1)},w^{(1)})-\mathcal{M}(v^{(2)},w^{(2)})\|_{\sols{T_\ast}{\kappa}{\kappa'}}
		\leq
			\frac{1}{2}
			\|(v^{(1)},w^{(1)})-(v^{(2)},w^{(2)})\|_{\sols{T_\ast}{\kappa}{\kappa'}}.
	\end{gather}
	We see that the map $\mathcal{M}$ is contraction on $\mathbf{B}_{T_\ast,M_\ast}$
	Therefore there exists a unique fixed point $(v,w)\in\mathbf{B}_{T_\ast,M_\ast}$ of $\mathcal{M}$,
	which is a solution on $[0,T_\ast]$.

	Next we show that the solution on $[0,T_\ast]$ is unique.
	Let $(v^{(1)},w^{(1)}),(v^{(2)},w^{(2)})\in\sols{T_\ast}{\kappa}{\kappa'}$ are solutions
	with a common initial condition $(v_0,w_0)$.
	We show that $(v^{(1)},w^{(1)})=(v^{(2)},w^{(2)})$.
	Taking $M>0$ such that
	\begin{align*}
		\|(v^{(1)},w^{(1)})\|_{\sols{T_\ast}{\kappa}{\kappa'}}
		\vee
		\|(v^{(2)},w^{(2)})\|_{\sols{T_\ast}{\kappa}{\kappa'}}
		\leq
			M,
	\end{align*}
	the similar arguments as above ensure that $\mathcal{M}$ is
	a contraction on $\mathbf{B}_{T_{**},M}$, where $T_{**}(\le T_*)$ depends on $M$.
	Hence $(v^{(1)},w^{(1)})$ and $(v^{(2)},w^{(2)})$ coincide on $[0,T_{\ast\ast}]$.
	We can continue this procedure on $[T_{**},2T_{**}]$, $[2T_{**},3T_{**}],\dots$.
	However, in these steps, we need to check
	that $(\tilde{v}^{(i)},\tilde{w}^{(i)})(t)=(v^{(i)},w^{(i)})(t+T_{**})$ satisfies
	\begin{align*}
		\|(\tilde{v}^{(i)},\tilde{w}^{(i)})\|_{\sols{T_\ast-T_{\ast\ast}}{\kappa}{\kappa'}}
		\leq
			M,
	\end{align*}
	since for example
	\begin{align*}
		\|\tilde{v}^{(i)}\|_{\cE_{T_\ast-T_{\ast\ast}}^{\frac{5}{6}-\kappa'}\HolBesSp{1-\kappa'}}
		&=
			\sup_{T_{\ast\ast}<t\leq T_{\ast}}(t-T_{\ast\ast})^{\frac{5}{6}-\kappa'}
				\|v^{(i)}(t)\|_{\HolBesSp{1-\kappa'}}\\
		&\leq
			\sup_{0<t\leq T_\ast}t^{\frac{5}{6}-\kappa'}
				\|v^{(i)}(t)\|_{\HolBesSp{1-\kappa'}}\\
		&=
			\|v^{(i)}\|_{\mathcal{E}_{T_\ast}^{\frac{5}{6}-\kappa'}\HolBesSp{1-\kappa'}}.
	\end{align*}
	Obviously $(\tilde{v}^{(i)},\tilde{w}^{(i)})$ is a solution with the initial condition
	$
		(v^{(1)},w^{(1)})(T_{\ast\ast})
		=
			(v^{(2)},w^{(2)})(T_{\ast\ast})
	$.
	Therefore we can iterate the above arguments on
	$[kT_{\ast},(k+1)T_{\ast\ast}\wedge T]$
	for $k=1,2,\dots$ and thus $(v^{(1)},w^{(1)})$ and $(v^{(2)},w^{(2)})$ coincide on $[0,T_\ast]$.

	We show the last assertion.
	From \eqref{eq_20161005054253} and \eqref{eq_20161005054259},
	we see that $T_\ast$ continuously depends on the initial condition $(v_0,w_0)$ and the driving vector $X$.
	Since $\const[2]$ depends on the driving vector $X$ continuously, $M_\ast$ is a continuous map from
	$(v_0,w_0)$ and $X$.
	From this fact and the continuity of $\const[\|X\|_{\drivers{1}{\kappa}}]$,
	we see the continuity of $T_\ast$.
	Hence we have $T_\ast^{(n)}\to T_\ast$.
	Without loss of generality, for fixed $t<T_\ast$, we assume that $T_\ast^{(n)}>t$ for every $n$.
	From \eqref{eq_20161005060055} and the continuity of $M_\ast$ with respect to $(v_0,w_0)$ and $X$,
	we see
	$
		\sup_n\|(v^{(n)},w^{(n)})\|_{\sols{T}{\kappa}{\kappa'}}
		<
			\infty
	$.
	From this fact and \eqref{eq_20161005061121},
	we can choose $\const[3]'$ and $\const[4]'$ such that
	\begin{multline*}
		\|
			(v,w)
			-
			(v^{(n)},w^{(n)})
		\|_{\sols{t}{\kappa}{\kappa'}}
		\leq
			\const[3]'
			\Big(
				\|v_0-v^{(n)}_0\|_{\HolBesSp{-\frac{2}{3}+\kappa'}}
				+
				\|w_0-w^{(n)}_0\|_{\HolBesSp{-\frac{1}{2}-2\kappa}}
			\Big)\\
			+
			\const[4]'
			t^{\frac{3}{2}\kappa}
			\Big(
				\|X-X^{(n)}\|_{\drivers{1}{\kappa}}
				+\|(v,w)-(v^{(n)},w^{(n)})\|_{\sols{t}{\kappa}{\kappa'}}
			\Big).
	\end{multline*}
	Hence we have $(v^{(n)},w^{(n)})\to(v,w)$ in $[0,t_*]$ for some $t_*\le t$
	depending on $\const[3]'$ and $\const[4]'$.
	Iterating this argument, we have the convergence in $[0,t]$.
	The proof is completed.
\end{proof}

\begin{remark}\label{rem_20161013061330}
	If $(v_0,w_0)\in\HolBesSp{1-\kappa'}\times\HolBesSp{\frac{3}{2}-2\kappa'}$,
	we obtain the local well-posedness on the space $\cL_T^{1-\kappa',1-\frac{\kappa'}{2}}\times\cL_T^{\frac{3}{2}-\kappa',1-\kappa'}$
	without explosion at $t=0$ by similar arguments.
\end{remark}

\begin{proposition}
	For every $(v_0,w_0)\in\HolBesSp{-\frac{2}{3}+\kappa'}\times\HolBesSp{-\frac{1}{2}-2\kappa}$
	and $X\in\drivers{T}{\kappa}$,
	there exists $T_{\mathrm{sur}}\in(0,\infty]$ such that the system \eqref{eq_20160920090410}
	has a unique solution $(v,w)\in\sols{t}{\kappa}{\kappa'}$ for every $t<T_{\mathrm{sur}}$,
	and
	\begin{align*}
		\lim_{t\uparrow T_{\mathrm{sur}}}
			(\|v\|_{C_t\HolBesSp{-\frac{2}{3}+\kappa'}}+\|w\|_{C_t\HolBesSp{-\frac{1}{2}-2\kappa}})
		=
			\infty
	\end{align*}
	unless $T_{\mathrm{sur}}=\infty$.
	Furthermore, the mapping from $(v_0,w_0,X)$ to the maximal solution $(v,w)$ is continuous in the sense that,
	for a sequence $\{(v_0^{(n)},w_0^{(n)},X^{(n)})\}$ which converges to $(v_0,w_0,X)$,
	we have
	$
		T_{\mathrm{sur}}
		\leq
			\liminf_{n\to\infty}T_{\mathrm{sur}}^{(n)}
	$ and
	\begin{align*}
		\|(v^{(n)},w^{(n)})-(v,w)\|_{\mathcal{D}_t^{\kappa,\kappa'}}\to0
	\end{align*}
	for every $t<T_{\mathrm{sur}}$.
\end{proposition}

\begin{proof}
	Let $(v,w)\in\sols{T_\ast}{\kappa}{\kappa'}$ be a unique solution on $[0,T_\ast]$ shown in \tref{thm_20160920085605}
	Because of \rref{rem_20161013061330}, we can start from $(v,w)(T_*)\in\HolBesSp{1-\kappa'}\times\HolBesSp{\frac{3}{2}-2\kappa'}$
	and construct a solution $(\bar{v},\bar{w})\in\sols{T_{\ast\ast}}{\kappa}{\kappa'}$ with $(\bar{v},\bar{w})(0)=(v,w)(T_\ast)$.
	Obviously the extended function
	\begin{align*}
		(\FourierCoeff{v},\FourierCoeff{w})(t)
		=
			\begin{cases}
				(v,w)(t)&t\in[0,T_*]\\
				(\bar{v},\bar{w})(t-T_*)&t\in[T_*,T_*+T_{**}]
			\end{cases}
	\end{align*}
	belongs to $\sols{T_\ast+T_{\ast\ast}}{\kappa}{\kappa'}$ and solves the system \eqref{eq_20160921004656}.
	Uniqueness on $[0,T_\ast+T_{\ast\ast}]$ also holds. We can iterate this argument until the time $T_{\mathrm{sur}}$,
	which is a supremum up to when the existence and uniqueness hold.

	The lower semi-continuity of $T_{\mathrm{sur}}$ follows from the continuity of $T_\ast$.
	Let $(v_0^{(n)},w_0^{(n)},X^{(n)})\to(v_0,w_0,X)$.
	For any fixed $t<T_{\mathrm{sur}}$, we can construct a unique solution in $[0,t]$
	by gluing finite number of local solutions as above.
	In this procedure, each of length of time interval converges,
	so that the solution $(v^{(n)},w^{(n)})$ exists in $[0,t]$ for sufficiently large $n$.
	This implies $t<\liminf_{n\to\infty}T_{\mathrm{sur}}^{(n)}$.

	Now assume that $T_{\mathrm{sur}}<\infty$. If
	\begin{align*}
		\lim_{t\uparrow T_{\mathrm{sur}}}
			(\|v\|_{C_t\HolBesSp{-\frac{2}{3}+\kappa'}}+\|w\|_{C_t\HolBesSp{-\frac{1}{2}-2\kappa}})
		<
			\infty,
	\end{align*}
	we can start from $(v,w)(T_{\mathrm{sur}}-\delta)\in\HolBesSp{-\frac{2}{3}+\kappa'}\times\HolBesSp{-\frac{1}{2}-2\kappa}$
	for small $\delta>0$ and construct a solution on $[0,T_\ast]$,
	where $T_\ast$ is uniform over $\delta$. This implies that for sufficiently small $\delta>0$,
	we can construct a solution on $[T_{\mathrm{sur}}-\frac{\delta}{2},T_{\mathrm{sur}}+\frac{\delta}{2}]$
	without explosion at the starting time. This is a contradiction,
	so we obtain the existence and uniqueness up to survival time with respect to the weaker norms.
\end{proof}

\subsection{Renormalized equation}\label{sec_20160920094017}

In this subsection, we show that a solution  in the sense of \tref{thm_20160920085605}
to the equation with a driving vector constructed
from a driving force $\xi\in C_T\HolBesSp{\beta}$ for $\beta>-2$
and renormalization constants solves the renormalized equation.

We fix complex constants $\constRe[1]{}$, $\constRe[2,1]{}$ and $\constRe[2,2]{}$
and define functions $X^\tau$ as in \tblref{tbl_20160920085835}
for every graphical symbols $\tau$ and construct the driving vector $X=(X^\A,\dots,X^\IAABoBB)$.
The Y/N in the Driver column in \tblref{tbl_20160920085835} indicates
whether the term $X^\tau$ is included in the definition of a driving vector or not.
The term $X^\tau$ with Driver column N is going to be used to define other terms.
For the definition of $I(\ast,\bullet)$, see \eqref{eq_20161010092955}.
Note that we can interpret the product in \tblref{tbl_20160920085835} in the usual sense
because $X^\A$ is a $\CmplNum$-valued continuous function by the assumption $\xi\in C_T\HolBesSp{\beta}$ for $\beta>-2$.
The number in Regularity column denotes the exponent $\alpha_\tau$ of the H\"older-Besov space $\HolBesSp{\alpha_\tau}$
where the term $X^\tau$ lives in. Precisely, $\alpha_\tau$ means $\alpha_\tau-\kappa$ for any $\kappa>0$ small enough.
\begin{table}[h]
  \begin{center}
	\caption{Definition of a driving vectors}\label{tbl_20160920085835}
	\begin{tabular}{|c|c|c|c|} \hline
Driver & Symbol & Definition & Regularity $\alpha_\tau$ \\ \hline
Y & $X^\A(=Z)$ & $I(X^\A_0,\whiteNoise)$ & $-1/2$  \\
N & $X^\B$ & $\CmplConj{X^\A}$ & $-1/2$  \\ \hline
Y & $X^\AA$ & $(X^\A)^2$ & $-1$  \\
Y & $X^\AB$ & $X^\A X^\B-\constRe[1]{}$ & $-1$  \\
N & $X^\BB$ & $(X^\B)^2$ & $-1$  \\
N & $X^\AAB$ & $X^\AA X^\B-2\constRe[1]{}X^\A$ & $-3/2$  \\
Y & $X^\IAA$ & $I(X^\IAA_0,X^\AA)$ & $+1$  \\
Y & $X^\IAB$ & $I(X^\IAB_0,X^\AB)$ & $+1$  \\
Y & $X^\IAAB(=W)$ & $I(X^\IAAB_0,X^\AAB)$ & $+1/2$  \\ \hline
Y & $X^\IAABoA$ & $X^\IAAB\reso X^\A$ & $0$  \\
Y & $X^\IAABoB$ & $X^\IAAB\reso X^\B$ & $0$  \\ \hline
Y & $X^\IAAoAB$ & $X^\IAA\reso X^\AB$ & $0$  \\
Y & $X^\IAAoBB$ & $X^\IAA\reso X^\BB-2\constRe[2,1]{}$ & $0$  \\
Y & $X^\IABoAB$ & $X^\IAB\reso X^\AB-\constRe[2,2]{}$ & $0$  \\
Y & $X^\IABoBB$ & $X^\IAB\reso X^\BB$ & $0$  \\ \hline
Y & $X^\IAABoAB$ & $X^\IAAB\reso X^\AB-2\constRe[2,2]{}X^\A$ & $-1/2$  \\
Y & $X^\IAABoBB$ & $X^\IAAB\reso X^\BB-2\constRe[2,1]{}X^\B$ & $-1/2$  \\ \hline
	\end{tabular}
  \end{center}
\end{table}

The next theorem is about the renormalized equation.
\begin{theorem}\label{thm_20160920094302}
  Let $0<\kappa<\kappa'<1/18$.
	Let $\xi\in C_T\HolBesSp{\beta}$ and $X^\A_0, X^\IAA_0, X^\IAB_0, X^\IAAB_0\in\HolBesSp{\beta+2}$ for $\beta>-2$.
	Construct $X\in\drivers{1}{\kappa}$ as in \tblref{tbl_20160920085835}.
	Let $(v,w)\in\sols{T}{\kappa}{\kappa'}$ be the solution to \eqref{eq_20160920090410}
	with the initial condition $(v_0,w_0)\in\HolBesSp{-\frac{2}{3}+\kappa'}\times\HolBesSp{-\frac{1}{2}-2\kappa}$
  for the driving vector $X$.
	Set $u=Z-\nu W+v+w$ and
	$
		\constRe{}
		=
		2
		(
			\constRe[1]{}
			-\CmplConj{\nu}\CmplConj{\constRe[2,1]{}}
			-2\nu\constRe[2,2]{}
		)
	$.
	Then $u$ solves
	\begin{align*}
		\partial_t u
		=
			(\ImUnit+\mu)\LaplaceOp u
			+
			\nu(1-|u|^2) u
			+
			\nu \constRe{} u
			+
			\whiteNoise,
		\qquad
		t>0,\quad
		x\in\Torus^3.
	\end{align*}
	with the initial condition $u_0=Z_0-\nu W_0+v_0+w_0$
	in the usual mild sense.
\end{theorem}

The next lemma plays a key role to prove \tref{thm_20160920094302}.
\begin{lemma}
	Let $(v,w)$ be the solution to \eqref{eq_20160920090410}.
	Set $u_2=v+w$. Then, we have
	\begin{multline}\label{eq_20160921005804}
		F(v,w)+G(v,w)\\
		\begin{aligned}
			&=
				-
				\nu
				\big\{
					(-\nu W+u_2)^2
					(\CmplConj{Z}-\CmplConj{\nu}\CmplConj{W}+\CmplConj{u_2})
					+
					2(-\nu W+u_2)(-\CmplConj{\nu}\CmplConj{W}+\CmplConj{u_2})Z\\
			&\phantom{=}\quad\qquad
					+2(-\nu W+u_2)X^\AB
					+(-\CmplConj{\nu}\CmplConj{W}+\CmplConj{u_2})X^\AA\\
			&\phantom{=}\quad\qquad
					+2(\CmplConj{\nu}\CmplConj{\constRe[2,1]{}}+2\nu\constRe[2,2]{})(Z-\nu W+u_2)
				\big\}\\
			&\phantom{=}\quad
				+
				(\nu+1)(Z-\nu W+u_2).
		\end{aligned}
	\end{multline}
\end{lemma}
\begin{proof}
	It follows from the definition that
	\begin{align}
		\label{eq_20160921001259}
		G_1(v,w)+G_2(v,w)
		=
			-\nu
			\big\{
				u_2^2(\CmplConj{Z}-\CmplConj{\nu}\CmplConj{W})
				+\CmplConj{u_2}(u_2^2+2u_2(Z-\nu W))
			\big\}.
	\end{align}
	We will show
	\begin{align}
		\label{eq_20160921001318}
		&
			G_3(v,w)+G_4(v,w)+G_5(v,w)\\\notag
		&\quad
		=
			-
			\nu
			\big\{
				u_2
				(
					2\nu\CmplConj{\nu}W\CmplConj{W}
					-2\nu W\CmplConj{Z}
					-2\CmplConj{\nu}\CmplConj{W}Z
				)
				+
				\CmplConj{u_2}
				(
					\nu^2W^2
					-2\nu WZ
				)\\\notag
		&\phantom{=}\quad\quad\qquad
				-\nu^2\CmplConj{\nu}W^2\CmplConj{W}
				+\nu^2W^2\CmplConj{Z}
				+2\nu\CmplConj{\nu}W\CmplConj{W}Z\\\notag
		&\phantom{=}\quad\quad\qquad
				-4\nu ((-\nu W+u_2)\rpara X^\IAB)\reso X^\AB\\\notag
		&\phantom{=}\quad\quad\qquad
				-2\CmplConj{\nu} ((-\CmplConj{\nu}\CmplConj{W}+\CmplConj{u_2})\rpara \CmplConj{X^\IAB})\reso X^\AA\\\notag
		&\phantom{=}\quad\quad\qquad
				-2\nu ((-\CmplConj{\nu}\CmplConj{W}+\CmplConj{u_2})\rpara X^\IAA)\reso X^\AB\\\notag
		&\phantom{=}\quad\quad\qquad
				-\CmplConj{\nu} ((-\nu W+u_2)\rpara \CmplConj{X^\IAA})\reso X^\AA\\\notag
		&\phantom{=}\quad\quad\qquad
				-2\nu W(\reso+\lpara)X^\AB
				-\CmplConj{\nu}\CmplConj{W}(\reso+\lpara)X^\AA\\\notag
		&\phantom{=}\quad\quad\qquad
				+2(\CmplConj{\nu}\CmplConj{\constRe[2,1]{}}+2\nu\constRe[2,2]{})(Z-\nu W+u_2)
			\big\}\\\notag
		&\phantom{=}\quad\quad
			+
			(\nu+1)(Z-\nu W+u_2),\\
		&
		\label{eq_20160921001334}
			G_6(v,w)+G_7(v,w)+G_8(v,w)\\\notag
		&\quad=
			-\nu
			\big\{
				2u_2(\reso+\lpara)X^\AB+\CmplConj{u_2}(\reso+\lpara)X^\AA\\\notag
		&\phantom{=}\quad\quad\qquad
				+4\nu((-\nu W+u_2)\rpara X^\IAB)\reso X^\AB\\\notag
		&\phantom{=}\quad\quad\qquad
				+2\CmplConj{\nu}((-\CmplConj{\nu}\CmplConj{W}+\CmplConj{u_2})\rpara\CmplConj{X^\IAB})\reso X^\AA\\\notag
		&\phantom{=}\quad\quad\qquad
				+2\nu((-\CmplConj{\nu}\CmplConj{W}+\CmplConj{u_2})\lpara X^\IAA)\reso X^\AB\\\notag
		&\phantom{=}\quad\quad\qquad
				+\CmplConj{\nu}((-\nu W+u_2)\rpara\CmplConj{X^\IAA})\reso X^\AA\notag
			\big\}.
	\end{align}
	Summing them up, we obtain
	\begin{multline*}
		G_1(v,w)+\cdots+G_8(v,w)\\
		\begin{aligned}
			&=
				-
				\nu
				\big\{
					\big\{
						u_2^2+2u_2(Z-\nu W)+\nu^2 W^2-2\nu W Z
					\big\}
					\CmplConj{u_2}\\
			&\phantom{=}\quad\qquad
					+
					\big\{
						u_2^2-2\nu Wu_2+2Zu_2+\nu^2 W^2-2\nu WZ
					\big\}
					(-\CmplConj{\nu}\CmplConj{W})\\
			&\phantom{=}\quad\qquad
					+
					\big\{
						u_2^2-2\nu W u_2+\nu^2 W^2
					\big\}
					\CmplConj{Z}\\
			&\phantom{=}\quad\qquad
					+2(-\nu W+u_2)(\reso+\lpara)X^\AB
					+(-\CmplConj{\nu}\CmplConj{W}+\CmplConj{u_2})(\reso+\lpara)X^\AA\\
			&\phantom{=}\quad\qquad
					+2(\CmplConj{\nu}\CmplConj{\constRe[2,1]{}}+2\nu\constRe[2,2]{})(Z-\nu W+u_2)
				\big\}\\
			&\phantom{=}\quad
				+
				(\nu+1)(Z-\nu W+u_2)
				+
				cv-w\\
			&=
				-
				\nu
				\big\{
					(-\nu W+u_2)^2
					(\CmplConj{Z}-\CmplConj{\nu}\CmplConj{W}+\CmplConj{u_2})
					+
					2(-\nu W+u_2)(-\CmplConj{\nu}\CmplConj{W}+\CmplConj{u_2})Z\\
			&\phantom{=}\quad\qquad
					+2(-\nu W+u_2)(\reso+\lpara)X^\AB
					+(-\CmplConj{\nu}\CmplConj{W}+\CmplConj{u_2})(\reso+\lpara)X^\AA\\
			&\phantom{=}\quad\qquad
					+2(\CmplConj{\nu}\CmplConj{\constRe[2,1]{}}+2\nu\constRe[2,2]{})(Z-\nu W+u_2)
				\big\}\\
			&\phantom{=}\quad
				+
				(\nu+1)(Z-\nu W+u_2),
		\end{aligned}
	\end{multline*}
	which implies the conclusion.

	For the rest of this proof, we prove  \eqref{eq_20160921001318} and \eqref{eq_20160921001334}.
	To show \eqref{eq_20160921001318}, we use the definition of $X^\IABoAB$ and \pref{prop_comm_20160919055939}.
	From them, we see
	\begin{multline*}
		W X^\IABoAB+4\nu \comR(W,X^\IAB,X^\AB)\\
		\begin{aligned}
			&=
				W\{(X^\IAB\reso X^\AB)-\constRe[2,2]{}\}
				+(W\rpara X^\IAB)\reso X^\AB
				-W(X^\IAB\reso X^\AB)\\
			&=
				(W \rpara X^\IAB)\reso X^\AB
				-\constRe[2,2]{} W
		\end{aligned}
	\end{multline*}
	A similar argument implies
	\begin{align*}
		\CmplConj{W}\CmplConj{X^\IABoBB}+\comR(\CmplConj{W},\CmplConj{X^\IAB},X^\AA)
		&=
			(\CmplConj{W}\rpara \CmplConj{X^\IAB})\reso X^\AA\\
		\CmplConj{W}X^\IAAoAB+\comR(\CmplConj{W},X^\IAA,X^\AB)
		&=
			(\CmplConj{W}\rpara X^\IAA)\reso X^\AB\\
		W\CmplConj{X^\IAAoBB}+\comR(W,\CmplConj{X^\IAA},X^\AA)
		&=
			(W\rpara \CmplConj{X^\IAA})\reso X^\AA-2\CmplConj{\constRe[2,1]{}}W.
	\end{align*}
	Applying these identities and the definitions of $X^\IAABoAB$ and $X^\IAABoBB$, we obtain
	\begin{align*}
		G_4(v,w)
		&=
			-
			\nu
			\big\{
				-\nu^2\CmplConj{\nu}W^2\CmplConj{W}
				+\nu^2W^2\CmplConj{Z}
				+2\nu\CmplConj{\nu}W\CmplConj{W}Z\\
		&\phantom{=}\quad\qquad
				+4\nu^2 \{(W\rpara X^\IAB)\reso X^\AB-\constRe[2,2]{}W\}
				+2\CmplConj{\nu}^2 (\CmplConj{W}\rpara \CmplConj{X^\IAB})\reso X^\AA\\
		&\phantom{=}\quad\qquad
				+2\nu\CmplConj{\nu} (\CmplConj{W}\rpara X^\IAA)\reso X^\AB
				+\nu\CmplConj{\nu} \{(W\rpara \CmplConj{X^\IAA})\reso X^\AA-2\CmplConj{\constRe[2,1]{}}W\}\\
		&\phantom{=}\quad\qquad
				-2\nu (W\reso X^\AB-2\constRe[2,2]{}X^\A)
				-2\nu W\lpara X^{\AB}\\
		&\phantom{=}\quad\qquad
				-\CmplConj{\nu}(\CmplConj{W\reso X^\BB}-2\CmplConj{\constRe[2,1]{}X^\B})
				-\CmplConj{\nu}\CmplConj{W}\lpara X^{\AA}
			\big\}\\
		&\phantom{=}\quad
			+(\nu+1)(Z-\nu W)\\
		&=
			-
			\nu
			\big\{
				-\nu^2\CmplConj{\nu}W^2\CmplConj{W}
				+\nu^2W^2\CmplConj{Z}
				+2\nu\CmplConj{\nu}W\CmplConj{W}Z\\
		&\phantom{=}\quad\qquad
				+4\nu(\nu W\rpara X^\IAB)\reso X^\AB
				+2\CmplConj{\nu} (\CmplConj{\nu}\CmplConj{W}\rpara \CmplConj{X^\IAB})\reso X^\AA\\
		&\phantom{=}\quad\qquad
				+2\nu (\CmplConj{\nu}\CmplConj{W}\rpara X^\IAA)\reso X^\AB
				+\CmplConj{\nu} (\nu W\rpara \CmplConj{X^\IAA})\reso X^\AA\\
		&\phantom{=}\quad\qquad
				-2\nu W(\reso+\lpara)X^\AB
				-\CmplConj{\nu}\CmplConj{W}(\reso+\lpara)X^\AA\\
		&\phantom{=}\quad\qquad
				+2(\CmplConj{\nu}\CmplConj{\constRe[2,1]{}}+2\nu \constRe[2,2]{})(Z-\nu W)
			\big\}\\
		&\phantom{=}\quad
			+(\nu+1)(Z-\nu W).
	\end{align*}
	We use the similar argument to obtain
	\begin{multline*}
		G_3(v,w)+G_5(v,w)\\
		\begin{aligned}
			&=
				-
				\nu
				\big\{
					u_2
					(
						2\nu\CmplConj{\nu}W\CmplConj{W}
						-2\nu W\CmplConj{Z}
						-2\CmplConj{\nu}\CmplConj{W}Z
					)
					+
					\CmplConj{u_2}
					(
						\nu^2W^2
						-2\nu WZ
					)\\
			&\phantom{=}\quad\qquad
					-4\nu (u_2\rpara X^\IAB)\reso X^\AB
					-2\CmplConj{\nu} (\CmplConj{u_2}\rpara \CmplConj{X^\IAB})\reso X^\AA\\
			&\phantom{=}\quad\qquad
					-2\nu (\CmplConj{u_2}\rpara X^\IAA)\reso X^\AB
					-\CmplConj{\nu} (u_2\rpara \CmplConj{X^\IAA})\reso X^\AA\\
			&\phantom{=}\quad\qquad
					+2(\CmplConj{\nu}\CmplConj{\constRe[2,1]{}}+2\nu\constRe[2,2]{})u_2
				\big\}\\
			&\phantom{=}\quad
				+
				(\nu+1)u_2.
		\end{aligned}
	\end{multline*}
	Combining them, we see \eqref{eq_20160921001318}.
	From the definition of $\com(v,w)$, we obtain \eqref{eq_20160921001334}.
	The proof is completed.
\end{proof}

\begin{proof}[Proof of \tref{thm_20160920094302}]
	Set $u_2=v+w$, $u_1=-\nu W+u_2$ and $u=Z+u_1$.
	Note that $u_2$ solves $\cL^1 u_2=F(v,w)+G(v,w)$.
	Substituting $X^\AB=Z\CmplConj{Z}-\constRe{}$ and $X^\AA=Z^2$ to \eqref{eq_20160921005804},
	we have
	\begin{align*}
		F(v,w)+G(v,w)
		&=
			-\nu (Z+u_1)^2(\CmplConj{Z}+\CmplConj{u_1})
			+\nu (Z+u_1)
			+2\nu\constRe{}(Z+u_1)\\
		&\phantom{=}\quad\qquad
			+\nu(Z^2\CmplConj{Z}-2\constRe[1]{}Z)
			+Z+u_1\\
		&=
			-\nu u^2\CmplConj{u}
			+\nu u
			+2\nu\constRe{}u
			+\nu(Z^2\CmplConj{Z}-2\constRe[1]{}Z)
			+u
	\end{align*}
	where $\constRe{}=\constRe[1]{}-\CmplConj{\nu}\CmplConj{\constRe[2,1]{}}+2\nu\constRe[2,2]{}$.
	Hence
	\begin{align*}
		\{\partial_t-(\ImUnit+\mu)\LaplaceOp\}u
		&=
			\cL^1 u-u\\
		&=
			\cL^1(Z-\nu W+u_2)-u\\
		&=
			\whiteNoise
			-\nu(Z^2\CmplConj{Z}-2\constRe[1]{}Z)
			+\{F(v,w)+G(v,w)\}
			-u\\
		&=
			-\nu u^2\CmplConj{u}
			+\nu u
			+2\nu\constRe{}u
			+\whiteNoise.
	\end{align*}
	The proof is completed.
\end{proof}

%% file: 0560_ConvOfDrivers.tex

\section{Proof of convergence of driving vectors}\label{sec_20161107073600}

This section is a probabilistic part of proof of \tref{thm_2017013154042}.
In this section, we construct a driving vector $X\in\drivers{T}{\kappa}$
associated to the white noise $\whiteNoise$ (\tref{thm_20161128051824}).
After that we derive the expression of renormalization constants
$\constRe[1]{\epsilon}$, $\constRe[2,1]{\epsilon}$ and $\constRe[2,2]{\epsilon}$
used in the construction of $X$ (\pref{prop_20161107071120})
and obtain the divergence rate of them (\pref{prop_20161107072011}).

First of all, we define Ornstein-Uhlenbeck like process $Z=Z(t,x)$, which is a seed of the driving vector.
The process $Z$ is defined as a stationary solution to the following equation:
\begin{align*}
	\partial_t Z
	=
		\{(\ImUnit+\mu)\LaplaceOp-1\} Z
		+
		\whiteNoise.
\end{align*}
The solution has a formal expression
\begin{align*}
	Z_t
	=
		I(\whiteNoise)_t
	=
		\int_{-\infty}^t
			P^1_{t-s}
			\whiteNoise_s\,
			ds
	=
		\sum_{k\in\Integers^3}
			\left(
				\int_{-\infty}^t
					P^1_{t-s}
					\FourierBase[k]
					\hat{\whiteNoise_s}(k)\,
					ds
			\right).
\end{align*}
Here, $I$ is defined by \eqref{eq_20161006075406}.
Since $Z$ is a distribution-valued process,
we cannot define processes such as $Z^2$ and $Z^2\CmplConj{Z}$ a priori.
To define such processes, we consider an approximation $\{Z^\epsilon\}_{0<\epsilon<1}$ of $Z$
and define $Z^2$ and $Z^2\CmplConj{Z}$ as renormalized limits of $(Z^\epsilon)^2$ and $Z^2\CmplConj{Z}$
in an appreciate topology, respectively.
To this end, we recall the smeared noise $\{\whiteNoise^\epsilon\}_{0<\epsilon<1}$ defined by \eqref{eq_20161201075726}
approximates the white noise $\whiteNoise$.
Using the approximation, we define
\begin{align}\label{eq_20161205034834}
	Z^\epsilon_t
	=
		\int_{-\infty}^t
			P^1_{t-s}
			\whiteNoise^\epsilon_s\,
			ds
	=
		\sum_{k\in\Integers^3}
			\chi^\epsilon(k)
			\left(
				\int_{-\infty}^t
					P^1_{t-s}
					\FourierBase[k]
					\hat{\whiteNoise_s}(k)\,
					ds
			\right).
\end{align}

We recall that the Fourier transform $\{\hat{\whiteNoise}(k)\}_{k\in\Integers^3}$ of $\whiteNoise$
has the same law of the white noise associated to $(E,\mathcal{B},d\mathfrak{m})$.
Here, $E=\RealNum\times\Integers^3$, $\mathcal{B}$ is the product $\sigma$-field of $\mathcal{B}(\RealNum)$ and $2^{\Integers^3}$
and $d\mathfrak{m}=dsdk$, where $ds$ and $dk$ are
the Lebesgue measure on $\RealNum$ and the counting measure $\Integers^3$, respectively.
Note that $d\mathfrak{m}$ is given by
\begin{align*}
	\mathfrak{m}(A)
	=
		\int_E
			\indicator{A}(s,k)\,
			dsdk
	=
		\sum_{k\in\Integers^3}
			\int_\RealNum
				\indicator{A}(s,k)\,
				ds.
\end{align*}
We denote by $\mathcal{B}^\ast$ the set of all elements $A\in \mathcal{B}$ such that $\mathfrak{m}(A)<\infty$.
Let
$
	M(A)
	=
		\sum_{k\in\Integers^3}
			\int_\RealNum
				\indicator{A}(s,k)
				\hat{\whiteNoise_s}(k)\,
				ds
$
for $A\in\mathcal{B}^\ast$.
Since
$\{M(A);A\in\mathcal{B}^\ast\}$ is a jointly isotropic complex normal such that
$\expect[M(A)\CmplConj{M(B)}]=\mathfrak{m}(A\cap B)$, we can define complex multiple It\^o-Wiener integrals $\WienerIntCmpl{p}{q}$
to calculate $(Z^\epsilon)^2$ and $(Z^\epsilon)^2\CmplConj{Z^\epsilon}$; see \secref{sec:itowienerintegral}.
By using them, we show their convergence after renormalization and construct the driving vector $X$.

Throughout this section, we use the notations in \secref{sec:itowienerintegral} and the following:
\begin{itemize}
	\item We use $m=(s,k)$, $n=(t,l)$, $\mu=(\sigma,k)$ and $\nu=(\tau,l)$ to denote a generic element in $E$.
	\item For $\nu_i=(\tau_i,l_i)$, we write $\nu_{-i}=(\tau_i,-l_i)$.
	\item For $p_1,\dots,p_n\in\Integers\setminus\{0\}$, we write $k_{p_1,\dots,p_n}=(k_{p_1},\dots,k_{p_n})$
		and $k_{[p_1,\dots,p_n]}=k_{p_1}+\dots+k_{p_n}$ for shorthand.
		We use the same abbreviation for $s$, $t$, $l$, $m$, $n$, $\sigma$, $\tau$, $\mu$ and $\nu$.
	\item We define
			$
				|k|_\ast
				=
					1+|k|
				=
					1+\sqrt{k_1^2+k_2^2+k_3^2}
			$
			for $k=(k_1,k_2,k_3)\in\Integers^3$
			and
			$
				|m|_\ast
				=
					|(s,k)|_\ast
				=
					1+|s|^{1/2}+|k|
			$.
			The same notations are used for $l$, $n$, $\mu$ and $\nu$.
\end{itemize}
Let $f:E^{p+q}\to\CmplNum$ satisfy
\begin{align*}
	\int_{\RealNum^{p+q}}
		|f((s,k)_{1,\dots,p},(t,l)_{1,\dots,q})|^2\,
		ds_1\cdots ds_p
		dt_1\cdots dt_q
	<
		\infty
\end{align*}
for every $k_1,\dots,k_p,l_1,\dots,l_q\in\Integers^3$.
For such $f$, we can define the Fourier transform $\FTTime f$ with respect to time parameters.
In particular, if $f$ is integrable and square-integrable with respect to the time parameters, then
$\FTTime f$ is given by
\begin{multline*}
	[\FTTime f]((\sigma,k)_{1,\dots,p},(\tau,l)_{1,\dots,q})
	=
		\int_{\RealNum^{p+q}}
			ds_1\cdots ds_p
			dt_1\cdots dt_q\\
			\times
			e^{-2\pi\ImUnit(\sigma_1s_1+\dots+\sigma_ps_p+\tau_1t_1+\dots+\tau_qt_q)}
			f((s,k)_{1,\dots,p},(t,l)_{1,\dots,q}).
\end{multline*}

\subsection{Convergence criteria}
In this subsection, we establish convergence criteria of It\^o-Wiener integrals.

\subsubsection{$\HolBesSp{\alpha}$-valued random variables}
We want to define a random field of the form
\begin{align*}
	X(x)
	=
		\WienerIntCmpl{p}{q}(f(x))
\end{align*}
for a kernel $f\in C(\Torus^3,L^\infty_{p,q})$ even if $f(x)\notin L^2_{p,q}$.
Here $L^\infty_{p,q}$ is the space of the essentially bounded measurable functions defined on $E^{p+q}$.
Assume now that
$
	\langle f,\phi\rangle
	=
		\int_{\Torus^3}
			f(x)
				\phi(x)\,
			dx
	\in
		L^2_{p,q}
$
for every $\phi\in\cD$
and define the family of random variables
\begin{align*}
	X(\phi)
	=
		\WienerIntCmpl{p}{q}(\langle f,\phi\rangle).
\end{align*}
If there exists a $\mathcal{D}'$-valued random variable $\tilde{X}$ such that
\begin{align*}
	\langle\tilde{X},\phi\rangle=X(\phi),
\end{align*}
then we write $\tilde{X}(x)=\WienerIntCmpl{p}{q}(f(x))$.

Now we define $X_j(x)=X((\FourierTrans^{-1}\DyaPartOfUnit[j])(x-\cdot))$.
If
$
	X
	=
		\sum_{j\ge-1}
			X_j
$
converges in $\mathcal{D}'$,
it satisfies $\langle X,\phi\rangle=X(\phi)$ for every $\phi\in\mathcal{D}$,
so we can write $X(x)=\WienerIntCmpl{p}{q}(f(x))$.

\begin{proposition}\label{prop_1481257897}
	Let $\alpha\in\RealNum$ and $p\in(1,\infty)$. If
	\begin{align*}
		C_{\alpha,p}
		=
			\sum_{j\ge-1}
				2^{(2\alpha p+1)j}
				\left(
					\sup_{x\in\Torus^3}
					\expect[|X_j(x)|^2]
				\right)^p
			<
				\infty,
	\end{align*}
	then $X=\sum_{j\ge-1}X_j$ converges in $L^{2p}(\Omega,\HolBesSp{\alpha})$ and we have
	\begin{align*}
		\expect[\|X\|_{\HolBesSp{\alpha}}^{2p}]
		\lesssim
			C_{\alpha,p}.
	\end{align*}
\end{proposition}
\begin{proof}
		Since
		$
			\langle f,(\FourierTrans^{-1}\DyaPartOfUnit[j])(x-\cdot)\rangle
			=
				\LPBlock[j]f(x)
		$,
		we have
		$
			\FourierTrans X_j(k)
			=
				\WienerIntCmpl{p}{q}(\DyaPartOfUnit[j] \FourierTrans f(k))
		$, which implies that the support of $\FourierTrans X_j(k)$ contained in an annulus.
		Hence we can apply {\cite[Lemma~2.69]{BahouriCheminDanchin2011}} to $X$.
		By a similar argument as {\cite[Lemma~5.3]{Hoshino2016arXiv}}, we see the assertion.
		(There, the following well-known property of Gaussian measures are used:
		on each fixed inhomogeneous Wiener chaos, all the $L^p$-norms, $1<p<\infty$, are equivalent.)
\end{proof}

\subsubsection{Good kernels}

We consider a random field of the form
\begin{align*}
	X(t,x)=\WienerIntCmpl{p}{q}(f_{(t,x)})
\end{align*}
for a kernel $f(t,\cdot)\in C(\Torus^3,L^\infty_{p,q})$
which satisfies the conditions as above for each fixed $t$.
We are interested in the case that $f$ satisfies the following good conditions.

\begin{definition}
	We say that a family $\{f_{(t,x)}\}_{t\ge0,x\in\Torus^3}$ is good if it has the form
	\begin{align*}
		f_{(t,x)}(m_{1,\dots,p},n_{1,\dots,q})
		=
			\FourierBase[k_{[1\dots p]}-l_{[1\dots q]}](x)
			H_t(m_{1,\dots,p},n_{1,\dots,q})
	\end{align*}
	for some $H_t\in L^\infty_{p,q}$ which is in $L^2$
	with respect to $(s_{1,\dots,p},t_{1,\dots,q})$ for each fixed $(k_{1,\dots,p},l_{1,\dots,q})$
	and $Q_t=\FTTime H_t$ satisfies
	\begin{align*}
		Q_t(\mu_{1,\dots,p},\nu_{1,\dots,q})
		=
			e^{-2\pi\ImUnit(\sigma_{[1\dots p]}+\tau_{[1\dots q]})t}
			Q_0(\mu_{1,\dots,p},\nu_{1,\dots,q}).
	\end{align*}
\end{definition}

For a function $f:E^{p+q}\to\CmplNum$, we set
\begin{align*}
	\tilde{\DyaPartOfUnit[j]}f(m_{1,\dots,p},n_{1,\dots,q})
	=
		\DyaPartOfUnit[j](k_{[1\dots p]}-l_{[1\dots q]})
		f(m_{1,\dots,p},n_{1,\dots,q}).
\end{align*}
We define
\begin{align*}
	R(\sigma_{1,\dots,p},\tau_{1,\dots,q})
	=
		\sigma_{[1\dots p]}+\tau_{[1\dots q]}.
\end{align*}

In order to estimate the Besov norm of $X$, it is enough to estimate $Q_0$.
\begin{proposition}
	Let $\{f_{(t,x)}\}_{t\ge0,x\in\Torus^3}$ be a good kernel.
	Assume that there exist $\beta\in\RealNum$, $\theta_0\in(0,2]$ and $C>0$ such that
	\begin{align*}
		\||R|^{\frac{\theta}2}\tilde{\DyaPartOfUnit[j]}Q_0\|_{L^2_{p,q}}\le C2^{(\beta+\theta)j}
	\end{align*}
	for every $j\ge-1$ and $\theta\in[0,\theta_0)$. Then we have
	\begin{align}\label{eq_1503553014}
		\expect
			[
				\|X\|_{C_T^\kappa\HolBesSp{\alpha-2\kappa}}^{2p}
			]
		\lesssim
			C^{2p},
	\end{align}
	for every $p\in(1,\infty)$, $\alpha<-\beta$ and $\kappa\in[0,\frac{\theta_0}2)$.
	Here $C_T^0\HolBesSp{\alpha}=C_T\HolBesSp{\alpha}$.
\end{proposition}

\begin{proof}
	Let $1<p<\infty$ satisfy $2(\alpha+\beta)p+1<0$.
	For every $t\in[0,\infty)$,	we have $X_t\in\HolBesSp{\alpha}$ from
	\begin{align*}
		\expect[\|X_t\|_{\HolBesSp{\alpha}}^{2p}]
		<
			\infty.
	\end{align*}
	We will show this inequality.
	Set
	$
		X_j(t,x)
		=
			\WienerIntCmpl{p}{q}
				(
					\langle
						f_{(t,\cdot)},
						(\FourierTrans^{-1}\DyaPartOfUnit[j])(x-\cdot)
					\rangle
				)
	$.
	Since
	$
		\langle
			f(t,\cdot),
			(\FourierTrans^{-1}\DyaPartOfUnit[j])(x-\cdot)
		\rangle
		=
			[\LPBlock[j]f_{(t,\cdot)}](x)
		=
			\tilde{\DyaPartOfUnit[j]}f_{(t,x)}
	$,
	we have
	\begin{multline*}
		\expect[|X_j(t,x)|^2]
		=
			\|\tilde{\DyaPartOfUnit[j]}f_{(t,x)}\|_{L^2_{p,q}}^2\\
		=
			\|\FTTime\tilde{\DyaPartOfUnit[j]} f_{(t,x)}\|_{L^2_{p,q}}^2
		=
			\|\tilde{\DyaPartOfUnit[j]} e^{-2\pi\ImUnit R t} Q_0\|_{L^2_{p,q}}^2
		=
			\|\tilde{\DyaPartOfUnit[j]} Q_0\|_{L^2_{p,q}}^2.
	\end{multline*}
	Using the assumption with $\theta=0$, we obtain
	$
		\expect[|X_j(t,x)|^2]
		\leq
			(
				\const
				2^{\beta j}
			)^2
	$.
	Hence
	\begin{align*}
		\const[\alpha,p]
		=
			\sum_{j=-1}^\infty
				2^{(2\alpha p+1)j}
				(
					\const
					2^{\beta j}
				)^{2p}
		\leq
			\const^{2p}
			\sum_{j=-1}^\infty
				2^{(2(\alpha+\beta)p+1)j}
		<
			\infty.
	\end{align*}
	From \pref{prop_1481257897}, we see the inequality.

	We show $X\in C^\kappa_T\HolBesSp{\alpha-2\kappa}$ and \eqref{eq_1503553014} for $\kappa\in(0,\theta_0/2)$.
	Set $\alpha'=\alpha-2\kappa$ and take $2\kappa<\theta<\theta_0$ such that $\alpha'+\beta+\theta<0$.
	For any $1<p<\infty$ such that $2(\alpha'+\beta+\theta)p+1<0$ and $(\theta-2\kappa)p>1$,
	we can show that
	\begin{gather*}
		\expect
			[
				\|X_t-X_s\|_{\HolBesSp{\alpha'}}^{2p}
			]
		\leq
			\const
			|t-s|^{p\theta},
	\end{gather*}
	where $\const$ is a positive constant independent of $s,t$.
	Note $(p\theta-1)/2p>\kappa$.
	These inequalities and the Kolmogorov continuity theorem \cite[Theorem~1.4.1]{Kunita1990}
	implies $X\in C^\kappa_T\HolBesSp{\alpha-2\kappa}$ and \eqref{eq_1503553014}.
	Next we show the assertions for $\kappa=0$.
	Let $\alpha<\alpha''<-\beta$.
	Then we see $X\in C^{\kappa'}_T\HolBesSp{\alpha''-2\kappa'}$
	for $\kappa'\in(0,\theta_0/2)$ by the above discussion.
	Choosing $\kappa'=(\alpha''-\alpha)/2$, we obtain $X\in C^{\kappa'}_T\HolBesSp{\alpha}$, which implies the conclusion.
\end{proof}

For a function $f:E^{p+q}\to\CmplNum$ and $\mu=(\sigma,k)$, we write
\begin{align*}
	\int_{\mu_{[1\dots p]}+\nu_{[(-1)\dots(-q)]}=\mu}
		f(\mu_{1,\dots,p},\nu_{1,\dots,q})
\end{align*}
for the integration over the ``hyperplane'' $\{\mu_{[1\dots p]}+\nu_{[(-1)\dots(-q)]}=\mu\}$.

\begin{proposition}\label{section6:from Q to X}
	Let $\{f_{(t,x)}\}_{t\ge0,x\in\Torus^3}$ be a good kernel.
	Assume that there exist $\gamma>1$, $\delta\ge0$ and $C>0$ such that
	\begin{align}\label{section6:we want gamma and delta}
		\int_{\mu_{[1\dots p]}+\nu_{[(-1)\dots(-q)]}=\mu}
			|Q_0(\mu_{1,\dots,p},\nu_{1,\dots,q})|^2
			\leq
			C
			|\mu|_*^{-2\gamma}
			|k|_*^{-2\delta}.
	\end{align}
	Then we have
	\begin{align*}
		\||R|^{\theta/2}\tilde{\DyaPartOfUnit[j]}Q_0\|_{L^2_{p,q}}
		\lesssim
			C2^{(\frac{5}{2}-\gamma-\delta+\theta)j}
	\end{align*}
	for every $\theta\in[0,\gamma-1)$. As a consequence, we have
	\begin{align*}
		\expect
			[
				\|X\|_{C_T^\kappa\HolBesSp{\alpha-2\kappa}}^{2p}
			]
		\lesssim C^{2p},
	\end{align*}
	for every $p\in(1,\infty)$, $\alpha<-\frac{5}{2}+\gamma+\delta$ and $\kappa\in[0,\frac{\gamma-1}{2}\wedge 1)$.
\end{proposition}

\begin{proof}
	Since
	\begin{multline*}
		\||R|^{\theta/2}\tilde{\DyaPartOfUnit[j]}Q_0\|_{L^2_{p,q}}^2\\
		\begin{aligned}
			&=
				\sum_{k\in\Integers^3}\DyaPartOfUnit[j](k)^2
					\int_\RealNum
						|\sigma|^\theta
						\left(
							\int_{\mu_{[1\dots p]}+\nu_{[(-1)\dots(-q)]}=\mu}
								|Q_0(\mu_{1,\dots,p},\nu_{1,\dots,q})|^2
						\right)\,
					d\sigma\\
			&\leq
				C
				\sum_{k\in\Integers^3}
					\DyaPartOfUnit[j](k)^2
					|k|_*^{-2\delta}
						\int_\RealNum
							|\mu|_*^{-2(\gamma-\theta)}\,
							d\sigma
		\end{aligned}
	\end{multline*}
	if $\gamma-\theta>1$, we have
	\begin{align*}
		\||R|^{\theta/2}\tilde{\DyaPartOfUnit[j]}Q_0\|_{L^2_{p,q}}^2
		&\lesssim
			C
			\sum_{k\in\Integers^3}
				\DyaPartOfUnit[j](k)^2
				|k|_*^{-2\delta}
				|k|_*^{2(1-\gamma+\theta)}\\
		&\lesssim
			C(2^j)^3
			(2^j)^{2(1-\gamma+\theta-\delta)}\\
		&=
			C2^{(5-2\gamma-2\delta+2\theta)j}.
	\end{align*}
	The proof is completed.
\end{proof}

\subsection{Definitions of driving vectors}

Since $Z$ is a distribution-valued process,
we cannot define a process such as $Z^2$, $Z\CmplConj{Z}$ and $Z^2\CmplConj{Z}$ a priori.
To define such processes, we consider an approximation $\{Z^\epsilon\}_{0<\epsilon<1}$ of $Z$
and define $Z^2$, $Z\CmplConj{Z}$ and $Z^2\CmplConj{Z}$
as renormalized limits of $(Z^\epsilon)^2$, $Z^\epsilon\CmplConj{Z^\epsilon}$ and $(Z^\epsilon)^2\CmplConj{Z^\epsilon}$.

\subsubsection{Ornstein-Uhlenbeck like process and its approximations}
We give an expression of $Z^\epsilon$ defined by \eqref{eq_20161205034834} in terms of It\^o-Wiener integral.
Since we have
$
	P^1_s\FourierBase[k]
	=
		h(s,k)
		\FourierBase[k]
$, where
$
	h(s,k)
	=
		e^{-\{4\pi^2(\ImUnit+\mu)|k|^2+1\}s}
$,
we see
\begin{align}\label{0550 Z is OU with xi^epsilon}
	Z^\epsilon_t
	=
		\sum_{k\in\Integers^3}
			\chi^\epsilon(k)
			\int_{-\infty}^t
				h(t-s,k)
				\FourierBase[k]
				\FourierCoeff{\whiteNoise_s}(k)\,
				ds.
\end{align}
Hence, we can write $Z^\epsilon_{(t,x)}=\WienerIntCmpl{1}{0}(f^\epsilon_{(t,x)})$ with
\begin{gather*}
	\begin{aligned}
		f^\epsilon_{(t,x)}(s,k)
		&=
			\FourierBase[k](x)
			H^\epsilon_t(s,k),
		&
		H^\epsilon_t(s,k)
		&=
			\chi^\epsilon(k)
			H_t(s,k),
	\end{aligned}\\
	H_t(s,k)
	=
		\indicator{[0,\infty)}(t-s)
		h(t-s,k).
\end{gather*}
Note that $Q_t=\FTTime H_t$ is given by
\begin{align*}
	Q_t(\sigma,k)
	=
		\frac{e^{-2\pi\ImUnit\sigma t}}{-2\pi\ImUnit\sigma+4\pi^2(\ImUnit+\mu)|k|^2+1}.
\end{align*}
In particular, we see $Q_t(\mu)=e^{-2\pi\ImUnit \sigma t}Q_0(\mu)$.
We simply write $Q_0=Q$.

\subsubsection{Definition of driving vectors}

For every $0<\epsilon<1$ and graphical symbols $\tau$,
we define distributions $X^{\epsilon,\tau}$ as in \tblref{tbl_20160723074305}.
\begin{table}[h]
  \begin{center}
	\caption{List of distributions}\label{tbl_20160723074305}
	\begin{tabular}{|c|c|c|c|} \hline
Driver & Distribution $X^{\epsilon,\tau}$ & Definition & Regularity $\alpha_\tau$ \\ \hline
Y & $X^{\epsilon,\A}$ & $Z^{\epsilon}$ & $-1/2$  \\
N & $X^{\epsilon,\B}$ & $\CmplConj{X^{\epsilon,\A}}$ & $-1/2$  \\ \hline
Y & $X^{\epsilon,\AA}$ & $(X^{\epsilon,\A})^2$ & $-1$  \\
Y & $X^{\epsilon,\AB}$ & $X^{\epsilon,\A}X^{\epsilon,\B}-\constRe[1]{\epsilon}$ & $-1$  \\
N & $X^{\epsilon,\BB}$ & $(X^{\epsilon,\B})^2$ & $-1$  \\
N & $X^{\epsilon,\AAB}$ & $X^{\epsilon,\AA}X^{\epsilon,\B}-2\constRe[1]{\epsilon}X^{\epsilon,\A}$ & $-3/2$  \\
Y & $X^{\epsilon,\IAA}$ & $I(X^{\epsilon,\AA})$ & $+1$  \\
Y & $X^{\epsilon,\IAB}$ & $I(X^{\epsilon,\AB})$ & $+1$  \\
Y & $X^{\epsilon,\IAAB}$ & $I(X^{\epsilon,\AAB})$ & $+1/2$  \\ \hline
Y & $X^{\epsilon,\IAABoA}$ & $X^{\epsilon,\IAAB}\reso X^{\epsilon,\A}$ & $0$  \\
Y & $X^{\epsilon,\IAABoB}$ & $X^{\epsilon,\IAAB}\reso X^{\epsilon,\B}$ & $0$  \\ \hline
Y & $X^{\epsilon,\IAAoAB}$ & $X^{\epsilon,\IAA}\reso X^{\epsilon,\AB}$ & $0$  \\
Y & $X^{\epsilon,\IAAoBB}$ & $X^{\epsilon,\IAA}\reso X^{\epsilon,\BB}-2\constRe[2,1]{\epsilon}$ & $0$  \\
Y & $X^{\epsilon,\IABoAB}$ & $X^{\epsilon,\IAB}\reso X^{\epsilon,\AB}-\constRe[2,2]{\epsilon}$ & $0$  \\
Y & $X^{\epsilon,\IABoBB}$ & $X^{\epsilon,\IAB}\reso X^{\epsilon,\BB}$ & $0$  \\ \hline
Y & $X^{\epsilon,\IAABoAB}$ & $X^{\epsilon,\IAAB}\reso X^{\epsilon,\AB}-2\constRe[2,2]{\epsilon}X^{\epsilon,\A}$ & $-1/2$  \\
Y & $X^{\epsilon,\IAABoBB}$ & $X^{\epsilon,\IAAB}\reso X^{\epsilon,\BB}-2\constRe[2,1]{\epsilon}X^{\epsilon,\B}$ & $-1/2$  \\ \hline
	\end{tabular}
  \end{center}
\end{table}
The operator $I$ is defined by \eqref{eq_20161006075406}
and the constants $\constRe[1]{\epsilon}$, $\constRe[2,1]{\epsilon}$ and $\constRe[2,2]{\epsilon}$
in \tblref{tbl_20160723074305} are defined by
\begin{align}\label{eq_20161015073329}
	\constRe[1]{\epsilon}
	&=
		\expect
			[
				X^{\epsilon,\A}_{(t,x)}
				X^{\epsilon,\B}_{(t,x)}
			],
	&
	\constRe[2,1]{\epsilon}
	&=
		\frac{1}{2}
		\expect
			[
				X^{\epsilon,\IAA}_{(t,x)}
				\reso
				X^{\epsilon,\BB}_{(t,x)}
			],
	&
	\constRe[2,2]{\epsilon}
	&=
		\expect
			[
				X^{\epsilon,\IAB}_{(t,x)}
				\reso
				X^{\epsilon,\AB}_{(t,x)}
			].
\end{align}
The other symbols and regularities have the same meanings as in \tblref{tbl_20160920085835}.
We set
\begin{multline*}
	X^\epsilon
	=
		(
			X^{\epsilon,\A},
			X^{\epsilon,\AA},
			X^{\epsilon,\AB},
			X^{\epsilon,\IAA},
			X^{\epsilon,\IAB},
			X^{\epsilon,\IAAB},\\
			X^{\epsilon,\IAABoA},
			X^{\epsilon,\IAABoB},
			X^{\epsilon,\IAAoAB},
			X^{\epsilon,\IAAoBB},
			X^{\epsilon,\IABoAB},
			X^{\epsilon,\IABoBB},
			X^{\epsilon,\IAABoAB},
			X^{\epsilon,\IAABoBB}
	).
\end{multline*}
The constants $\constRe[1]{\epsilon}$, $\constRe[2,1]{\epsilon}$ and $\constRe[2,2]{\epsilon}$
look dependent on $(t,x)$ and the dyadic partition $\{\DyaPartOfUnit[m]\}_{m=-1}^\infty$ of unity.
However, we will show that they are not in \pref{prop_20161107071120}.

\subsubsection{It\^o-Wiener integral expressions of driving vectors}
We give expressions of $X^{\epsilon,\tau}$ by It\^o-Wiener integrals.

We start to discuss with $\tau=\A$, $\B$, $\AA$, $\AB$, $\BB$, $\AAB$, $\IAA$, $\IAB$, $\IAAB$.
We denote by $p(\tau)$ and $q(\tau)$ the number of circles and squares in $\tau$, respectively.
We write
\begin{align*}
	\chi^\epsilon(k_{1,\dots,p},l_{1,\dots,q})
	=
		\prod_{i=1}^p\chi^\epsilon(k_i)
		\prod_{j=1}^q\chi^\epsilon(l_j).
\end{align*}

\begin{proposition}\label{prop_20161205074106}
	Let $\tau=\A$, $\B$, $\AA$, $\AB$, $\BB$, $\AAB$, $\IAA$, $\IAB$, $\IAAB$, $p=p(\tau)$ and $q=q(\tau)$.
	Then
	$
		X_{(t,x)}^{\epsilon,\tau}
		=
			\WienerIntCmpl{p}{q}(f_{(t,x)}^{\epsilon,\tau})
		$,
	where
	$
		f_{(t,x)}^{\epsilon,\tau}
		=
			f_{(t,x)}^{\epsilon,\tau}(m_{1,\dots,p},n_{1,\dots,q})
	$
	is a good kernel with functions $H_t^{\epsilon,\tau}$ and $Q_0^{\epsilon,\tau}$ defined as follows.
\begin{enumerate}[(1)]
	\item	We have
			$
				H_t^{\epsilon,\tau}(m_{1,\dots,p},n_{1,\dots,q})
				=
					\chi^\epsilon(k_{1,\dots,p},l_{1,\dots,q})
					H_t^\tau(m_{1,\dots,p},n_{1,\dots,q})
			$,
			where $\{H_t^\tau\}_{t\ge0}\in L_{p,q}^2$ is given as follows.
	\begin{itemize}
		\item	$H_t^\A(m_1)=H_t(m_1)$ and $H_t^\B(n_1)=\overline{H_t(n_1)}$.
		\item	For $\tau=\AA,\AB,\BB,\AAB$,
				\begin{align*}
					H_t^\tau(m_{1,\dots,p},n_{1,\dots,q})
					=
						\prod_{i=1}^p
							H_t^\A(m_i)
						\prod_{j=1}^q
							H_t^\B(n_j).
				\end{align*}
		\item	Let $\tau_0=\AA,\AB,\IAAB$ for $\tau=\IAA,\IAB,\IAAB$, respectively.
				\begin{align*}
					H_t^\tau(m_{1,\dots,p},n_{1,\dots,q})
					=
						\int_\RealNum
							H_t(u,k_{[1\dots p]}-l_{[1\dots q]})
							H_u^{\tau_0}(m_{1,\dots,p},n_{1,\dots,q})\,
							du.
				\end{align*}
	\end{itemize}
		\item	We have
				$
					Q_0^{\epsilon,\tau}(\mu_{1,\dots,p},\nu_{1,\dots,q})
					=
						\chi^\epsilon(k_{1,\dots,p},l_{1,\dots,q})
						Q_0^\tau(\mu_{1,\dots,p},\nu_{1,\dots,q})
				$,
				where $Q_0^\tau\in L_{p,q}^2$ is given as follows.
		\begin{itemize}
		\item $Q_0^\A(\mu_1)=Q(\mu_1)$ and $Q_0^\B(\nu_1)=\overline{Q_0(-\nu_{-1})}$.
		\item For $\tau=\AA,\AB,\BB,\AAB$,
				\begin{align*}
					Q_0^\tau(\mu_{1,\dots,p},\nu_{1,\dots,q})
					=
						\prod_{i=1}^p
							Q_0^\A(\mu_i)
						\prod_{j=1}^q
							Q_0^\B(\nu_j).
				\end{align*}
		\item Let $\tau_0=\AA,\AB,\IAAB$ for $\tau=\IAA,\IAB,\IAAB$, respectively.
				\begin{align*}
					Q_0^\tau(\mu_{1,\dots,p},\nu_{1,\dots,q})
					=
						Q(\mu_{[1\dots p]}+\nu_{[(-1)\dots(-q)]})
						Q_0^{\tau_0}(\mu_{1,\dots,p},\nu_{1,\dots,q}).
				\end{align*}
		\end{itemize}
	\end{enumerate}
	In the above, we regard $H_t^\tau$ as a function with respect to
	$n_{1,\dots,q}$ and $m_{1,\dots,p}$ for $p=0$ and $q=0$, respectively.
	In particular, $H_t^\tau$ is a constant for $p=q=0$.
	We use the same convention for $Q_0^\tau$.
\end{proposition}
\begin{proof}
	The assertion follows from \pref{2:product}.
\end{proof}

From this proposition, we can guess the limit $X^\tau$ of $\{X^{\epsilon,\tau}\}_{0<\epsilon<1}$ as follows:
\begin{definition}
	Let $\tau\in\{\A,\B,\AA,\AB,\BB,\AAB,\IAA,\IAB,\IAAB\}$, $p=p(\tau)$ and $q=q(\tau)$. We define
	\begin{align*}
		f_{(t,x)}^\tau(m_{1,\dots,p},n_{1,\dots,q})
		=
			\FourierBase[k_{[1\dots p]}-l_{[1\dots q]}]
			H_t^\tau(m_{1,\dots,p},n_{1,\dots,q})
	\end{align*}
	and $X^\tau_{(t,x)}=\WienerIntCmpl{p}{q}(f_{(t,x)}^\tau)$.
\end{definition}

Next, we consider $X^{\epsilon,\tau}$ for
$\tau=\IAABoA$, $\IAABoB$, $\IAAoAB$, $\IAAoBB$, $\IABoAB$, $\IABoBB$, $\IAABoAB$, $\IAABoBB$.
For these $\tau$, we define $(\tau_1,\tau_2)=(\IAAB,\A)$, $(\IAAB,\B)$, $(\IAA,\AB)$,
$(\IAA,\BB)$, $(\IAB,\AB)$, $(\IAB,\BB)$, $(\IAAB,\AB)$, $(\IAAB,\BB)$, respectively.
We simply write $p_i=p(\tau_i)$ and $q_i=q(\tau_i)$ for $i=1,2$. We set $p=p_1+p_2$ and $q=q_1+q_2$.

We define the function $\psi_\circ:\Integers^3\times\Integers^3\to\RealNum$ by
\begin{align}\label{eq_20161206032529}
	\psi_\circ(k,l)
	=
		\sum_{|i-j|\le1}
			\DyaPartOfUnit[i](k)\DyaPartOfUnit[j](l).
\end{align}

\begin{proposition}\label{section6:expansion of tau1 rs tau2}
	For above $(\tau_1,\tau_2)$, it holds that
	\begin{align*}
		X_t^{\epsilon,\tau_1}\reso X_t^{\epsilon,\tau_2}(x)
		=
			\sum_g
				\WienerIntCmpl{p-\#g}{q-\#g}(f_{(t,x)}^{\epsilon,(\tau_1,\tau_2,g)}),
	\end{align*}
	where $g$ runs over all of the graphs consisting of disjoint edges
	\begin{align*}
		(i,j)\in\{1,\dots,p_1\}\times\{q_1+1,\dots,q\}\cup\{p_1+1,\dots,p\}\times\{1,\dots,q_1\},
	\end{align*}
	and $f_{(t,x)}^{\epsilon,(\tau_1,\tau_2,g)}$ is a good kernel with functions $H_t^{\epsilon,(\tau_1,\tau_2,g)}$ and $Q_0^{\epsilon,(\tau_1,\tau_2,g)}$ defined as follows.
	\begin{enumerate}[(1)]
		\item	$H_t^{\epsilon,(\tau_1,\tau_2,\emptyset)}$ is given by
				\begin{multline*}
					H_t^{\epsilon,(\tau_1,\tau_2,\emptyset)}(m_{1,\dots,p},n_{1,\dots,q})
					=
						\psi_\circ(k_{[1\dots p_1]}-l_{[1\dots q_1]},k_{[(p_1+1)\dots p]}-l_{[(q_1+1)\dots q]})\\
						\times
						H_t^{\epsilon,\tau_1}(m_{1,\dots,p_1},n_{1,\dots,q_1})
						H_t^{\epsilon,\tau_2}(m_{(p_1+1),\dots,p},n_{(q_1+1),\dots,q}).
				\end{multline*}
				For general $g$, $H_t^{\epsilon,(\tau_1,\tau_2,g)}$ is given by
				\begin{align*}
					H_t^{\epsilon,(\tau_1,\tau_2,g)}(m_{1,\dots,p},n_{1,\dots,q}\setminus g)
					=
						\int_{E^{2\#g}}
							H_t^{\epsilon,(\tau_1,\tau_2,\emptyset)}(m_{1,\dots,p},n_{1,\dots,q})\,
							d(m,n)_g,
				\end{align*}
				where $(m_{1,\dots,p},n_{1,\dots,q}\setminus g)$ means that variables $(m_i,n_j)$
				are removed for all $(i,j)\in g$ and $d(m,n)_g=\prod_{(i,j)\in g}\delta(m_i-n_j)dm_idn_j$.
		\item	$Q_0^{\epsilon,(\tau_1,\tau_2,\emptyset)}$ is given by
				\begin{multline*}
					Q_0^{\epsilon,(\tau_1,\tau_2,\emptyset)}(\mu_{1,\dots,p},\nu_{1,\dots,q})
					=
						\psi_\circ(k_{[1\dots p_1]}-l_{[1\dots q_1]},k_{[(p_1+1)\dots p]}-l_{[(q_1+1)\dots q]})\\
						\times
						Q_0^{\epsilon,\tau_1}(\mu_{1,\dots,p_1},\nu_{1,\dots,q_1})
						Q_0^{\epsilon,\tau_1}(\mu_{(p_1+1),\dots,p},\nu_{(q_1+1),\dots,q}).
				\end{multline*}
				For general $g$, $Q_0^{\epsilon,(\tau_1,\tau_2,g)}$ is given by
				\begin{align*}
					Q_0^{\epsilon,(\tau_1,\tau_2,g)}(\mu_{1,\dots,p},\nu_{1,\dots,q}\setminus g)
					=
						\int_{E^{2\#g}}
							Q_0^{\epsilon,(\tau_1,\tau_2,\emptyset)}(\mu_{1,\dots,p},\nu_{1,\dots,q})\,
							d(\mu,\nu)_g,
				\end{align*}
				where $(\mu_{1,\dots,p},\nu_{1,\dots,q}\setminus g)$ means that variables $(\mu_i,\nu_j)$
				are removed for all $(i,j)\in g$ and $d(\mu,\nu)_g=\prod_{(i,j)\in g}\delta(\mu_i+\nu_{-j})d\mu_id\nu_j$.
	\end{enumerate}
\end{proposition}
For example,
\begin{align*}
	Q^{\epsilon,(\IAAB,\A,\emptyset)}_0(\mu_{1,2,3},\nu_1)
	&=
		\psi_\circ(k_{[12]}-l_1,k_3)
		\chi^\epsilon(k_1,k_2,k_3,l_1)\\
	&\phantom{=}\quad\qquad
		\times
		Q(\mu_{[12]}+\nu_{-1})
		Q(\mu_1)
		Q(\mu_2)
		\CmplConj{Q(-\nu_{-1})}
		Q(\mu_3),\\
	Q^{\epsilon,(\IAAB,\A,(3,1))}_0(\mu_1,\mu_2)
	&=
		\int_{E^2}
			\psi_\circ(k_{[12]}-l_1,k_3)
			\chi^\epsilon(k_1,k_2,k_3,l_1)\\
	&\phantom{=}\quad\qquad
			\times
			Q(\mu_{[12]}+\nu_{-1})
			Q(\mu_1)
			Q(\mu_2)
			\CmplConj{Q(-\nu_{-1})}\\
	&\phantom{=}\quad\qquad
			\times
			Q(\mu_3)\,
			\delta(\mu_3+\nu_{-1})\,
			d\mu_3d\nu_1.
\end{align*}

\begin{proof}
	Contraction formula of $H_t^{\epsilon,(\tau_1,\tau_2,g)}$ is trivial from the product formula of Wiener chaoses.
	For example, we see
	\begin{align*}
		X^{\epsilon,\IAABoA}_{(t,x)}
		&=
			X^{\epsilon,\IAAB}_{(t,x)}
			\reso
			X^{\epsilon,\A}_{(t,x)}\\
		&=
			\sum
				[\LPBlock[m_1]X^{\epsilon,\IAAB}_t](x)
				[\LPBlock[m_2]X^{\epsilon,\A}_t](x)\\
		&=
			\sum
				\WienerIntCmpl{2}{1}(\DyaPartOfUnit[m_1]f^{\epsilon,\IAAB}_{(t,x)})
				\WienerIntCmpl{1}{0}(\DyaPartOfUnit[m_2]f^{\epsilon,\A}_{(t,x)})\\
		&=
			\WienerIntCmpl{3}{1}
				(
					f^{\epsilon,(\IAAB,\A,\emptyset)}_{(t,x)}
				)
			+
			\WienerIntCmpl{2}{0}
				(
					f^{\epsilon,(\IAAB,\A,(3,1))}_{(t,x)}
				),
	\end{align*}
	where the summation runs over integers $m_1,m_2\geq -1$ with $|m_1-m_2|\leq 1$.
	In order to obtain contraction formula of $Q_0^{\epsilon,(\tau_1,\tau_2,g)}$, we use Plancherel's formula
	\begin{align*}
		\int_{\RealNum^2}
			f(s)g(t)\delta(s-t)\,
			dsdt
		=
			\int_{\RealNum^2}
				\hat{f}(\sigma)\hat{g}(\tau)\delta(\sigma+\tau)\,
				d\sigma d\tau.
	\end{align*}
	Note that $\mu_i+\nu_{-j}=0$ if and only if $\sigma_i+\tau_j=0,k_i=l_j$. This formula is obtained as follows.
	\begin{align*}
		\int_{\RealNum^2}
			f(s)g(t)\delta(s-t)\,
			dsdt
		&=
			\int_\RealNum
				f(s)\overline{\bar{g}(s)}\,
				ds
		=
			\int_\RealNum
				\hat{f}(\sigma)\overline{\hat{\bar{g}}(\sigma)}\,
				d\sigma\\
		&=
			\int_\RealNum
				\hat{f}(\sigma)\hat{g}(-\sigma)\,
				d\sigma
		=
			\int_{\RealNum^2}
				\hat{f}(\sigma)\hat{g}(\tau)\delta(\sigma+\tau)\,
				d\sigma d\tau.
	\end{align*}
	The proof is completed.
\end{proof}

In \tblref{tbl_contraction_symbols}, we give a list of all contractions $g$ for each $(\tau_1,\tau_2)$
and define the corresponding symbols $(\tau_1,\tau_2,g)$.
Note that the graphs in the same line gives the same kernel $f^{\epsilon,(\tau_1,\tau_2,g)}$,
so we write $(\tau_1,\tau_2,g)$ by the same symbol.
By taking the renormalization into account, we have the following decompositions:
\begin{gather}\label{eq_1481533581}
	\begin{aligned}
		X^{\epsilon,\IAABoA}
		&=
			\WienerIntCmpl{3}{1}(f^{\epsilon,\IAABoA})
			+\WienerIntCmpl{2}{0}(f^{\epsilon,\IAADoC}),\\
		X^{\epsilon,\IAABoB}
		&=
			\WienerIntCmpl{2}{2}(f^{\epsilon,\IAABoB})
			+2\WienerIntCmpl{1}{1}(f^{\epsilon,\IACBoD}),\\
		X^{\epsilon,\IAAoAB}
		&=
			\WienerIntCmpl{3}{1}(f^{\epsilon,\IAAoAB})
			+2\WienerIntCmpl{2}{0}(f^{\epsilon,\IACoAD}),\\
		X^{\epsilon,\IAAoBB}
		&=
			\WienerIntCmpl{2}{2}(f^{\epsilon,\IAAoBB})
			+4\WienerIntCmpl{1}{1}(f^{\epsilon,\IACoBD}),\\
		X^{\epsilon,\IABoAB}
		&=
			\WienerIntCmpl{2}{2}(f^{\epsilon,\IABoAB})
			+\WienerIntCmpl{1}{1}(f^{\epsilon,\ICBoAD})
			+\WienerIntCmpl{1}{1}(f^{\epsilon,\IADoCB}),\\
		X^{\epsilon,\IABoBB}
		&=
			\WienerIntCmpl{1}{3}(f^{\epsilon,\IABoBB})
			+2\WienerIntCmpl{0}{2}(f^{\epsilon,\IADoBD}),\\
		X^{\epsilon,\IAABoAB}
		&=
			\WienerIntCmpl{3}{2}(f^{\epsilon,\IAABoAB})
			+2\WienerIntCmpl{2}{1}(f^{\epsilon,\IACBoAD})
			+\WienerIntCmpl{2}{1}(f^{\epsilon,\IAADoBC})
			+2\WienerIntCmpl{1}{0}(\mathfrak{R}f^{\epsilon,\IACDoCD}),\\
		X^{\epsilon,\IAABoBB}
		&=
			\WienerIntCmpl{2}{3}(f^{\epsilon,\IAABoBB})
			+4\WienerIntCmpl{1}{2}(f^{\epsilon,\IACBoBD})
			+2\WienerIntCmpl{0}{1}(\mathfrak{R}f^{\epsilon,\IBCCoDD})
	\end{aligned}
\end{gather}
where
\begin{align*}
	\mathfrak{R}f^{\epsilon,\IACDoCD}
	&=
		f^{\epsilon,\IACDoCD}
		-\constRe[2,2]{\epsilon} f^{\epsilon,\A},
	&
	\mathfrak{R}f^{\epsilon,\IBCCoDD}
	&=
		f^{\epsilon,\IBCCoDD}
		-\constRe[2,1]{\epsilon} f^{\epsilon,\B}.
\end{align*}
\begin{table}[h]
  \begin{center}
	\caption{List of graphical symbols for contractions}\label{tbl_contraction_symbols}
\begin{tabular}{cccc}\hline

$\tau_1$&$\tau_2$&$g$&$(\tau_1,\tau_2,g)$\\\hline

\multirow{2}{*}{$\bigIAAB$}&$\multirow{2}{*}{$\bigA{3}$}$&$\emptyset$&$\IAABoA$\rule[-1.5mm]{0mm}{5mm}\\\cline{3-4}
&&$\{(3,1)\}$&$\IAADoC$\rule[-1.5mm]{0mm}{5mm}\\\hline

\multirow{2}{*}{$\bigIAAB$}&$\multirow{2}{*}{$\bigB{2}$}$&$\emptyset$&$\IAABoB$\rule[-1.5mm]{0mm}{5mm}\\\cline{3-4}
&&$\{(1,2)\},\{(2,2)\}$&$\IACBoD$\rule[-1.5mm]{0mm}{5mm}\\\hline

\multirow{2}{*}{$\bigIAA$}&$\multirow{2}{*}{$\bigAB{3}{1}$}$&$\emptyset$&$\IAAoAB$\rule[-1.5mm]{0mm}{5mm}\\\cline{3-4}
&&$\{(1,1)\},\{(2,1)\}$&$\IACoAD$\rule[-1.5mm]{0mm}{5mm}\\\hline

\multirow{3}{*}{$\bigIAA$}&$\multirow{3}{*}{$\bigBB{1}{2}$}$&$\emptyset$&$\IAAoBB$\rule[-1.5mm]{0mm}{5mm}\\\cline{3-4}
&&$\{(1,1)\},\{(1,2)\},\{(2,1)\},\{(2,2)\}$&$\IACoBD$\rule[-1.5mm]{0mm}{5mm}\\\cline{3-4}
&&$\{(1,1),(2,2)\},\{(1,2),(2,1)\}$&$\ICCoDD$\rule[-1.5mm]{0mm}{5mm}\\\hline

\multirow{4}{*}{$\bigIAB$}&$\multirow{4}{*}{$\bigAB{2}{2}$}$&$\emptyset$&$\IABoAB$\rule[-1.5mm]{0mm}{5mm}\\\cline{3-4}
&&$\{(1,2)\}$&$\ICBoAD$\rule[-1.5mm]{0mm}{5mm}\\\cline{3-4}
&&$\{(2,1)\}$&$\IADoCB$\rule[-1.5mm]{0mm}{5mm}\\\cline{3-4}
&&$\{(1,2),(2,1)\}$&$\ICDoCD$\rule[-1.5mm]{0mm}{5mm}\\\hline

\multirow{2}{*}{$\bigIAB$}&$\multirow{2}{*}{$\bigBB{2}{3}$}$&$\emptyset$&$\IABoBB$\rule[-1.5mm]{0mm}{5mm}\\\cline{3-4}
&&$\{(1,2)\},\{(1,3)\}$&$\IADoBD$\rule[-1.5mm]{0mm}{5mm}\\\hline

\multirow{4}{*}{$\bigIAAB$}&$\multirow{4}{*}{$\bigAB{3}{2}$}$&$\emptyset$&$\IAABoAB$\rule[-1.5mm]{0mm}{5mm}\\\cline{3-4}
&&$\{(1,2)\},\{(2,2)\}$&$\IACBoAD$\rule[-1.5mm]{0mm}{5mm}\\\cline{3-4}
&&$\{(3,1)\}$&$\IAADoBC$\rule[-1.5mm]{0mm}{5mm}\\\cline{3-4}
&&$\{(1,2),(3,1)\},\{(2,2),(3,1)\}$&$\IACDoCD$\rule[-1.5mm]{0mm}{5mm}\\\hline

\multirow{3}{*}{$\bigIAAB$}&$\multirow{3}{*}{$\bigBB{2}{3}$}$&$\emptyset$&$\IAABoBB$\rule[-1.5mm]{0mm}{5mm}\\\cline{3-4}
&&$\{(1,2)\},\{(1,3)\},\{(2,2)\},\{(2,3)\}$&$\IACBoBD$\rule[-1.5mm]{0mm}{5mm}\\\cline{3-4}
&&$\{(1,2),(2,3)\},\{(1,3),(2,2)\}$&$\IBCCoDD$\rule[-1.5mm]{0mm}{5mm}\\\hline

\end{tabular}
\end{center}
\end{table}

Finally, we define a process $X^\tau$, which is a candidate of the limit of $\{X^{\epsilon,\tau}\}_{0<\epsilon<1}$.
It may be natural to define $H_t^{(\tau_1,\tau_2,g)}$ by the same way as in \pref{section6:expansion of tau1 rs tau2}
by replacing $H_t^{\epsilon,\tau_i}$ by $H_t^{\tau_i}$ for $i=1,2$. This definition makes sense if $\#g=0,1$,
however, does not if $\#g=2$. In Section \ref{section6:estimate of higher order terms},
we will show that there exist kernels $\mathfrak{R}f^\tau$ for $\tau=\IACDoCD,\IBCCoDD$ such that
\begin{align*}
	\mathfrak{R}f^{\epsilon,\IACDoCD}
	&\to
		\mathfrak{R}f^\IACDoCD,
	&
	\mathfrak{R}f^{\epsilon,\IBCCoDD}
	&
		\to\mathfrak{R}f^\IBCCoDD
\end{align*}
as $\epsilon\to 0$.

\begin{definition}
For $\tau\in\{\IAABoA,\IAABoB,\IAAoAB,\IAAoBB,\IABoAB,\IABoBB\}$, we define
\begin{align*}
	X^\tau_{(t,x)}
	=
		\sum_{\#g=0,1}
			\WienerIntCmpl{p-\#g}{q-\#g}(f_{(t,x)}^\tau).
\end{align*}
For $\tau=\IAABoAB,\IAABoBB$, we define
\begin{align*}
	X^\IAABoAB
	&=
		\WienerIntCmpl{3}{2}(f^\IAABoAB)
		+2\WienerIntCmpl{2}{1}(f^\IACBoAD)
		+\WienerIntCmpl{2}{1}(f^\IAADoBC)
		+2\WienerIntCmpl{1}{0}(\mathfrak{R}f^\IACDoCD),\\
	X^\IAABoBB
	&=
		\WienerIntCmpl{2}{3}(f^\IAABoBB)
		+4\WienerIntCmpl{1}{2}(f^\IACBoBD)
		+2\WienerIntCmpl{0}{1}(\mathfrak{R}f^\IBCCoDD).
\end{align*}
\end{definition}

The following is the main  theorem in this section:
\begin{theorem}\label{thm_20161128051824}
	Let $\kappa>0$ and $T>0$.
	Then, we have
	\begin{align*}
		\lim_{\epsilon\downarrow 0}
			\expect[\|X-X^\epsilon\|_\drivers{T}{\kappa}^p]
		=
			0
	\end{align*}
	for every $1<p<\infty$.
\end{theorem}

\begin{remark}
	The limit process $X$ in \tref{thm_20161128051824} is given explicitly by generalized It\^o-Wiener integrals.
	Since the expression of kernels are independent of $\chi$, so is $X$.
\end{remark}

The proof of this theorem will be given in the next section.

\subsection{Proof of convergence of driving vectors}

In this section, we show the convergence $X^{\epsilon,\tau}\to X^\tau$ for all $\tau$.
As stated above, they have the good kernels.
Hence, it is sufficient to estimate $Q_0^\tau$ and $Q_0^\tau-Q_0^{\epsilon,\tau}$,
due to \pref{section6:from Q to X}.

\subsubsection{Useful estimates}

In order to estimate $Q_0^\tau$ and $Q_0^\tau-Q_0^{\epsilon,\tau}$,
we use the following lemmas many times.
\begin{lemma}\label{section6:-a-b to -a-b+5}
If $\alpha,\beta\in(0,5)$ and $\alpha+\beta>5$, we have
$$
\int_E|\mu|_*^{-\alpha}|\nu-\mu|_*^{-\beta}d\mu\lesssim|\nu|_*^{-\alpha-\beta+5}.
$$
\end{lemma}
\begin{proof}
		We modify {\cite[Lemma~9.8]{GubinelliPerkowski2017}} to the three dimensional case.
\end{proof}

\begin{lemma}\label{section6:reduce singularity}
	The function $\psi_\circ$ defined by \eqref{eq_20161206032529} is bounded
	and supported in the set $\{(k,l);C^{-1}|l|_*\le|k|_*\le C|l|_*\}$ for some $C>0$.
	Moreover, we have
	\begin{align*}
		|\psi_\circ(k,l)|\lesssim|k+l|_*^{-\theta}|l|_*^\theta
	\end{align*}
	for every $\theta>0$.
\end{lemma}

\begin{proof}
	The properties $|\psi_\circ(k,l)|\le1$ and $\psi_\circ(k,l)>0$ imply $|k|_*\approx|l|_*$ are trivial.
	We show the last property.
	Since if $(k,l)\in\supp(\psi_\circ)$ then $|l|_*/|k|_*\gtrsim1$ and $|k+l|_*\le|k|_*+|l|_*\lesssim|k|_*$,
	we have
	\begin{align*}
		|\psi_\circ(k,l)|\lesssim|k|_*^{-\theta}|l|_*^\theta\lesssim|k+l|_*^{-\theta}|l|_*^\theta
	\end{align*}
	for every $\theta>0$.
\end{proof}

\subsubsection{Lower order terms}

Now we consider $X^\tau$ for $\tau=\A$, $\B$, $\AA$, $\AB$, $\BB$, $\AAB$, $\IAA$, $\IAB$, $\IAAB$.

\begin{proposition}\label{prop_1481179726000}
For $\tau=\A$, $\B$, $\AA$, $\AB$, $\BB$, $\AAB$, $\IAA$, $\IAB$, $\IAAB$, we have
\begin{gather*}
	\int_{\mu_{[1\dots p]}+\nu_{[(-1)\dots(-q)]}=\mu}
		|
			Q_0^\tau(\mu_{1,\dots,p},\nu_{1,\dots,q})
		|^2
	\lesssim
		|\mu|_*^{-2\gamma_\tau},\\
	\int_{\mu_{[1\dots p]}+\nu_{[(-1)\dots(-q)]}=\mu}
		|(Q_0^\tau-Q_0^{\epsilon,\tau})(\mu_{1,\dots,p},\nu_{1,\dots,q})|^2
	\lesssim
		\epsilon^\lambda
		|\mu|_*^{-2\gamma_\tau+\lambda},
\end{gather*}
for every $\lambda\in(0,1]$, where $\gamma_\tau=2\,(\tau=\A,\B), \frac32\,(\tau=\AA,\AB,\BB),1\,(\tau=\AAB),\frac72\,(\tau=\IAA,\IAB),3\,(\tau=\IAAB)$.
\end{proposition}

\begin{proof}
	For $\tau=\A,\B$, the required estimates is trivial from $|Q(\mu)|\lesssim|\mu|_*^{-2}$. Indeed,
	\begin{align*}
		|Q_0^\A(\mu_1)|^2
		&=
			|Q(\mu_1)|^2\lesssim|\mu_1|_*^{-4},
		&
		|Q_0^\B(\nu_1)|^2
		&=
			|Q(\nu_{-1})|^2\lesssim|\nu_1|_*^{-4}.
	\end{align*}
	For $\tau=\AA$, from Lemma \ref{section6:-a-b to -a-b+5} we have
	\begin{align*}
	\int_{\mu_{[12]}=\mu}|Q_0^\AA(\mu_{1,2})|^2\lesssim\int_{\mu_{[12]}=\mu}|\mu_1|_*^{-4}|\mu_2|_*^{-4}\lesssim|\mu|_*^{-3}.
	\end{align*}
	The case $\tau=\AB,\BB$ are parallel. For $\tau=\AAB$, we have
	\begin{align*}
		\int_{\mu_{[12]}+\nu_{-1}=\mu}
			|Q_0^\AAB(\mu_{1,2},\nu_1)|^2
		\lesssim
			\int_{\mu_{[12]}+\nu_{-1}=\mu}
				|\mu_1|_*^{-4}|
				\mu_2|_*^{-4}
				|\nu_{-1}|_*^{-4}
		\lesssim
			|\mu|_*^{-2}.
	\end{align*}
	For $\tau=\IAA,\IAB,\IAAB$, we have
	\begin{multline*}
		\int_{\mu_{[1\dots p]}+\nu_{[(-1)\dots(-q)]}=\mu}
			|Q_0^\tau(\mu_{1,\dots,p},\nu_{1,\dots,q})|^2\\
		\begin{aligned}
			&\lesssim
				|\mu|_*^{-4}
				\int_{\mu_{[1\dots p]}+\nu_{[(-1)\dots(-q)]}=\mu}
					|Q_0^{\tau_0}(\mu_{1,\dots,p},\nu_{1,\dots,q})|^2\\
			&\lesssim
				\begin{cases}
					|\mu|_*^{-7},	&\tau=\IAA,\IAB,\\
					|\mu|_*^{-6},	&\tau=\IAAB.
				\end{cases}
		\end{aligned}
	\end{multline*}
	Here, we used \pref{prop_20161205074106} (2).
	The required estimates of $Q_0^\tau-Q_0^{\epsilon,\tau}$ is obtained by similar computations
	by using $Q_0^\tau-Q_0^{\epsilon,\tau}=(1-\chi^\epsilon)Q_0^\tau$ and the inequality
	\begin{multline}\label{eq_20161206035717}
		|1-\chi^\epsilon(k_{1,\dots,p},l_{1,\dots,q})|\\
		\lesssim
			\epsilon^\lambda
			\left(
				\sum_{i=1}^p
					|k_i|^\lambda
				+
				\sum_{j=1}^q
					|l_j|^\lambda
			\right)
		\lesssim
			\epsilon^\lambda
			\left(
				\sum_{i=1}^p
					|\mu_i|_*^\lambda
				+
				\sum_{j=1}^q
					|\nu_j|_*^\lambda
			\right)
	\end{multline}
	for every $\lambda\in(0,1]$.
\end{proof}

\begin{proof}[Proof of \tref{thm_20161128051824} for $\A$, $\AA$, $\AB$, $\IAA$, $\IAB$, $\IAAB$]
	\pref[section6:from Q to X]{prop_1481179726000} imply the conclusion.
	Note that we need to prove $X^\IAAB$ is H\"older continuous in time.
\end{proof}

\subsubsection{Higher order terms}\label{section6:estimate of higher order terms}

Now we consider $X^\tau$ for $\tau=\IAABoA$, $\IAABoB$, $\IAAoAB$, $\IAAoBB$, $\IABoAB$, $\IABoBB$, $\IAABoAB$, $\IAABoBB$.
We define $(\tau_1,\tau_2)$ for each $\tau$ as in \pref{section6:expansion of tau1 rs tau2}.
We note that $X^\tau$ is written as a sum of It\^o-Wiener integrals
which have good kernels $f^{(\tau_1,\tau_2,g)}$ for $\#g=0,1$ and $\mathfrak{R}f^{(\tau_1,\tau_2,g)}$ for $\#g=2$
such as \eqref{eq_1481533581}.
We will estimate these functions for the case $\#g=0,1,2$ separately.

First we consider the functions $Q_0^\tau=Q_0^{(\tau_1,\tau_2,\emptyset)}$.

\begin{proposition}\label{prop_1481179894}
For $\tau=\IAABoA$, $\IAABoB$, $\IAAoAB$, $\IAAoBB$, $\IABoAB$, $\IABoBB$, $\IAABoAB$, $\IAABoBB$, we have
\begin{gather*}
	\int_{\mu_{[1\dots p]}+\nu_{[(-1)\dots(-q)]}=\mu}
		|
			Q_0^\tau(\mu_{1,\dots,p},\nu_{1,\dots,q})
		|^2
	\lesssim
		|\mu|_*^{-2\gamma_\tau+\kappa}
		|k|_*^{-2\delta_\tau-\kappa},\\
	\int_{\mu_{[1\dots p]}+\nu_{[(-1)\dots(-q)]}=\mu}
		|
			(Q_0^\tau-Q_0^{\epsilon,\tau})(\mu_{1,\dots,p},\nu_{1,\dots,q})
		|^2
	\lesssim
		\epsilon^\lambda
		|\mu|_*^{-2\gamma_\tau+\kappa+\lambda}
		|k|_*^{-2\delta_\tau-\kappa+\lambda},
\end{gather*}
for every small $\kappa>0$ and $\lambda>0$, where
\begin{align*}
	(\gamma_\tau,\delta_\tau)
	=
		\begin{cases}
			(2,\frac{1}{2}),	&\tau=\IAABoA,\IAABoB,\\
			(\frac{3}{2},1),	&\tau=\IAAoAB,\IAAoBB,\IABoAB,\IABoBB,\\
			(\frac{3}{2},\frac{1}{2}),	&\tau=\IAABoAB,\IAABoBB.
		\end{cases}
	\end{align*}
\end{proposition}
\begin{proof}
	We have
	\begin{multline*}
		\int_{\mu_{[1\dots p]}+\nu_{[(-1)\dots(-q)]}=\mu}
			|Q_0^\tau(\mu_{1,\dots,p},\nu_{1,\dots,q})|^2\\
		\begin{aligned}
			&
				\lesssim
					\int_{\mu_1'+\mu_2'=\mu}
						\psi_\circ(k_1',k_2')^2
						\int_{\mu_{[1\dots p_1]}+\nu_{[(-1)\dots(-q_1)]}=\mu_1'}
							|Q_0^{\tau_1}(\mu_{1,\dots,p_1},\nu_{1,\dots,q_1})|^2\\
			&\qquad
						\times
						\int_{\mu_{[(p_1+1)\dots p]}+\nu_{[(-q_1-1)\dots(-q)]}=\mu_1'}
							|Q_0^{\tau_2}(\mu_{(p_1+1),\dots,p},\nu_{(q_1+1),\dots,q})|^2\\
			&\lesssim
				\int_{\mu_1'+\mu_2'=\mu}
					\psi_\circ(k_1',k_2')^2
					|\mu_1'|_*^{-2\gamma_{\tau_1}}
					|\mu_2'|_*^{-2\gamma_{\tau_2}}.
		\end{aligned}
	\end{multline*}
	We estimate them by using \pref{prop_1481179726000} and \lref{section6:reduce singularity}.
	In the case $\tau=\IAABoA,\IAABoB$, \pref{prop_1481179726000} implies $(\gamma_{\tau_1},\gamma_{\tau_2})=(3,2)$.
	Applying \lref{section6:reduce singularity}, we have
	\begin{align*}
		\psi_\circ(k_1',k_2')^2
		|\mu_1'|_\ast^{-6}
		|\mu_2'|_\ast^{-4}
		&\lesssim
			|k'_1+k'_2|_\ast^{-1-\kappa}
			|k'_1|_\ast^{1+\kappa}
			|\mu_1'|_\ast^{-6}
			|\mu_2'|_\ast^{-4}\\
		&\leq
			|k|_\ast^{-1-\kappa}
			|\mu_1'|_\ast^{-5+\kappa}
			|\mu_2'|_\ast^{-4}.
	\end{align*}
	for any $\mu_1'+\mu_2'=\mu$.
	Hence
	\begin{align*}
		\int_{\mu_1'+\mu_2'=\mu}
			\psi_\circ(k_1',k_2')^2
			|\mu_1'|_*^{-6}
			|\mu_2'|_*^{-4}
		&\lesssim
			|k|_*^{-1-\kappa}
			\int_{\mu_1'+\mu_2'=\mu}
			|\mu_1|_*^{-5+\kappa}
			|\mu_2|_*^{-4}\\
		&\lesssim
			|k|_*^{-1-\kappa}
			|\mu|_*^{-4+\kappa}.
	\end{align*}
	For $\tau=\IAAoAB$, $\IAAoBB$, $\IABoAB$, $\IABoBB$,
	$(\gamma_{\tau_1},\gamma_{\tau_2})=(\frac72,\frac32)$ implies
	\begin{align*}
		\int_{\mu_1'+\mu_2'=\mu}
			\psi_\circ(k_1',k_2')^2
			|\mu_1'|_*^{-7}
			|\mu_2'|_*^{-3}
		&\lesssim
			|k|_*^{-2-\kappa}
			\int_{\mu_1'+\mu_2'=\mu}
				|\mu_1|_*^{-5+\kappa}
				|\mu_2|_*^{-3}\\
		&\lesssim
			|k|_*^{-2-\kappa}
			|\mu|_*^{-3+\kappa}.
	\end{align*}
	For $\tau=\IAABoAB$, $\IAABoBB$, $(\gamma_{\tau_1},\gamma_{\tau_2})=(3,\frac32)$ implies
	\begin{align*}
		\int_{\mu_1'+\mu_2'=\mu}
			\psi_\circ(k_1',k_2')^2
			|\mu_1'|_*^{-6}
			|\mu_2'|_*^{-3}
		&\lesssim
			|k|_*^{-1-\kappa}
			\int_{\mu_1'+\mu_2'=\mu}
				|\mu_1|_*^{-5+\kappa}
				|\mu_2|_*^{-3}\\
		&\lesssim
			|k|_*^{-1-\kappa}
			|\mu|_*^{-3+\kappa}.
	\end{align*}
	By noting \eqref{eq_20161206035717}, we can estimate $Q_0^\tau-Q_0^{\epsilon,\tau}$ in a similar way.
\end{proof}

Next we consider the functions $Q_0^\tau$
for $\tau=\IAADoC$, $\IACBoD$, $\IACoAD$, $\IACoBD$, $\ICBoAD$, $\IADoCB$, $\IADoBD$, $\IACBoAD$, $\IAADoBC$, $\IACBoBD$.
We show three propositions; \prefs{prop_1481179916}{prop_1481179943}{prop_1481179962}.
\begin{proposition}\label{prop_1481179916}
For $\tau=\IAADoC$, $\IACBoD$, we have
\begin{gather*}
	\int_{\mu_{[1\dots p]}+\nu_{[(-1)\dots(-q)]}=\mu}
		|Q_0^\tau(\mu_{1,\dots,p},\nu_{1,\dots,q})|^2
	\lesssim
		|\mu|_*^{-5},\\
	\int_{\mu_{[1\dots p]}+\nu_{[(-1)\dots(-q)]}=\mu}
		|(Q_0^\tau-Q_0^{\epsilon,\tau})(\mu_{1,\dots,p},\nu_{1,\dots,q})|^2
	\lesssim
		\epsilon^\lambda
		|\mu|_*^{-5+\lambda},
\end{gather*}
for every small $\lambda>0$.
\end{proposition}

\begin{proof}
	We consider the case $\tau=\IAADoC$. Note
	\begin{align*}
		Q_0^\IAADoC(\mu_{1,2})
		&=
			\int_{E^2}
				\psi_\circ(k_{[12]}-l_1,k_3)
				Q(\mu_{[12]}+\nu_{-1})\\
		&\phantom{=}\quad\qquad
				\times
				Q(\mu_1)
				Q(\mu_2)
				\CmplConj{Q(-\nu_{-1})}
				Q(\mu_3)\,
				\delta(\mu_3+\nu_{-1})\,
				d\mu_3d\nu_1\\
		&=
			Q(\mu_1)
			Q(\mu_2)
			\int_E
				\psi_\circ(k_{[12]}-k_3,k_3)
				Q(\mu_{[12]}-\mu_3)
				|Q(\mu_3)|^2\,
				d\mu_3.
	\end{align*}
	By using the estimate $|\psi_\circ(k_{[12]}-k_3,k_3)|\leq 1$,
	we have
	\begin{align*}
		|Q_0^\IAADoC(\mu_{1,2})|
		\lesssim
			|\mu_1|_*^{-2}
			|\mu_2|_*^{-2}
			\int_E
				|\mu_{[12]}-\mu_3|_*^{-2}
				|\mu_3|_*^{-4}\,
				d\mu_3
		\lesssim
			|\mu_1|_*^{-2}
			|\mu_2|_*^{-2}
			|\mu_{[12]}|_*^{-1}.
	\end{align*}
	Hence
	\begin{align*}
		\int_{\mu_{[12]}=\mu}
			|Q_0^\IAADoC(\mu_{1,2})|^2
		\lesssim
			|\mu|_*^{-2}
			\int_{\mu_{[12]}=\mu}
				|\mu_1|_*^{-4}
				|\mu_2|_*^{-4}
		\lesssim
			|\mu|_*^{-5}.
	\end{align*}

	We will show the second inequality for $\tau=\IAADoC$.
	The inequality \eqref{eq_20161206035717} implies
	\begin{multline*}
		|Q_0^\IAADoC(\mu_{1,2})-Q_0^{\epsilon,\IAADoC}(\mu_{1,2})|
		\lesssim
			|\mu_1|_\ast^{-2}
			|\mu_2|_\ast^{-2}
			\cdot
			\epsilon^{\lambda/2}
			\{
				|\mu_1|_\ast^\lambda
				+|\mu_2|_\ast^\lambda
				+|\mu_{[12]}|_\ast^\lambda
			\}^{1/2}
			|\mu_{[12]}|_\ast^{-1}.
	\end{multline*}
	Hence we obtain the second inequality for $\tau=\IAADoC$.

	The assertion for $\IACBoD$ is verified in the same way.
\end{proof}

\begin{proposition}\label{prop_1481179943}
	For $\tau=\IACoAD,\IACoBD,\ICBoAD,\IADoCB,\IADoBD$, we have
	\begin{gather*}
		\int_{\mu_{[1\dots p]}+\nu_{[(-1)\dots(-q)]}=\mu}
			|Q_0^\tau(\mu_{1,\dots,p},\nu_{1,\dots,q})|^2
		\lesssim
			|\mu|_*^{-4+\kappa}|k|_*^{-1-\kappa},\\
		\int_{\mu_{[1\dots p]}+\nu_{[(-1)\dots(-q)]}=\mu}
			|(Q_0^\tau-Q_0^{\epsilon,\tau})(\mu_{1,\dots,p},\nu_{1,\dots,q})|^2
		\lesssim
			\epsilon^\lambda
			|\mu|_*^{-4+\kappa+\lambda}
			|k|_*^{-1-\kappa},
	\end{gather*}
	for every small $\kappa>0$ and $\lambda>0$.
\end{proposition}

\begin{proof}
	We give a proof for the case $\tau=\IACoAD$ only, because we can show the other cases in the same way.
	Note
	\begin{align*}
		Q_0^\IACoAD(\mu_{1,3})
		&=
			\int_{E^2}\psi_\circ(k_{[12]},k_3-l_1)
				Q(\mu_{[12]})\\
		&\phantom{=}\quad\qquad
				\times
				Q(\mu_1)
				Q(\mu_2)
				Q(\mu_3)
				\CmplConj{Q(-\nu_{-1})}\,
				\delta(\mu_2+\nu_{-1})\,
				d\mu_2d\nu_1\\
		&=
			Q(\mu_1)
			Q(\mu_3)
			\int_E
				\psi_\circ(k_1+k_2,k_3-k_2)
				Q(\mu_{[12]})
				|Q(\mu_2)|^2\,
				d\mu_2.
	\end{align*}
	\lref{section6:reduce singularity} implies
	\begin{align*}
		\psi_\circ(k_1+k_2,k_3-k_2)
		\lesssim
			|k_1+k_3|_\ast^{-\frac{1+\kappa}{2}}
			|k_1+k_2|_\ast^{\frac{1+\kappa}{2}}
		\leq
			|k_1+k_3|_\ast^{-\frac{1+\kappa}{2}}
			|\mu_{[12]}|_\ast^{\frac{1+\kappa}{2}}.
	\end{align*}
	Combining them, we have
	\begin{align*}
		|Q_0^\IACoAD(\mu_{1,3})|
		&\lesssim
			|\mu_1|_*^{-2}
			|\mu_3|_*^{-2}
			|k_1+k_3|_\ast^{-\frac{1+\kappa}2}
			\int_E
				|\mu_{[12]}|_*^{-\frac{3-\kappa}{2}}
				|\mu_2|_*^{-4}\,
				d\mu_2\\
		&\lesssim
			|\mu_1|_*^{-2}
			|\mu_3|_*^{-2}
			|k_1+k_3|_\ast^{-\frac{1+\kappa}{2}}
			|\mu_1|_*^{-\frac{1-\kappa}{2}}
	\end{align*}
	Hence
	\begin{align*}
		\int_{\mu_{[13]}=\mu}|Q_0^\IACoAD(\mu_{1,3})|^2
		\lesssim
			|k|_*^{-1-\kappa}
			\int_{\mu_{[13]}=\mu}
				|\mu_1|_*^{-5+\kappa}
				|\mu_3|_*^{-4}
		\lesssim
			|k|_*^{-1-\kappa}
			|\mu|_*^{-4+\kappa}.
	\end{align*}
	In a similar way, we see
	\begin{multline*}
		|Q_0^\IACoAD(\mu_{1,3})-Q_0^{\epsilon,\IACoAD}(\mu_{1,3})|\\
		\lesssim
			|\mu_1|_*^{-2}
			|\mu_3|_*^{-2}
			|k_1+k_3|_\ast^{-\frac{1+\kappa}{2}}
			\cdot
			\epsilon^{\lambda/2}
			\{
				|\mu_1|_\ast^\lambda
				+|\mu_2|_\ast^\lambda
				+|\mu_{[12]}|_\ast^\lambda
			\}^{1/2}
			|\mu_1|_\ast^{-\frac{1-\kappa}{2}},
	\end{multline*}
	which implies the second assertion.
	The proof has been completed.
\end{proof}

\begin{proposition}\label{prop_1481179962}
	For $\tau=\IACBoAD$, $\IAADoBC$, $\IACBoBD$, we have
	\begin{gather*}
		\int_{\mu_{[1\dots p]}+\nu_{[(-1)\dots(-q)]}=\mu}
			|Q_0^\tau(\mu_{1,\dots,p},\nu_{1,\dots,q})|^2
		\lesssim
			|\mu|_*^{-4+\kappa}
			|k|_*^{-\kappa},\\
		\int_{\mu_{[1\dots p]}+\nu_{[(-1)\dots(-q)]}=\mu}
			|(Q_0^\tau-Q_0^{\epsilon,\tau})(\mu_{1,\dots,p},\nu_{1,\dots,q})|^2
		\lesssim
			\epsilon^\lambda
			|\mu|_*^{-4+\kappa+\lambda}
			|k|_*^{-\kappa}
	\end{gather*}
	for every small $\kappa>0$ and $\lambda>0$.
\end{proposition}

\begin{proof}
	Here, we will show the assertion for $\tau=\IACBoAD$ only.
	Note
	\begin{align*}
		Q_0^\IACBoAD(\mu_{1,3},\nu_1)
		&=
			\int_{E^2}
				\psi_\circ(k_{[12]}-l_1,k_3-l_2)
				Q(\mu_{[12]}+\nu_{-1})
				Q(\mu_1)
				Q(\mu_2)
				\CmplConj{Q(-\nu_{-1})}\\
		&\qquad\times
				Q(\mu_3)
				\CmplConj{Q(-\nu_{-2})}\,
				\delta(\mu_2+\nu_{-2})\,
				d\mu_2d\nu_2\\
		&=
			Q(\mu_1)
			Q(\mu_3)
			\CmplConj{Q(-\nu_{-1})}\\
		&\qquad\times
			\int_E
				\psi_\circ(k_{[12]}-l_1,k_3-k_2)
				Q(\mu_{[12]}+\nu_{-1})
				|Q(\mu_2)|^2\,
				d\mu_2.
	\end{align*}
	We use \lref{section6:reduce singularity} to obtain
	\begin{align*}
		\psi_\circ(k_{[12]}-l_1,k_3-k_2)
		\lesssim
			|k_{[13]}-l_1|_\ast^{-\frac{\kappa}{2}}
			|k_{[12]}-l_1|_\ast^{\frac{\kappa}{2}}
		\lesssim
			|k_{[13]}-l_1|_\ast^{-\frac{\kappa}{2}}
			|\mu_{[12]}+\nu_{-1}|_\ast^{\frac{\kappa}{2}}.
	\end{align*}
	Hence
	\begin{multline*}
		|Q_0^\IACBoAD(\mu_{1,3},\nu_1)|\\
		\begin{aligned}
			&\lesssim
				|\mu_1|_*^{-2}
				|\mu_3|_*^{-2}
				|\nu_{-1}|_*^{-2}
				|k_{[13]}-l_1|_*^{-\frac{\kappa}2}
				\int_E
					|\mu_{[12]}+\nu_{-1}|_*^{-2+\frac{\kappa}2}
					|\mu_2|_*^{-4}\,
					d\mu_2\\
			&\lesssim
				|\mu_1|_*^{-2}
				|\mu_3|_*^{-2}
				|\nu_{-1}|_*^{-2}
				|k_{[13]}-l_1|_*^{-\frac{\kappa}2}
				|\mu_1+\nu_{-1}|_*^{-1+\frac{\kappa}2},
		\end{aligned}
	\end{multline*}
	which implies
	\begin{multline*}
		\int_{\mu_{[13]}+\nu_{-1}=\mu}
			|Q_0^\IACBoAD(\mu_{1,3},\nu_1)|^2\\
		\begin{aligned}
			&\lesssim
				|k|_*^{-\kappa}
				\int_{\mu_{[13]}+\nu_{-1}=\mu}
					|\mu_1|_*^{-4}
					|\mu_3|_*^{-4}
					|\nu_{-1}|_*^{-4}
					|\mu_1+\nu_{-1}|_*^{-2+\kappa}\\
			&\lesssim|k|_*^{-\kappa}|\mu|_*^{-4+\kappa}.
		\end{aligned}
	\end{multline*}
	In addition, we have
	\begin{multline*}
		|Q_0^\IACBoAD(\mu_{1,3},\nu_1)-Q_0^{\epsilon,\IACBoAD}(\mu_{1,3},\nu_1)|
		\lesssim
			|\mu_1|_*^{-2}
			|\mu_3|_*^{-2}
			|\nu_{-1}|_*^{-2}
			|k_{[13]}-l_1|_*^{-\frac{\kappa}2}\\
			\times
			\epsilon^{\lambda/2}
			\{|\mu_1|_\ast^\lambda+|\mu_3|_\ast^\lambda+|\nu_{-1}|_\ast^\lambda+|\mu_1+\nu_{-1}|_\ast^\lambda\}^{1/2}
			|\mu_1+\nu_{-1}|_*^{-1+\frac{\kappa}2},
	\end{multline*}
	which implies the conclusion.
	The proof is completed.
\end{proof}

Finally we consider the functions $\mathfrak{R}Q_0^\tau$ for $\tau=\IACDoCD,\IBCCoDD$.
First of all, we have to define the renormalized kernels $\mathfrak{R}f^\tau$. Since
\begin{align*}
	\mathfrak{R}Q_0^{\epsilon,\IACDoCD}
	&=
		Q_0^{\epsilon,\IACDoCD}
		-Q_0^{\epsilon,\ICDoCD}Q_0^{\epsilon,\A},\\
	\mathfrak{R}Q_0^{\epsilon,\IBCCoDD}
	&=
		Q_0^{\epsilon,\IBCCoDD}
		-Q_0^{\epsilon,\ICCoDD}Q_0^{\epsilon,\B},
\end{align*}
we have
\begin{align*}
	\mathfrak{R}Q_0^{\epsilon,\IACDoCD}(\mu_1)
	&=
		\chi^\epsilon(k_1)
		Q(\mu_1)
		\int_{E^2}
			\chi^\epsilon(k_{2,3})^2
			|Q(\mu_2)|^2
			|Q(\mu_3)|^2\\
	&\phantom{=}\qquad
			\times
			\delta_{0,\mu_1}
				\{
					\psi_\circ(\cdot+k_2-k_3,k_3-k_2)
					Q(\cdot+\mu_2-\mu_3)
				\}\,
			d\mu_2
			d\mu_3,\\
	\mathfrak{R}Q_0^{\epsilon,\IBCCoDD}(\nu_1)
	&=
		\chi^\epsilon(k_1)
		\CmplConj{Q(-\nu_{-1})}
		\int_{E^2}
			\chi^\epsilon(k_{1,2})^2
			|Q(\mu_1)|^2
			|Q(\mu_2)|^2\\
	&\phantom{=}\qquad
			\times
			\delta_{0,-\nu_{-1}}
				\{
					\psi_\circ(\cdot+k_{[12]},-k_{[12]})
					Q(\cdot+\mu_{[12]})
				\}\,
			d\mu_1
			d\mu_2,
\end{align*}
where $\delta_{\mu_1,\mu_2}f=f(\mu_2)-f(\mu_1)$ is the difference operator. We set
\begin{align*}
	\mathfrak{R}Q_0^{\IACDoCD}(\mu_1)
	&=
		Q(\mu_1)
		\int_{E^2}
			|Q(\mu_2)|^2
			|Q(\mu_3)|^2\\
	&\phantom{=}\qquad
			\times
			\delta_{0,\mu_1}
				\{
					\psi_\circ(\cdot+k_2-k_3,k_3-k_2)
					Q(\cdot+\mu_2-\mu_3)
				\}\,
			d\mu_2
			d\mu_3,\\
	\mathfrak{R}Q_0^{\IBCCoDD}(\nu_1)
	&=
		\CmplConj{Q(-\nu_{-1})}
		\int_{E^2}
			|Q(\mu_1)|^2
			|Q(\mu_2)|^2\\
	&\phantom{=}\qquad
			\times
			\delta_{0,-\nu_{-1}}
				\{
					\psi_\circ(\cdot+k_{[12]},-k_{[12]})
					Q(\cdot+\mu_{[12]})
				\}\,
			d\mu_1
			d\mu_2,
\end{align*}
if they are well-defined. The required kernels $\mathfrak{R}f^\tau$ is defined by
good kernels with $\mathfrak{R}H_t^\tau:=\FTTime^{-1}(e^{-2\pi\ImUnit Rt}Q_0^\tau)$.
The following proposition implies that these kernels are well-defined and $\mathfrak{R}f^{\epsilon,\tau}$ converges to $\mathfrak{R}f^\tau$.

\begin{proposition}\label{prop_1481534768}
For $\tau=\IACDoCD,\IBCCoDD$, we have
\begin{gather*}
	|\mathfrak{R}Q_0^\tau(\mu)|^2
	\lesssim
		|\mu|_*^{-4+\kappa},\\
	|(\mathfrak{R}Q_0^\tau-\mathfrak{R}Q_0^{\epsilon,\tau})(\mu)|^2
	\lesssim
		\epsilon^\lambda
		|\mu|_*^{-4+\kappa+\lambda}
\end{gather*}
for every small $\kappa>0$ and $\lambda>0$.
\end{proposition}

In order to prove this proposition,
we extend the domains of $\psi_\circ=\psi_\circ(k,l)$ and $Q=Q(\sigma,k)$ into $\RealNum^6$ and $\RealNum^4$
in a natural way, respectively.
We write $k=(k^1,k^2,k^3)\in\RealNum^3$.
Then, we have the following estimate of their derivatives:
\begin{lemma}\label{lem_1481250926}
	For every $0<\kappa<1$, it holds that
	\begin{align*}
		|\partial_{k^\alpha}\psi_\circ(k,l)|
		&\lesssim
			\indicator{|k|\approx|l|}
			|k|_\ast^{-1+\kappa},
		&
		|\partial_\sigma Q(\mu)|
		&\lesssim
			|\mu|_*^{-4},
		&
		|\partial_{k^\alpha}Q(\mu)|
		&\lesssim
			|\mu|_*^{-3}.
	\end{align*}
\end{lemma}
\begin{proof}
	The latter two inequality can be shown easily. We show the first inequality.
	Note that $(k,l)\in\supp\psi_\circ$ implies $|k|\approx|l|$.
	We see
	\begin{align*}
		\partial_{k^\alpha}\psi_\circ(k,l)
		=
			\sum_{|i-j|\leq 1}
				2^{-i}
				\partial_{k^\alpha}
				\DyaPartOfUnit[0](2^{-i}k)
				\DyaPartOfUnit[j](l)
		=
			\sum_{i\geq -1}
				2^{-i}
				\partial_{k^\alpha}
				\DyaPartOfUnit[0](2^{-i}k)
				\sum_{j;|i-j|\leq 1}
					\DyaPartOfUnit[j](l).
	\end{align*}
	In this calculation, we abused the symbols $\DyaPartOfUnit[-1]$ and $\DyaPartOfUnit[0]$.
	We see that the compactness of $\supp\DyaPartOfUnit[0]$ implies
	$
		|
			\partial_{k^\alpha}
			\DyaPartOfUnit[0](2^{-i}k)
		|
		\lesssim
			|2^{-i}k|_\ast^{-1+\kappa}
		\leq
			2^{-i(1-\kappa)}
			|k|_\ast^{-1+\kappa}
	$
	and
	the summation
	$
		\sum_{j;|i-j|\leq 1}
		\DyaPartOfUnit[j](l)
	$
	has an upper bound independent of $i$.
	Hence,
	\begin{align*}
		|\partial_{k^\alpha}\psi_\circ(k,l)|
		\lesssim
			\sum_{|i-j|\leq 1}
				2^{-i(2-\kappa)}
				|k|_\ast^{-1+\kappa}
		\lesssim
			|k|_\ast^{-1+\kappa}.
	\end{align*}
	The proof is completed.
\end{proof}

\begin{proof}[Proof of \pref{prop_1481534768}]
	We focus on $\tau=\IACDoCD$.

	First we esimate  the difference operator part in $\mathfrak{R}Q_0^\IACDoCD$.
	We write $\tau^s\mu=(\tau^2\sigma,\tau k)$ for $\mu\in E$ and $\tau\in[0,1]$.
	The fundamental theorem of calculus and \lref{lem_1481250926} imply
	\begin{multline*}
		|\delta_{0,\mu_1}\{\psi_\circ(\cdot+k_2,-k_2)Q(\cdot+\mu_2)\}|
		=
			\left|
				\int_0^1
					\frac{d}{d\tau}
						\{\psi_\circ(\cdot,-k_2)Q\}
					(\tau^s\mu_1+\mu_2)\,
					d\tau
			\right|\\
		\lesssim
			|k_1|
			\int_0^1
				\indicator{|\tau k_1+k_2|\approx|k_2|}
				|\tau k_1+k_2|_\ast^{-1+\kappa}
				|\tau^s\mu_1+\mu_2|_*^{-2}\,
				d\tau\\
			+
			|\sigma_1|
			\int_0^1
				\tau
				|\tau^s\mu_1+\mu_2|_*^{-4}\,
				d\tau
			+
			|k_1|
			\int_0^1
				|\tau^s\mu_1+\mu_2|_*^{-3}\,
				d\tau.
	\end{multline*}
	Hence we have
	\begin{multline*}
		\int_{E^2}
			|Q(\mu_2)|^2
			|Q(\mu_3)|^2
			|\delta_{0,\mu_1}\{\psi_\circ(\cdot+k_2-k_3,k_3-k_2)Q(\cdot+\mu_2-\mu_3)\}|\,
			d\mu_2d\mu_3\\
		\lesssim
			A_1+A_2+A_3,
	\end{multline*}
	where
	\begin{align*}
		A_1
		&=
			\int_{E^2}
				d\mu_2d\mu_3\,
				|\mu_2|_*^{-4}
				|\mu_3|_*^{-4}
				|k_1|
				\int_0^1
					\indicator{|\tau k_1+k_2-k_3|\approx|k_2-k_3|}\\
			&\phantom{=}\quad\qquad\qquad\qquad
					\times
					|\tau k_1+k_2-k_3|_\ast^{-1+\kappa}
					|\tau^s\mu_1+\mu_2-\mu_3|_*^{-2}\,
					d\tau,\\
		A_2
		&=
			\int_{E^2}
				d\mu_2d\mu_3\,
				|\mu_2|_*^{-4}
				|\mu_3|_*^{-4}
				|\sigma_1|
				\int_0^1
					\tau
					|\tau^s\mu_1+\mu_2-\mu_3|_*^{-4}\,
					d\tau,\\
		A_3
		&=
			\int_{E^2}
				d\mu_2d\mu_3\,
				|\mu_2|_*^{-4}
				|\mu_3|_*^{-4}
				|k_1|
				\int_0^1
					|\tau^s\mu_1+\mu_2-\mu_3|_*^{-3}\,
					d\tau.
	\end{align*}

	We estimate the terms $A_1$, $A_2$ and $A_3$.
	Note that \lref{section6:-a-b to -a-b+5} holds even if $\nu\in\RealNum^4$.
	We start the estimates with $A_1$.
	By changing variables with $\mu'_2=\mu_2$ and $\mu'_3=\mu_2-\mu_3$
	and the Fubini theorem, we have
	\begin{align*}
		A_1
		&=
			|k_1|
			\int_{E^2}
				d\mu'_2
				d\mu'_3\,
				|\mu'_2|_\ast^{-4}
				|\mu'_2-\mu'_3|_\ast^{-3}\\
		&\phantom{=}\quad\qquad
				\times
				\left(
					\int_0^1
						\indicator{|\tau k_1+k'_3|\approx|k'_3|}
						|\tau k_1+k'_3|_\ast^{-1+\kappa}
						|\tau^s\mu_1+\mu'_3|_\ast^{-2}\,
						d\tau
				\right)\\
		&\lesssim
			|k_1|
			\int_E
				d\mu'_3\,
				|\mu'_3|^{-3}\,
				\left(
					\int_0^1
						\indicator{|\tau k_1+k'_3|\approx|k'_3|}
						|\tau k_1+k'_3|_\ast^{-1+\kappa}
						|\tau^s\mu_1+\mu'_3|_\ast^{-2}\,
						d\tau
				\right)\\
		&=
			|k_1|
			\int_0^1
				d\tau\,
				\int_E
					d\mu'_3\,
					\indicator{|\tau k_1+k'_3|\approx|k'_3|}
					|\tau k_1+k'_3|_\ast^{-1+\kappa}
					|\tau^s\mu_1+\mu'_3|_\ast^{-2}
					|\mu'_3|^{-3}.
	\end{align*}
	The Young inequality implies
	\begin{align*}
		\int_\RealNum
			d\sigma'_3\,
			|\tau^s\mu_1+\mu'_3|_\ast^{-2}
			|\mu'_3|_\ast^{-3}
		&\leq
			\int_\RealNum
				d\sigma'_3\,
				\left(
					|\tau^s\mu_1+\mu'_3|_\ast^{-5}
					+
					|\mu'_3|_\ast^{-5}
				\right)\\
		&=
			|\tau k_1+k'_3|_\ast^{-3}
			+
			|k'_3|_\ast^{-3}.
	\end{align*}
	Hence
	\begin{multline*}
		\int_E
			d\mu'_3\,
			\indicator{|\tau k_1+k'_3|\approx|k'_3|}
			|\tau k_1+k'_3|_\ast^{-1+\kappa}
			|\tau^s\mu_1+\mu'_3|_\ast^{-2}
			|\mu'_3|^{-3}\\
		\lesssim
			\sum_{k'_3}
				\indicator{|\tau k_1+k'_3|\approx|k'_3|}
				|\tau k_1+k'_3|_\ast^{-1+\kappa}
				\left(
					|\tau k_1+k'_3|_\ast^{-3}
					+
					|k'_3|_\ast^{-3}
				\right)\\
		\lesssim
			\sum_{k'_3}
				\indicator{|\tau k_1+k'_3|\approx|k'_3|}
				|k'_3|_\ast^{-4+\kappa}
		\lesssim
			\sum_{k'_3:|\tau k_1|\lesssim |k'_3|}
				|k'_3|_\ast^{-4+\kappa}
		\lesssim
			|\tau k_1|_\ast^{-1+\kappa}.
	\end{multline*}
	Here, we used that
	$|\tau k_1|\leq |\tau k_1+k'_3|+|k'_3|\lesssim |k'_3|$
	in the case that $|\tau k_1+k'_3|\approx|k'_3|$.
	Combining them and using that $|\tau^s\mu|_*\ge\tau|\mu|_*$ for every $\tau\in[0,1]$, we obtain
	\begin{align*}
		A_1
		\lesssim
			|k_1|
			\int_0^1
				d\tau
				(\tau |k_1|_\ast)^{-1+\kappa}
		\lesssim
			|k_1|_\ast^\kappa
		\leq
			|\mu_1|_\ast^\kappa.
	\end{align*}
	The estimate of $A_1$ has finished.

	The estimates of $A_2$ and $A_3$ is easy.
	Indeed, we have
	\begin{align*}
		A_2
		&\lesssim
			|\sigma_1|
			\int_0^1
				\tau
				|\tau^s\mu_1|_*^{-2}\,
				d\tau,
		&
		A_3
		&\lesssim
			|k_1|
			\int_0^1
				|\tau^s\mu_1|_*^{-1}\,
				d\tau.
	\end{align*}
	Since $|\tau^s\mu|_*\ge\tau|\mu|_*$ for every $\tau\in[0,1]$, we have
	\begin{gather*}
		A_2
		\lesssim
			|\sigma_1|
			\int_0^1
				\tau|\tau^s\mu_1|_*^{\kappa-2}\,
				d\tau
		\lesssim
			|\sigma_1|
			|\mu_1|_*^{\kappa-2}
			\int_0^1
				\tau^{\kappa-1}\,
				d\tau
		\lesssim
			|\mu_1|_*^\kappa,\\
		A_3
		\lesssim
			|k_1|
			\int_0^1
				|\tau^s\mu_1|_*^{\kappa-1}\,
				d\tau
		\lesssim
			|k_1|
			|\mu_1|_*^{\kappa-1}
			\int_0^1
				\tau^{\kappa-1}\,
				d\tau
		\lesssim
			|\mu_1|_*^\kappa
	\end{gather*}
	for every $\kappa\in(0,1)$.

	Hence
	$$
	|\mathfrak{R}Q_0^\tau(\mu)|^2\lesssim|\mu|_*^{-4+\kappa}.
	$$
	We obtain the estimate of $|(\mathfrak{R}Q_0^\tau-\mathfrak{R}Q_0^{\epsilon,\tau})(\mu)|^2$ in a similar way.
	We can see the assertion is valid for $\tau=\IBCCoDD$ in the same way.
\end{proof}

\begin{proof}[Proof of \tref{thm_20161128051824} for $\IAABoA$, $\IAABoB$, $\IAAoAB$, $\IAAoBB$, $\IABoAB$, $\IABoBB$, $\IAABoAB$, $\IAABoBB$]
	We will use \pref{section6:from Q to X}.
	The constant $(\gamma_\tau,\delta_\tau)$ in \pref{prop_1481179894} satisfies
	\begin{align}\label{eq_1481534569}
		\gamma_\tau+\delta_\tau
		=
			\begin{cases}
				\frac{5}{2},	&\tau=\IAABoA,\IAABoB,\\
				\frac{5}{2},	&\tau=\IAAoAB,\IAAoBB,\IABoAB,\IABoBB,\\
				2,	&\tau=\IAABoAB,\IAABoBB.
			\end{cases}
	\end{align}
	The assertions for the case $\IAABoA$ and $\IAABoB$ follow from \eqref{eq_1481534569} and \pref{prop_1481179916}.
	For $\IAAoAB$, $\IAAoBB$, $\IABoAB$ and $\IABoBB$, we see the assertion from \eqref{eq_1481534569} and \pref{prop_1481179943}.
	For $\IAABoAB$ and $\IAABoBB$, we use \eqref{eq_1481534569}, \pref[prop_1481179962]{prop_1481534768}.
\end{proof}

\begin{remark}\label{0550 X is invariant under diff approx}
We can construct another sequence $\{\tilde{X}^\epsilon\}$ of driving vectors from the space-time smeared noise
$\tilde{\whiteNoise}^\epsilon$ defined by \eqref{eq_1503279335}.
As stated in \rref{0550 tx regularize is ok},
the limit driving vector does not change.
In order to show this fact, for simplicity, we consider the case that temporal and spatial variables are separated: $\varrho^\epsilon(t,x)=\varrho_0^\epsilon(t)\varrho_1^\epsilon(x)$.
Here, $\varrho_0^\epsilon(t)=\epsilon^{-2}\varrho_0(\epsilon^{-2}t)$ and $\varrho_1^\epsilon(x)=\epsilon^{-3}\varrho_1(\epsilon^{-1}x)$
for even functions $\varrho_0$ and $\varrho_1$.
Since the noise is smeared in time, the solution $\tilde{Z}^\epsilon$ of
$\partial_t\tilde{Z}^\epsilon=\{(\ImUnit+\mu)\LaplaceOp-1\} \tilde{Z}^\epsilon+\tilde{\xi}^\epsilon$
is given by the same formula as \eqref{0550 Z is OU with xi^epsilon}, with $\hat{\xi}_s(k)$ replaced by the convolution $\int\hat{\xi}_u(k)\varrho_0^\epsilon(s-u)\,du$. By shifting the mollifier to the heat kernel, we have the formula
$$
\tilde{Z}^\epsilon_t
	=
		\sum_{k\in\Integers^3}
			\int_{-\infty}^\infty
				\tilde{H}_t^\epsilon(s,k)
				\FourierBase[k]
				\FourierCoeff{\whiteNoise_s}(k)\,
				ds,
$$
where
$
	\tilde{H}_t^\epsilon(s,k)
	=
		\chi^\epsilon(k)
		\int
			H_t(u,k)
			\varrho^\epsilon_0(s-u)\,
			du
$
and
$
	\chi=\FourierTrans\varrho_1
$.
Then the corresponding Fourier transform $\tilde{Q}_t^\epsilon=\FTTime\tilde{H}_t^\epsilon$ is given by
$$
\tilde{Q}_t^\epsilon(\sigma,k)
=
	\varphi_0(\epsilon^2 \sigma)
	\chi^\epsilon(k)
	Q_t(\sigma,k),
$$
where $\varphi_0=\FTTime\varrho_0$.
This implies $Q_0-\tilde{Q}_0^\epsilon$ has the good estimate as in \pref{prop_1481179726000},
so that $\tilde{Z}^\epsilon$ converges to the same limit $Z$ as that of $Z^\epsilon$.
In the proof, we replace $\chi^\epsilon(k)$ in \pref{prop_1481179726000} by $\varphi_0(\epsilon^2 \sigma)\chi^\epsilon(k)$.
By similar arguments, we can see the invariance of the limit for all other elements of $\tilde{X}^\epsilon$.
Since $\tilde{Z}^\epsilon$ is stationary in time, we can define the new renormalization constants
$\tilde{\mathfrak{c}}^\epsilon_1$, $\tilde{\mathfrak{c}}^\epsilon_{2,1}$ and $\tilde{\mathfrak{c}}^\epsilon_{2,2}$
in the same way as \eqref{eq_20161015073329} and see the new constants depend only on $\epsilon$.
However, they may not coincide with $\constRe[1]{\epsilon}$, $\constRe[2,1]{\epsilon}$ and $\constRe[2,2]{\epsilon}$.
\end{remark}

\subsection{Properties of renormalization constants}
In this subsection, we study the renormalization constants
$\constRe[1]{\epsilon}$, $\constRe[2,1]{\epsilon}$ and $\constRe[2,2]{\epsilon}$
defined by \eqref{eq_20161015073329}.
We use \pref[prop_20161205074106]{section6:expansion of tau1 rs tau2}
to show that they are independent of $(t,x)$
and the choice of the dyadic partition $\{\DyaPartOfUnit[m]\}_{m=-1}^\infty$ of unity (\pref{prop_20161107071120})
and obtain the divergence rate (\pref{prop_20161107072011}).
\pref{prop_20161107071120} and \tref{thm_20160920094302} imply the renormalized equation \eqref{eq:rcgl}
does not depend on the choice of the dyadic partition of unity.
Hence the solution to \eqref{eq:rcgl} is independent of the partition.

\subsubsection{Expression of renormalization constants}
We obtain explicit expressions of $\constRe[1]{\epsilon}$, $\constRe[2,1]{\epsilon}$ and $\constRe[2,2]{\epsilon}$ as follows:
\begin{proposition}\label{prop_20161107071120}
	We have the following:
	\begin{align}
		\label{eq_20161107064910}
		\constRe[1]{\epsilon}
		&=
			\sum_{k\in\Integers^3}
				\frac{\chi^\epsilon(k)^2}{2(4\pi^2|k|^2+1)},\\
		\label{eq_20161015075934}
		\constRe[2,1]{\epsilon}
		&=
			\sum_{k_1,k_2\in\Integers^3}
				\frac{\chi^{\epsilon}(k_1,k_2)^2}{4(4\pi^2\mu|k_1|^2+1)(4\pi^2\mu|k_2|^2+1)(\alpha_1+\ImUnit\beta_1)},\\
		\label{eq_20161015075943}
		\constRe[2,2]{\epsilon}
		&=
			\sum_{k_1,l_1\in\Integers^3}
				\frac{\chi^{\epsilon}(k_1,l_1)^2}{4(4\pi^2\mu|k_1|^2+1)(4\pi^2\mu|l_1|^2+1)(\alpha_2+\ImUnit\beta_2)},
	\end{align}
	where
	\begin{align*}
		\alpha_1
		&=
			4\pi^2\mu
			(
				|k_1+k_2|^2
				+
				|k_1|^2
				+
				|k_2|^2
			)
			+
			3,
		&
		\beta_1
		&=
			4\pi^2
			(
				|k_1+k_2|^2
				-
				|k_1|^2
				-
				|k_2|^2
			),\\
		\alpha_2
		&=
			4\pi^2\mu
			(
				|k_1-l_1|^2
				+
				|k_1|^2
				+
				|l_1|^2
			)
			+
			3,
		&
		\beta_2
		&=
			4\pi^2
			(
				|k_1-l_1|^2
				-
				|k_1|^2
				 +
				|l_1|^2
			).
	\end{align*}
\end{proposition}
\begin{proof}
	From \pref{prop_20161205074106}, we have
	\begin{multline*}
		\constRe[1]{\epsilon}
		=
			\sum_{k\in\Integers^3}
				\int_\RealNum
					f^{\epsilon,\A}_{(t,x)}(s,k)
					f^{\epsilon,\B}_{(t,x)}(s,k)\,
					ds\\
		=
			\sum_{k\in\Integers^3}
				\int_\RealNum
					\chi^\epsilon(k)^2
					|H_t(s,k)|^2\,
					ds
		=
			\sum_{k\in\Integers^3}
				\frac{\chi^\epsilon(k)^2}{2(4\pi^2|k|^2+1)},
	\end{multline*}
	which is \eqref{eq_20161107064910}.

	We show \eqref{eq_20161015075934}.
	From \pref[prop_20161205074106]{section6:expansion of tau1 rs tau2}, we have
	\begin{align*}
		\constRe[2,1]{\epsilon}
		&=
			H_t^{\epsilon,(\IAA,\BB,\{(1,1),(2,2)\})}\\
		&=
			\int_{E^2}
				\psi_\circ(k_1+k_2,-(k_1+k_2))
				H_t^{\epsilon,\IAA}(m_{1,2})
				H_t^{\epsilon,\BB}(m_{1,2})\,
				dm_1
				dm_2.
	\end{align*}
	By using $\psi_\circ(k_1+k_2,-(k_1+k_2))=1$ and
	\begin{gather*}
		\int_\RealNum
			ds\,
			H^\A_u(s,k)
			H^\B_t(s,k)
		=
			\frac{1}{2(4\pi^2\mu|k|^2+1)}
			H^\B_t(u,k),\\
		\int_\RealNum
			dt\,
			H^\B_u(t,l)
			H^\A_t(t,l)
		=
			\frac{1}{2(4\pi^2\mu|l|^2+1)}
			H^\A_t(u,l),
	\end{gather*}
	we obtain the assertion.
	We can show \eqref{eq_20161015075943} in the same way.
\end{proof}

\subsubsection{Divergence rate of renormalization constants}
Here, we show the following proposition concerning divergence rate of the the renormalization constants:
\begin{proposition}\label{prop_20161107072011}
	There exists a positive constant $\const$ such that
	\begin{gather}
		\label{eq_20161107071738}
		\const^{-1}\epsilon^{-1}
		\leq
			\constRe[1]{\epsilon}
		\leq
			\const\epsilon^{-1},\\
		\label{eq_20161107071756}
		\const^{-1}
		\log\epsilon^{-1}
		\leq
			|\constRe[2,1]{\epsilon}|
		\leq
			\const
			\log\epsilon^{-1},\\
		\label{eq_20161107071819}
		\const^{-1}
		\log\epsilon^{-1}
		\leq
			|\constRe[2,2]{\epsilon}|
		\leq
			\const
			\log\epsilon^{-1}
	\end{gather}
	for any $0<\epsilon<1$.
\end{proposition}

Since the estimate \eqref{eq_20161107071738} follows easily from \eqref{eq_20161107064910},
we show \eqref{eq_20161107071756} and \eqref{eq_20161107071819} for the rest of this subsection.
\begin{lemma}\label{lem_20161016033338}
	Let $\alpha_1$, $\beta_1$, $\alpha_2$ and $\beta_2$ be as in \pref{prop_20161107071120}.
	There exist positive constants $\const[1]$ and $\const[2]$ such that
	\begin{gather*}
		\begin{aligned}
			\Re \frac{1}{\alpha_1+\ImUnit\beta_1}
			&\geq
				\const[1]
				(|k_1|^2+|k_2|^2+1)^{-1},
			&
			\left|
				\frac{1}{\alpha_1+\ImUnit\beta_1}
			\right|
			\leq
				\const[2]
				(|k_1|^2+|k_2|^2+1)^{-1},
		\end{aligned}\\
		\begin{aligned}
			\Re \frac{1}{\alpha_2+\ImUnit\beta_2}
			&\geq
				\const[1]
				(|k_1|^2+|l_1|^2+1)^{-1},
			&
			\left|
				\frac{1}{\alpha_2+\ImUnit\beta_2}
			\right|
			\leq
				\const[2]
				(|k_1|^2+|l_1|^2+1)^{-1},
		\end{aligned}
	\end{gather*}
	for any $k_1$, $k_2$ and $l_1$.
\end{lemma}

\begin{proof}
	We show the assertion for $1/(\alpha_1+\ImUnit\beta_1)$.
	Since
	$
		|k_1|^2+|k_2|^2+1
		\lesssim
			\alpha_1
		\lesssim
			|k_1|^2+|k_2|^2+1
	$
	and
	$
		|\beta_1|
		\lesssim
			|k_1|^2+|k_2|^2+1
	$,
	we see
	\begin{align*}
		\Re\frac{1}{\alpha_1+\ImUnit\beta_1}
		=
			\Re\frac{\alpha_1-\ImUnit\beta_1}{\alpha_1^2+\beta_1^2}
		=
			\frac{\alpha_1}{\alpha_1^2+\beta_1^2}
		\gtrsim
			\frac{|k_1|^2+|k_2|^2+1}{(|k_1|^2+|k_2|^2+1)^2}
	\end{align*}
	and
	\begin{align*}
		\left|
			\frac{1}{\alpha_1+\ImUnit\beta_1}
		\right|
		=
			\left|
				\frac{\alpha_1-\ImUnit\beta_1}{\alpha_1^2+\beta_1^2}
			\right|
		\lesssim
			\frac{|k_1|^2+|k_2|^2+1}{(|k_1|^2+|k_2|^2+1)^2}.
	\end{align*}
	The assertion is verified.
	We can show the assertion for $1/(\alpha_2+\ImUnit\beta_2)$ in the same way.
\end{proof}

\begin{proof}[Proof of \pref{prop_20161107072011}]
	We first prove of the lower estimate.
	We show that there exists a constant $\const[1]$ such that
	\begin{align*}
		\Re \constRe[2,1]{\epsilon},
		\Re \constRe[2,2]{\epsilon}
		\geq
			\const[1] \log\epsilon^{-1},
	\end{align*}
	for any $0<\epsilon<1$.

	We consider only $\constRe[2,1]{\epsilon}$
	and estimate the real part of summands in \eqref{eq_20161015075934}.
	\pref{prop_20161107071120} and \lref{lem_20161016033338} imply
	\begin{align*}
		\Re \constRe[2,1]{\epsilon}
		&\gtrsim
			\sum_{k_1,k_2\in\Integers^3}
				\chi^\epsilon(k_1)^2
				\chi^\epsilon(k_2)^2
				(|k_1|^2+1)^{-1}
				(|k_2|^2+1)^{-1}
				(|k_1|^2+|k_2|^2+1)^{-1}\\
		&\geq
			\sum_{k_1,k_2\in\Integers^3}
				\chi^\epsilon(k_1)^2
				\chi^\epsilon(k_2)
				(|k_1|^2+|k_2|^2+1)^{-3}\\
		&\geq
			1+\log\epsilon^{-1},
	\end{align*}
	which implies the lower estimate of $\Re \constRe[2,1]{\epsilon}$.
	We can obtain that of $\Re \constRe[2,2]{\epsilon}$ by the same way.

	Next we prove the upper estimate.
	We show that there exists a constant $\const[1]$ such that
	\begin{align*}
		|\constRe[2,1]{\epsilon}|,
		|\constRe[2,2]{\epsilon}|
		\leq
			\const[2] \log\epsilon^{-1}
	\end{align*}
	for any $0<\epsilon<1$.

	We consider only $\constRe[2,1]{\epsilon}$.
	\pref{prop_20161107071120} and \lref{lem_20161016033338} imply
	\begin{align*}
		|\constRe[2,1]{\epsilon}|
		&\lesssim
			\sum_{k_1,k_2\in\Integers^3}
				\chi^\epsilon(k_1)^2
				\chi^\epsilon(k_2)^2
				(|k_1|^2+1)^{-1}
				(|k_2|^2+1)^{-1}
				(|k_1|^2+|k_2|^2+1)^{-1}\\
		&\leq
			\sum_{|k_1|,|k_2|\leq\epsilon^{-1}}
				(|k_1|^2+1)^{-1}
				(|k_2|^2+1)^{-1}
				(|k_1|^2+|k_2|^2+1)^{-1}.
	\end{align*}
	In this estimate, we used $\supp\chi\subset B(0,1)$.
	We divide the region of the summation
	$
		\{
			(k_1,k_2);
			\text{$|k_1|\leq\epsilon^{-1}$ and $|k_2|\leq\epsilon^{-1}$}
		\}
		\subset
			\Integers^3\times\Integers^3
	$
	into
	\begin{align*}
		A_1
		&=
			\{(k_1,k_2);\text{$|k_1|\leq 2$ or $|k_2|\leq 2$}\},\\
		A_2
		&=
			\{(k_1,k_2);2\leq|k_1|\leq|k_2|\leq\epsilon^{-1}\},\\
		A_3
		&=
			\{(k_1,k_2);2\leq|k_2|\leq|k_1|\leq\epsilon^{-1}\}.
	\end{align*}
	The summation over $A_1$ is estimated as follows:
	\begin{align*}
		\sum_{(k_1,k_2)\in A_1}
			(|k_1|^2+1)^{-1}
			(|k_2|^2+1)^{-1}
			(|k_1|^2+|k_2|^2+1)^{-1}
		&\lesssim
			\sum_{k\in\Integers^3}
			(|k|^2+1)^{-2}\\
		&<
			\infty.
	\end{align*}
	For the summation over $A_2$, we have
	\begin{multline*}
		\sum_{(k_1,k_2)\in A_2}
			(|k_1|^2+1)^{-1}
			(|k_2|^2+1)^{-1}
			(|k_1|^2+|k_2|^2+1)^{-1}\\
		\begin{aligned}
			&\leq
				\sum_{(k_1,k_2)\in A_2}
					|k_1|^{-2}
					|k_2|^{-2}
					|k_2|^{-2}\\
			&\leq
				\sum_{k_2:2\leq|k_2|\leq\epsilon^{-1}}
					\left(
						\sum_{k_1:2\leq|k_1|\leq|k_2|}
							|k_1|^{-2}
					\right)
					|k_2|^{-4}\\
			&\lesssim
				\sum_{2\leq|k_2|\leq\epsilon^{-1}}
					|k_2|
					|k_2|^{-4}\\
			&\lesssim
				\log\epsilon^{-1}.
		\end{aligned}
	\end{multline*}
	The summation over $A_3$ has the same upper bound.
	Hence we see $|\constRe[2,1]{\epsilon}|\lesssim\log\epsilon^{-1}$.
	We can prove $|\constRe[2,2]{\epsilon}|\lesssim\log\epsilon^{-1}$ by the same way.
	The proof is completed.
\end{proof}

%% file: 0700_ItoWienerInt.tex

\section{Complex multiple It\^o-Wiener integral}\label{sec:itowienerintegral}

We recall some notations and properties of complex multiple Wiener integrals from \cite{Ito1952}.

A complex random variable $Z$ is called isotropic complex normal
if $\Re Z$ and $\Im Z$ are independent, has the same law with mean $0$ and $(\Re Z,\Im Z)$ is jointly normal.
A system of complex random variables $\{Z_\lambda\}$ is called jointly isotropic complex normal
if $\sum_{i=1}^nc_iZ_{\lambda_i}$ is isotropic complex normal
for any $n$, any $c_1,\dots,c_n\in\CmplNum$, and any indices $\lambda_1,\dots,\lambda_n$.
Note that the isotropic complex normal $\{Z_\lambda\}$ satisfies $\expect[Z_\lambda Z_\mu]=0=\expect[\CmplConj{Z_\mu}\CmplConj{Z_\nu}]$.
The distribution of jointly isotropic complex normal system $\{Z_\lambda\}$
is uniquely determined by the positive-definite matrix
$V_{\lambda\mu}=\expect[Z_\lambda\overline{Z_\mu}]$ (\cite[Theorem~2.3]{Ito1952}).

Let $(E,\mathcal{B},\mathfrak{m})$ be a $\sigma$-finite, atomless measure space,
and $\mathcal{B}^\ast$ be the set of all elements $A\in\mathcal{B}$ such that $\mathfrak{m}(A)<\infty$.
Then there exists a jointly isotropic complex normal system $\{M(A)\,;A\in\mathcal{B}^*\}$
defined on a probability space $(\probSp,\sigmaField,\prob)$ such that
\begin{align*}
	\expect[M(A)\CmplConj{M(B)}]=\mathfrak{m}(A\cap B),
\end{align*}
and its distribution is uniquely determined (\cite[Theorem~3.1]{Ito1952}).

Now the complex multiple It\^o-Wiener integral of $f\in L_{p,q}^2=L^2(E^p\times E^q)$ is defined as follows.
Let $S_{p,q}$ be the set of $L_{p,q}^2$ functions of the form
\begin{align*}
	f
	=
		\sum_{i_1,\dots,i_p,j_1,\dots,j_q=1}^n
			a_{i_1\dots i_pj_1\dots j_q}
			\indicator{E_{i_1}\times\dots\times E_{i_p}\times E_{j_1}\times\dots\times E_{j_q}}
\end{align*}
with $n\in\Integers_+$, where $E_1,\dots,E_n$ are any disjoint sets of $\mathcal{B}^*$
and $\{a_{i_1\dots i_pj_1\dots j_q}\}$ is a set of complex numbers such that $=0$
unless $i_1,\dots,i_p,j_1,\dots,j_q$ are all different.
For $f$ of this form, we define
\begin{align*}
	\WienerIntCmpl{p}{q}(f)
	=
		\sum_{i_1,\dots,i_p,j_1,\dots,j_q=1}^n
			a_{i_1\dots i_pj_1\dots j_q}
			M(E_{i_1})
			\cdots
			M(E_{i_p})
			\CmplConj{M(E_{j_1})}
			\cdots
			\CmplConj{M(E_{j_q})}.
\end{align*}
This functional has the property
$
	\expect[|\WienerIntCmpl{p}{q}(f)|^2]
	\leq
		p!q!\|f\|_{L_{p,q}^2}^2
$.
We defined the It\^o-Wiener integral for non-symmetric functions,
hence, we cannot expect the equality in this inequality.
Since $S_{p,q}$ is dense in $L_{p,q}^2$,
the integral $\WienerIntCmpl{p}{q}$ is uniquely extended into continuous linear map from $L_{p,q}^2$ to $L^2(\prob)$.
We set $L_{0,0}^2=\CmplNum$ and $\WienerIntCmpl{0}{0}(c)=c$.
From \cite[Theorem~7]{Ito1952}, we have
\begin{gather*}
	\expect[|\WienerIntCmpl{p}{q}(f)|^2]
	\leq p!q!\|f\|_{L_{p,q}^2},\\
	\expect[\WienerIntCmpl{p}{q}(f)\CmplConj{\WienerIntCmpl{r}{s}(g)}]
	=
		0\quad \text{for $(p,q)\neq(r,s)$}.
\end{gather*}

The product formula is important.
For $0\leq r_1\leq p_1\wedge q_2$ and $0\leq r_1\leq q_1\wedge p_2$,
we denote by $\mathcal{F}(p_1,q_1;p_2,q_2;r_1,r_2)$ the set of graphs consisting of disjoint $r_1+r_2$ edges
\begin{align*}
	(p',q')\in\{1,\dots,p_1\}\times\{q_1+1,\dots,q_1+q_2\}\cup\{p_1+1,\dots,p_1+p_2\}\times\{1,\dots,q_1\}.
\end{align*}
For $(f,g)\in L_{p_1,q_1}^2\times L_{p_2,q_2}^2$ and $\gamma\in\mathcal{F}(p_1,q_1;p_1,q_2;r_1,r_2)$,
we define $f\otimes_\gamma g\in L_{p_1+p_2-(r_1+r_2),q_1+q_2-(r_1+r_2)}^2$ by
\begin{multline*}
	(f\otimes_\gamma g)
		(t_1,\dots, t_{p_1+p_2},s_1,\dots, s_{q_1+q_2}\setminus\{(t_{p'},s_{q'})\}_{(p',q')\in\gamma})\\
	=
		\int_{E^{r_1+r_2}}
			h(\{(t_{p'},s_{q'})\}_{(p',q')\in\gamma})\,
			\prod_{(p',q')\in\gamma}dm(t_{p'},s_{q'}),
\end{multline*}
where $h:E^{r_1+r_2}\to\CmplNum$ is defined by
\begin{multline*}
	h(\{(t_{p'},s_{q'})\}_{(p',q')\in\gamma})
	=
		f(t_1,\dots,t_{p_1},s_1,\dots,s_{q_1})\\
		\times
		g(t_{p_1+1},\dots,t_{p_1+p_2},s_{q_1+1},\dots,s_{q_1+q_2})|_{t_{p'}=s_{q'}, (p',q')\in\gamma}.
\end{multline*}

\begin{theorem}[{\cite[Theorem~9]{Ito1952}}, {\cite[Proposition~1.1.2]{Nualart2006}}]\label{2:product}
	For every $f\in L_{p_1,q_1}^2$ and $g\in L_{p_2,q_2}^2$, we have
	\begin{multline*}
		\WienerIntCmpl{p_1}{q_1}(f)
		\WienerIntCmpl{p_2}{q_2}(g)\\
		=
			\sum_{r_1=0}^{p_1\wedge q_2}
			\sum_{r_2=0}^{p_2\wedge q_1}
			\sum_{\gamma\in\mathcal{F}(p_1,q_1;p_2,q_2;r_1,r_2)}
				\WienerIntCmpl{p_1+p_2-(r_1+r_2)}{q_1+q_2-(r_1+r_2)}(f\otimes_\gamma g).
	\end{multline*}
\end{theorem}
For example, we have
\begin{align*}
	\WienerIntCmpl{2}{1}(f)
	\WienerIntCmpl{0}{2}(g)
	&=
		\WienerIntCmpl{2}{3}(f\otimes_{\emptyset}g)\\
	&\phantom{=}\quad
		+\WienerIntCmpl{1}{2}(f\otimes_{\{(1,1)\}}g)
		+\WienerIntCmpl{1}{2}(f\otimes_{\{(1,2)\}}g)\\
	&\phantom{=}\quad
		+\WienerIntCmpl{1}{2}(f\otimes_{\{(2,2)\}}g)
		+\WienerIntCmpl{1}{2}(f\otimes_{\{(2,3)\}}g)\\
	&\phantom{=}\quad
		+\WienerIntCmpl{0}{1}(f\otimes_{\{(1,2),(2,3)\}}g)
		+\WienerIntCmpl{0}{1}(f\otimes_{\{(1,3),(2,2)\}}g).
\end{align*}

%% file: ComplexGinzburgLandauEq.bbl
\begin{thebibliography}{99}

\bibitem[AK02]{AransonKramer2002}
Igor~S. Aranson and Lorenz Kramer.
\newblock The world of the complex ginzburg-landau equation.
\newblock {\em Rev. Mod. Phys.}, 74:99--143, 2002.

\bibitem[BBF15]{BailleulBernicotFrey2015Arvix}
Isma\"el Bailleul, Fr\'ed\'eric Bernicot, and Dorothee Frey.
\newblock Higher order paracontrolled calculus, 3d-{PAM} and multiplicative {B}urgers equations, 2015, arXiv:1506.08773.

\bibitem[BCD11]{BahouriCheminDanchin2011}
Hajer Bahouri, Jean-Yves Chemin, and Rapha{\"e}l Danchin.
\newblock {\em Fourier analysis and nonlinear partial differential equations}, volume 343 of {\em Grundlehren der Mathematischen Wissenschaften [Fundamental Principles of Mathematical Sciences]}.
\newblock Springer, Heidelberg, 2011.

\bibitem[BK16]{BerglundKuehn2016}
Nils Berglund and Christian Kuehn.
\newblock Regularity structures and renormalisation of {F}itz{H}ugh-{N}agumo {SPDE}s in three space dimensions.
\newblock {\em Electron. J. Probab.}, 21:Paper No. 18, 1--48, 2016.

\bibitem[BS04a]{Barton-Smith2004b}
Marc Barton-Smith.
\newblock Global solution for a stochastic {G}inzburg-{L}andau equation with multiplicative noise.
\newblock {\em Stochastic Anal. Appl.}, 22(1):1--18, 2004.

\bibitem[BS04b]{Barton-Smith2004a}
Marc Barton-Smith.
\newblock Invariant measure for the stochastic {G}inzburg {L}andau equation.
\newblock {\em NoDEA Nonlinear Differential Equations Appl.}, 11(1):29--52, 2004.

\bibitem[CC13]{CatellierChouk2013Arxiv}
R\'emi Catellier and Khalil Chouk.
\newblock Paracontrolled distributions and the 3-dimensional stochastic quantization equation.
\newblock 2013, arXiv:1310.6869.

\bibitem[FH14]{FrizHairer2014}
Peter~K. Friz and Martin Hairer.
\newblock {\em A course on rough paths}.
\newblock Universitext. Springer, Cham, 2014.

\bibitem[FH17]{FunakiHoshino2016Arxiv}
Tadahisa Funaki and Masato Hoshino.
\newblock A coupled {KPZ} equation, its two types of approximations and existence of global solutions.
\newblock {\em J. Funct. Anal.}, 273:1165--1204, 2017.

\bibitem[GIP15]{GubinelliImkellerPerkowski2015}
Massimiliano Gubinelli, Peter Imkeller, and Nicolas Perkowski.
\newblock Paracontrolled distributions and singular {PDE}s.
\newblock {\em Forum Math. Pi}, 3:e6, 75, 2015.

\bibitem[GP17]{GubinelliPerkowski2017}
Massimiliano Gubinelli and Nicolas Perkowski.
\newblock K{PZ} {R}eloaded.
\newblock {\em Comm. Math. Phys.}, 349(1):165--269, 2017.

\bibitem[Hai02]{Hairer2002}
Martin Hairer.
\newblock Exponential mixing properties of stochastic {PDE}s through asymptotic coupling.
\newblock {\em Probab. Theory Related Fields}, 124(3):345--380, 2002.

\bibitem[Hai14]{Hairer2014a}
Martin Hairer.
\newblock A theory of regularity structures.
\newblock {\em Invent. Math.}, 198(2):269--504, 2014.

\bibitem[Hos16]{Hoshino2016}
Masato Hoshino.
\newblock K{PZ} equation with fractional derivatives of white noise.
\newblock {\em Stoch. Partial Differ. Equ. Anal. Comput.}, 4(4):827--890, 2016.

\bibitem[Hos17a]{Hoshino2016arXiv}
Masato Hoshino.
\newblock Paracontrolled calculus and {F}unaki-{Q}uastel approximation for the {KPZ} equation.
\newblock {\em Stochastic Process. Appl.}, in press.

\bibitem[Hos17b]{Hoshino2017b}
Masato Hoshino.
\newblock Global well-posedness of complex Ginzburg-Landau equation with a space-time white noise.
\newblock {\em Ann. Inst. Henri Poincaré Probab. Stat.}, in press.

\bibitem[It{\^o}52]{Ito1952}
Kiyosi It{\^o}.
\newblock Complex multiple {W}iener integral.
\newblock {\em Jap. J. Math.}, 22:63--86, 1952.

\bibitem[KS04]{KuksinShirikyan2004}
Sergei Kuksin and Armen Shirikyan.
\newblock Randomly forced {CGL} equation: stationary measures and the inviscid limit.
\newblock {\em J. Phys. A}, 37(12):3805--3822, 2004.

\bibitem[Kun90]{Kunita1990}
Hiroshi Kunita.
\newblock {\em Stochastic flows and stochastic differential equations},
  volume~24 of {\em Cambridge Studies in Advanced Mathematics}.
\newblock Cambridge University Press, Cambridge, 1990.

\bibitem[Kup16]{Kupiainen2016}
Antti Kupiainen.
\newblock Renormalization group and stochastic {PDE}s.
\newblock {\em Ann. Henri Poincar\'e}, 17(3):497--535, 2016.

\bibitem[MW16]{MourratWeber2016Arxiv}
Jean-Christophe Mourrat and Hendrik Weber.
\newblock Global well-posedness of the dynamic $\phi^4_3$ model on the torus.
\newblock 2016, arXiv:1601.01234.

\bibitem[Nua06]{Nualart2006}
David Nualart.
\newblock {\em The {M}alliavin calculus and related topics}.
\newblock Probability and its Applications (New York). Springer-Verlag, Berlin, second edition, 2006.

\bibitem[Oda06]{Odasso2006}
Cyril Odasso.
\newblock Ergodicity for the stochastic complex {G}inzburg-{L}andau equations.
\newblock {\em Ann. Inst. H. Poincar\'e Probab. Statist.}, 42(4):417--454, 2006.

\bibitem[PG11]{PuGuo2011}
Xueke Pu and Boling Guo.
\newblock Momentum estimates and ergodicity for the 3{D} stochastic cubic {G}inzburg-{L}andau equation with degenerate noise.
\newblock {\em J. Differential Equations}, 251(7):1747--1777, 2011.

\bibitem[Yan04]{Yang2004}
Desheng Yang.
\newblock The asymptotic behavior of the stochastic {G}inzburg-{L}andau equation with multiplicative noise.
\newblock {\em J. Math. Phys.}, 45(11):4064--4076, 2004.

\bibitem[ZZ15]{ZhuZhu2015}
Rongchan Zhu and Xiangchan Zhu.
\newblock Three-dimensional {N}avier-{S}tokes equations driven by space-time white noise.
\newblock {\em J. Differential Equations}, 259(9):4443--4508, 2015.

\end{thebibliography}
